\documentclass[letterpaper, 11pt]{amsart}
\usepackage{amssymb,amsmath,amsfonts,amsthm}
\usepackage{graphicx}
\usepackage[normalem]{ulem}
\graphicspath{ {./images/} }
\usepackage{setspace}

\usepackage{mathrsfs}  
\usepackage{psfrag}
\usepackage{mathtools}
\usepackage{color}
\usepackage{todonotes}
\usepackage{enumitem}
\definecolor{Red}{rgb}{1,0,0}
\definecolor{Blue}{rgb}{0,0,1}
\definecolor{Green}{rgb}{0,1,0}
\definecolor{magenta}{rgb}{1,0,.6}
\definecolor{gold}{rgb}{.6,.5,0}
\definecolor{orange}{rgb}{1,0.4,0}
\definecolor{darkgreen1}{rgb}{0, .35, 0}
\definecolor{darkgreen}{rgb}{0, .6, 0}
\definecolor{darkred}{rgb}{.75,0,0}

\usepackage{hyperref}
\hypersetup{colorlinks=true, citecolor=darkgreen, linkcolor=blue, hypertexnames=false}

\usepackage{chngcntr}
\counterwithin{equation}{section}

\DeclarePairedDelimiter{\floor}{\lfloor}{\rfloor}
\DeclarePairedDelimiter{\roof}{\lceil}{\rceil}
\theoremstyle{plain}
\newtheorem{main}{Theorem}

\newtheorem{theorem}{Theorem}[section]
\newtheorem{lemma}[theorem]{Lemma}
\newtheorem{proposition}[theorem]{Proposition}

\newtheorem{conjecture}{Conjecture}

\theoremstyle{remark}
\newtheorem{remark}[theorem]{Remark}
\newtheorem{definition}[theorem]{Definition}

\newcommand\numberthis{\addtocounter{equation}{1}\tag{\theequation}}

\newcommand{\Sing}{\operatorname{Sing}}

           \def\ea{\end{array}}
          \def\ec{\end{center}}
     \def\ed{\end{description}}
        \def\ee{\end{equation}}
       \def\eea{\end{eqnarray}}
     \def\eeaa{\end{eqnarray*}}
 \def\et{\end{thebibliography}}
\def\bib{\bibitem}

\def\bM{{\bf{M}}}

\def\Orb{{\rm Orb}}

\def\Cl{{\rm Cl}}

\def\Reg{{\rm Reg}}
\def\Sing{{\rm Sing}}

\def\CR{{\rm CR}}

\newcommand{\interior}[1]{%
	{\kern0pt#1}^{\mathrm{o}}%
}

\def\supp{\operatorname{supp}}

\def\cG{{\mathcal G}}

\def\cD{{\mathcal D}}
\def\cC{{\mathcal C}}
\def\cO{{\mathcal O}}

\def\cU{{\mathcal U}}
\def\cV{{\mathcal V}}
\def\cR{{\mathcal R}}
\def\cB{{\mathcal B}}

\def\cF{{\mathcal F}}
\def\cM{{\mathcal M}}
\def\cN{{\mathcal N}}
\def\cP{{\mathcal P}}

\def\cR{{\mathcal R}}

\def\cS{{\mathcal S}}

\def\loc{\operatorname{loc}}

\def\vep{\varepsilon}

\def\RR{{\mathbb R}}

\def\ZZ{{\mathbb Z}}
\def\NN{{\mathbb N}}

\def\Var{\operatorname{Var}}
\def\Exp{\operatorname{Exp}}

\def\dxt{\delta_{(x,t)}}

	\title[An Ergodic Spectral Decomposition Theorem]{An Ergodic Spectral Decomposition Theorem for Singular Star Flows}
\date{\today}
\author{Maria Jose Pacifico, Fan Yang and Jiagang Yang}

\address{Instituto de Matem\'atica, Universidade Federal do Rio de Janeiro, C. P. 68.530, CEP 21.945-970,  Rio de Janeiro, RJ, Brazil.}
 \email{pacifico@im.ufrj.br }

\address{Department of Mathematics, Wake Forest University, Winston-Salem, NC, USA.}
\email{yangf@wfu.edu}

\address{Departamento de Geometria, Instituto de Matem\'atica e Estat\'\i stica, Universidade Federal Fluminense, Niter\'oi, Brazil}
\email{yangjg@impa.br}

\thanks{Pacifico’s work was partially supported by CAPES-Finance Code 001, CNPq Projeto Universal No. 404943/2023-3, CNPq-Brazil grant 307776/2019-0 and by
	Foundation for Research Support of the State of Rio de Janeiro (FAPERJ) grant CNE
	E-26/202.850/2018(239069). J.\ Yang’s work was partially supported by CAPES Finance Code 001, CNPq-Brazil grant 312054/2023-8, CNPq-Projeto Universal No.
	404943/2023-3, PRONEX, and
	MATH-AmSud 220029. F.\ Yang’s work was partially supported by National Science
	Foundation (NSF) grant DMS-2418590.}

\begin{document}

\begin{abstract} 
For Axiom A diffeomorphisms and flows, Smale’s Spectral Decomposition Theorem asserts that the non-wandering set decomposes into finitely many isolated hyperbolic basic sets, each given by a homoclinic class. For singular star flows, which may be viewed as “Axiom A flows with singularities”, the corresponding spectral decomposition remains open and is known as the Spectral Decomposition Conjecture.

We provide a positive answer to an ergodic formulation of this conjecture: {$C^1$-open and dense among} singular star flows with positive topological entropy, there is a unique measure of maximal entropy. More generally, we prove the uniqueness of equilibrium states for Hölder continuous potentials under a mild and  {natural} pressure gap condition. We further establish that  {$C^1$-open and dense} star flows are almost expansive and that the topological pressure of continuous potentials varies continuously with respect to the vector field in the $C^1$ topology.

Our approach combines ergodic and geometric arguments adapted to the multi-singular setting.
In this context, classical hyperbolic tools such as uniform local product structure or invariant splittings on the tangent bundle are no longer available.
To overcome this, we develop new mechanisms to control the geometry of orbit segments and to produce transversal intersections on large subsets uniformly detected by good invariant measures.
These ingredients allow us to extend classical arguments to the multi-singular setting through structural properties of equilibrium states combined with refined shadowing and specification at the level of invariant measures.
\hspace{0.5cm}
\end{abstract}

%%%%%%%%%%%%AQUI ABAIXO ESTA O ORIGINAL ANTES DE MELHORAR

%\begin{abstract}
%For Axiom A diffeomorphisms and flows, Smale’s Spectral Decomposition Theorem asserts that the non-wandering set decomposes into finitely many isolated hyperbolic basic sets, each given by a homoclinic class. For singular star flows, which may be viewed as “Axiom A flows with singularities”, the corresponding spectral decomposition remains open and is known as the Spectral Decomposition Conjecture.
%
%We provide a positive answer to an ergodic formulation of this conjecture. $C^1$ open and densely, singular star flows with positive topological entropy admit only finitely many ergodic measures of maximal entropy. More generally, we prove finiteness of equilibrium states for Hölder continuous potentials under a mild and optimal pressure gap condition. We further establish that $C^1$ open and densely, star flows are almost expansive and that the topological pressure of continuous potentials varies continuously with respect to the vector field in the $C^1$ topology.
%
%Our approach is purely ergodic and does not rely on any topological decomposition. In the multi-singular setting, neither a spectral decomposition of the phase space nor a dominated splitting on the tangent bundle is available. The results are obtained instead from structural properties of equilibrium states combined with refined shadowing and specification arguments at the level of invariant measures.
%	\hspace{0.5cm}
%\end{abstract}

\maketitle

\tableofcontents

\section{Introduction}\label{s.intro}
Smale's Spectral Decomposition Theorem \cite{Smale} provides perhaps one of the most important descriptions of the topological structure of Axiom A systems, including both diffeomorphisms and non-singular flows (i.e., $|X(x)|\ne 0$ for all $x\in\bM$). It states that the non-wandering set decomposes into finitely many transitive, isolated pieces, each of which is the closure of homoclinic orbits of a hyperbolic periodic orbit. Each of such pieces is called a hyperbolic basic set. From there, studying the topological and statistical properties of Axiom A systems can be reduced to the corresponding analysis on each hyperbolic basic set, with finiteness allowing one to pass from local to global conclusions.

{It is important to emphasize, however, that classical spectral decomposition results, including those of Smale and their ergodic counterparts developed by Bowen, concern the structure of a single dynamical system. {In these settings, the conclusions rely on strong hyperbolicity assumptions for the given system, such as uniform hyperbolicity. In contrast, the framework considered in this paper replaces this by a significantly weaker assumption — namely, the hyperbolicity of critical elements — imposed robustly on nearby systems through the star property.} }

{Moreover, even at the level of a single system, flows with singularities exhibit fundamentally different features: they are typically not uniformly hyperbolic and are not structurally stable. {This lack of uniform hyperbolicity further highlights the need for new approaches beyond the classical theory.}}

{Thus, in order to pursue a “Spectral Decomposition Theorem” in this context, one must first identify an appropriate notion of hyperbolicity for singular flows.}
{Motivated by this need,} {this naturally leads to the concept of {\em star systems}.} A system is said to have the star property if, for every system sufficiently close in the $C^1$ topology, all periodic orbits and singularities (if any) are hyperbolic. {In particular, the star condition replaces uniform hyperbolicity by the weaker requirement that all critical elements are hyperbolic, while maintaining a form of robustness under perturbations.}

For diffeomorphisms and non-singular flows, it was shown by Franks \cite{Franks} and Liao \cite{Liao81} that structural stability implies the star property. Later, Liao \cite{Liao81} and Ma~n'e \cite{Mane} (see also \cite{Palis88}) proved that for diffeomorphisms the star property is equivalent to Axiom A. The same result was subsequently established for non-singular star flows by Hayashi \cite{Hayashi}, and independently by Gan and Wen \cite{GW}.

Following this perspective, singular star flows can be regarded as “Axiom A systems with singularities”, and the Spectral Decomposition Conjecture for such systems can be formulated as follows (see Section~\ref{ss.crc} for the definition of chain recurrence classes).

\begin{conjecture}\cite{Palis, ZGW, SGW}\label{c.sdt}
	$C^1$ open and densely (or $C^1$ generically), singular star flows have only finitely many non-trivial chain recurrence classes, each of which is the homoclinic class of a hyperbolic periodic orbit. 
\end{conjecture}

The main novelty of this work lies in treating the conjecture from a purely ergodic viewpoint, without relying on a topological decomposition of the phase space. 
In contrast with the sectional-hyperbolic setting studied previously, we allow singularities of different indices to coexist within the same chain recurrence class and develop new arguments that compensate for the lack of a global geometric structure.

The first step towards this conjecture is to characterize the hyperbolicity of singular star flows. To this end, singular-hyperbolicity and sectional-hyperbolicity were proposed in \cite{MPP}, \cite{MM} and \cite{ZGW}; they capture the structure of the famous classical Lorenz attractor \cite{Lo63}, the geometric Lorenz model \cite{ABS, Gu76,GW79}, and their higher dimensional counterparts. More recently, multi-singular hyperbolic (for the precise definition, see Section~\ref{ss.def}) was proposed in \cite{BD21} as a generalization of sectional-hyperbolicity. In particular, it allows singularities of different indices to coexist in a chain recurrence class, for which a robust example was constructed in \cite{daLuz}.  Furthermore, the authors of \cite{BD21} proved that $C^1$ open and densely, multi-singular hyperbolicity is equivalent to the star property; see Theorem~\ref{t.BD} below.
Therefore, the Spectral Decomposition Conjecture is equivalent to the following conjecture:
\begin{conjecture}\label{c.sdt1}
	Every multi-singular hyperbolic chain recurrence class is isolated.
\end{conjecture}

A partial answer to this question was obtained by the authors of this article in a previous work \cite{PYY23a}; see Theorem 2.24 below. In summary, the only remaining case where Conjecture 2 might fail is when the topological entropy of the chain recurrence class is zero. In this case, the class supports no ergodic measure other than Dirac measures supported on singularities.

In this paper, we consider the Spectral Decomposition Conjecture for star flows from an ergodic-theoretic perspective. Our goal is to establish the existence and finiteness of a particular family of ergodic invariant probability measures, namely the equilibrium states of Hölder continuous potentials. This family includes, in particular, the measure of maximal entropy and the equilibrium states associated with the geometric potential.

{A key difficulty in this setting comes from the interplay between singularities and the global structure of the dynamics. In particular, the main difficulty addressed in the present work is not merely the presence of singularities — a situation already considered in our previous paper [73] — but rather the coexistence of singularities of different indices within the same chain recurrence class, combined with the absence of a usable topological characterization of the system.}

In this setting, the classical topological approach underlying Smale’s spectral decomposition and its extensions is no longer available. Consequently, new ideas are required. The ergodic approach developed here is designed to bypass this lack of topological structure and will be made precise in the results that follow.

In the general multi-singular setting, the classical tools based on dominated splittings on the tangent bundle and geometric decompositions are no longer available. As a consequence, the methods developed for sectional-hyperbolic attractors cannot be directly adapted to this context.

Our main conceptual contribution is the development of an ergodic framework grounded on a precise control of the dynamics near singularities, which allows us to extend Bowen’s approach to the multi-singular setting.
Instead of relying on a spectral decomposition into hyperbolic basic pieces, we derive finiteness and uniqueness results directly from structural properties of equilibrium states, combined with refined control of shadowing and specification at the level of ergodic measures. 

As a consequence, our results establish finiteness and uniqueness of equilibrium states under natural generic assumptions, providing an ergodic counterpart to the spectral decomposition conjecture in a setting where no topological classification of the dynamics is available.

 We denote by $\mathscr X^{1,*}_+(\bM)$ the set of $C^1$ star vector fields on $\bM$ with positive topological entropy. Our main results are: 

\begin{main}\label{mc.C}
	There exists a $C^1$ open and dense set $\cU\subset \mathscr X^{1,*}_+(\bM)$, such that every $X\in\cU$ with positive topological entropy has a unique ergodic measure of maximal entropy. 
\end{main}

\begin{main}\label{mc.B}
	For every $h>0$, there exists a $C^1$ open and dense set $\cU\subset \mathscr X^{1,*}_+(\bM)$, such that  every $X\in \cU$ has only finitely many chain recurrence classes with topological entropy greater than $h$.  For each such chain recurrence class $C$, $X|_C$ has a unique measure of maximal entropy. Furthermore, the number of such classes varies upper semi-continuously with respect to the system in $C^1$ topology.
\end{main}

A key ingredient in the proof of Theorems \ref{mc.C} and \ref{mc.B} is a structural result for individual multi-singular hyperbolic chain recurrence classes (see Theorems \ref{m.B} and \ref{m.C} below). In this setting, the lack of a global hyperbolic structure prevents the direct use of classical tools and requires a finer analysis of the dynamics at the level of invariant measures and orbit segments. These structural features fundamentally change the nature of the problem and require new analytical and ergodic ideas.

To overcome these challenges, we develop tools that were not available in our previous work. In particular, we introduce an argument based on the space of ergodic measures to establish shadowing properties and transversal intersections. A central new ingredient in this direction is Proposition \ref{p.transversal} (Proposition 6.2), which provides a mechanism to produce transversal intersections and plays a key role in establishing the specification property in the absence of a global hyperbolic structure. This provides a structural mechanism that does not rely on a topological decomposition of the phase space, nor on the existence of a dominated splitting on the tangent bundle, and may have broader applicability beyond the present setting — for instance, to geodesic flows on non-compact surfaces.

{The cornerstone of our technical approach is Theorem \ref{m.B}, which establishes existence and uniqueness of equilibrium states and can be viewed as an ergodic counterpart of Bowen’s classical theorem in the presence of singularities. {While a key step in the proof of the Bowen property is Proposition \ref{p.key}, which bridges Bowen balls in the ambient metric with Liao’s tubular neighborhoods near singularities and allows the classical Bowen argument to be implemented in this setting,} {the main new contribution of this work lies in the establishment of the specification property through Proposition \ref{p.transversal}, which requires a substantially different approach and does not follow from previous techniques.} Complementing this, Theorem \ref{m.C} ensures robustness under $C^1$ perturbations, confirming that the ergodic structures we identify persist in nearby systems.}
\vspace{0.1cm}

Theorem \ref{mc.C} is a special case of a more general result on equilibrium states. Given $\phi: \bM\to \mathbb R$, we denote by $\mathscr X^{1,*}_\phi(\bM)$ the set of $C^1$ star vector fields for which 
\begin{equation}\label{e.pgap.star}
\phi(\sigma) < P(\phi, X), \forall \sigma\in\Sing(X),
\end{equation}
where $P(\phi, X)$ is the topological pressure of $\phi$.

\begin{main}\label{mc.C1}
	For every H\"older continuous function $\phi:\bM\to\RR$, there exists a $C^1$ open and dense set $\cU\subset 
	\mathscr X^{1,*}_\phi(\bM)$, such that every $X\in\cU$ has a unique ergodic equilibrium state for $\phi.$
\end{main}

We also obtain the upper semi-continuity on the number of equilibrium states. For the precise statement, see Theorem~\ref{t.mc.B1}.

\begin{remark}\label{r.Pineq1}
	Under the star property, there can only be finitely many singularities, each of which is hyperbolic and non-degenerate. As a result, the assumption $\phi(\sigma)<P(\phi,X)$ is a finite, verifiable condition among star flows. 
	
	We remark that for a given star vector field $X$, $\phi(\sigma)<P(\phi,X)$ is a $C^0$ open condition among continuous functions, since $P(\phi,X)$ varies continuously w.r.t.\,the potential function \cite[Theorem 9.7]{Wal}. Furthermore, for a given continuous potential function $\phi$, it is a $C^1$ open condition among singular star flows with positive topological entropy due to Theorem~\ref{m.A1} below. 
	
	Previously we proved in \cite{PYY23a} that for $C^1$ generic flows,  $\phi(\sigma)<P(\phi,X)$ is a $C^0$-dense condition on each sectional-hyperbolic attractor. Note that such attractors must satisfy the star property locally. 
\end{remark}

\begin{remark}\label{r.Pineq2}
	For each $\sigma\in\Sing(X)$ we denote by $\mu_\sigma$ the point mass on $\sigma$. Then, by the variational principle, for any continuous potential function $\phi:\bM\to \RR$ we have 
	$$
	\phi(\sigma) = h_{\mu_\sigma}(X) + \int \phi\,d\mu_\sigma \le P(\phi,X);
	$$ 
	equality holds if and only if $\mu_\sigma$ is an equilibrium state of $\phi$. Therefore, \eqref{e.pgap.star} is the same as asking that the point masses on singularities are not equilibrium states. 
	
	In a subsequent work \cite{SY} we will show that this condition is indeed optimal. In particular, we will construct examples of singular star flows together with $C^\infty$ potential functions $\phi:\bM\to\RR$, such that ergodic  equilibrium states of $\phi$ are {\em exactly} the point masses of singularities of $X$. In particular, $\phi$ can have multiple equilibrium states co-existing on the same chain recurrence class.
\end{remark}

Prior to this paper, for singular star flows, only the finiteness of physical measures for (a $C^1$ open and dense family of) $C^\infty$ star flows has been obtained in \cite{CWYZ}. The proof exploits the fact that physical measures must satisfy Pesin's entropy formula and have absolutely continuous conditional measures on the center unstable manifolds.   In comparison, for star diffeomorphisms and non-singular star flows, the counterpart of Theorem \ref{mc.C} is well-known. The proof of the non-singular case combines the Spectral Decomposition Theorem of Smale together with the seminal work of Bowen \cite{B75_2}, which shows that every hyperbolic homoclinic class can only support one MME (or one equilibrium state for every H\"older continuous potential function). For singular star flows, we will first show that every {\em isolated} non-trivial\footnote{Here being non-trivial means that $\Lambda$ is not a periodic orbit or a singularity.} chain recurrence class can only support one equilibrium state (Theorem~\ref{m.B}); then, we will use the result of our previous work \cite{PYY23a} to study how such classes can approach a non-isolated class. One crucial step is to show that the metric entropy varies upper semi-continuously as a function of the vector field. This is achieved by the following theorem:
	
\begin{main}\label{mc.A}
	$C^1$ open and densely, every star flow is almost expansive. Furthermore, the scale of expansivity can be made uniform in a $C^1$ small neighborhood. 
\end{main}

{In the next subsection, we present the main technical results: expansivity, continuity of the topological pressure, and uniqueness of equilibrium states for each multi-singular hyperbolic chain recurrence class. 
The cornerstone is Theorem \ref{m.B}, which establishes existence and uniqueness of equilibrium states and can be viewed as an ergodic counterpart of Bowen’s classical theorem in the presence of singularities. 
{A key step in the proof of the Bowen property is Proposition \ref{p.key}, which bridges Bowen balls in the ambient metric with Liao’s tubular neighborhoods near singularities and allows the classical Bowen argument to be implemented in this setting.} 
{The main new ingredient, however, lies in the establishment of the specification property through Proposition \ref{p.transversal}, which provides a mechanism for producing transversal intersections and requires a substantially different approach from previous works.} Theorem \ref{m.C} ensures robustness under $C^1$ perturbations, and together these results yield the finiteness and uniqueness statements in Theorems \ref{mc.C} and \ref{mc.B}, highlighting both the novelty and stability of the multi-singular hyperbolic framework.}

\subsection{Statement of technical results}\label{ss.statement}
Throughout this article,  $\mathscr{X}^1(\bM)$ denote the collection of $C^1$ vector fields on a compact Riemannian manifold $\bM$ without boundary. Given $X\in\mathscr{X}^1(\bM)$, we write $(f_t)_{t\in\RR}$ for the one-parameter flow generated by $X$, and $\Sing(X)$ for the collection of singularities of $X$. Given an invariant set $\Lambda$, we write 
$\Sing_\Lambda(X) = \Sing(X)\cap\Lambda$ for the singularities in $\Lambda$. We shall always assume that all singularities in $\Lambda$ are hyperbolic and non-degenerate (meaning that the eigenvalues are non-zero). Such an assumption is $C^r$ open and dense for every $r\ge 1$ due to the Kupka-Smale theorem. We will also frequently assume, in Section~\ref{s.Pliss} and onward, that singularities in $\Lambda$ are {\em active}; roughly speaking, this means that singularities are approximated by regular orbits in $\Lambda$. For the precise definition, see Definition~\ref{d.active.sing} below.

Now we are ready to state our main results on multi-singular hyperbolic sets of a $C^1$ vector field. The precise definition is given in Section~\ref{s.preliminary}, Definition~\ref{d.multising}. 

\begin{main}\label{m.A}
	Let $\bM$ be a compact Riemannian manifold without boundary,  $X\in\mathscr X^1(\bM)$  and $\Lambda\subset\bM$ be a multi-singular hyperbolic compact invariant set of $X$.
	Then there exist $\vep>0$, a $C^1$ open neighborhood $\cU$ of $X$ and an open neighborhood $U$ of $\Lambda$, such that every $Y\in \cU$ is almost expansive at scale $\vep$ on its maximal invariant set in $U$.
\end{main}

As a corollary of Theorem~\ref{m.A}, we obtain the  continuity of the topological pressure for continuous potential functions.
\begin{main}\label{m.A1}
	Let $\bM$ be a compact Riemannian manifold without boundary,  $X\in\mathscr X^1(\bM)$  and $\Lambda\subset\bM$ be a multi-singular hyperbolic compact invariant set of $X$ that is isolated  and has positive topological entropy.\footnote{As we have shown in \cite{PYY23a}, for $C^1$ generic star flows, every non-trivial, isolated chain recurrence class (which must be multi-singular hyperbolic due to Theorem~\ref{t.BD}) must have positive topological entropy.} Assume that $\phi: \bM\to \RR$ is a continuous function. Then, the topological pressure $P(\phi,X|_\Lambda)$ varies continuously w.r.t.\,the vector field $X$ in $C^1$ topology in the following sense: for every $\vep>0$, there exist a $C^1$ open neighborhood $\cU$ of $X$ and an open neighborhood $U$ of $\Lambda$, such that for every $Y\in\cU$ and its maximal invariant set $\Lambda_Y\subset U$, the topological pressure satisfies
	$$
	|P(\phi, X|_{\Lambda}) - P(\phi, Y|_{\Lambda_Y})|<\vep.
	$$
	
	In particular, the topological entropy of $X|_\Lambda$ varies continuously in $C^1$ topology.
\end{main}

\begin{remark}\label{r.ThmA}
	It is important to note that, in Theorem~\ref{m.A} and \ref{m.A1} there is no assumption on $\Lambda$ being transitive or chain transitive. We also do not assume that singularities in $\Lambda$ are active. 
\end{remark}

\begin{remark}\label{r.ThmA1}
	The definition of multi-singular hyperbolicity in this article  (Definition~\ref{d.multising}) follows \cite{CDYZ}. It is equivalent to the original definition by Bonatti and da Luz (see Definition~\ref{d.multising.BD} in the Appendix) under the extra assumption that all singularities in $\Lambda$ are active. See \cite[Theorem D and E]{CDYZ}. This assumption is used in Theorem~\ref{m.B} below but not in Theorem~\ref{m.A} or \ref{m.A1}. However, in Section~\ref{ss.ThmA.proof} we will show that for every {\em regular} (i.e., not a point mass of a singularity) ergodic measure $\mu$ on $\Lambda$, every singularity in $\supp\mu$ is active. This will solve the discrepancy between these two definitions. 
\end{remark}

The following result should be viewed as the ergodic counterpart of Bowen’s theorem in the multi-singular hyperbolic setting, and constitutes a central result of this paper.

\begin{main}\label{m.B}
Let $\bM$ be a compact Riemannian manifold without boundary,  $X\in\mathscr X^1(\bM)$, and $\Lambda\subset\bM$ be a multi-singular hyperbolic compact invariant set of $X$ that is an isolated chain recurrence class. Further assume that:
\begin{enumerate}
	\item[(A)] all singularities in $\Lambda$ are non-degenerate and active;
	\item[(B)] $\Lambda$ contains a periodic orbit; furthermore, all periodic orbits in $\Lambda$ are pairwise homoclinically related;
	\item[(C)] $\phi:\bM\to \RR$ is a H\"older continuous function whose topological pressure  satisfies
	\begin{equation}\label{e.pgap}
		\phi(\sigma)< P(\phi, X|_\Lambda), \,\forall \sigma\in \Sing(X|_\Lambda).
	\end{equation}
\end{enumerate}
Then there exists a unique equilibrium state $\mu_\phi$ for $X|_\Lambda$ supported on $\Lambda$ which is ergodic.  In particular, $X|_\Lambda$ has a unique measure of maximal entropy if the topological entropy of $X|_\Lambda$ is positive. 
\end{main}

Indeed we will obtain a strong statement: the conclusion of Theorem~\ref{m.B} remains true under $C^1$ small perturbation of the vector fields. This will be key to obtain $C^1$ openness in Theorem~\ref{mc.C}.

\begin{main}\label{m.C}
	There exists a $C^1$ residual set $\mathcal R\subset \mathscr X^1(\bM)$, such that for every $X\in\mathcal R$, every non-trivial, isolated chain recurrence class $\Lambda$ that is multi-singular hyperbolic, and every H\"older continuous function $\phi:\bM\to\mathbb R$ satisfying \eqref{e.pgap}, there exist a $C^1$ neighborhood $\cU$ of $X$ and an open neighborhood $U$ of $\Lambda$, such that for every $C^1$ vector field $Y\in\cU$ and the maximal invariant set $\tilde \Lambda_Y$ of $Y$ in $U$, there exists a unique equilibrium state of $\phi$ for $Y|_{\tilde\Lambda_Y}$. In particular, there exists a unique MME for $Y|_{\tilde\Lambda_Y}$.
\end{main}

\subsection{Structure of the paper and the sketch of the proof of Theorem~\ref{m.B} and \ref{m.C}}

In Section \ref{s.preliminary} we collect some useful tools to be used throughout the article. Among them, Liao's theory on singular flows has been covered in full detail in our previous paper \cite[Section 2]{PYY23}. Therefore we will only state the relevant results without providing the proof.
%\textcolor{magenta}{For the convenience of the reader, we emphasize that only those aspects directly needed in the sequel will be recalled.}
%\textcolor{magenta}{We briefly recall only the results needed in the sequel.}

In literature there are two (almost) equivalent definitions of multi-singular hyperbolicity. The definition used in this article, i.e., Definition \ref{d.multising}, is a modified version of \cite{CDYZ}. The equivalence of our definition with that in \cite{CDYZ} and \cite{BD21} can be found in the Appendix.

In this article we will consider the maximal invariant set $\Lambda$ in an open isolating neighborhood $U$. In this setting one cannot simply apply the (improved) CT-criterion from \cite{PYY21} on $\Lambda$ since the shadowing orbits obtained from the specification property is not necessarily contained in $\Lambda$.
{This issue reflects the lack of an a priori topological control of orbit segments in the multi-singular setting.}
In our previous paper \cite{PYY23} we were able to bypass this issue by showing that when $\Lambda$ is a sectional-hyperbolic attractor, one can choose the shadowing orbit from $\Lambda$. However, this is no longer the case when $\Lambda$ is not sectional-hyperbolic (or when $\Lambda$ is sectional-hyperbolic but is not an attractor). Therefore we need a further improvement \cite{PYYY24} over the (already) improved CT-criterion, which can be found in Section \ref{s.CT}.

Section \ref{s.fakefoliation} contains the first major tool of this article: fake foliations on the normal plane of regular points that are invariant under the sectional Poincar'e map. These foliations are tangent to their respective cones, and form a local product structure on the normal plane.
{They provide a substitute for invariant foliations in a setting where no global hyperbolic structure is available.}
The downside of such a construction is that they are only well-defined within Liao's ``relative uniform'' scale (uniformly small $\mbox{const}\cdot |X(x)|$), and therefore may not remain well-posed for points in the Bowen ball, particularly when those points get close to singularities. Nonetheless, we will show that distance along such foliations are uniformly
{contracted and expanded}
over carefully parsed orbit segments (Lemma \ref{l.Elarge}, \ref{l.Flarge} and \ref{l.expansion.reg}). As a result, typical points must have trivial infinite Bowen balls, leading to Theorem~\ref{m.A}. Then, we conclude Section \ref{s.fakefoliation} with the proof of Theorem~\ref{m.A1}.

The rest of the paper is devoted to the proof of Theorems \ref{m.B} and \ref{m.C}. {Some tools used here are taken from our previous work \cite{PYY23}, where we proved that every sectional-hyperbolic attractor, including the classical Lorenz attractor, supports a unique equilibrium state.}
However, {several major difficulties arise in comparison:}
the main obstruction now lies in the absence of a dominated splitting on the tangent bundle and the presence of singularities of different indices.
For sectional-hyperbolic flows, the tangent bundle has a dominated splitting $E^s\oplus F^{cu}$ and the stable direction $E^s$ is uniformly contracted by the tangent flow. In contrast, {for multi-singular hyperbolic flows, there is no invariant splitting on the tangent bundle. Instead, the normal bundle, defined only on the punctured manifold $\bM\setminus \Sing(X)$, splits into $E_N\oplus F_N$, with neither bundle uniformly contracted or expanded by the linear Poincaré flow.}

To deal with this non-uniformity and non-compactness, we need to consider the co-existence of four types of Pliss times: two for the contraction/expansion of the linear Poincaré flow along each bundle, and two more for the forward and backward recurrence to neighborhoods of singularities.
{This refined decomposition of orbit segments will be crucial in both the Bowen and specification arguments.}
Results concerning those Pliss times can be found in Section \ref{s.Pliss}. Section \ref{s.Pliss} also introduces infinite Pliss times, relevant for the specification property.

{Section \ref{s.mainthm.proof} contains the proof of Theorem \ref{m.B}, assuming that the Bowen property and the specification property hold on a collection of good orbit segments (Theorems \ref{t.Bowen} and \ref{t.spec}). {While these two properties play complementary roles in the argument, their proofs rely on substantially different mechanisms.}}

On the one hand, a crucial ingredient in establishing the Bowen property is Proposition \ref{p.key}, which ensures that points in the Bowen ball of a good orbit segment are $\rho$-scaled shadowed by the orbit, allowing the construction of a local product structure and control of expansion and contraction along the orbit.
This provides the geometric control needed to reproduce Bowen’s bounded distortion argument in the present setting.
This proposition underlies the bounded distortion estimates that make the Bowen property valid in the multi-singular hyperbolic setting.

On the other hand, the specification property is obtained through a different and genuinely new approach, based on Proposition \ref{p.transversal}, which provides a mechanism to produce transversal intersections and plays a central role in the proof.
This distinction highlights the different nature of the two arguments: geometric control for the Bowen property versus transversal structure for specification.

\vspace{0.1cm}

Section \ref{s.bowen} and \ref{s.spec} are devoted to the proof of Theorem \ref{t.Bowen} and \ref{t.spec} and will become significantly different from \cite{PYY23}. In Section \ref{s.bowen}, we prove Theorem \ref{t.Bowen}, which establishes that for orbit segments where all four aforementioned Pliss times co-exist, one obtains bounded distortion estimates for the Birkhoff sum of the potential function.

The argument combines the hyperbolicity of the scaled linear Poincaré flow with recurrence estimates to control the geometry of nearby orbits. More precisely, for a good orbit segment $(x,t)$, points in the Bowen ball are {\em scaled shadowed} by the orbit of $x$ (Proposition \ref{p.key}), meaning that the distance at time $s\in [0,t]$ is proportional to $|X(x_s)|$.

{This control ensures that the fake foliations constructed earlier remain well-defined along the entire orbit segment, allowing us to reproduce the standard bounded distortion argument and thereby establish the Bowen property in the multi-singular hyperbolic setting.}

%Section \ref{s.bowen} and \ref{s.spec} are devoted to the proof of Theorem \ref{t.Bowen} and \ref{t.spec} and will become significantly different from \cite{PYY23}. In Section \ref{s.bowen} we will prove Theorem \ref{t.Bowen}, which states that for orbit segments where all four aforementioned Pliss times co-exist, one has bounded distortion estimates for the Birkhoff sum of the potential function.
%The argument combines hyperbolicity of the scaled linear Poincaré flow with recurrence estimates to show that, for a good orbit segment $(x,t)$, points in the Bowen ball are {\em scaled shadowed} by the orbit of $x$ (Proposition \ref{p.key}), meaning that the distance between the orbit of $x$ and $y$ at time $s\in [0,t]$ is proportional to $|X(x_s)|$.
%This ensures that the fake foliations constructed earlier remain well-posed for all orbit segments starting in the Bowen ball and allows us to reproduce the standard bounded distortion argument, yielding the Bowen property.
\vspace{0.1cm}

The specification property requires a very subtle treatment and is addressed in Section \ref{s.spec}. Our previous paper \cite{PYY23} heavily depends on the existence of a hyperbolic periodic orbit whose unstable manifold transversely intersects the stable manifold of every regular point.
Such a property has been established for sectional-hyperbolic attractors, but is not known for multi-singular hyperbolic flows or for sectional hyperbolic sets that are not attractors.\footnote{Such an example in dimension 4 has been constructed in an ongoing work.} Indeed, very little is known about the topological structure of such classes, except that they are ($C^1$ generically) homoclinic classes (\cite{PYY23a}).

To overcome this, we shall assume that all periodic orbits are pairwise homoclinically related (Assumption (B) of Theorem~\ref{m.B}). This is a mild $C^1$-generic condition sufficient for the specification property.

The main new ingredient here is Proposition \ref{p.transversal}, which provides a mechanism to produce transversal intersections on a large subset of $\Lambda$ detected by invariant measures, even in the absence of a global hyperbolic structure.

Building on this mechanism, we identify a large subset of $\Lambda$ with good uniform hyperbolic behavior, near which a suitable local structure can be recovered (e.g., transversal intersections between invariant manifolds and controlled recurrence to the singular set), allowing the construction of shadowing orbits and the proof of the specification property.

%The specification property requires a very subtle treatment, and is addressed in Section \ref{s.spec}. Our previous paper \cite{PYY23} heavily depends on the existence of a hyperbolic periodic orbit whose unstable manifold transversely intersects the stable manifold of every regular point. 
%Such a property has been established for sectional-hyperbolic attractors, but is not known for multi-singular hyperbolic flows or for sectional hyperbolic sets that are not attractors.\footnote{Such an example in dimension 4 has been constructed in an ongoing work.} Indeed, very little is known about the topological structure of such classes, except that they are ($C^1$ generically) homoclinic classes (\cite{PYY23a}). 
%To overcome this, we shall assume that all periodic orbits are pairwise homoclinically related (Assumption (B) of Theorem~\ref{m.B}). 
%This is a mild $C^1$-generic condition sufficient for the specification property. The main new ingredient here is Proposition \ref{p.transversal}, which provides a mechanism to produce transversal intersections within a large subset of $\Lambda$ detectable by invariant measures.
%From there, we consider a large subset of $\Lambda$ with good, uniform hyperbolic properties, near which a suitable topological structure can be recovered (e.g., transversal intersections between invariant manifolds and good recurrence to the singular set).
\vspace{0.1cm}

This mechanism allows us to establish the specification property (Theorem \ref{t.spec}), which constitutes one of the main new contributions of this work. Our proof also provides an alternative approach in the classical Lorenz attractor case, relaxing assumptions from \cite{PYY23}.

Section~\ref{s.opendense} contains the proof of Theorem~\ref{m.C}. The proof resembles that of Theorem~\ref{m.B}; in particular, the proof of the Bowen property (Theorem~\ref{t.Bowen}) is untouched. 
For the specification property, 
the key step is to show that Assumption (B) implies that for nearby vector fields $Y$, all measures with large entropy or pressure are homoclinically related. Once this is done, we can reproduce the proof of Theorem~\ref{t.spec} and complete the proof of Theorem~\ref{m.C}.
\vspace{0.1cm}

Finally, in Section \ref{s.star} we will apply Theorem \ref{m.A} to \ref{m.C} to obtain the finiteness of MME and general equilibrium states for star flows and their perturbations, proving Theorem~\ref{mc.C} to \ref{mc.A}.
\vspace{0.1cm}

The technical difficulties in this paper stem from the lack of invariant tangent splittings and from the coexistence of singularities of different indices, which fundamentally distinguishes the present setting from the classical sectional-hyperbolic case. The discussion above highlights why the arguments developed for sectional-hyperbolic attractors cannot be directly extended to the general multi-singular setting.

\subsection{Sectional hyperbolicity and multi-singular hyperbolicity: a quick comparison} \label{ss.quick}
For the convenience of our readers and to highlight the technical difficulties that we face in this paper (compared to \cite{PYY23}), we provide in this subsection a comparison between sectional hyperbolic attractors and general multi-singular hyperbolic sets (the latter includes the former as a special case).

We start with the definition of sectional hyperbolicity.

\begin{definition}\label{d.sectionalhyperbolic}
	A compact invariant set $\Lambda$ of a $C^1$ vector field $X$ is called {\em sectional-hyperbolic}, if the tangent bundle admits a dominated splitting $T_\Lambda\bM = E^{ss}\oplus F^{cu}$, such that $Df_t\mid_{E^{ss}}$ is uniformly contracting, and $Df_t\mid_{F^{cu}}$ is {\em sectional-expanding:}  there are constants $C_1>0,\lambda>1$ such that for every $x\in\Lambda$ and any subspace $V_x\subset F^{cu}_x$ with $\dim V_x = 2$, we have 
	\begin{equation}\label{e.sect.hyp}
		|\det Df_t(x)|_{V_x}| \ge C_1 \lambda ^t \mbox{ for all } t>0.
	\end{equation}
\end{definition}

Not every sectional-hyperbolic set is an attractor. Those that are attractors enjoy some additional properties:
\begin{enumerate}
	\item $\Lambda$ has a well-defined stable foliation.
	\item Singularities in $\Lambda$ must have the same stable index (\cite{MM, SGW}).
	\item $X|_\Lambda$ has positive topological entropy (\cite{PYY}).
	\item $\Lambda$ contains a hyperbolic periodic orbit $\gamma$ (\cite{PYY}) whose stable manifold is dense in a neighborhood of $\Lambda$ \cite{CY}.
	\item The unstable manifold of $\gamma$ is a {\em quasi $u$-section}: it transversely intersects with the stable manifold of every regular point (\cite{CY,PYY23}).
	\item $\Lambda$ is robustly transitive (\cite{PYY23}).
\end{enumerate}

On the other hand, multi-singular hyperbolic sets (the precise definition is postponed to the next section) may not have any dominated splitting on the tangent bundle. Instead, the splitting exists on the normal bundle $N = E_N\oplus F_N$, and neither bundle is uniform under the tangent dynamics (in which case is the linear Poincar\'e flow). In this case, it can be shown that $\Lambda$ has a dominated splitting on the tangent bundle if, and only if, all singularities in $\Lambda$ have the same stable index, in which case $\Lambda$ is sectional-hyperbolic (\cite{SGW}). In general, singularities in $\Lambda$ may have different indices (differ by one) (\cite{SGW}). 

For multi-singular hyperbolic sets $\Lambda$ that are not sectional-hyperbolic, the topological properties are largely unknown. The only existing result, to our knowledge, is that $C^1$ generically, every such class must be isolated and contains a periodic orbit \cite{PYY23a}.  

We summarize the main structural differences between sectional-hyperbolic attractors and general multi-singular hyperbolic sets. The table below highlights which geometric features rely on the existence of a dominated splitting on the tangent bundle and which may break down when singularities of different indices coexist.
\vspace{0.1cm}

\begin{center}
\resizebox{\textwidth}{!}{
\renewcommand{\arraystretch}{1.3}
\begin{tabular}{ |c|c|c| } 
\hline
 & SH attractor & MH isolated set \\ 
\hline
Isolation & Yes & Yes \\
\hline 
Dominated splitting & $T_\Lambda\bM = E^{ss}\oplus F^{cu}$ 
& $N_\Lambda = E_N\oplus F_N$ \\ 
\hline
Singularities & Same stable index 
& Indices may differ \\ 
\hline
Invariant foliations & Stable foliation $\mathcal F^s$ 
& Not necessarily defined \\ 
\hline 
Robust transitivity & Yes 
& Unknown \\ 
\hline
Unstable manifolds & Quasi $u$-section property 
& Not known \\ 
\hline
Dynamical behavior & Attracting structure 
& Saddle-type behavior possible \\ 
\hline
\end{tabular}
}
\end{center}

\vspace{0.2cm}

We remark that not all multi-singular hyperbolic chain recurrence classes are attractors (see for instance \cite{daLuz}), especially when $\Lambda$ indeed contains singularities of different indices. In this case, orbits near $\Lambda$ may eventually escape a filtrating neighborhood of $\Lambda$ and move towards other chain recurrence classes.

\section{Preliminary}\label{s.preliminary}
For the convenience of our readers, all the notations are consistent with our previous paper~\cite{PYY23}. For certain notations and tools such as the Liao's tubular neighborhood theory, we invite our readers to the comprehensive discussion in~\cite[Section 2 and 3]{PYY23}. 

Throughout this paper,  a singular flow is a flow where $\Sing(X)\ne\emptyset$. As a standard assumption, we will assume that all the singularities are hyperbolic and non-degenerate (i.e., the derivatives are invertible). This is a $C^r$ open and dense assumption in the space of $C^r$ vector fields, for every $r\ge 1.$

As in~\cite{PYY21,PYY23}, we shall identify $(x,t)\in\bM\times\RR^+$ with  the orbit segment $\{f_s(x):0\le s< t\}$. We will also use the standard notation 
$$
x_t = f_t(x).
$$
In comparison, we will use superscripts when referring to a sequence of points, e.g., $x^1,x^2,\ldots.$ In this case, $(x^i)_s$ refers to $f_s(x^i)$. 
\subsection{Multi-singular hyperbolicity}\label{ss.def}
Bonatti and da Luz introduced multi-singular hyperbolicity in~\cite{BD21} to capture the hyperbolic structure of singular star flows. Their definition involves {reparametrization cocycles}\footnote{Roughly speaking, Bonatti and da Luz's definition requires that under a proper reparametrization which is given by a multiplicative cocycle over the flow, the {\em extended linear Poincar\'e flow} (the extension of the linear Poincar\'e flow to the projective tangent bundle) is uniformly hyperbolic.} for the extended flow on the Grassmannian manifold. Later, the authors of~\cite{CDYZ} introduced a new definition which, in most situations, are equivalent to the previous one (see Theorem~\ref{t.sec.hyp}). This new definition only involves the well-known notion of the linear Poincar\'e flow, and is more suitable for the structure this paper.

The set of regular points is denoted by
$$
\Reg(X) = \bM\setminus \Sing(X). 
$$
For a compact invariant set $\Lambda$, the normal bundle over $\Lambda$ is defined over $\Lambda\cap \Reg(X)$ as the orthogonal complement of the flow direction in the tangent bundle. More precisely:
$$
N_\Lambda = \bigsqcup_{x\in\Lambda\cap\Reg(X)} N(x) \mbox{, where } N(x) =\langle X(x)\rangle^\perp\mbox{ is the normal plane at $x$}.
$$ 
When no confusion is caused, we will sometimes drop the index $\Lambda$. 

Writing $\pi_{N,x}: T_x\bM\to N(x)$ for the orthogonal projection to the normal plane $N(x)$, we define the {\em linear Poincar\'e flow} $\psi_t: N\to N$ to be the projection of the tangent flow to the normal bundle. To be more precise, given $v\in N(x)$, we let
\begin{equation}\label{e.Poin}
	\psi_t(v) = \pi_{N,x_t}\circ Df_t(v) = Df_t(v) -\frac{< Df_t(v), X(x_t) >}{\|X(x_t)\|^2}X(x_t),
\end{equation}
where $< .,. >$ is the inner product on $T_xM$ given by the Riemannian metric. The {\em scaled linear Poincar\'e flow} $\psi^*_t$ is the linear Poincar\'e flow scaled by the flow speed, that is,
\begin{equation}\label{e.scaledP}
	\psi_t^*(v) = \frac{\psi_t(v)}{\|Df_t\mid_{\langle X(x) \rangle}\|} = \frac{|X(x)|}{|X(x_t)|}\psi_t(v).
\end{equation}
The scaled linear Poincar\'e flow takes into consideration the flow speed at each regular point, and is compatible with Liao's theorem on the tubular neighborhoods. 
More properties of the scaled linear Poincar\'e flow will be discussed in Section~\ref{ss.Liao}.

As shown in \cite{SGW} and \cite{BD21}, star flows may not have any dominated splitting on the tangent bundle. Therefore we must consider dominated splittings on the normal bundle:
\begin{definition}\label{d.sing.dom.spl}
	For $X\in\mathscr{X}^1(\bM)$ and compact invariant set $\Lambda$, a {\em singular dominated splitting of index $i$} is a decomposition of the normal bundle $N_{\Lambda} = E_N\oplus F_N$ which is invariant for the linear Poincar\'e flow $(\psi_t)$, such that
	\begin{enumerate}
		\item $\dim E_N = i$;
		\item the splitting is dominated for the linear Poincar\'e flow: there exist constants $C>0$ and $\lambda>1$, such that 
		\begin{equation}\label{e.dom.spl}
		\|\psi_t\mid_{E_N(x)}\| \cdot\|\psi_{-t}\mid_{F_N(x_t)}\| < C\lambda^{-t},\,\,\forall x\in\Lambda\cap\Reg(X), t>0.
		\end{equation}		
		\item for each singularity $\sigma\in\Sing_\Lambda(X)$:
		\begin{enumerate}
			\item  either there exists a dominated splitting of the form $T_\sigma\bM = E_\sigma^{ss}\oplus F_\sigma$ for the tangent flow $(Df_t)$ with $\dim E_\sigma^{ss} = i$, such that $E_\sigma^{ss}$ is uniformly contracted, and 
			$$
			W^{ss}(\sigma) \cap \Lambda = \{\sigma\};
			$$
 			in this case we say that $\sigma$ is of {\em Lorenz-type}, and denote the collection of such singularities by $\Sing^+_\Lambda(X)$;
			\item or there exists a dominated splitting of the form $T_\sigma\bM = E_\sigma \oplus E_\sigma^{uu}$ for the tangent flow $(Df_t)$ with $\dim E_\sigma^{uu} = \dim F_N = \dim\bM-1-i$, such that $E_\sigma^{uu}$ is uniformly expanded, and 
			$$
			W^{uu}(\sigma) \cap \Lambda = \{\sigma\};
			$$
			in this case we say that $\sigma$ is of {\em reverse Lorenz-type}, and denote the collection of such singularities by $\Sing^-_\Lambda(X)$;
		\end{enumerate}
	\end{enumerate}
\end{definition}

\begin{remark}\label{r.symmetry}
	It is clear from the definition that a singular dominated splitting for $X$  of index $i$ is a singular dominated splitting for $-X$ of index $\bM-1-i$. Furthermore, there is a strong symmetry in the classification of singularities into Lorenz and reverse Lorenz-type:
	\begin{equation}\label{e.symmetry.1}
		\Sing^+_\Lambda(X) = \Sing^-_\Lambda(-X),
	\end{equation}
	i.e., every Lorenz-type singularity for $X$ is a reverse Lorenz-type singularity for $-X$, and vice versa. This observation will be used frequently in the rest of this paper.
\end{remark}

\begin{remark}\label{r.dom.scaled}
	As we explained in our previous paper~\cite[Section 2.2]{PYY23}, one can define {dominated splittings for the scaled linear Poincar\'e flow} by replacing $\psi_t$ in~\eqref{e.dom.spl} with $\psi_t^*$. Since $$
	\frac{|\psi_t^*(u)|}{|\psi_t^*(v)|} = \frac{\frac{|\psi_t(u)|}{\|Df_t\mid_{\langle X(x)\rangle}\|}}{\frac{|\psi_t(v)|}{\|Df_t\mid_{\langle X(x)\rangle}\|}}=\frac{|\psi_t(u)|}{|\psi_t(v)|},
	$$
	every dominated splitting for $(\psi_t)_t$ is a dominated splitting for $(\psi_t^*)_t$ with the same constants, and vice versa.
\end{remark}

In this paper, we shall assume w.l.o.g. that $E_N$ and $F_N$ are orthogonal at every regular point by changing the metric if necessary. By speeding up the flow (i.e., replacing $X$ with $cX$ for some $c>1$)\footnote{There is a one-to-one correspondence between equilibrium states of $X$ and $cX$; as a result, the existence and uniqueness of equilibrium states for $cX$ is equivalent to the existence and uniqueness of equilibrium states for $X$ (with the corresponding potential functions); see Section~\cite[Section 4]{PYY21} for detail.} we may assume the following {\em one-step domination}:
\begin{equation}\label{e.onestepDS1}
	\|\psi^*_1\mid_{E_N(x)}\| \|\psi^*_{-1}\mid_{F_N(x_t)}\| \le \frac12.
\end{equation}

 Given $\alpha>0$, one can define the $(\alpha,F_N)$-cone field on the normal bundle by
$$
C^N_\alpha(F_N(x)) = \{v = v_{E_N} + v_{F_N}\in N(x): |v_{E_N}|\le\alpha|v_{F_N}| \}.
$$ 
Then~\eqref{e.onestepDS1} implies the forward invariance of the $(\alpha,F_N)$-cone field  under $\psi^*_1$. The $(\alpha,E_N)$-cone field can be defined similarly, and is backward invariant.

Note that {\em a priori}, a singularity can be both of Lorenz and reverse Lorenz-type. For example, an isolated singularity whose dynamics does not interact with the rest of $\Lambda$. To rule out such singularities, we define active singularities as follow.
\begin{definition}\label{d.active.sing}
	We say that a hyperbolic singularity $\sigma\in\Lambda$ is active in $\Lambda$, if both 
	$W^s(\sigma)\cap \Lambda\setminus \{\sigma\}$ and $W^u(\sigma)\cap\Lambda\setminus \{\sigma\}$ are non-empty.  
\end{definition}
In other words, $\sigma$ is active if there exist regular orbits in $\Lambda$ that approach $\sigma$ (to arbitrarily small scales) and then leave $\sigma$. 

\begin{remark}\label{r.active}
	It is proven in~\cite[Theorem D and E]{CDYZ} that if all singularities in $\Lambda$ are active, then the definition in~\cite{CDYZ} is equivalent to that in~\cite{BD21}. This assumption is mild, as it is satisfied by all non-trivial chain recurrence classes (for the precise definition, see Section~\ref{ss.crc}). Furthermore, every compact invariant set $\Lambda$ can be written as a compact invariant  subset $\tilde\Lambda$ with all singularities active, together with (finitely many) singularities that are not active, each of which is connected to $\tilde \Lambda$ with either its stable manifold or unstable manifold, but not both. In this case, the non-trivial dynamics is supported on $\tilde\Lambda.$ If we assume that Assumption (C) of Theorem~\ref{m.B} holds (which we will do throughout this paper), then the point masses of singularities are not equilibrium states. Furthermore, the stable and unstable manifolds of singularities do not support non-trivial invariant measures due to Poincar\'e Recurrence Theorem. As a result, equilibrium states of $X|_\Lambda$, if exists (the existence will be given by Theorem~\ref{m.A} and~\cite{B72}), must be supported on $\tilde\Lambda$. Then one can apply Theorem \ref{m.B} to $X|_{\tilde\Lambda}$ to obtain the uniqueness of equilibrium states on $\tilde\Lambda$, and therefore on $\Lambda$. 
\end{remark}

Now we are ready to introduce multi-singular hyperbolicity.

\begin{definition}\label{d.multising}
	A compact invariant set $\Lambda$ for a $C^1$ vector field $X$ is multi-singular hyperbolic of index $i$, if 
	\begin{enumerate}
		\item $\Lambda$ admits a singular dominated splitting $E_N\oplus F_N$ of index $i$;
		\item there exists $\eta>1$ such that for every open isolating neighborhood $V$ of $\Sing_\Lambda(X)$ that is sufficiently small, there exists $T_V>0$ such that the following inequalities hold:\footnote{Here $\floor{t}$ is the integer part of $t$.} 
		\begin{equation}\label{e.hyp0}
			\prod_{i=0}^{\floor{t}-1}\left\|\psi_1|_{E_N(x_i)}\right\|\le \eta^{-\floor{t}},   \prod_{i=0}^{\floor{t}-1}\big\|\psi_{-1}|_{F_N((x_{\floor{t}-i}))}\big\|\le \eta^{-t},
		\end{equation}	
		\begin{equation}\label{e.hyp}
			\prod_{i=0}^{\floor{t}-1}\left\|\psi^*_1|_{E_N(x_i)}\right\|\le \eta^{-\floor{t}}, \mbox{ and } \prod_{i=0}^{\floor{t}-1}\big\|\psi_{-1}^*|_{F_N((x_{\floor{t}-i}))}\big\|\le \eta^{-t},
		\end{equation}
		whenever $x,x_t\in  \Lambda\cap V^c$ and $t>T_V$;
		\item each singularity in $\Lambda$ is hyperbolic with splitting $T_\sigma\bM = E_\sigma^{ss}\oplus E_\sigma^c\oplus E_\sigma^{uu}$, satisfying 
		$$
		\dim E_\sigma^{ss} = \dim E_N, \,\,\dim E_\sigma^{uu} = \dim F_N,\,\, \dim E^c= 1; 
		$$
		furthermore, the largest negative Lyapunov exponent $\lambda_\sigma^{ss}$ along $E_\sigma^{ss}$, the smallest Lyapunov exponent $\lambda_\sigma^{uu}$ along $E_\sigma^{uu}$ and the center Lyapunov exponent $\lambda_\sigma^c$ satisfy
		\begin{equation}\label{e.sing.Lya}
		0<|\lambda_\sigma^c|<\min(-\lambda_\sigma^{ss}, \lambda_\sigma^{uu}).
		\end{equation}
	\end{enumerate} 
\end{definition}

\begin{remark}
	For a multi-singular hyperbolic set $\Lambda$ of index $i$, all periodic orbits in $\Lambda$ are hyperbolic and have the same stable index (i.e., the dimension of the stable subspace) which is $i$; on the other hand, the indices of singularities must be $i$ or $i+1$ depending on whether $\lambda^c_\sigma$ is positive or negative. We call $i$ the index of $\Lambda$.
\end{remark}

\begin{remark}\label{r.hyp.measure}
	Recall that an invariant ergodic measure for a flow $X$ that is not the point mass of a singularity (i.e., a regular measure) is called hyperbolic, if the only zero Lyapunov exponent of $\mu$ (defined using the tangent flow $(Df_t)$) comes from the flow direction. It is easy to prove (see, for instance, \cite[Theorem 2.12]{PYY}) that for every regular measure (i.e., a measure that is not  a point mass of a singularity), the Lyapunov exponents defined by $(Df_t)$ (except for the zero exponent from the flow direction) coincide with those defined by $(\psi_t)$ and those by $(\psi_t^*)$. As a result, a regular measure $\mu$ is hyperbolic if and only if the Lyapunov exponents defined by $(\psi_t)$ are all non-zero.  Definition~\ref{d.multising} (2) implies that all Lyapunov exponents $\lambda_i(\mu)$ of $\mu$ must satisfy $|\lambda_i(\mu)|\ge \eta$. This shows that every regular measure is hyperbolic.
\end{remark}

\begin{remark}
	Later we will show that for certain regular orbits $(x,t)$ that ``stay close to $\Lambda$'', there exists a singular dominated splitting on the normal bundle over the finite orbit segment $(x,t)$ such that \eqref{e.hyp0} and \eqref{e.hyp} hold. Such a statement involves the extended flow on the Grassmannian manifold and is therefore postponed to the next subsection. See Lemma~\ref{l.robust.dom.spl} and \ref{l.robust.hyp}.
\end{remark}

Definition~\ref{d.multising} is slightly different from~\cite[Definition 1.6]{CDYZ}; in the latter, item (2) was stated as

\begin{enumerate}
	\item[(2)'] there exist $\eta>1$, $T > 0$ and a compact isolating neighborhood  $V$ of $\Sing_\Lambda(X)$ such that
	$$
	\|\psi_t|_{E_N(x)}\|\le \eta^{-t}, \mbox{ and } \|\psi_{-t}|_{F_N(x_t)}\|\le \eta^{-t},
	$$
	whenever $x,x_t\in \Lambda\cap V^c$ and $t>T$.
\end{enumerate}
(2)' is useful when considering Lyapunov exponents. 
In comparison, Definition~\ref{d.multising} (2) is easier to use in combination with the Pliss Lemma (see Theorem~\ref{t.pliss}) to produce hyperbolic times. For more details, see Section~\ref{s.Pliss}. 

We shall prove the following proposition, which states that these two definitions are equivalent. 

\begin{proposition}\label{p.equivalent}
	Let $\Lambda$ be a compact invariant set with all singularities hyperbolic and active. Then:
	\begin{itemize}
		\item (2) implies (2)'.
		\item \cite[Definition 1.6]{CDYZ} implies (2).
	\end{itemize}
	Consequently, Definition~\ref{d.multising} and \cite[Definition 1.6]{CDYZ} are equivalent. 
\end{proposition}

The proof of this proposition can be found in Appendix~\ref{A.multi-sing}.

For the next lemma, recall the definition of $\Sing_\Lambda^\pm(X)$ in Definition \ref{d.sing.dom.spl} (3).
\begin{lemma}\label{l.lorenz-type}
	Let $\Lambda$ be a multi-singular hyperbolic set for a $C^1$ vector field $X$ with all singularities active. Then 
	\begin{equation}
	\begin{split}
	&\Sing_\Lambda^+(X) = \{\sigma\in \Sing_\Lambda(X): \lambda_\sigma^c<0\}, \mbox{ and }\\
		&\Sing_\Lambda^-(X) = \{\sigma\in \Sing_\Lambda(X): \lambda_\sigma^c>0\}.
	\end{split}
	\end{equation}
	In particular, 
	$$
	\Sing_\Lambda^+(X)\cap \Sing_\Lambda^-(X)=\emptyset.
	$$
\end{lemma}

\begin{proof} We only prove that  $\Sing_\Lambda^+(X) = \{\sigma\in \Sing_\Lambda(X): \lambda_\sigma^c<0\}$. The same argument applies to $\Sing_\Lambda^-(X)$ by considering $-X$ (keep in mind Remark~\ref{r.symmetry}).
	
	``$\subset$'': take $\sigma\in \Sing_\Lambda^+(X)$, we have $W^{ss}(\sigma)\cap\Lambda\setminus\{\sigma\}=\emptyset$. On the other hand, since $\sigma$ is active, it holds that $W^s(\sigma)\cap\Lambda\setminus\{\sigma\}\ne\emptyset$. This shows that $W^{ss}(\sigma)\ne W^s(\sigma)$, indicating that $E^c_\sigma<0$.
	
	``$\supset$'': let $\sigma\in\Sing_\Lambda(X) $ be such that $\lambda_\sigma^c<0$. Then $W^{u}(\sigma) = W^{uu}(\sigma)$. Assume for contradiction's sake that $\sigma\in \Sing_\Lambda^-(X)$, then $W^{u}(\sigma)\cap\Lambda\setminus\{\sigma\} = W^{uu}(\sigma)\cap\Lambda\setminus\{\sigma\}=\emptyset$. This contradicts with the assumption that $\sigma$ is active. Therefore we must have $\sigma\in \Sing_\Lambda(X)\setminus \Sing_\Lambda^-(X)= \Sing_\Lambda^+(\sigma)$.
\end{proof}

The next two theorems, taken from~\cite{CDYZ}, will be useful in establishing the robust behavior of multi-singular hyperbolic flows.

\begin{theorem}\label{t.robust.dom}\cite[Theorem A]{CDYZ}
Let $X\in\mathscr{X}^1(\bM)$ and $\Lambda$ be a compact invariant set admitting a singular dominated splitting of index $i$. Then there exist a $C^1$ neighborhood $\cU$ of $X$ and a neighborhood $U$ of $\Lambda$ such that for any $Y \in \cU$, the maximal invariant set in $U$ admits a singular dominated splitting of index $i$.
\end{theorem}

\begin{theorem}\label{t.robust.hyp}\cite[Theorem B]{CDYZ}
Let $X\in\mathscr{X}^1(\bM)$ and $\Lambda$ be a compact invariant set that is multi-singular hyperbolic. Then there exist a $C^1$ neighborhood $\cU$ of $X$ and a neighborhood $U$ of $\Lambda$ such that the maximal invariant set of $Y \in \cU$ in $U$ is multi-singular hyperbolic.
\end{theorem}

We conclude this subsection with the following theorem on the relation between multi-singular hyperbolicity and sectional hyperbolicity (see~\cite[Definition 1]{PYY23}).
\begin{theorem}\label{t.sec.hyp}\cite[Theorem C]{CDYZ}
	Let $X\in\mathscr{X}^1(\bM)$ and $\Lambda$ be a compact invariant set such that every singularity in $\Lambda$ is active. Then:
	\begin{enumerate}
		\item $\Lambda$ is uniformly hyperbolic if and only if $\Lambda$ is multi-singular hyperbolic and does not contain any	singularity;
		\item $\Lambda$ is sectional-hyperbolic (for $X$ or $-X$) if and only if $\Lambda$ is multi-singular hyperbolic and all the singularities in $\Lambda$ have the same stable index.
	\end{enumerate}
\end{theorem}

In particular, if the indices of all singularities in $\Lambda$ are equal to $i$, then $\Lambda$ is sectional-hyperbolic. If all the indices are $i+1$, then $\Lambda$ is sectional-hyperbolic for $-X$.

\subsection{How regular orbits approach an active singularity: dynamics on the Grassmannian manifold}\label{ss.active}
Below we consider how regular orbits in a multi-singular hyperbolic set $\Lambda$ approach and escape active singularities. We will see that when a regular orbit $\gamma$ in $\Lambda$ approaches some active $\sigma \in\Lambda$:
\begin{enumerate}
	\item if $\sigma\in\Sing_\Lambda^+(\Lambda)$, then $\gamma$ must approach $\sigma$ along the one-dimensional $E^c_\sigma$ direction ($\lambda_\sigma^c<0$);
	\item if $\sigma\in\Sing_\Lambda^-(\Lambda)$, then $\gamma$ must escape $\sigma$ along the one-dimensional $E^c_\sigma$ direction ($\lambda_\sigma^c>0$).
\end{enumerate}
These two statements are symmetric by replace $X$ with $-X$. Therefore we shall only state and prove case (1).

To rigorously formulate this phenomenon, we introduce the {\em extended linear Poincar\'e flow} which was first developed in~\cite{LGW} (see also \cite{SGW, GY} and \cite{CY1, CY25}). We remark that this concept will only be used in this subsection and the appendix, so uninterested readers can safely skip it.

Denote by 
$$
G^1 = \{L_x: L_x \mbox{ is a 1-dimensional subspace of }T_x\bM, x\in \bM\}
$$
the Grassmannian manifold of $\bM$. 
Given a $C^1$ flow $(f_t)_t$, the tangent flow $(Df_t)_t$ acts naturally on $G^1$ by mapping each $L_x$ to $Df_t(L_x)$. 

Write $\beta:G^1\to \bM$ and $\xi: T\bM\to \bM$ the bundle projection. The pullback bundle of $T\bM$:
$$
\beta^*(T\bM) = \{(L_x,v)\in G^1\times T\bM: \beta(L)= \xi(v)=x\} 
$$
is a vector bundle over $G^1$ with dimension $\dim \bM$. The tangent flow $(Df_t)_t$ lifts naturally to $\beta^*(TM)$:
$$
Df_t(L_x,v) =(Df_t(L_x),Df_t(v)).
$$

Recall that the linear Poincar\'e flow $\psi_t$  projects the tangent flow to the normal bundle $N$. The key observation is that this projection can be defined not only w.r.t.the normal bundle but to the orthogonal complement of any section $\{L_x:x\in \bM\}\subset G^1$.

To be more precise, given $L = \{L_x:x\in \bM\}$ we write 
$$
N^L = \{(L_x,v)\in \beta^*(TM): v \perp L_x\}.
$$
Then $N^L$, consisting of vectors perpendicular to $L$, is a sub-bundle of $\beta^*(T\bM)$ over $G^1$ with dimension $\dim \bM-1$. The {\em extended linear Poincar\'e flow} is then defined as 
$$
\psi_t = \psi_{t}^L: N^L\to N^L,\,\,
\psi_t(L_x,v) = \pi(Df_t(L_x,v)), 
$$
where $\pi$ is the orthogonal projection from fibres of $\beta^*(TM)$ to the corresponding fibres of $N^L$ along $L$.

If we define the map 
\begin{equation}\label{e.zeta}
\zeta: \Reg(X)\to G^1,\,\, \zeta(x) = \langle X(x)\rangle,
\end{equation}
i.e., $\zeta$ maps every regular point $x$ to the unique $L_x\in G^1$ with  $\beta(L_x)=x$ such that $L_x$ is generated by the flow direction at $x$, then the extended linear Poincar\'e flow on $N^{\zeta(\Reg(X))}$ can be naturally identified with the linear Poincar\'e flow defined earlier. On the other hand, given any invariant set $\Lambda$ of the flow, consider the compact set: 
\begin{equation}\label{e.BLambda}
\mathfrak B(\Lambda) = \Cl\left({\zeta(\Lambda\cap \Reg(X))}\right)
\end{equation}
where  $\Cl$ denotes the closure of a set. In other words, $\mathfrak B(\Lambda) $ consists of those directions in $G^1$ that can be approximated by the flow direction of regular points in $\Lambda$. 
If $\Lambda$ contains no singularity, then $\mathfrak B(\Lambda) $ can be seen as a natural copy of $\Lambda$ in $G^1$ equipped with the direction of the flow at each point of $\Lambda$. If $\sigma\in\Lambda$ is a singularity, then $ \mathfrak B(\Lambda) $ also contains all the directions in $\beta^{-1}(\sigma)$ that can be approximated by the limiting flow direction as the orbit of regular points in $\Lambda$ approach $\sigma$. %In other words, one replaces the singularity $\sigma$ by a subset of the sphere $\beta^{-1}(\sigma)$. 
%The extended Poincar\'e flow restricted to $B(\Lambda) $ can be seen as  the continuous extension of the linear Poincar\'e flow on $\Lambda$. %A similar treatment can be applied to the scaled linear Poincar\'e flow $\psi_t^*$.
This motivates us to define, for $\sigma\in\Sing_\Lambda(X)$,
$$
\mathfrak B_\sigma(\Lambda) = \{L\in \mathfrak B(\Lambda): \beta(L)=\sigma\}.
$$ 
So we have $\mathfrak B(\Lambda) = \bigcup_{\sigma\in\Lambda}\mathfrak B_\sigma(\Lambda) \cup \zeta(\Lambda\cap \Reg(X))$.

The following lemma is obtained from the proof of~\cite[Lemma 4.4]{LGW}. See also \cite{MPP}.

\begin{lemma}\label{l.lgw}
	Let $\Lambda$ be a multi-singular hyperbolic set for a $C^1$ vector field $X$ with all singularities active. Then:
	\begin{itemize}
		\item for every $\sigma\in\Sing_\Lambda^+(X)$ and every $L\in\mathfrak B_\sigma(\Lambda)$, one has
		$$
		L\subset E^c_\sigma\oplus E^{uu}_\sigma.
		$$
		\item for every $\sigma\in\Sing_\Lambda^-(X)$ and every $L\in\mathfrak B_\sigma(\Lambda)$, one has
		$$
		L\subset E^{ss}_\sigma\oplus E^c_\sigma.
		$$
	\end{itemize} 
\end{lemma}

We remark that the proof of the first part is essentially the same as~\cite[Lemma 2.15]{PYY23}, despite the latter being stated for sectional-hyperbolic  (in which case all singularities in $\Lambda$ are of Lorenz-type). The only place where sectional hyperbolicity is used in~\cite[Lemma 2.15]{PYY23} is to obtain $W^{ss}(\sigma)\cap\Lambda\setminus \{\sigma\}=\emptyset$. In our case, this is guaranteed by the definition of singular dominated splitting, see Definition~\ref{d.sing.dom.spl} (3), Case (a). The second part follows by considering $-X$.

Next we define the ``center cone'' at the tangent space of a singularity $\sigma$ as
$$
C_\alpha(E_\sigma^c) = \{v\in T_\sigma\bM, v=v^{ss}+v^c+v^u, \max\left\{|v^{ss}, v^u|\right\}\le\alpha |v^c| \}
$$
(note that this cone is not invariant under the tangent flow unless one further  requires  that $v^u=0$)
and consider 
$$
\cC_{\alpha}(E_\sigma^c) = \exp_{\sigma}\left(C_\alpha(E_\sigma^c)\right)
$$
its image under the exponential map. $\cC_{\alpha}(E_\sigma^c)$ can be considered as a {\em cone on the manifold} around the center direction of $\sigma$ despite that $\sigma$ may not have a center manifold. The domination between $E^{ss}_\sigma$ and $E^c_\sigma$ implies that this cone has certain invariance property under $f_1$, in the sense that 
$$
f_1\left(\cC_{\alpha}(E_\sigma^c)\cap W^s(\sigma)\right)\subset \cC_{\theta\alpha}(E_\sigma^c\cap W^s(\sigma)),\mbox{ for some }\theta\in(0,1). 
$$
Then a similar argument as in the previous lemma shows that for some $T>0$ and every $t>T$, $y_t\in	\cC_{\alpha}(E_\sigma^c)\cap W^s(\sigma)$. This observation leads to the following lemma. 

\begin{lemma}\label{l.nearEc}
	Under the same assumption as Lemma~\ref{l.lgw}, for every $\alpha>0$, there exist $r_0>0$ and $\overline r\in (0,r_0)$, such that for every $r\in(0,\overline r]$ and every $x\in B_{r}(\sigma)\cap\Lambda$, letting $t_x= \sup\{t>0: (x_{-t}, 0)\subset B_{r_0}(\sigma)\}$ then one has
	$$
	x_{-t_x} \in \partial B_{r_0}(\sigma)\cap \cC_{\alpha}(E_\sigma^c).
	$$ 
\end{lemma}
The proof is omitted. See the proof of~\cite[Lemma 2.16]{PYY23}, and keep in mind $W^{ss}(\sigma)\cap\Lambda\setminus \{\sigma\}=\emptyset.$

These two lemmas combined show that near a Lorenz-type singularity $\sigma$, regular orbits in $\Lambda$ can only approach $\sigma$ along the one-dimensional center direction. By considering $-X$, we see that regular orbit can only escape reverse Lorenz-type singularities along the one-dimensional $E^c$ direction.

Next we turn our attention to finite orbit segments $(x,t)$ that stay in a small neighborhood of $\Lambda$.  A priori, there may not exists a dominated splitting on the normal bundle over $(x,t)$ that is the continuation of that on $\Lambda$ (see for instance \cite{MPP99}), 
%nor can one get any statement similar to Lemma~\ref{l.nearEc}. To see this,  just consider $x\in W^{ss}_{loc}(\sigma)$ for $\sigma\in \Sing^+_\Lambda(X)$; for any $t>0$, $(x,t)$ stays close to $\Lambda$, but there cannot be any dominated splitting on the normal bundle of $(x,t)$ (\cite{LGW, MPP}) 
nor does $(x,t)$ enter $B_r(\sigma)$ through the center direction (for instance, consider $x$ on the strong stable manifold of a Lorenz-type singularity). 

To avoid this issue one must carefully choose the neighborhood $U$ and only consider certain ``good'' orbits. This is summarized in the following lemma.

\begin{lemma}\label{l.nbhd.1}
	Let $\Lambda$ be a multi-singular hyperbolic set for a $C^1$ vector field $X$. Then, there exist $r_0>0$, an open neighborhood $U$ of $\Lambda$, and $\overline r\in (0,r_0)$, such that the following statement hold.
	
	For every $r\in (0,\overline r]$, every $x\in U\cap (B_{r_0}(\Sing_\Lambda(X)))^c$, every $t>0$ such that $(x,t)\subset U$, and every $\sigma\in\Sing^+_\Lambda(X)$,  assume that:
	\begin{enumerate}
		\item $t_x>0$ satisfies $x_{t_x}\in \partial B_{r_0}(\sigma)$;
		\item $t_x'>t_x$ satisfies $x_{t_x'}\in B_{r}(\sigma)$;
		\item $(x_{t_x}, t_x'-t_x)\subset B_{r_0}(\sigma)$.
	\end{enumerate}
	Then one has 
	$$
	x_{t_x} \in \partial B_{r_0}(\sigma)\cap \cC_{\alpha}(E_\sigma^c).
	$$
	A similar statement holds for $\sigma\in \Sing^-_\Lambda(X)$. Furthermore, the same holds with $U$ replaced with a sub-neighborhood of $\Lambda$ in $U$.
\end{lemma}

The lemma states that  for a Lorenz-type singularity $\sigma$ and for orbit segments in $U$ that start outside $B_{r_0}(\sigma)$ and enters the smaller ball $B_{r}(\sigma)$, it can only enter the larger ball $B_{r_0}(\sigma)$ along the one-dimensional center direction. In other words, one can choose the neighborhood $U$ to avoid 
entering $B_{r_0}(\sigma)$ along $W^{ss}(\sigma)$.  Similarly, it can only escape $B_{r_0}(\sigma)$ for a reverse Lorenz-type singularity along the one-dimensional center direction. We remark that this lemma does not require $\Lambda$ to be isolated. 
The proof of this lemma can be found in Appendix \ref{A.multi-sing}.

Next we turn our attention to the dominated splitting on the normal bundle over $(x,t)$ that satisfies (1) to (3) of the previous lemma. For this purpose, we write 
$$
N_{(x,t)} =  \bigsqcup_{s\in[0,t)} N(x_s).
$$

\begin{lemma}\label{l.robust.dom.spl}
	Under the assumptions of Lemma~\ref{l.nbhd.1}, one can shrink $U$ to obtain constants $C>0$ and $\lambda>1$, such that 
	for every orbit segment $(x,t)\subset U$ satisfying $x, x_t\notin B_{r_0}(\Sing_\Lambda(X))$, the following statements hold:
	\begin{enumerate}
		\item There exists a splitting $N_{(x,t)}=E_N\oplus F_N$ with $\dim E_N = i$ that is invariant for the linear Poincar\'e flow  in the sense that for every $s\in (0,t)$ and $s'\in (0, t-s)$,
		$$
		\psi_{s'}(E_N(x_s)) = E_N(x_{s+s'});\,\, \psi_{s'}(F_N(x_s)) = F_N(x_{s+s'}).
		$$
		\item The splitting is dominated in the sense that for all $s\in (0,t) \mbox{ and } s'\in (0,t-s),$ one has
		\begin{equation}\label{e.dom.spl1}
			\|\psi_{s'}\mid_{E_N(x_s)}\| \cdot\|\psi_{-{s'}}\mid_{F_N(x_{s+s'})}\| < C\lambda^{-t}.
		\end{equation}
	\end{enumerate}
\end{lemma}

Then, we establish exponential contraction / expansion along $E_N$ and $F_N$, respectively, for finite orbit segments in $U$ that start and end outside a small neighborhood of $\Sing_\Lambda(X)$.

\begin{lemma}\label{l.robust.hyp}
	Under the assumptions of Lemma~\ref{l.nbhd.1}, one can further shrink $r_0$ and $U$ so that there exists a constant $\eta>1$, such that for every open isolating neighborhood $W$ of $\Sing_\Lambda(X)$ that is contained in $B_{r_0}(\Sing_\Lambda(X))$, there exists a constant $T_{W}>0$ such that the following hold:
	
	\noindent For orbit segments $(x,t)\in U$, assume that 
	\begin{enumerate}
		\item there exist $t_1,t_2\ge 0$ such that $(x_{-t_1}, t_1+t+t_2)\in \cO(U)$ and $x_{-t_1}\notin B_{r_0}(\Sing_\Lambda(X)), x_{t+t_2}\notin B_{r_0}(\Sing_\Lambda(X))$; in particular, by Lemma~\ref{l.robust.dom.spl}, the ``larger'' orbit segment $(x_{-t_1}, t_1+t+t_2)$ that contains $(x,t)$ has a dominated splitting $E_N\oplus F_N$ on its normal bundle;
		\item $x,x_t\notin W$ and $t>T_{W}$.
	\end{enumerate}
	Then we have
	\begin{equation}\label{e.hyp1}
		\prod_{i=0}^{\floor{t}-1}\left\|\psi_1|_{E_N(x_i)}\right\|\le \eta^{-\floor{t}},   \prod_{i=0}^{\floor{t}-1}\big\|\psi_{-1}|_{F_N((x_{\floor{t}-i}))}\big\|\le \eta^{-t},
	\end{equation}	
	\begin{equation}\label{e.hyp2}
		\prod_{i=0}^{\floor{t}-1}\left\|\psi^*_1|_{E_N(x_i)}\right\|\le \eta^{-\floor{t}}, \mbox{ and } \prod_{i=0}^{\floor{t}-1}\big\|\psi_{-1}^*|_{F_N((x_{\floor{t}-i}))}\big\|\le \eta^{-t},
	\end{equation}
\end{lemma}

\begin{remark}\label{r.U}
	The neighborhood $U$ in Lemma \ref{l.nbhd.1} to \ref{l.robust.hyp} can be replaced by any open neighborhood of $\Lambda$ that is contained in $U$. 
\end{remark}

The previous two lemmas states that for any finite orbit segment in $U$ (not necessarily in $\Lambda$) that starts outside $B_{r_0}(\Sing_\Lambda)(X)$, there exists a dominated splitting on the normal bundle of $(x,t)$; furthermore, the exponential contraction / expansion behavior as in Definition \ref{d.multising}, Equations \eqref{e.hyp0} and \eqref{e.hyp} remain true as long as $x$ and $x_t$ are taken uniformly away from singularity, and the orbit segment is sufficiently long.

The proof of these lemmas can be found in Appendix \ref{A.multi-sing}. We remark that these results do not immediately follow from Definition \ref{d.sing.dom.spl} and \ref{d.multising} due to the lack of compactness and uniform estimates on the normal bundle of $\Lambda$. Indeed their proof requires the original definition of Bonatti and da Luz (Definition \ref{d.multising.BD} in Appendix \ref{A.multi-sing}) and is a result of the compactness of $\mathfrak B(\Lambda)$ in the Grassmannian manifold.  Indeed, we will show that by choose $U$ carefully, for any orbit segment $(x,t)\in\cO(U)$ that start outside $B_{r_0}(\Sing_\Lambda(X))$, $\zeta((x,t))$ is an orbit segment for the extended flow on $G^1$ that is contained in a small neighborhood of $\mathfrak{B}(\Lambda)$ where $\zeta$ is given by \eqref{e.zeta} and $\mathfrak{B}(\Lambda)$ is defined by \eqref{e.BLambda}. Then the result will follow from the hyperbolicity for (the extension of) the  splitting $E_N\oplus F_N$ on the compact set $\mathfrak{B}(\Lambda)$.

\subsection{Chain recurrence classes of generic vector fields}\label{ss.crc}
In this section we collect some well-known results on the chain recurrence classes of $C^1$ generic vector fields. Here a property is called $C^1$ generic, if it is satisfied by a $C^1$ residual set of vector fields, i.e., it is satisfied on a dense $G_\delta$ set.

Following the famous work of Conley \cite{Con}, an $\vep-$chain from $x$ to $y$ is a sequence of points $x=x^0,\ldots, x^n = y$ together with times $t_i\ge 1$, $i=0,\ldots, n-1$ such that $d((x^{i})_{t_{i}}, x^{i+1})<\vep$, for all $i=0,\ldots, n-1$. We say that $y$ is chain attainable from $x$, or $x\mapsto y$, if there exists an $\vep-$chain from $x$ to $y$ for every $\vep>0$. This defines a relation: 
$$
x\sim y \mbox{ if and only if } x\mapsto y \mbox{ and } y\mapsto x.
$$
This relation is not necessarily an equivalence relation since it may not be reflexive. To solve this issue, one considers the {\em chain recurrent set} 
$$
\CR(X) = \{x\in\bM: x\sim x\}.
$$
Then ``$\sim$'' is an equivalence relation on $\CR(X)$, whose equivalence classes are called chain recurrence classes. For each $x\in\CR(X)$ we will denote by $C(x)$ the chain recurrence class that contains $x$.

It is easy to see from the definition that for a sequence of chain recurrence classes of a given vector field, the Hausdorff limit is contained in a chain recurrence class. More generally, chain recurrence classes vary upper semi-continuously: if $\{Y_n\}$ is a sequence of vector fields converging to $X$ in $C^1$ topology, and $C_n$ are chain recurrence classes of $Y_n$, then the Hausdorff limit of $C_n$ is contained in a chain recurrence class of $X$.

We also remark that by the Poincar\'e Recurrence Theorem, the support of every ergodic invariant measure $\mu$ is contained in some chain recurrence class.

Below we gather some well-known results concerning the chain recurrence classes of $C^1$ generic vector fields.  

\begin{lemma}\label{l.generic}
	The following properties are $C^1$ generic:
	\begin{enumerate}
		\item (Kupka-Smale Theorem) $X$ is Kupka-Smale: every critical element is hyperbolic, and the stable manifold of any critical element intersects the unstable manifold of any other critical element transversely.
		\item (\cite{BC}) if a chain recurrence class $C$ contains a periodic orbit $\gamma$, then $C$ coincides with the homoclinic class $H(\gamma)$.
		\item (\cite{Cr06}) Every non-trivial chain recurrence class is the Hausdorff limit of a sequence of periodic orbits.
		\item %\bl{
		(\cite{Abd}) For each hyperbolic saddle $\gamma$, the homoclinic class of $\gamma$ depends continuously on $X$ and the continuation of $\gamma$.%}
	\end{enumerate}
\end{lemma}
We will revisit Lemma~\ref{l.generic} for star flows shortly. See Theorem~\ref{t.dichotomy} below.

\subsection{Singular star flows}\label{ss.star}
We start with the precise definition of the star property.
\begin{definition}\label{d.star}
	We say that a $C^1$ vector field $X\in\mathscr{X}^1(\bM)$ has the star property, if there exists a $C^1$ neighborhood $\cU$ of $X$ such that for every $Y\in\cU$, every critical element (i.e., singularity and periodic orbit) of $Y$ is hyperbolic. 
\end{definition}

It has been shown for diffeomorphisms and non-singular flows that the star property implies uniform hyperbolicity. 
For singular flows, however, the situation is much more complicated. These flows are generally not structure stable and not uniformly hyperbolic. An attempt to characterize the hyperbolicity of singular star flows was first made in~\cite{MPP99} where singular hyperbolicity was proposed. It was proven that in dimension three, $C^1$ generically, the star property is equivalent to singular hyperbolicity~\cite{SGW}. However, the situation is very different in higher dimensions. This is caused by the coexistence of singularities of difference indices (in which case they can only differ by one, due to \cite{SGW}) in the same chain recurrence class; see \cite{SGW}. Such a robust example was constructed in~\cite{daLuz} in dimension five. Then in~\cite{BD21}, the following theorem was proven.

\begin{theorem}\label{t.BD}\cite[Theorem 3]{BD21}
$C^1$ open and densely among star flows, the chain recurrent set $\CR(X)$ is contained in the union of finitely many pairwise disjoint filtrating regions in which X is multi-singular hyperbolic.
\end{theorem}
This theorem shows that multi-singular hyperbolicity is indeed the counterpart of uniformly hyperbolicity when the system has singularities. 

For Theorems \ref{mc.C} to \ref{mc.A} on star flows, we need the following theorem concerning the topological structure of multi-singular hyperbolic chain recurrence classes. 

\begin{theorem}\cite[Theorem A]{PYY23a}\label{t.dichotomy}
	There is a residual set $\cR$ of $C^1$ star flows, such that for every $X \in \cR$ and every non-trivial chain recurrence class $C$ of $X$, we have 
	\begin{enumerate}
		\item if $h_{top}(X|_C)>0$, then $C$ contains some periodic point $p$ and is an isolated homoclinic class;
		\item if $h_{top}(X|_C)=0$, then $C$ is sectional-hyperbolic for $X$ or $-X$, and contains no periodic orbits. In this case, every ergodic invariant measure $\mu$ with $\supp\mu\subset C$ must satisfy $\mu = \delta_{\sigma}$ for some $\sigma\in \Sing_C(X)$.
	\end{enumerate}
\end{theorem}

So now Conjecture \ref{c.sdt} and \ref{c.sdt1} becomes:
\begin{conjecture}\label{c.aperiodic}
	Case (2) of Theorem~\ref{t.dichotomy} does not exist for generic singular star flows. In particular, every non-trivial sectional hyperbolic chain recurrence class must support an ergodic invariant measure that is not the point mass of a singularity. 
\end{conjecture}

\subsection{Liao's theory on the scaled linear Poincar\'e flow}\label{ss.Liao}
Below we list a number of results concerning the properties of Liao's tubular neighborhood and the scaled linear Poincar\'e flow $(\psi^*_t)_t$. A comprehensive discussion as well as the proofs can be found in~\cite[Section 2.3]{PYY23}. 

Recall the definition of the linear Poincar\'e flow and its scaling from~\eqref{e.Poin} and~\eqref{e.scaledP}, and note that they are not uniformly continuous since the flow direction $\langle X\rangle$ is only defined on the open set $\Reg(X)$ and is not uniformly continuous. To solve this issue, Liao introduced a tubular neighborhood theory for flow orbits, and prove that the linear Poincar\'e flow is uniformly continuous within the scale $\rho|X|$ (see Proposition~\ref{p.tubular3} for the precise statement). For this reason, in order to establish properties for which the uniform continuity is crucial (e.g., existence of the invariant manifolds, local product structure, and the Bowen property) we need to consider flow orbits that are within $\rho|X|$ scale of each other. This motivates the definition of the scaled shadowing property (Definition~\ref{d.shadow1}).

We start with the boundedness of the (scaled) linear Poincar\'e flow. 
\begin{lemma}\label{l.scaledflow} \cite[Lemma 2.1]{GY} (see also \cite[Lemma 2.1]{PYY23}) \label{l.psi.bdd}
	Let $X$ be a $C^1$ vector field. Then, for any $\tau>0$, there exists $C_\tau>0$ such that for any $t\in[-\tau,\tau]$,
	$$
	\|\psi_t\|\le C_\tau, \mbox{ and }\|\psi^*_t\|\le C_\tau.
	$$
\end{lemma}

Next, we define, for a regular point $x$, and a constant $\mathfrak{d}_0>0$ less than the injectivity radius of $\bM$,
$$
\cN(x) = \exp_x(N_{\mathfrak{d}_0}(x)),
$$
and 
$$
\cN_{\rho}(x) =\{y\in \cN(x): d_{\cN(x)}(x,y)<\rho\}
$$
the projection of the normal plane to the manifold $\bM$.  Note that {\em a priori} $\cN(x)$ may contain a singularity.

To avoid such an issue, for every $x\in\Reg(X)$ and $\rho>0$, denote by 
$$
U_{\rho|X(x)|}(x)   = \left\{v+tX(x)\in T_x\bM: v\in N(x), |v|\le \rho |X(x)|, |t|<\rho\right\}
$$
a flow box on the tangent space of $x$ with size $\rho|X(x)|$. Then we define the map:
\begin{equation}\label{e.Fx}
	F_x: U_{\rho|X(x)|}(x)\to\bM,\,\, F_x(v+tX(x)) = f_t(\exp_{x}(v)),
\end{equation} 

\begin{proposition}\label{p.Fx}\cite[Proposition 2.2]{WW}
	For any $C^1$ vector field $X$ on $\bM$, there exists $\overline\rho_0>0$ such that for any regular point $x\in\Reg(X)$, the map $F_x: U_{\overline\rho_0|X(x)|}(x)\to \bM$ is an embedding whose image contains no singularity of $X$, and satisfies $m(D_pF_x)\ge\frac13$ and $\|D_pF_x\|\le 3$ for every $p\in  U_{\overline\rho_0|X(x)|}(x)$.
\end{proposition}

\begin{remark}\label{r.Liao.robust}
	It is worth noting that the constant $\overline\rho_0$ in Proposition~\ref{p.Fx}, as well as the constant $C_\tau$ in Lemma~\ref{l.scaledflow}, can be chosen continuously with respect to $X$ in $C^1$ topology. Indeed, in the proof of~\cite[Proposition 2.1 and 2.2]{WW}, the constant $\overline \rho_0$ is taken to be $1/(10L_X)$, where 
	\begin{equation}\label{e.lip}
		L_X = \sup\{\|DX\|\}.
	\end{equation}
\end{remark}

For $\rho_0=\frac{\overline\rho_0}{3}$, let
\begin{equation}\label{e.Px1}
	\cP_x:  B_{\rho_0|X(x)|}(x)\to \cN_{\rho_0|X(x)|}(x),\,\, \cP_x(y) = F_x(v) = \exp_x(v). 
\end{equation}
In other words, $\cP_x$ projects every point $y\in B_{\rho_0|X(x)|}(x)$ to the normal plane $\cN_{\rho_0|X(x)|}(x)$ along the flow line through the point $y$. In particular, there exists a function 
$$\tau_x(y): B_{\rho_0|X(x)|}(x)\to [-\overline\rho_0,\overline\rho_0]$$
such that 
\begin{equation}\label{e.Px2}
	\cP_x(y) = f_{\tau_x(y)}(y). 
\end{equation}

Next, consider the holonomy along flow orbits. In order for such holonomy maps to be well-defined, one must avoid all singularities (i.e., one must work with the scale $\rho|X|$ for some $\rho$ sufficiently small). This is taken care of by the next proposition.

\begin{proposition}\label{p.tubular}\cite[Lemma 2.2]{GY}
	There exists $\rho_0>0$ and $K_0>1$, such that for every $\rho\le\rho_0$ and every regular point $x$ of $X$, the holonomy map along flow lines
	$$
	\cP_{1,x}: \cN_{\rho K_0^{-1} |X(x)|}(x) \to \cN_{\rho|X(x_1)|}(x_1)
	$$ 
	is well-defined, differentiable, and injective. %Moreover, for $y\in \cN_{\rho K_0^{-1} |X(x)|}(x)$,  the orbit segment from $y$ to $\cP_{1,x}(y)$ is contained in  $B_\rho^*(x,1)$.
\end{proposition}
Applying Proposition~\ref{p.tubular} recursively, we obtain
\begin{proposition}\label{p.tubular1}
	For every $\rho\le\rho_0$, every regular point $x$ of $X$ and $T\in\NN$, the holonomy map along flow lines
	$$
	\cP_{T,x}: \cN_{\rho K_0^{-T} |X(x)|}(x) \to \cN_{\rho|X(x_1)|}(x_1)
	$$ 
	is well-defined, differentiable, and injective.
\end{proposition}

\begin{remark}\label{r.Liao.robust1}
Note  that $\cP_{t,x}(x) = x_t$; however, $\cP_{t,x}(y)$ is generally different from $y_t$ for $y\in  \cN_{\rho_0K_0^{-1}  |X(x)|}(x)$. To deal with this issue, let us give an alternate definition for $\cP_{t,x}$. Let $L_X = \sup\{\|DX\|\}$ as before; It follows (\cite[p.3191]{WW}) that
\begin{equation}\label{e.ballinclusion}
	f_t\left(B_{e^{-2L|t|}\rho_0|X(x)|}(x)\right)\subset B_{\rho_0|X(x_t)|}(x_t). 
\end{equation}
Then one can define the sectional Poincar\'e map $\cP_{t,x}$ as
\begin{equation}\label{e.Ptx1}
	\cP_{t,x}: \cN_{e^{-2L|t|}\rho_0  |X(x)|}(x) \to \cN_{\rho_0|X(x_t)|}(x_t),\,\, \cP_{t,x} = \cP_{x_t}\circ f_t.
\end{equation}
It is straightforward to check that this definition of $\cP_{t,x}$ is consistent with the previous one, and $K_0$ in Proposition~\ref{p.tubular} can be taken as $e^{2L}$. As a result, $K_0$ can be chosen uniformly in a small neighborhood of $X$ in $C^1$ topology.
\end{remark}

For $x\in\Reg(X), y\in \cN_{\rho_0K_0^{-1}  |X(x)|}(x)$ and $t\in [0,1]$, we define
\begin{equation}\label{e.tau}
	\tau_{x,y}(t) = t+\tau_{x_t}(y_t)
\end{equation}
where $\tau_{x_t}(y_t)$ is given by~\eqref{e.Px2}. $\tau_{x,y}(t)$ is well-defined due to~\eqref{e.ballinclusion}. Then \eqref{e.Ptx1} implies that
\begin{equation}\label{e.tau1}
	f_{\tau_{x,y}(t)}(y)=\cP_{t,x}(y)\in \cN_{\rho|X(x_t)|}(x_t), \forall t\in [0,1].
\end{equation}

The next lemma controls the derivative of $\tau_{x,y}(1)$.
\begin{lemma}\label{l.tau}\cite[Lemma 2.7]{PYY23}
	For every $x\in\Reg(X)$ and $t\in [0,1]$, $\tau_{x,y}(t)$ is differentiable as a function of $y\in \cN_{\rho_0K_0^{-1}|X(x)|}(x)$; furthermore, there exits $K_\tau>0$ such that for all $x\in \Reg(X)$ it holds
	\begin{equation}\label{e.Ktau}
		|X(x)|\cdot \sup\left\{\left\|D_y \tau_{x,y}(1)\right\| : y\in\cN_{\rho_0K_0^{-1}|X(x)|}(x) \right\}\le K_\tau.
	\end{equation}
\end{lemma}

\begin{remark}\label{r.Ktau.robust}
	In the proof of \cite[Lemma 2.7]{PYY23}. the constant $K_\tau$ was taken to be  $3\sup\|Df_1\|  \cdot \sup\|Df_{-1}\|$. Then it is clear that $K_\tau$ can be chosen continuously w.r.t.\,$X$ in $C^1$ topology.
\end{remark}

Next, let us consider the smoothness of $\cP_{1,x}$. Fix $\overline\rho_0$ given by Proposition~\ref{p.Fx}, we lift the holonomy map $\cP_{t,x}$ to the normal bundle and scale it to obtain
\begin{equation}\label{e.p}
	P_{t,x}: N_{\overline\rho_0K_0^{-1}|X(x)|}(x)\to N_{\overline\rho_0|X(x)|}(x),\,\,P_{t,x} = \exp_{x_t}^{-1}\circ\cP_{t,x}\circ\exp_x,
\end{equation}
and 
\begin{equation}\label{e.p*}
	P_{t,x}^*: N_{\overline\rho_0K_0^{-1}}(x)\to N_{\overline\rho_0}(x),\,\,P_{t,x}^*(y) :=\frac{P_{t,x}(y|X(x)|)}{|X(x_t)|}.
\end{equation}
 We have $D_x\cP_{t,x} = D_0P_{t,x} = \psi_t(x)$, and $D_0P_{t,x}^* = \psi^*_t(x)$. It should also be noted that the domain of  $P^*_{1,x}$ has a uniform size $\overline \rho_0 K_0^{-1}$ independent of $x$. The following propositions present some uniform estimates for $DP_{t,x}$ and $DP_{t,x}^*$.
\begin{proposition}\cite[Lemma 2.3 and 2.4]{GY} \cite[Lemma 2.9]{PYY23}\label{p.tubular3}
	The following results holds at every regular point $x\in \Reg(X)$ with all constants $a,\rho,\rho_1,K_1',K_1^*$ uniformly in $x$:
	\begin{enumerate}
		\item 	$DP_{1,x}$ is uniformly continuous at a uniformly relative scale in the following sense: for every $a>0$ and $\rho\in [0,\rho_0]$, there exists $0<\rho_1<\rho$ such that if $y,z\in \cN_{\rho K_0^{-1}|X(x)|}(x)$ with $d(y,y')<\rho_1 K_0^{-1} |X(x)|$, then we have
		$$
		\|D_y\cP_{1,x} - D_{z}\cP_{1,x}\|<a.
		$$
		\item $DP^*_{1,x}$ is uniformly continuous at a uniform scale (not just relative!) in the following sense: for every $a>0$ there exists $\rho_1>0$ such that for $y,z\in N(x)$, if $d(y,z)<  \rho_1$ then 
		$$
		\|D_yP^*_{1,x} - D_{z}P^*_{1,x}\|<a.
		$$
		\item  there exists $K_1'>0$ such that 
		$$
		\|D\cP_{1,x}\|\le K_1',\,\, \mbox{ and }\,\,	\|D(\cP_{1,x})^{-1}\|\le K_1'.
		$$
		\item there exists $K_1^*\ge K_1'$ such that  $\|DP^*_{1,x}\|\le K^*_1$;
		\item for every $\rho\in(0,\rho_0 K_0^{-1}]$, there exists $\rho'\in(0,\rho]$ such that for every  regular point $x$, we have 
		$$
		\cP_{1,x}\left(\cN_{\rho |X(x)|}(x)\right)\supset \cN_{\rho'|X(x_1)|}(x_1).
		$$
	\end{enumerate}
\end{proposition}

Next, we introduce the concept of {\em $\rho$-scaled shadowing} which tracks a point $y$ as it moves alongside the orbit of $x$ withing the scale $\rho|X|$. This definition is motivated by~\eqref{e.tau1}.
\begin{definition}\label{d.shadow1}
	For $0<\rho\le \rho_0$, we say that the orbit of $y$ is {\em $\rho$-scaled shadowed} by the orbit of $x$ up to  time $T\in(0,+\infty)$, if there exists a strictly increasing, continuous function
	$$
	\tau_{x,y}(t): [0,T]\to [0,\infty) 
	$$ 
	with $\tau_{x,y}(0) = 0$, such that for every $t\in[0,T]$, it holds
	$$
	f_{\tau_{x,y}(t)}(y)\in \cN_{\rho|X(x_t)|}(x_t). 
	$$
\end{definition}		
Note that the definition above requires $y\in \cN_{\rho|X(x)|}(x)$. Also, we have 
$$
f_{\tau_{x,y}(t)}(y) = \cP_{t,x}(y);
$$
that is, this definition of $\tau_{x,y}$ is consistent with Equation \eqref{e.tau}.

\begin{remark}\label{r.shadowing}
	By Proposition~\ref{p.tubular} and Proposition~\ref{p.tubular1} and the discussion following them,  for $T>0$, if $y\in \cN_{\rho K_0^{-T}|X(x)|}(x)$, then the orbit of $y$ is $\rho$-scaled shadowed by the orbit of $x$ up to  time $T$.
\end{remark}
\begin{lemma}\label{l.shadowing}\cite[Lemma 2.11]{PYY23}
	Let $\rho\in (0,\rho_0]$, and assume that the orbit of $y$ is $\rho$-scaled shadowed by the orbit of $x$ up to  time $t$. Then the following statements hold:
	\begin{enumerate}
		
		\item For every $0\le s < s' \le t$, the orbit of $\cP_{s,x}(y)$ is $\rho$-scaled shadowed by the orbit of $x_{s}$ up to   time $s'-s$. 
		\item If, in addition, that the orbit of $\cP_{t,x}(y)$ is $\rho$-scaled shadowed by the orbit of $x_t$ up to  time $t'$, then the orbit of $y$ is $\rho$-scaled shadowed by the orbit of $x$ up to  time $t+t'$.
	\end{enumerate}
\end{lemma}
The next lemma will be useful when establishing the scaled shadowing property.
\begin{lemma}\label{l.shadowing2}\cite[Lemma 2.12]{PYY23}
	Let $\rho\in (0,\rho_0]$ and $x\in\Reg(X)$. Assume that the point $y\in \cN_{\rho K_0^{-1}|X(x)| }(x)$ satisfies the following property:  for every $k\in [1,\floor{t}]\cap\NN$, %\footnote{Here $\floor{t}$ denotes the integer part of $t$.} 
	the point 
	$$
	y^{k,*} = \cP_{1,x_{k-1}}\circ\cdots\circ\cP_{1,x}(y)
	$$
	exists and is contained in $\cN_{\rho K_0^{-1}|X(x_k)|}(x_k)$.  Then the orbit of $y$ is $\rho$-scaled shadowed by the orbit of $x$ up to  time $\floor{t}+1$.
\end{lemma}

We conclude this section by the following remark on the robustness of the constants on the vector field $X$.

\begin{remark}\label{r.Liao.robust2}
All the constants in the results of this subsection, namely $ C_\tau, \rho_0,$ $\rho_1,  K_\tau, K_0, K_1', K_1^*$ can be chosen continuously with respect to the vector field $X$ in $C^1$ topology. See Remark~\ref{r.Liao.robust}, Remark~\ref{r.Liao.robust1} and the proof of~\cite[Lemma 2.7]{PYY23}.
\end{remark}

\subsection{Almost Expansivity} 
\begin{definition}
	Given a vector field $X$, $x\in\bM$ and $\vep>0$, we define the {\em bi-infinite Bowen ball} to be 
	$$
	\Gamma_\vep = \{y\in\bM: d(x_t,y_t)\le \vep, \forall t\in\RR\}. 
	$$
	We say that $X$ is {\em almost expansive at scale $\vep>0$}, if the set 
	$$
	\Exp(\vep): = \{x\in\bM: \Gamma_\vep(x)\subset (x_{-s},2s) \mbox{ for some } s=s(x)>0\} 
	$$
	satisfies $\mu(\Exp(\vep))=1$ for every $X$-invariant probability measure $\mu$.
	
	Finally, we say that $X$ is $C^r$ robustly almost expansive at scale $\vep>0$ (for a given $r\ge 1$), if there exists a $C^r$ neighborhood $\cU$ of $X$, such that every $Y\in\cU$ is almost expansive at scale $\vep$.
\end{definition}
In this paper, we shall only consider the case $r=1$.

By~\cite[Proposition 2.4]{LVY}, (robust) almost expansivity at scale $\vep$ implies (robust) entropy expansivity at the same scale. By~\cite{B72}, the metric entropy varies upper semi-continuously as a function of the invariant measure under weak-$^*$ topology.\footnote{In \cite{SYY} it is proven that, for singular flows away from homoclinic tangencies, the metric entropy varies upper semi-continuously provided that the limit measure $\mu$ satisfies $\mu(\Sing(X)) = 0$. Here we remove this extra assumption for star flows.} This leads to the following proposition.

\begin{proposition}\label{p.existence}
	If $X$ is almost expansive at some scale $\vep>0$, then every continuous potential function has an equilibrium state.
\end{proposition}

{
\section{A Further Improved Climenhaga-Thompson Criterion}\label{s.CT}
The original Climenhaga-Thompson Criterion \cite{CT16} and its improvement in \cite{PYY23} provide a verifiable way to prove the uniqueness of equilibrium states for continuous flows on metric spaces. In this article, however, we need a further improved criterion  that we recently obtained in \cite{PYYY24} which applies to the maximal invariant set $\Lambda$ in an open set $U$.

In this section, we assume that $(f_t)_t$ is a continuous flow on a compact metric space $\bM$, and $\Lambda$ is a compact invariant set of $(f_t)_t$ with an isolating  neighborhood $U$; that is
$$
\Lambda = \bigcap_{t\in\RR} f_t(U).
$$ 
Note that every neighborhood of $\Lambda$ in $U$ is also an isolating neighborhood of $\Lambda$. In particular, we take $U_1$ a neighborhoods of $\Lambda$ such that 
$$
\Lambda \subset (U_1)^\circ\subset U_1 \subset U^\circ\subset U.
$$

We consider the following types of orbit segments:
\begin{itemize}
	\item For a set $I\subset \RR$ we write $f_{I}(x) = \{f_sx: s\in I\}$. 
	\item Every element $(x,t)\in \Lambda\times\RR^+$ is identified with the orbit segment $f_{[0,t)}(x) = \{f_sx: s\in [0,t)\}$; these are finite orbit segments that are contained in $\Lambda$. With slight abuse of notation we consider $(x,0)$ as the empty set rather than the singleton $\{x\}$. 
	\item We write $\cO(U_1)$ for the subset of $U_1\times \RR$ such that every $(x,t)\in \cO(U_1)$ satisfies
	$$
	f_sx\in U_1, \forall s\in [0,t);
	$$
	then we identify $(x,t)$ with the orbit segment $f_{[0,t)}(x)$. In other words, $\cO(U_1)$ consists of finite orbit segments that are entirely contained in $U_1$.
	\item $\cO(U_2)$ is defined similarly. 
\end{itemize}
It follows from invariance that
$$
\Lambda\times \RR\subset \cO(U_1) \subset \cO(U).
$$

Next we consider the pressure on a collection of finite orbit segments. For an orbit segment collection $\cC$, 
we define 
$$
(\cC)_t = \{x: (x,t)\in \cC\}. 
$$
Let $\phi: U\to \RR$ be a continuous function, and  $\delta>0,\vep>0$ be two positive (small) constants. 
We write 
\begin{equation}\label{e.Phi}
	\Phi_\vep(x,t) = \sup_{y\in B_{t,\vep}(x)}\int_0^t\phi(f_s(y))\,ds,
\end{equation}
with $\vep = 0$ being the standard Birkhoff integral
$$
\Phi_0(x,t) =\int_0^t\phi(f_s(y))\,ds.
$$
Putting $\Var(\phi,\vep) = \sup\{|\phi(x)-\phi(y)|: d(x,y)<\vep\}$, we obtain the trivial bound
$$
|\Phi_\vep(x,t) - \Phi_0(x,t)|\le t\Var(\phi,\vep).
$$
The {\em (two-scale) partition function} $\Lambda(\cC,\phi,\delta,\vep,t)$ for a collection of finite orbit segments $\cC$ and $t>0$ is defined as 
\begin{equation}\label{e.Lambda}
	\Lambda(\cC,\phi,\delta,\vep,t)=\sup\left\{\sum_{x\in E}e^{\Phi_\vep(x,t)}:E\subset (\cC)_t \mbox{ is $(t,\delta)$-separated}\right\}.
\end{equation}
%Henceforth we shall always assume that $\delta,\vep$ are taken small enough so that
%\begin{equation}\label{e.smallconstant}
%	B_*(\Lambda)\subset U_1,\,\, \mbox{ and }\,\, B_*(U_1)\subset U,\,\,* = \delta,\vep,
%\end{equation}
%where for a positive number $r>0$ and a set $A$, $B_r(A)$ is the $r$-neighborhood of $A$:
%$$
%B_r(A) = \bigcup_{x\in A} B_r(x). 
%$$

The pressure of $\phi$ on $\cC$ with scale $\delta,\vep$ is defined as
\begin{equation}\label{e.P1}
	P(\cC,\phi,\delta,\vep) = \limsup_{t\to+\infty}\frac1t \log\Lambda(\cC,\phi,\delta,\vep,t).
\end{equation}
The monotonicity of the partition function in both $\delta$ and $\vep$ can be naturally translated to $P$. When $\vep=0$ we will often write $P(\cC,\phi,\delta)$, and let
\begin{equation}\label{e.P2}
	P(\cC,\phi) = \lim_{\delta\to0} P(\cC,\phi,\delta). 
\end{equation}
When $\cC =  \Lambda\times \RR^+$, this coincides with the standard definition of the topological pressure $P(\phi,X|_\Lambda)$ for the restriction of the flow on $\Lambda$.

\begin{definition}\cite[Definition 2.3]{CT16}
	A decomposition  $(\cP,\cG,\cS)$ for a collection of finite orbit segments $\cD$ consists of three collections $\cP,\cG,\cS$ and three functions $p,g,s:\cD\to \RR^+$ such that for every $(x,t)\in \cD$, the values $p=p(x,t), g=g(x,t)$ and $s=s(x,t)$ satisfy $t=p+g+s$, and
	\begin{equation}\label{e.decomp}
		(x,p)\in\cP,\hspace{0.5cm} (f_px,g)\in\cG,\hspace{0.5cm} (f_{p+g}x,s)\in\cS.
	\end{equation}
	Given a decomposition $(\cP,\cG,\cS)$ and a real number $M\ge0$, we write $\cG^M$ for the set of orbit segments $(x,t)\in \cD$ with $p\le  M$ and $s\le  M$.
\end{definition}

	As mentioned earlier, in this article we must deal with three ``spaces'' of finite orbit segments: $\Lambda\times\RR^+\subset \cO(U_1)\subset \cO(U)$. Suppose there exists a $(\cP_{1},\cG_1,\cS_1)-$decomposition on $\cD_1\subset \cO(U_1)$ (here we use the subscript to highlight the fact that the decomposition is defined for orbit segments in $\cO(U_1)$); then the following properties hold:
	\begin{enumerate}
		\item $\cP_1,\cG_1,\cS_1$ are subsets of $\cO(U_1)$; consequently, orbit segments in them are completely contained in $U_1$;
		\item by restricting the functions $p,g,s$ to $\cD_0:= \cD_1\cap\Lambda\times\RR^+$ {(which may be empty)} and considering $\cP_0 := \cP_1\cap\Lambda\times\RR^+$ and similarly defining $\cG_0,\cS_0$, one obtains a  $(\cP_0,\cG_0,\cS_0)-$decomposition of $\cD_0\subset \Lambda\times\RR^+$. 
	\end{enumerate}

%We assume that $\bM\times \{0\}$ (whose elements are identified with empty sets) belongs to $\cP\cap\cG\cap\cS$. This allows us to decompose orbit segments in trivial ways. 
Following~\cite[(2.9)]{CT16}, for $\cC\in\bM\times\RR^+$ we define the slightly larger collection $[\cC]\supset \cC$ to be
\begin{equation}\label{e.[C]}
	[\cC]:= \{(x,n)\in \bM\times \mathbb{N}:(f_{-s}x,n+s+t) \in\cC \mbox{ for some } s,t\in[0,1)\}.
\end{equation}
This allows us to pass from continuous time to discrete time.

\begin{definition}\label{d.Bowen}
	Given $\cC\subset  \bM\times \RR^+$, a potential $\phi$ is said to have the Bowen property on $\cC$ at scale $\vep>0$, if there exists $K>0$ such that
	\begin{equation}\label{e.Bowen}
		\sup\left\{|\Phi_0(x,t) - \Phi_0(y,t)|:(x,t)\in\cC, y\in B_{t,\vep}(x)\right\}\le K.
	\end{equation}
\end{definition}
\begin{definition}
	We say that $\cG\subset\bM\times\RR^+$ has weak specification at scale $\delta$ if there exists $\tau>0$ such that for every finite orbit collection $\{(x^i,t_i)\}_{i=1}^k\subset \cG$, there exists a point $y\in\bM$ (not necessarily in $\Lambda$, even when $\cG\subset \Lambda\times\RR^+$) and a sequence of ``gluing times'' $\tau_1,\ldots,\tau_{k-1}$ with $\tau_i\le \tau$ such that for $s_j = \sum_{i=1}^{j}t_i+\sum_{i=1}^{j-1}\tau_i$ and $s_0=\tau_0=0$, we have 
	\begin{equation}\label{e.spec}
		d_{t_j}(f_{s_{j-1}+\tau_{j-1}}(y), x^j)<\delta \mbox{ for every } 1\le j\le k.
	\end{equation}
	The constant $\tau = \tau(\delta)$ is referred to as the {\em maximum gap size}.

\end{definition}
\begin{definition}\label{d.tailspec}
	We say that $\cG$ has {\em tail (W)-specification}  at scale $\delta$ if there exists $T_0>0$ such that $\cG\cap(\bM\times [T_0,\infty))$ has (W)-specification at scale $\delta$. We may also say that $\cG$ has (W)-specification at scale $\delta$ for $t>T_0$ if we need to declare the choice of $T_0$.
	%\bl{
	Furthermore, in the case $\cG_0\subset \Lambda\times\RR^+$, we say that $\cG_0$ has weak (tail) specification at scale $\delta$ with shadowing orbits in $U_1$, if $(y,s_k)\in \cO(U_1)$. %}
	%\bl{
	Similarly, we say that $\cG_1\subset \cO(U_1)$ has weak (tail) specification at scale $\delta$ with shadowing orbits in $U$, if $(y,s_k)\in \cO(U)$. %}
\end{definition}

We are finally ready to state the main tool of this paper.

\begin{theorem}\cite[Theorem A]{PYYY24}\label{t.improvedCL1}
	Let $(f_t)_{t\in\RR}$ be a continuous flow on a compact metric space $X$. Assume that $\Lambda$ is a compact invariant set of $(f_t)$ that is isolated with isolating neighborhood $U$, and $\phi: U\to\RR$ a continuous potential function. Suppose that there exist $\vep>0,\delta>0$ with $\vep\ge1000\cdot \delta$, such that  $P_{\exp}^\perp (\phi,\vep;\Lambda)<P(\phi,X|_\Lambda)$. Also assume there exist  a neighborhoods $U_1\subset U$ of $\Lambda$ and $\cD_1\subset \cO(U_1)$ which admits a  $(\cP_1,\cG_1,\cS_1)-$decomposition that induces a decomposition  $(\cP_0,\cG_0,\cS_0)$ of  $\cD_0= \cD_1\cap\Lambda\times\RR^+$ with the following properties: 
	\begin{enumerate}[label={(\Roman*)}]
		\item[($I_0$)] $(\cG_0)^1$ has tail (W)-specification at scale $\delta$ with shadowing orbits contained in $\cO(U_1)$;
			\item [($I_1$)] $(\cG_1)^1$ has tail (W)-specification at scale $\delta$ with shadowing orbits contained in $\cO(U)$;
		\item[(II)]  $\phi$ has the Bowen property at scale $\vep$ on $\cG_1$;
		\item[($III$)] $P((\cO(U_1)\setminus\cD_1)\cup [\cP_1]\cup[\cS_1], \phi,\delta,\vep)<P(\phi,X|_\Lambda)$. 
	\end{enumerate}
	Then there exists a unique equilibrium state for the potential $\phi$ whose support is contained in $\Lambda$ and is ergodic.
\end{theorem}
Here, following our previous notation,  
\begin{equation}\label{e.g11}
	(\cG_1)^1 = \{(x,t)\in \cD_1: p\le 1,s\le 1\},		
\end{equation}
and $(\cG_0)^1  = \{(x,t)\in \cD_0: p\le 1,s\le 1\}= (\cG_1)^1\cap\Lambda\times\RR^+$.

We conclude this section by the following two lemmas. 
\begin{lemma}\cite[Lemma 3.1]{PYYY24}\label{l.equalpressure}
	Let $\Lambda$ be an isolated compact invariant set with isolating neighborhood $U$. Let $U_1$ be any neighborhood of $\Lambda$ that is contained in $U$, and $\phi:U\to\RR$ a continuous function. Then we have
	$$
	P(\cO(U),\phi) = P(\cO(U_1),\phi) = P(\Lambda\times\RR^+, \phi) = P(\phi,X|_\Lambda). 
	$$
\end{lemma}
\begin{proof}
	This is a simple consequence of the proof of the variational principle \cite{Wal}. Note that by local maximality, for any sequence of orbit segments $(x^i,t_i)\in \cO(U_1)$ with $t_i\to\infty$, any limit point of the empirical measure (in weak-* topology) must be supported in $\Lambda$. This implies that $P(\cO(U),\phi)\le P(\phi,X|_\Lambda)$. The reversed inequality follows from inclusion.  
\end{proof}
%We conclude the preliminary by remarking that this version of Theorem~\ref{t.improvedCL1} is sightly different from~\cite[Theorem A]{PYY21} in that the constant $2000$ in~\cite[Theorem A]{PYY21} is replaced by $1000L_X$. The reason behind this can be found in~\cite[Remark 3.4]{PYY23}.

\begin{lemma}\cite[Remark 2.2]{PYYY24}
	Condition $(III)$ in Theorem~\ref{t.improvedCL1} implies the following ``pressure gap'' property on $\Lambda$: 
	$$
	P((\Lambda\times\RR^+\setminus\cD_0)\cup [\cP_0]\cup[\cS_0], \phi,\delta,\vep)<P(\phi,X|_\Lambda).
	$$
	In particular, $\cD_0\ne\emptyset$ and satisfies $P(\cD_0, \phi) =  P(\phi,X|_\Lambda).$
\end{lemma}
\begin{proof}
	The first part of the lemma follows from the observation that 
	$$
	[\cP_0]\subset [\cP_1], [\cS_0]\subset [\cS_1], \text{ and } \Lambda\times\RR^+\setminus\cD_0\subset (\cO(U_1)\setminus\cD_1).
	$$
	For the ``in particular'' part, note that 
	$$P(\phi,X|_\Lambda) = P(\Lambda\times\RR^+,\phi) =  \max\left\{P(\Lambda\times\RR^+\setminus\cD_0,\phi), P(\cD_0,\phi)\right\},
	$$
	and the previous inequality shows that $P(\Lambda\times\RR^+\setminus\cD_0,\phi)<P(\phi,X|_\Lambda)$.
\end{proof}
}

\section{Fake foliation charts on the normal bundle, almost expansivity and the continuity of topological pressure}\label{s.fakefoliation}
In this section we prove Theorem~\ref{m.A} and \ref{m.A1}. The main ingredients are:
\begin{itemize}
	\item the construction of a local product structure on the normal plane of regular points, given by fake foliations tangent to the $(\alpha,E_N)$-cone and $(\alpha,F_N)$-cone, respectively;
	\item controlling the contraction and expansion along fake foliations:
	\begin{itemize}
		\item for orbit segments away from singularities (Lemma~\ref{l.expansion.reg}); and
		\item for orbit segments near singularities (Lemma~\ref{l.Elarge} and \ref{l.Flarge}).
	\end{itemize} 
\end{itemize}
The fake foliation charts constructed in Section~\ref{ss.fake} will also be crucial in the proof of the Bowen property in Section~\ref{s.bowen}. 

We remark that the proof of Theorem~\ref{m.A} does not rely on, and indeed covers the result of~\cite[Theorem A]{PYY}, since every sectional-hyperbolic set is multi-singular hyperbolic.

\subsection{A fake foliation chart on the normal bundle}\label{ss.fake}
The first step is to construct a fake foliation chart on the normal plane $\cN(x)$ of every regular point $x\in\Lambda$, consisting of two foliations $\cF_{x,\cN}^E$ and $\cF_{x,\cN}^F$ tangent to some small cones of $E_N(x)$ and $F_N(x)$, respectively. These foliations will be invariant under the holonomy maps $\cP_{1,x}$. In view of Proposition~\ref{p.tubular3}, they only exist at the scale $\rho_0|X(x)|$. %\bl{
We will also explain how to obtain fake foliations for finite orbit segments $(x,t)\in \cO(U)$ with similar properties. %}

We remark that the construction here is the same as~\cite[Section 6.1 and 6.2]{PYY23}. The only difference is that in~\cite{PYY23} the dominated splitting $E_N\oplus F_N$ was obtain by projecting the sectional-hyperbolic splitting $E^s\oplus F^{cu}$ to the normal bundle (see~\cite[Lemma 2.2]{PYY23}). In this paper, the dominated splitting is given by the definition of multi-singular hyperbolicity.

We start with the Hadamard-Perron Theorem.
\begin{theorem}[Hadamard-Perron Theorem]\cite[Theorem 6.2.8]{Katok} and~\cite[Section 3]{BW}\label{t.fakeleaves} Fix $\zeta<\eta$
	and let  $\{A_x\}_{x\in I}$ be a $GL(n,\RR)$ cocycle over an invertible dynamical system $T:I\to I$ (here we do not assume $I$ to be compact or even a metric space) such that 
	\begin{itemize}
		\item  $\forall x\in I, u\in\RR^k$ and $v\in \RR^{n-k}$, $ A_x(u,v) = (A_x^1(u), A_x^2(v))$;
		\item $\|(A^1_x)^{-1}\|\le \eta^{-1}$, and $\|A_x^2\|\le \zeta$.
	\end{itemize}
	Then for $0<\alpha<\min\{1, \sqrt{\eta/\zeta}-1\}$ and $\iota>0$ sufficiently small, 
	for every $\{f_x\}_{x\in I}$ a family of $C^1$ diffeomorphisms of $\RR^n$ such that 
	$$
	f_x(u,v) = (A_x^1(u)+f_x^1(u,v), A_x^2(v) + f_x^2(u,v))
	$$
	with $\|f_x^1\|_{C^1}<\iota$ and $\|f_x^2\|_{C^1}<\iota$ for all $x\in I$, there exist two foliations $\cF^E_{x}$ and $\cF^{F}_{x}$, with the following properties:
	\begin{enumerate}
		\item (tangent to respective cones) each leaf of $\cF^E_{x}$ is the graph of a $C^1$ function $h_x^E:\RR^{n-k}\to \RR^k$ with $\|Dh_x^E\|_{C^0}<\alpha$ for all $x\in I$; similarly, each leaf of $\cF^{F}_{x}$ is the graph of a $C^1$ function $h_x^{F}:\RR^{k}\to \RR^{n-k}$ with $\|Dh_x^{F}\|_{C^0}<\alpha$;
		\item (invariance) for every $x\in I$ and $y\in\RR^n$, it holds for $*=E,F$:
		$$
		f_x\left(\cF^{*}_{x}(y)\right) = \cF^{*}_{Tx}(f_x(y));
		$$
		%\item (exponential expansion in $\cF^s_x$) let $\zeta' = (1+\alpha)(\zeta+\iota(1+\alpha))< \frac{\eta}{1+\alpha}-\iota :=\eta'$; then $f_x$ exponentially contracts distance in $\cF^s_x$
		\item (product structure) for every $y\in \RR^n$ and every $x\in I$, there exist two unique points $y^E\in \cF^{E}_{x}(0),y^{F}\in \cF^{F}_{x}(0)$ such that 
		$$
		\{y\}=\cF^{F}_{x}(y^E)\pitchfork \cF^{E}_{x}(y^{F}).
		$$ 
	\end{enumerate}
\end{theorem}

\subsubsection{Fake foliations for orbits in $\Lambda$}

Fix $\alpha>0$ small enough \footnote{In this paper we shall not specify how small $\alpha$ needs to be. We invite our readers to~\cite[Section 6.1]{PYY}. Roughly speaking, $\alpha$ need to be chosen so that the transition between different foliation charts within $\alpha$ cones of $E_N$ and $F_N$ can be nicely controlled. Such a control is needed in the proof of Lemma~\ref{l.Elarge} and~\ref{l.Flarge}.  Also see the start of~\cite[Section 6.2.2]{PYY}.}, we apply Theorem~\ref{t.fakeleaves} for this $\alpha$ (decrease it if necessary) to the family of scaled sectional Poincar\'e maps  $\{P^*_{1,x}\}$ defined by~\eqref{e.p*} over the time-one map $T=f_1$ on the index set $I=\Reg(X)$. In particular, we have  $n=\dim \bM-1$, $\RR^k=E_N$ and $\RR^{n-k}=F_N$. Here each normal plane $N(x)$ is identified with $\RR^n$, and the $C^1$ diffeomorphisms $f_x$ are defined as
\begin{equation}\label{e.tildeP}
\tilde P_{1,x} := \begin{cases}
	P_{1,x}^*, &|y|<\upsilon\\
	g_x, &|y|\in[\upsilon,K\upsilon]\\
	\psi^*_{1,x}, &|y|>K\upsilon
\end{cases}
\end{equation}
where $\upsilon$ is taken small enough such that the maps $P^*_{1,x}$ are $\iota$-close to $\psi^*_{1,x}$  under $C^1$ topology where $\iota=a>0$ is given by the previous theorem (see Proposition~\ref{p.tubular3} (2)); also recall that 
$$
D_0\tilde P_{1,x} =D_0P_{1,x}^*=\psi^*_{1,x};
$$
$K>0$ is a constant large enough such that the smooth bump functions $g_x$ satisfy $\|g_x-\psi^*_{1,x}\|_{C^1}<\iota$.  The construction of $\tilde P_{1,x}$ guarantees that 
$$
\|\tilde P_{1,x}-\psi^*_{1,x}\|_{C^1}<\iota,
$$
provided that $\upsilon$ is taken small enough and $K$ sufficiently large.

Then Theorem~\ref{t.fakeleaves} gives us two foliations $\cF^{E,*}_{x,N}$ and $\cF^{F,*}_{x,N}$. Here the sub-index $N$ highlights the fact that these foliations are defined on the (scaled) normal plane $N(x)\subset T_x\bM$ for every $x\in \Reg(X)$. The $*$ in the superscript is due to the foliations being defined using the scaled maps $P^*_{1,x}$. Note that these foliations are tangent to the $(\alpha, E_N)$-cone and $(\alpha, F_N)$-cone, respectively. They are invariant in the sense that 
$$
\tilde P_{1,x}(\cF^E_{x,N}) = \cF^E_{x_1,N},\,\,\, \tilde P_{1,x}(\cF^{F}_{x,N}) = \cF^{F}_{x_1,N}.
$$ 
It is worth pointing out that these foliations are likely not invariant under $\tilde P_{t,x}$ for $t\notin \ZZ$. 

Finally, we rescale these foliations by the flow speed $|X(x)|$ and push them to the manifold $\bM$ via the exponential map $\exp_x$. To be more precise, we define the maps
$$
S_x: N(x)\to N(x), S_x(u) = |X(x)|u
$$
and the foliations
$$
\cF^i_{x,\cN} = \exp_{x}\left(S_x\left(\cF^i_{x,N}\right)\right)\subset \cN_{{\rho_0}|X(x)|}(x) , i=E,F,
$$
Note that $\cF^i_{x,\cN}$ are well-defined on $\cN_{{\rho_0}|X(x)|}(x)$ \footnote{Indeed, these foliations are well-defined at a uniform scale $\cN(x) = \exp_x (N_{\mathfrak{d}_0}(x))$ where $\mathfrak{d}_0$ is the injectivity radius; however, since $\cP_{1,x}$ is only defined at a uniformly relative scale $\rho_0|X(x)|$, the invariance of those foliations only holds at a uniformly relative scale.} and are invariant under the maps $\cP_{1,x}$ in the following sense:
\begin{equation}\label{e.inv}
	\cP_{1,x}\left(\cF^{F}_{x,\cN} (y)\right)\supset \cF^{F}_{x_1,\cN} (\cP_{1,x}(y)),\,\, \mbox{ for } y\in \cN_{\rho_0K_0^{-1}|X(x)|}(x), x\in\Reg(X);
\end{equation}
and a similar statement holds for $\cF^{s}_{x,\cN}$ for the maps $\cP_{-1,x}$. It is also worth pointing out that these foliations are likely not invariant under $\cP_{t,x}$ for $t\notin \ZZ$.

The foliations $\cF^i_{x,\cN}$ form a local product structure on $\cN_{{\rho_0}|X(x)|}(x)$ in the following sense. For every $y\in \cN_{{\rho_0}|X(x)|}(x)$, we can write
$$
y=[y^{E},y^{F}]
$$
where $y^{{E}}\in \cF^E_{x,\cN}(x)$ (this is the leaf of $\cF^E_{x,\cN}$ that contains $x$), $y^{{F}}\in \cF^F_{x,\cN}(x)$ such that 
$$
\{y\}= \cF^F_{x,\cN}(y^{{E}})\pitchfork  \cF^E_{x,\cN}(y^{{F}}).
$$ 
See Figure~\ref{f.EF}.
\begin{figure}[h!]
	\centering
	\def\svgwidth{\columnwidth}
	\includegraphics[scale=0.6]{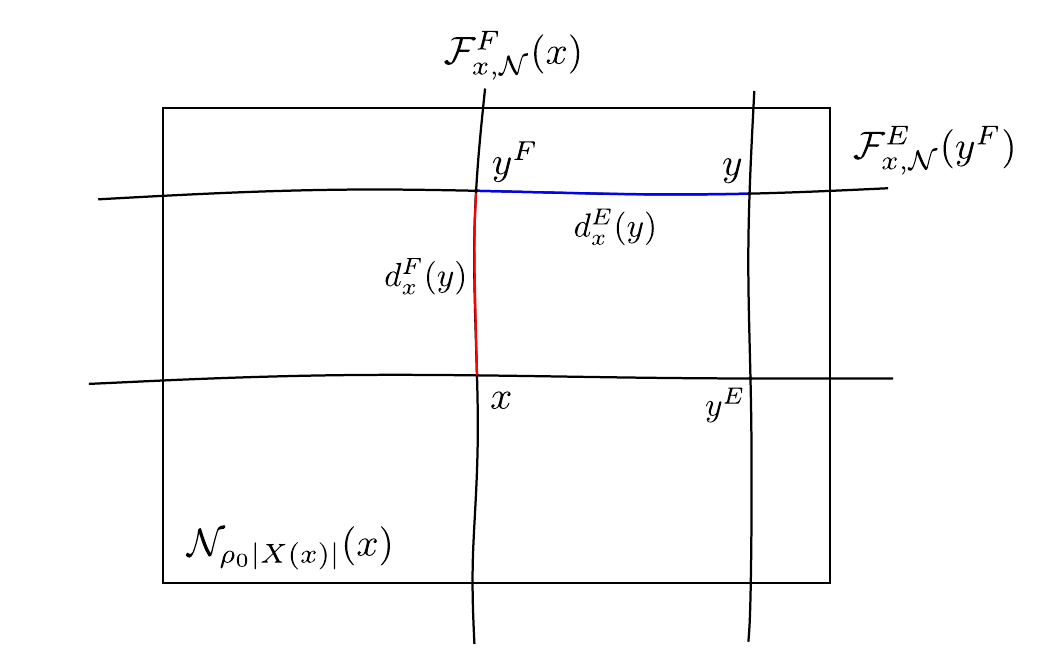}
	\caption{The $E$-length and $F$-length of $y$.}
	\label{f.EF}
\end{figure}

From now on, to simplify notation and highlight the structure of the proof, we make the following definition:
\begin{definition}\label{d.EFcomparison}
	Let $x\in\Reg(X)$ and $y\in \cN_{\rho_0|X(x)|}(x)$. Let $y^E, y^F\in \cN(x)$ be such that $y=[y^E,y^F]$. We define:
	$$
	d^E_x(y) = d_{\cF_{x,\cN}^E}( y^{F}, y),\,\, \mbox{ and }\,\, d^F_x(y) = d_{\cF_{x,\cN}^{F}}( x,y^{F}), 
	$$ 
	where $d_\cF(\cdot,\cdot)$ is the distance within the submanifold $\cF$. 
	
	We will sometimes refer to $d^E_x(y)$ and $d^F_x(y)$ as the $E$-length and $F$-length of $y$. See Figure~\ref{f.EF}. 
\end{definition}

Note that the $E$-length is defined using $y^F$ and $y$ as opposed to $x$ and $y^E$. Alternatively, one could use $x$ and $y^E$ to defined $d^E_x(y)$, and use $y^E$ and $y$ to define $d^F_x(y)$. This does not cause any substantial difference in the rest of this paper. 

%Note that the $E$-length of $y$ is not defined using $x$ and $y^E$.
{
\subsubsection{Fake foliations for finite orbit segments in $\cO(U)$}\label{sss.fakefoliation.nbhd}
Next we consider finite orbit segments that stay close to $\Lambda$. Similar to Section \ref{s.CT} we fix isolating neighborhoods $U$ of $\Lambda$, and denote by $\cO(U)$ the collection of finite orbit segments that are entirely contained in $U$. Furthermore, we shall require that $U$ satisfies Lemma~\ref{l.nbhd.1}, \ref{l.robust.dom.spl} and \ref{l.robust.hyp}. Then for every $(x,t)\in \cO(U)\setminus \Lambda\times\RR^+$ with $x,x_t\notin B_{r_0}(\Sing_\Lambda(X))$ and $r_0$ given by Lemma \ref{l.nbhd.1}, we obtain a dominated splitting $N_{(x,t)} = E_N\oplus F_N$ with contraction / expansion property under the (scaled) linear Poincar\'e flow.

Let $(x,t)$ as above be fixed. To apply the Hadamard-Perron Theorem and obtain fake foliations on the normal plane of points in $(x,t)$, one considers the following family of $C^1$ diffeomorphism:
$$
(f_s)_{s\in \RR} = \begin{cases}
 \tilde P_{1,x_s}, &s\in [0,t);\\
  \psi^*_{1,x_t}, & s\ge t\\
  \psi^*_{1,x}, & s<0
\end{cases}.
$$
where $\tilde P_{1,y}$ is given by \eqref{e.tildeP}. In other words, one ``extends'' the finite cocycle $(f_s)_{s\in[0,t)}$ to an infinite cocycle $(f_s)_{s\in \RR}$ by repeating with the linear map $\psi^*_{1,x_t}$ for $s \ge t$, and repeating with the linear map $\psi^*_{1,x}$ for $s< 0$. Then, Hadamard-Perron Theorem gives foliations $\cF_{x_s,N}^{E, (x,t),*}$ and $\cF_{x_s,N}^{F,(x,t),*}$ for $s\in [0,t)$. These foliations depend not only on the point $x$ but also on the orbit segment $(x,t)$, thus the superscript $(x,t)$. We then rescale these foliations by $|X(x_s)|$ and  pushed to the manifold by the exponential map $\exp_{x_s}(\cdot)$. This results in fake foliations $\cF_{x_s,\cN}^{i,(x,t)}$, $i=E,F$ on the normal plane of $x_s$, $s\in [0,t)$ that are invariant under the map $\cP_{1,x_s}$ in the following sense: 
$$
	\cP_{1,x_s}\left(\cF^{F,(x,t)}_{x_s,\cN} (y)\right)\supset \cF^{F,(x,t)}_{x_{s+1},\cN} (\cP_{1,x_s}(y)),\, \mbox{for } y\in \cN_{\rho_0K_0^{-1}|X(x)|}(x), s\in [0,t-1), 
$$
and
$$
\cP_{1,x_s}\left(\cF^{E,(x,t)}_{x_s,\cN} (y)\right)\subset  \cF^{E,(x,t)}_{x_{s+1},\cN} (\cP_{1,x_s}(y)),\, \mbox{for } y\in \cN_{\rho_0K_0^{-1}|X(x)|}(x), s\in [0,t-1).
$$
Note that invariance holds only if one does not iterate beyond $[0,t)$.  These foliations form a local product structure on  $\cN_{{\rho_0}|X(x_s)|}(x)$ as in the previous case, and the $E$- and $F$-lengths can be defined in the same way as before. 

\begin{remark}
	Note that for $x\in \Lambda$, the foliations $\cF^{i}_{x,\cN}$ depend only on $x$ (and its entire orbit); in comparison, for $(x,t)\in \cO(U)\setminus \Lambda\times\RR^+$ the fake foliations depend on both $x$ and $t.$
\end{remark}

\begin{remark}
	The construction in $\cO(U)\setminus \Lambda\times\RR^+$ is somewhat flexible. For example, one can use the diffeomorphism $\tilde P_{1,x}$ instead of $\psi^*_{1,x}$ for $s<0$. This results in different fake foliations and therefore a different local product structure, but does not substantially affect the rest of the proof.
	
	Alternatively one can even obtain two $C^1$ fake foliations (not just continuous foliations with $C^1$ leaves) in the following way: at $x$, take $\cF_{x,N}^{F,*}$ to be the $C^\infty$ foliation whose leaves are $(n-k)$-dimensional planes parallel to $F_N(x)$. At each $x_s$ for $s\in (0,t)$,  let $\cF_{x_s,N}^{F,*} = \tilde P_{s,x}(\cF_{x,N}^{F,*})$. Then $\cF_{x_s,N}^{F,*}$ is a $C^1$ foliation tangent to the $F_{N}$-cone due to the domination between $E_N$ and $F_N$  and consequently the forward invariance of the $F_{N}$-cone field. To obtain the other foliation $\cF_{x_s,N}^{E,*}$, repeat the previous construction with $-X$ starting on $N(x_t)$. These foliations form a local product structure and provide well-defined $E$- and $F$-lengths.  The reason we did not choose this construction is for the sake of consistency with $(x,t)\in \Lambda\times \RR^+$.
	
	%Also note that the alternate construction starts at certain reference points ($x$ for $\cF_{x,N}^{F,*}$, $x_t$ for $\cF_{x,N}^{E,*}$). For $x\in\Lambda$, this issue makes their construction depending on orbit segments $(x,t)$ rather than on the point $x$ itself, causing further complication in the notations and the proof of Theorem \ref{m.A}. 
	
	Finally we remark that not all orbit segments in $\cO(U)$ have fake foliations constructed on their normal planes. This is due to the potential lack of dominated splitting on their normal bundles (for instance, points on the local strong stable manifold of a singularity $\sigma\in\Sing^+_\Lambda(X)$) (see \cite{MPP99}); only those start and end outside $B_{r_0}(\Sing_\Lambda(X))$ do. 
\end{remark}

We conclude this subsection by noting that the previous construction does not require $\Lambda$ to be isolated. Indeed the construction still holds if $\Lambda$ is replaced by the maximal invariant set $\tilde\Lambda$ in a small neighborhood of $\Lambda$, due to the flow being multi-singular hyperbolic on $\tilde\Lambda$ (Theorem~\ref{t.robust.hyp}).
}

\subsection{Contraction and expansion near singularities}\label{ss.sing}
In this section we establish the contraction and expansion of $E$- and $F$-lengths. They do not immediately follow from \eqref{e.hyp} due to a lack of uniform continuity for the linear Poincar\'e flow. 

In view of the symmetry (Remark~\ref{r.symmetry}), we will only consider Lorenz-type singularities $\Sing_\Lambda^+(X)$. The same holds for reverse Lorenz-type by considering $-X$ (and switch $E$ with $F$).

Let us start with the choice of a small neighborhood for each singularity. Fix $r_0>0$ small enough such that Lemma~\ref{l.nearEc} holds for $\alpha>0$ chosen in Section~\ref{ss.fake} and take $r\in (0,r_0)$. For each $\sigma\in\Sing_\Lambda(X)$ and $x\in {B_{r}(\sigma)}$, we let 
$$
t^-=t^-(x)=\sup\{s>0: (x_{-s},s)\subset \overline{B_{r_0}(\sigma)}\},\,\mbox{ and }
$$ 
$$
t^+=t^+(x)=\sup\{s>0: (x,s)\subset \overline{B_{r_0}(\sigma)}\}.
$$
Then the orbit segment $(x_{-t^-},t^-+t^+)$ is contained in the closed ball $\overline{B_{r_0}(\sigma)}$, with $x_{t^\pm}\in \partial B_{r_0}(\sigma)$. Now, we define
\begin{equation}\label{e.W.def}
	W_r(\sigma) = \bigcup_{x\in B_{r }(\sigma)} \{x_s:s\in(-t^-,t^+)\}.
\end{equation}
In other words, $W_{r }(\sigma)$ is a flow-saturated neighborhood of $\sigma$ that satisfies
$$
B_r(\sigma)\subset W_r(\sigma) \subset B_{r_0}(x).
$$
%We shall also write $W_r = \bigcup_{\sigma\in\Sing_\Lambda(X)}W_r(\sigma)$.
Note that $\partial B_{r_0}(\sigma)\cap \overline W_{r }(\sigma)$ consists of two pieces $\partial W_{r }^-(\sigma)$ and $\partial W_{r }^+(\sigma)$, such that if $x\in \partial W_{r }^-(\sigma)$ then there exists $x^+\in \partial W_{r }^+(\sigma)$ on the forward orbit of $x$ such that the orbit segment from  $x$ to $x^+$ is contained in $W_{r }(\sigma)$. In other words, points in $\partial W_{r }^-(\sigma)$ are moving towards $\sigma$ and points in $\partial W_{r }^+(\sigma)$ are moving away from $\sigma$. 

Clearly $W_{r }(\sigma)$ is an open neighborhood of $\sigma$. Meanwhile, define the open set 
$$
W_{r }:=\bigcup_{\sigma\in\Sing_\Lambda(X)}W_{r }(\sigma),
$$ then points in $(W_{r })^c$ 
are bounded away from all singularities. In particular, the flow speed on $(W_{r })^c$ is bounded away from zero, and the ``uniformly relative'' scale $\rho_0|X(x)|$ as in Section~\ref{ss.Liao} is indeed uniform.

The following two lemmas, taken from~\cite[Section 6]{PYY23}, will play central roles in the proof of expansivity and the Bowen property.  For $x\in\partial W_r^-\cap\Lambda$ we write $t(x) = \sup\{s>0: (x,s)\in W_r\}$ and $x^+= x_{t(x)}\in \partial W_r^+$; take $y\in B_{t(x),\vep}(x)\cap \Lambda$ and let $y^+ = \cP_{x^+}(y_{t(x)})\in \cN_{\rho_0|X(x^+)|}(x^+)$. Then, the following two lemmas  state that {for $r>0$ sufficiently small, for  orbit segments $(x,t)\in \cO(U)$ such that  $x\notin B_{r_0}(\Sing_\Lambda(X))$ and enters $B_{r}(\sigma)$ (note that they can only enter $B_{r_0}(\sigma)$ through the center direction, due to Lemma~\ref{l.nbhd.1})} :
\begin{enumerate}
	\item if upon {\em entering} $W_{r }$ we have the $F$-length of $y$ larger than the $E$-length at $y$, then the same is true when the orbit of $y$ leaves $W_{r }$;
	\item on the other hand, if upon {\em leaving} $W_{r }$ we have the $E$-length of $y^+$ larger than the $F$-length at $y^+$, then the same is true when $y$ enters $W_{r }$;
	%\item in the second case, the orbit of $y$ is scaled shadowed by the orbit of $x$ before leaving $W_r$;
	\item in both cases, we see forward/backward expansion on the larger coordinate at an exponential speed.
\end{enumerate}

% Note that $y_i$ may not be on $\cN(x_i)$, $0\le i \le t(x)$. However, by the implicit function theorem there exists a change of time $\tilde t_x(y): B_{t(x),\vep}(x)\to\RR^+$ with $\sup_{t,y}\left\{\left|\frac{d}{dt}\tilde t_x(t,y)-1\right|\right\}$ small enough (shrink $\vep$ if necessary), such that 
%$$
%\tilde t_x(i,y)\in \cN(x_i), 0\le i\le t(x)
%$$

Recall the definition of $\rho$-scaled shadowing from Definition~\ref{d.shadow1} and its properties from Lemma \ref{l.shadowing2}.
\begin{lemma}\label{l.Elarge}\cite[Lemma 6.7]{PYY23}
	Let $\Lambda$ be a multi-singular hyperbolic compact invariant set, and $\sigma\in\Sing_\Lambda^+(X)$ be an active singularity. 
	Then, for all $r_0>0$ sufficiently small, there exists an open isolating neighborhood $U$ of $\Lambda$, such that for every $\rho\in (0,\rho_0]$, 
	there exist $\overline r\in(0,{r_0}),\vep_1>0$ such that for $0<r <\overline r$ and $0<\vep<\vep_1$, the following holds:\\
	{Let $(x,t)\in\cO(U)$ satisfy  $x\in \partial W_{r }^-$ and $ x^+ =f_{t(x)}(x)\in \partial W_r^+$.} For $y\in B_{t(x),\vep}(x)\cap \cN(x)$, assume that $y=[y^E,y^F]$ and that $y^+ = [y^{+,E}, y^{+,F}]$  where $y^+=\cP_{x^+}(y_{t(x)})\in \cN_{\rho_0|X(x^+)|}(x^+)$. Finally, assume that 
	\begin{equation}\label{e.Elarge}
		d^E_{x^+}(y^+)\ge d^F_{x^+}(y^+) , \hspace{0.5cm} \mbox{($y^+$ has large $E$-length at $x^+$)}
	\end{equation}
	then we have the following statements:
	\begin{enumerate}
		\item ($y$ has large $E$-length at $x$) we have $d^E_{x}(y)\ge d^F_{x}(y) $;
		\item (backward expansion on $E$-length) there exists $\lambda_{\sigma,E}>1$ independent of $r$, $\vep$, $x$ and $y$  such that 
		$$d^E_{x}(y)\ge\lambda_{\sigma,E}^{t(x)}d^E_{x^+}(y^+).$$
		\item  (scaled shadowing until leaving $W_{r }$) $y$ is $\rho$-scaled shadowed by the orbit of $x$ up to  time $t(x)$.
	\end{enumerate}
\end{lemma}

Similarly, we have 	

\begin{lemma}\label{l.Flarge}\cite[Lemma 6.8]{PYY23}
	Let $\Lambda$ be a multi-singular hyperbolic compact invariant set, and $\sigma\in\Sing_\Lambda^+(X)$ be an active singularity. 
	Then, for all $r_0>0$ sufficiently small, {there exists an open isolating neighborhood $U$ of $\Lambda$}, %such that for every $\rho\in (0,\rho_0]$, 
	and $\overline r\in(0,{r_0}),\vep_1>0$ such that for $0<r <\overline r$ and $0<\vep<\vep_1$, the following holds:\\
	{Let $(x,t)\in\cO(U)$ satisfy  $x\in \partial W_{r }^-$ and $ x^+ =f_{t(x)}(x)\in \partial W_r^+$.} For $y\in B_{t(x),\vep}(x)\cap \cN(x)$, assume that $y=[y^E,y^F]$ and that $y^+ = [y^{+,E}, y^{+,F}]$  where $y^+=\cP_{x^+}(y_{t(x)})\in \cN_{\rho_0|X(x^+)|}(x^+)$. Finally, assume that 
	\begin{equation}\label{e.Flarge}
		d^F_x(y)\ge d^E_x(y), \hspace{1cm} \mbox{($y$ has large $F$-length at $x$)}
	\end{equation}
	then we have the following statements:
	\begin{enumerate}
		\item ($y^+$ has large $F$-length at $x^+$) we have $$ d^F_{x^+}(y^+)\ge d^E_{x^+}(y^+);$$
		\item (forward expansion on $F$-length) there exists $\lambda_{\sigma,F}>1$ independent of $r$, $\vep$, $x$ and $y$ such that 
		$$ d^F_{x^+}(y^+)\ge\lambda_{\sigma,F}^{t(x)}d^F_{x}(y).$$
	\end{enumerate}
	
\end{lemma}	
Here note that the local product structure $[\cdot,\cdot]$ is well-defined at $y$ and $y^+$ (even if $x\notin \Lambda$) since both points are away from all singularities (once $r_0$ is fixed). Therefore one can find $c_{r_{0}}>0$ such that $|X(y)|, |X(y^+)|>c_{r_{0}}$ for all $y$ and $y^+$ involved in the statement of Lemma~\ref{l.Elarge} and \ref{l.Flarge}. %Then take $\vep_1< \rho_0c_{r_{0}}$.

The proof of both lemmas can be found in~\cite[Appendix C]{PYY23}. There are only two differences to note:
\begin{itemize}
	\item In \cite[Lemma 6.7, 6.8]{PYY23} we required that $x\in\Lambda$; this is to guarantee that the $x$ can only enter $B_{r_0}(\sigma)$ through the center direction (Lemma \ref{l.nearEc}). Here for orbit segments in $\cO(U)$ that are not in $\Lambda$, we use Lemma \ref{l.nbhd.1} instead to get the same properties.
	\item The proof of Lemma 6.7 and 6.8 in \cite{PYY23} does not involve sectional hyperbolicity. The proof \cite{PYY23} only uses sectional hyperbolicity to obtain Equation (C.15). In the current setting, the same estimate follows from \eqref{e.sing.Lya}, which means that $\lambda_\sigma^c+\lambda^{uu}_\sigma>0$ if $\sigma$ is of Lorenz-type.\footnote{Indeed, since $\lambda_\sigma^c+\lambda^{uu}_\sigma > 0$, the singleton $\{\sigma\}$ is sectional-hyperbolic. This suffices for the proof of~\cite[Lemma 6.8]{PYY23} to proceed, since it only involves a local analysis near a Lorenz-type singularity. }
\end{itemize}

Finally, we remark that by considering $-X$, these lemmas yield similar statements for reverse Lorenz-type singularities (in which case one switches $E$ and $F$).

\subsection{Proof of Theorem~\ref{m.A}}\label{ss.ThmA.proof}
We first state the precise result to be proven in this section.

\begin{theorem}\label{t.exp}
	Let $\Lambda$ be a compact, invariant multi-singular hyperbolic set (either in the sense of Definition~\ref{d.multising} or in the sense of Bonatti and da Luz, i.e., Definition~\ref{d.multising.BD}) for a $C^1$ vector field $X$. Then, there exist $\vep_{\Exp}>0$, a $C^1$ neighborhood $\cU$ of $X$ and an open neighborhood $U$ of $\Lambda$, such that for every $Y\in \cU$ and the maximal invariant set $\tilde \Lambda_Y$ of $Y$ in $U$, the following hold:
	
	For every ergodic invariant probability measure $\mu$ supported on $\tilde \Lambda_Y$ and every point $x\in\supp\mu\setminus \left(\bigcup_{\sigma\in \Sing(Y)\cap\tilde\Lambda_Y} W^s(\sigma)\cup W^u(\sigma)\right)$, there exists $s=s(x)>0$ such that the bi-infinite Bowen ball of scale $\vep_{\Exp}$ satisfies
	$$
	\Gamma_{\vep_{\Exp}}(x)\subset (x_{-s},2s).
	$$
\end{theorem}
In other words, the expansivity can only fail for points on the stable or the unstable manifold of a singularity. Note that for every ergodic invariant measure $\mu$ it holds that $\mu\left(W^s(\sigma)\setminus \sigma\right)=0$ due to the Poincar\'e Recurrence Theorem; similarly $\mu\left(W^u(\sigma)\setminus \sigma\right) = 0$. Then Theorem~\ref{m.A} is an immediately corollary of Theorem~\ref{t.exp}. 

Below we prove Theorem~\ref{t.exp}.

Given a  compact invariant multi-singular hyperbolic set $\Lambda$ for a $C^1$ vector field $X$, we denote by $\cU$ the neighborhood of $X$ and $U$ the neighborhood of $\Lambda$ given by Theorem~\ref{t.robust.hyp}. Let $\tilde\Lambda$ be the maximal invariant set of $X$ in $U$ which is also multi-singular hyperbolic. Below we will prove that there exists $\vep_{\Exp}>0$ such that $X|_{\tilde\Lambda}$ is almost expansive at scale $\vep_{\Exp}$. One can easily check that our choice of $\vep_{\Exp}$ depends continuously on the vector field $X$ on the $C^1$ topology, therefore the same conclusion holds for every $Y\in\cU$ (by shrinking $\cU$ and decreasing $\vep_{\Exp}$ when necessary).

First we address the minor issue mentioned in Remark~\ref{r.ThmA}, namely the discrepancy between two definitions of multi-singular hyperbolicity.

\begin{lemma}\label{l.sing.active}
	Let $\mu$ be an ergodic invariant measure of $X|_{\tilde\Lambda}$ that is not a point mass of a singularity. Then every hyperbolic singularity $\sigma\in\supp\mu$ is active. In particular, on $\supp\mu$, two definitions of multi-singular hyperbolicity coincide.
\end{lemma} 

The proof is a simply application of the Birkhoff Ergodic Theorem, and is therefore omitted. 

%\begin{proof}%
%	Let $\sigma\in\supp\mu$ be a singularity. Then $\mu(\sigma)=0$ as $\mu$ is ergodic. Assuming that $\sigma$ is not active, we consider the case where 
%	$$
%	W^{u}(\sigma)\cap\tilde\Lambda = \{\sigma\}.
%	$$
%	The other case can be handled similarly by considering $-X$.
	
%	We claim that there exists a neighborhood $U_\sigma$ of $\sigma$, such that $U_\sigma\cap\tilde\Lambda\subset W^s(\sigma)$. If this is not the case, then one has an isolating neighborhood $U^0$ and a sequence of regular points $\{x^n\}$ in $\tilde\Lambda$ with $x^n\to \sigma$ and $x^n$ are not contained in $W^s(\sigma)$. Then they must eventually exit $U^0$. Let (recall that $(x,t)$ is the orbit segment starting at $x$ up to time $t$)
%	$$
%	\tau_n=\sup\{t>0: (x^n,t)\subset U^0\}.
%	$$
%	we have $\tau_n\to\infty$ and $x^n_{\tau_n}:=f_{\tau_n}(x^n)\in \partial U^0$ for every $n$. By taking a subsequence, assume that $x^n_{\tau_n}\to y\in\partial U^0$. Then $y\in\tilde\Lambda\cap\Reg(X)$, and its entire backward orbit is contained in $U^0$. Then $y$ must be contained on the unstable manifold of $\sigma$, a contradiction.
	
%	So we have $U_\sigma\cap\tilde\Lambda\subset W^s(\sigma)$ for some neighborhood $U_\sigma$. This means that $\mu(U_\sigma) = \mu(U_\sigma\cap\tilde\Lambda)= \mu(W^s(\sigma))= 0$ due to the Poincar\'e Recurrence Theorem, which contradicts with the assumption that $\sigma\in\supp\mu$. We conclude the proof.
%\end{proof}

{Before stepping into the next subsection, we remark that in Section \ref{ss.ThmA.proof}, we will only  consider points in the support of an ergodic measure on $\tilde \Lambda$. Consequently, the fake foliations only depend on $x$, and their invariance holds on the entire orbit of $x$. 
}

\subsubsection{Choice of parameters}\label{ss.choice.vexp}
We first fix $r_0>0$ such that Lemma~\ref{l.nearEc}, Lemma~\ref{l.Elarge} and Lemma~\ref{l.Flarge} hold for all singularities in $\tilde\Lambda$ (note that there are only finitely many singularities since they are all hyperbolic). Then, with $V = {B_{r_0}(\Sing_{\tilde\Lambda}(X))}$ an isolating neighborhood of $\Sing_{\tilde\Lambda}(X)$, we obtain $T_V$ as in Definition~\ref{d.multising} (2) so that \eqref{e.hyp} holds for $t>T_V$ and $x,x_t\notin V$. Let $K_1'$ be given by Proposition~\ref{p.tubular3} (3) which is the upper bound for $\|D\cP_{1.x}\|$ and $\|D(\cP_{1.x})^{-1}\|$.

Next, we let $\lambda_{\sigma,*}$, $*=E,F$ be given by Lemma~\ref{l.Elarge} and Lemma~\ref{l.Flarge}, and $\eta>1$ given by Definition~\ref{d.multising} (2). Fix 
\begin{equation*}\label{e.choice.eta'}
1<\eta'<\eta,
\end{equation*}
and
\begin{equation}\label{e.choice.lambda0}
1<\lambda_0<\min \left\{\lambda_{\sigma,*}, *=E,F, \sigma\in\Sing_{\tilde\Lambda}(X) \right\}.
\end{equation}
We shall further increase $T_V$ (which we assume w.l.o.g. to be an integer) if necessary so that (recall $L_X$ defined by~\eqref{e.lip}) 
\begin{equation}\label{e.choice.omega}
 \omega:= L_X\cdot  (\eta')^{-T_V} <1.
\end{equation}

Now we choose $r$. 
With $a = (\eta-\eta')/2$ we apply Proposition~\ref{p.tubular3} to find $\rho_1< \rho_0$ small enough so that Proposition~\ref{p.tubular3} (1) and (2) hold for $y,z\in \cN_{\rho_1 K_0^{-1}|X(x)|}(x)$ at every regular point $x$. 
Let $\overline r\in (0,r_0)$ and $\vep_1>0$ be given by Lemma~\ref{l.Elarge} and~\ref{l.Flarge} applied to $\rho_1$ and $r_0$. For $r\le\overline r$, recall the definition of $W_r(\sigma)$ from~\eqref{e.W.def}, and note that for each $r$ and singularity $\sigma$, there exists $s_r(\sigma)>0$ such that for every $x\in B_r(\sigma)$, 
$$
t^++t^->s_r(\sigma);
$$  
furthermore, $s_r(\sigma)\to\infty$ as $r\to 0$. This allows us to pick $r\in (0,\overline r]$ sufficiently small so that 
\begin{equation}\label{e.choice.b}
	b:=\frac{\lambda_0^{s_r}}{L_X(K_1')^{T_V}}>1,
\end{equation}
where $s_r = \min \{s_r(\sigma):\sigma\in\Sing_{\tilde\Lambda}(X)\}$. The purpose of this step is to use the expansion given by Lemma~\ref{l.Elarge} and~\ref{l.Flarge} near singularities to overcome the possibility that the time spend outside $\cup_\sigma W_r(\sigma)$ is too small for the hyperbolicity given by Definition~\ref{d.multising} (2) to take effect.

Finally we are ready to pick $\vep_{\Exp}$. Note that $W_r = \cup_{\sigma\in\Sing_{\tilde\Lambda}(X)} W_r(\sigma)$ is an open neighborhood of $\sigma\in\Sing(X)$. We let 
\begin{equation}\label{e.choice.vep}
	\vep_{\Exp} = \frac12\min\left\{\rho_1 K_0^{-1}\cdot\inf_{y\in W_r^c} |X(y)| ,\,\, \vep_1\right\}>0. 
\end{equation}
This choice of $\vep_{\Exp}$ means that %for every $\rho<\rho_1$, 
for every orbit segment $(x,t)$  that is entirely outside $W_r^c$ and every $y\in B_{t,\vep_{\Exp}} (x)\cap \cN_{\rho_0|X(x)|}(x)$, the orbit of $y$ is $\rho_1$-scaled shadowed by the orbit of $x$ up to time $t$. Furthermore, the local product structure defined in Section~\ref{ss.fake} exists at every $\cP_{s,x}(y)$ for $y\in  B_{t,\vep_{\Exp}} (x)\cap \cN_{\rho_0|X(x)|}(x)$ and $s\in[0,t]$.

\subsubsection{Contraction, expansion and domination away from singularity}
We continue our discussion on orbit segments outside $W_r$. The following two lemmas provide control over the $E$ and $F$-length of $y$ as it is $\rho$-scaled shadowed by the orbit of $x$.

\begin{lemma}\cite[Lemma 7.4]{PYY23}\label{l.reg}
	There exists  $\lambda_R>1$, such that for every orbit segment $(x,t)\subset W_r^c$ and $y\in B_{t,\vep_{\Exp}}(x)$ with $t\ge 1$, we have 
	\begin{enumerate}[label=(\alph*)]
		\item assume that 
		\begin{equation}\label{e.Flarge1}
			d^F_x(\cP_x(y))\ge d^E_x(\cP_x(y)), \hspace{1cm} \mbox{($F$ large 	near $x$)}
		\end{equation}
		then we have 
		\begin{equation*}
			d^F_{x_t}(\cP_{x_t}(y_t))\ge \lambda_R^t\cdot d^E_{x_t}(\cP_{x_t}(y_t)); \hspace{1cm} \mbox{($F$ exponentially large 	near $x_t$)}
		\end{equation*}
		\item  assume that 
		\begin{equation}\label{e.Elarge1}
			d^E_{x_t}(\cP_{x_t}(y_t))\ge d^F_{x_t}(\cP_{x_t}(y_t)), \hspace{1cm} \mbox{($E$ large 	near $x_t$)}
		\end{equation}
		then we have 
		\begin{equation*}
			d^E_x(\cP_x(y))\ge \lambda_R^t\cdot d^F_x(\cP_x(y)). \hspace{1cm} \mbox{($E$ exponentially large 	near $x$)}
		\end{equation*}
	\end{enumerate} 
\end{lemma}
This is a simple consequence of the dominated splitting $E_N\oplus F_N$ and the uniform continuity of $D\cP_{1,x}$. The proof is omitted.

\begin{remark}
	Unlike in Definition \ref{d.multising}, Lemma \ref{l.Elarge} or \ref{l.Flarge}, here we do not need $t$ to be large.
\end{remark}

\begin{lemma}\label{l.expansion.reg}
	For every orbit segment $(x,t)\subset W_r^c$ with $x,x_t\notin V$, the following statements hold for every $y\in B_{t,\vep_{\Exp}}(x)$:
	\begin{enumerate}
		\item if $t\le T_V$, 
		$$
		d^{F}_{x}(\cP_x(y)) \le L_X (K_1')^{T_V}  \cdot d^{F}_{x_t}(\cP_{x_t}(y_t));
		$$
		\item if $t> T_V$, then for $\omega<1$ defined by~\eqref{e.choice.omega},
		$$
		d^{F}_{x}(\cP_x(y)) \le L_X (\eta')^{-t}  \cdot d^{F}_{x_t}(\cP_{x_t}(y_t))\le  \omega\cdot d^{F}_{x_t}(\cP_{x_t}(y_t)). 
		$$
	\end{enumerate}
	A similar statement holds for the $E$-length of $y$ by considering $-X$. 
\end{lemma}

\begin{proof}	
	For simplicity we write $y^s = \cP_{x_s}(y_s)$. As before,  We use $\floor{a}$ to denote the integer part of $a$.
	
	\noindent Case (1). If $t< T_V$,  let $\ell$ be any arc in $\cN_{\rho_0|X(x_t)|} (x_t)$ joining $x_t$ and $\left(y^t\right)^F$ with  length equals $d^F_{x_t}(y^t)$. Then on $\ell$, $\cP_{-t,x_t}$ is well-defined. By Proposition~\ref{p.tubular3} (3) we have 
	\begin{align*}
		d_x^F(\cP_x(y)) &\le \cP_{-t, x_t} (\ell) \le 
		\left( \sup_{z\in \ell}\left\| D_z\cP_{-t, x_t} \right\|\right) d_{x_t}^F(y^t)\\ 
		&\le L_X \left(\sup_{z\in \cP_{t-\floor{t}, x_t}(\ell)}\left\| D_z\cP_{-\floor{t}, x_{\floor{t}}} \right\|\right) d_{x_{t}}^F(y^t)\\
		&\le L_X (K_1')^{T_V}\cdot  d_{x_t}^F(y^t).
	\end{align*}
	
	\noindent Case (2). Let $t>T_V$, and $\ell$ be chosen as in Case (1). This time we use~\eqref{e.hyp0}, Proposition~\ref{p.tubular3} (1) and the choice of $\vep_{\Exp}$ to obtain 
	\begin{align*}
		&d_x^F(\cP_x(y))\\ %&\le \cP_{-t, x_t} (\ell) \le 	\left(\sup_{z\in \ell}\left\| D_z\cP_{-t, x_t} \right\|\right) d_{x_t}^F(y^t)\\ 
		&\le L_X \prod_{i=0}^{\floor{t}-1}
		\left(\sup_{z\in \cP_{-(t -\floor{t}+i),x_t}(\ell) } \|D_z\cP_{-1,x_{\floor{t}-i}}|_{F_{N}(x_{\floor{t}-i})}\| + a\right)
		\cdot  d_{x_t}^F(y^t)\\
		&\le L_X \prod_{i=0}^{\floor{t}-1}
		\left( \|\psi_{-1,x_{\floor{t}-i}}\| + a \right)
		\cdot  d_{x_t}^F(y^t)\\
		&\le L_X (\eta')^{-\floor{t}} \cdot  d_{x_t}^F(y^t) \le \omega\cdot  d_{x_t}^F(y^t),
	\end{align*}
	as needed.
\end{proof}

Note that here we only need estimates on $(\psi_t)$ as opposed to $(\psi_t^*)$ because $(x,t)$ are outside $W_r$, therefore Liao's tubular neighborhood has uniform diameter. In comparison,  estimates on $(\psi_t^*)$ will be crucial in Section~\ref{s.Pliss} and~\ref{s.bowen} when we use simultaneous Pliss times for $(\psi_t^*)$ to show the scaled shadowing property for points in the Bowen ball.

\subsubsection{Proof of Theorem~\ref{t.exp}}

Given an ergodic invariant probability measure $\mu$, if it is a point mass of a singularity $\sigma$, then the hyperbolicity of $\sigma$ implies that for $\mu$ a.e. point (in this case, $\sigma$ itself), the bi-infinite Bowen point is just the point $\{\sigma\}$. So below we consider the case when $\mu$ is not a point mass of a singularity. By Lemma~\ref{l.sing.active}, all singularities in $\supp\mu$ are active, and two definitions of multi-singular hyperbolicity coincides. Therefore all the results in the previous subsection apply.

We assume that $x\in\supp\mu$ is not contained in the stable or the unstable manifold of any singularity. We remark that the argument below applies to points that are not Birkhoff typical points of $\mu$. 

We will show that if $y\in \Gamma_{\vep_{\Exp}}(x)$ where $\Gamma_\vep(x)$ is the bi-infinite Bowen ball at $x$ with scale $\vep$, then $y\in \Orb(x)$. First, observe that by invariance it suffices to consider $x\notin W_r$. In this case, the $E$ and $F$-length of $y$ are well-defined. We will show that 
$$
d^E_x(\cP_x(y)) = d^F_x(\cP_x(y)) = 0.
$$  
Assuming for the sake of contradiction that this is not true, we may require that 
\begin{equation}\label{e.exp.Flarge}
d^E_x(\cP_x(y)) \le d^F_x(\cP_x(y)),\mbox{ and }  d^F_x(\cP_x(y))>0.
\end{equation}
The other case can be handled in the same way by considering $-X$. 

Given such a point $x$, we parse the time interval $[0,\infty)$ as
$$
0\le T_1^i < T_1^o < T_2^i < T_2^o <\cdots,
$$
where (here `$i$' and `$o$' refer to getting {\em in} and {\em out} of $W_r$):
\begin{itemize}
	\item {all $T_k^*$, $*= i,o$, $k=1,2,\ldots$ are integers;}
	\item the (open) orbit segment $(x_{T_k^i+1}, T_k^o-T_k^i-2)$ is contained in $W_r$;  consequently, $T_k^o-T_k^i>s_r$ where $s_r = \min \{s_r(\sigma):\sigma\in\Sing_{\tilde\Lambda}(X)\}$;
	\item  the (closed) orbit segment $(x_{T_k^o}, T_{k+1}^i-T_k^o)$ is not in $ W_r$;
	\item $x_{T_k^*}\notin V$, $i=o,i$.
\end{itemize}
i.e., the forward orbit of $x$ enters $W_r$ for the $k$th time in the time interval $[T_k^i, T_{k}^i+1]$, and then leaves $W_r$ for the $k$th time in the time interval $[T_k^o-1,T_k^o]$. 
By the definition of $W_r$, orbits can only enter and leave $W_r$ through $\partial W_r^\pm\subset V$, so we have 
$$
x_{T_{k}^{*}}\notin V, k=1,\ldots, * = i,o. 
$$ 
This is why we can first fix $r_0$, then decrease $r$ without affecting $T_V = T_{B_{r_0}(\Sing_{\tilde\Lambda}(X))}$.
The same parsing applied to the backward orbit of $x$ yields $T_{k}^*$, $k=-1,-2,\ldots,$ $*=i,o$.

Since we do not require $x$ to be a Birkhoff typical point of $\mu$, we do not know if its forward orbit enters $W_r$ infinitely many times. Then, there are two cases to consider. 

\medskip
\noindent{\bf  Case 1.} The forward orbit of $x$ enters $W_r$ infinitely many times.

In this case, we first apply Lemma~\ref{l.reg} (a) on the orbit segment given by the time interval $(0, T_1^i)$ to obtain
$$
d^F_{x_{T_1^i}}\left(\cP_{x_{T_1^i}}(y_{T_1^i})\right)\ge d^E_{x_{T_1^i}}\left(\cP_{x_{T_1^i}}(y_{T_1^i})\right). \hspace{1cm} (\mbox{$F$ larger than $E$ at } x_{T_1^i})
$$
Then, we consider the time interval $(T_1^i,T_1^o)$. During this time period, the orbit of $x$ enters $W_r(\sigma)$ of an active singularity $\sigma$. 

\medskip
\noindent{\bf Sub-case 1}. $\sigma$ is of Lorenz-type. Then we apply Lemma~\ref{l.Flarge} (1) to get
$$
d^F_{x_{T_1^o}}\left(\cP_{x_{T_1^o}}(y_{T_1^o})\right)\ge d^E_{x_{T_1^o}}\left(\cP_{x_{T_1^o}}(y_{T_1^o})\right). \hspace{1cm} (\mbox{$F$ larger than $E$ at } x_{T_1^o})
$$

\medskip
\noindent{\bf Sub-case 2}. $\sigma$ is of reverse Lorenz-type. Then we apply Lemma~\ref{l.Elarge} (1) for $-X$ (keep in mind Remark~\ref{r.symmetry}, switch $E$ and $F$, and switch $i$ and $o$) to get
$$
d^F_{x_{T_1^o}}\left(\cP_{x_{T_1^o}}(y_{T_1^o})\right)\ge d^E_{x_{T_1^o}}\left(\cP_{x_{T_1^o}}(y_{T_1^o})\right). \hspace{1cm} (\mbox{$F$ larger than $E$ at } x_{T_1^o})
$$
These two cases lead to the same conclusion.

Moving on to the time interval $(T_1^o,T_2^i)$ which is outside $W_r$, Lemma~\ref{l.reg} (a) again yields 
$$
d^F_{x_{T_2^i}}\left(\cP_{x_{T_2^i}}(y_{T_2^i})\right)\ge d^E_{x_{T_2^i}}\left(\cP_{x_{T_2^i}}(y_{T_2^i})\right). \hspace{1cm} (\mbox{$F$ larger than $E$ at } x_{T_2^i})
$$

Recursively, we get that for every $k\in\NN$, $*=i,o$, 
\begin{equation}\label{e.exp.Flarge11}
d^F_{x_{T_k^*}}\left(\cP_{x_{T_k^*}}(y_{T_k^*})\right)\ge d^E_{x_{T_k^*}}\left(\cP_{x_{T_k^*}}(y_{T_k^*})\right). \hspace{0.5cm} (\mbox{$F$ larger than $E$ at every } x_{T_k^*})
\end{equation}

Next we consider the expansion for the $F$-length. It uses a recursive argument similar to the previous one, so we will only show the expansion up to the time $T_2^i$.  

For the time interval $(T_1^i, T_1^o)$ spent inside $W_r$, if the associated singularity $\sigma$ is of Lorenz-type, then we use Lemma~\ref{l.Flarge} (2) to get (recall the definition of $\lambda_0$ by~\eqref{e.choice.lambda0})
\begin{equation}\label{e.exp.11}
d^F_{x_{T_1^i}}\left(\cP_{x_{T_1^i}}(y_{T_1^i})\right)\le \lambda_0^{-s_r}\cdot  d^F_{x_{T_1^o}}\left(\cP_{x_{T_1^o}}(y_{T_1^o})\right). \hspace{0.5cm} (\mbox{exponentially large $F$-length})
\end{equation}
If the associated singularity $\sigma$ is of reverse Lorenz-type, then the same conclusion follows from Lemma~\ref{l.Elarge} (2) applied to $-X$. 

For the time interval $(T_1^o, T_2^i)$ which is outside $W_r$,  we apply Lemma~\ref{l.expansion.reg} to get
\begin{equation}\label{e.exp.22}
d^F_{x_{T_1^o}}\left(\cP_{x_{T_1^o}}(y_{T_1^o})\right)\le
\begin{cases}
	L_X(K_1')^{T_V}\cdot d^F_{x_{T_2^i}}\left(\cP_{x_{T_2^i}}(y_{T_2^i})\right), &\mbox{ if $T_2^i -T_1^o \le T_V$};\\
	\omega \cdot d^F_{x_{T_2^i}}\left(\cP_{x_{T_2^i}}(y_{T_2^i})\right), &\mbox{ if $T_2^i -T_1^o > T_V$}.
\end{cases}
\end{equation}

Combining~\eqref{e.exp.11} and~\eqref{e.exp.22} and keeping in mind the choice of $b>1$ by~\eqref{e.choice.b}, we have 
$$
d^F_{x_{T_1^i}}\left(\cP_{x_{T_1^i}}(y_{T_1^i})\right)\le \max\{\omega,b^{-1}\}\cdot d^F_{x_{T_2^i}}\left(\cP_{x_{T_2^i}}(y_{T_2^i})\right);
$$
i.e., between two consecutive entries to $W_r$, the $F$-length of $y$ is expanded  by a factor  that is greater than one.

Recursively, for every $k\ge 2$ we have 
\begin{equation}\label{e.exp.33}
d^F_{x_{T_1^i}}\left(\cP_{x_{T_1^i}}(y_{T_1^i})\right)\le (\max\{\omega,b^{-1}\})^{k-1}\cdot d^F_{x_{T_{k}^i}}\left(\cP_{x_{T_k^i}}(y_{T_k^i})\right).
\end{equation}
 Note that $d^F_{x_{T_{k}^i}}\left(\cP_{x_{T_k^i}}(y_{T_k^i})\right)$ remains bounded as it is always contained in Liao's tubular neighborhood of size $\rho_0|X|$ thanks to the choice of $\vep_{\Exp}$. Also note $\max\{\omega,b^{-1}\}<1$. Therefore, in order for~\eqref{e.exp.33} to hold for all $k\ge 2$, one must have 
$$
d^F_{x_{T_1^i}}\left(\cP_{x_{T_1^i}}(y_{T_1^i})\right)=0,
$$
contradicting our assumption that $ d^F_x(\cP_x(y))>0$.

\medskip
\noindent{\bf  Case 2.} The forward orbit of $x$ enters $W_r$ only finitely many times.

In this case, \eqref{e.exp.Flarge11} still hold for every $T_k^*$, $*=i,o$. Let $T_N^o$ be the last time that the forward orbit of $x$ leaves $W_r$ (and thus exits $V = B_{r_0}(\Sing_{\tilde\Lambda}(X))$). Then  \eqref{e.exp.Flarge11} gives
$$
d^F_{x_{T_N^o}}\left(\cP_{x_{T_N^o}}(y_{T_N^o})\right)\ge d^E_{x_{T_N^o}}\left(\cP_{x_{T_N^o}}(y_{T_N^o})\right). \hspace{1cm} (\mbox{$F$ larger than $E$ at } x_{T_N^o})
$$
For the time interval $(T_N^o,+\infty)$, the corresponding orbit segment does not enter $W_r$ again, although it may enter $V = B_{r_0}(\Sing_{\tilde \Lambda}(X))$. Therefore Lemma~\ref{l.expansion.reg} yields, for every $t> T_V$ such that $x_t\notin V$ (note that such $t$ exists and can be made arbitrarily large), 
\begin{equation*}\label{e.exp.44}
	d^F_{x_{T_N^o}}\left(\cP_{x_{T_N^o}}(y_{T_N^o})\right)\le 
		L_X(\eta')^{-t}\cdot d^F_{x_{t+T_N^o}}\left(\cP_{x_{t+T_N^o}}(y_{t+T_N^o})\right).
\end{equation*}
The boundedness of $ d^F_{x_{t+T_N^o}}\left(\cP_{x_{t+T_N^o}}(y_{t+T_N^o})\right)$ for all $t>T_V$ means that $d^F_{x_{T_N^o}}\left(\cP_{x_{T_N^o}}(y_{T_N^o})\right)=0$ and therefore  $d^F_x(\cP_x(y))=0$, which is a contradiction.

We conclude the proof of Theorem~\ref{t.exp} and Theorem~\ref{m.A}.

\subsection{Proof of Theorem~\ref{m.A1}} Let $\Lambda$ be a multi-singular hyperbolic compact invariant set of a $C^1$ vector field $X$ with positive topological entropy and is isolated, i.e., it is the maixmal invariant set in a neighborhood $U$. It has been shown in \cite[Proposition 1.1]{LSWW} (which follows the classical work of Katok \cite{Katok80, Katok}; see also \cite[Theorem B]{PYY21} and \cite{LLL}) that $\Lambda$ can be approximated by a hyperbolic horseshoe $\Lambda^0$, away from all singularities, whose topological entropy can be made arbitrarily close to that of  $\Lambda$. 

Let $\phi: \bM\to\RR$ be a continuous potential function.  By slightly modifying the result of \cite{Gel}, one can require that the topological pressure of $\phi$ on $\Lambda^0$ is close to that of $\Lambda$. Since the topological pressure of a continuous function varies lower semi-continuously on a hyperbolic horseshoe, it follows that the topological pressure of $\phi$ varies lower semi-continuously w.r.t.\,the vector field in $C^1$ topology in the following sense: If $Y_n$ is a sequence of vector field converging to $X$ in $C^1$ topology, and $\Lambda_{Y_n}$ is the maximal invariant set of $Y_n$ in $U$, then we have (note that the continuation of $\Lambda^0$, denoted by $\Lambda^0_{Y_n}$, is contained in $\Lambda_{Y_n}$) 
$$
\liminf_{n\to\infty} P(\phi,Y_n|_{\Lambda_{Y_n}}) \ge P(\phi,X|_\Lambda).
$$

On the other hand, the robustness of almost expansivity in Theorem \ref{m.A} together with the classical work of Bowen \cite{B72} shows that the metric entropy varies upper semi-continuously w.r.t.\,both the invariant measure (in weak-*  topology) and the vector field (in $C^1$ topology). It then follows that the topological pressure of a continuous function varies upper semi-continuously w.r.t.\,the system in $C^1$ topology. With the lower semi-continuity obtained earlier, we conclude that the topological pressure is a continuous function of $X$, finishing the proof of Theorem~\ref{m.A1}.

\section{The Pliss lemma, hyperbolic times, and forward / backward recurrence Pliss times}\label{s.Pliss}
In this section we introduce the Pliss lemma~\cite{Pliss} and Pliss times. It was the main tool in our previous paper~\cite{PYY23} when we deal with equilibrium states of sectional-hyperbolic attractors. In this paper, we will consider six types of Pliss times (in comparison, in~\cite{PYY23} there were only two types of Pliss times):
\begin{enumerate}
	\item $E$-forward hyperbolic times for the forward  iteration along an orbit segment $(x,t)$;
	\item $F$-backward hyperbolic times for the backward iteration along an orbit segment $(x,t)$;
	\item forward recurrence Pliss times of an orbit segment $(x,t)$ visiting a small neighborhood of $\Sing_\Lambda(X)$;
	\item backward recurrence Pliss times of an orbit segment $(x,t)$ visiting a small neighborhood of $\Sing_\Lambda(X)$;
	\item points that are $E$-forward hyperbolic times for their entire infinite forward orbit; and 
	\item points that are $F$-backward hyperbolic times for their entire infinite backward orbit.
\end{enumerate} 
The goal of this section is to prove, under proper choices of certain parameters:
\begin{itemize}
	\item the co-existence of $E$-forward hyperbolic times and forward recurrence Pliss times of an orbit segment $(x,t)$;
	\item the co-existence of $F$-backward hyperbolic times and backward recurrence Pliss times of an orbit segment $(x,t)$;
	\item the existence of infinite hyperbolic times as typical points of every ergodic invariant measure {\em except} the point mass of singularities.
\end{itemize}
The first two cases are symmetric by considering $-X$.

We start with the Pliss lemma.
\begin{theorem} (The Pliss Lemma,~\cite{Pliss} and \cite[Lemma 11.8]{Mane.book})\label{t.pliss}
	Given $A\ge b_2 > b_1 >0$, let 
	$$
	\theta_0 = \frac{b_2-b_1}{A-b_1}.
	$$
	Then, given any real numbers $a_1, a_2, \ldots, a_N$ such that 
	$$
	\sum_{j=1}^N a_j\ge b_2N, \mbox{ and } a_j\le A \mbox{ for every } 1\le j \le N,
	$$
	there exist $\ell >\theta_0 N$ and $1\le n_1 <\cdots c_\ell\le N$ so that 
	$$
	\sum_{j=n_i}^{n-1}a_j\ge b_1(n-n_i) \mbox{ for every } n_i< n \le N \mbox{ and } i=1,\ldots,\ell.
	$$
\end{theorem}

Given a sequence $\{a_i\}_{i=1}^n$ and $b_1>0$. We say that it is a $b_1$-forward Pliss sequence, if for every $0< j\le n$ it holds 
\begin{equation}\label{e.pliss}
	\sum_{i=1}^{n-1}a_i\ge b_1 n. 
\end{equation}
Backward Pliss sequences can be defined similarly by considering the sequence $\tilde a_i = a_{n-i+1}$.

The next lemma is straightforward.
\begin{lemma}\label{l.gluepliss} Assume that 
	$\{a_i\}_{i=1}^n$ and $\{a'_i\}_{i=1}^m$ are two $b_1$-forward Pliss sequences. Define the sequence $\{c_j\}_{j=1}^{n+m}$ as
	$$
	c_j = \begin{cases}
		a_j,& j\le n\\
		a'_{j-n},& j>n.
	\end{cases}
	$$
	Then the sequence $\{c_j\}_{j=1}^{n+m}$ is a $b_1$-forward Pliss sequence. 
\end{lemma} 
A similar statement holds for backward Pliss sequences.

\subsection{$E$ and $F$-hyperbolic times for an orbit segment $(x,t)$}\label{ss.hyptime}
In this section, we assume that $\Lambda$ is a multi-singular hyperbolic compact invariant set for the vector field $X$, with singular dominated splitting $E_N\oplus F_N$ on the normal bundle $N_\Lambda$. 

First we applied the Pliss lemma on the scaled linear Poincar\'e flow. %For this purpose, we use the notation $\psi^*_{t,x}$ to highlight the base point of the scaled linear Poincar\'e flow. In other words, $\psi^*_{t,x}$ is the restriction of $\psi_t^*$ to the normal space $N(x)$.
\begin{definition}\label{d.EFPliss}
	Let $\lambda>1$. For $x\in\Reg(X)$ and $t>1$:
	\begin{itemize}
		\item We say that $x$ is a $(\lambda,E)$-forward hyperbolic time for the scaled linear Poincar\'e flow $(\psi_t^*)$ and the orbit segment $(x,t)$, if for every $j=1,\ldots,\floor{t}$ we have 
		\begin{equation}\label{e.Epliss}
			\prod_{i=0}^{j-1}	 \|\psi^*_{1}\mid_{E_N(x_{i})}\| \le \lambda^{-j}.
		\end{equation}
		\item We say that $x_t$ is a $(\lambda,F)$-backward hyperbolic time for the scaled linear Poincar\'e flow $(\psi_t^*)$ and the orbit segment $(x,t)$, if for every $j=1,\ldots,\floor{t}$ we have 
		\begin{equation}\label{e.Fpliss}
			\prod_{i=0}^{j-1}	 \|\psi^*_{-1}\mid_{F_N(x_{t-i})}\| \le \lambda^{-j}.
		\end{equation}
	\end{itemize} 
\end{definition}

The next lemma deals with the existence of $E$ and $F$-hyperbolic times. Recall $r_0$ and $U$ from previous sections, in particular, Lemma~\ref{l.robust.hyp}, \ref{l.Elarge} and \ref{l.Flarge}.
 \begin{lemma}\label{l.hyptime}
	There exists $\lambda_0>1$ and $\theta_0\in(0,1)$ such that for any open neighborhood $W$ of $\Sing_\Lambda(X)$, there exists  a constant  $T_{W}>0$ such that the following hold:
	
	\noindent {For any orbit segment $(x,t)\in \cO(U)$ satisfying the following properties:
	\begin{enumerate}
		\item  there exist $t_1,t_2\ge 0$ such that $(x_{-t_1}, t_1+t+t_2)\in \cO(U)$ and $x_{-t_1}\notin B_{r_0}(\Sing_\Lambda(X)), x_{t+t_2}\notin B_{r_0}(\Sing_\Lambda(X))$; %in particular, by Lemma~\ref{l.robust.dom.spl} and \ref{l.robust.hyp}, the ``larger'' orbit segment $(x_{-t_1}, t_1+t+t_2)$ that contains $(x,t)$ has a dominated splitting $E_N\oplus F_N$ on its normal bundle, and consequently, fake foliations are well-defined on the normal planes with size proportional to the flow speed;
		\item $x,x_t\notin W$ and $t>T_{W}$;
	\end{enumerate}
	we have:}
	\begin{itemize}
		\item there exist $\theta_0 \floor{t}$ many natural numbers $n_i$  such that each $x_{n_i}$ is a $(\lambda_0,E)$-forward hyperbolic time for $(\psi_t^*)$ and the orbit segment $(x_{n_i},t-n_i)$. 
		\item there exist $\theta_0 \floor{t}$ many natural numbers $n_i$  such that each $x_{t-n_i}$ is a $(\lambda_0,F)$-backward hyperbolic time for $(\psi_t^*)$ and the orbit segment $(x,t-n_i)$. 
	\end{itemize}  
\end{lemma}

\begin{proof}
	These two statements are symmetric by considering $-X$, so we shall only prove the first case. 
	
	Let $\eta>1$ and $T_{W}>0$ be given by Lemma~\ref{l.robust.hyp}, and fix any $\lambda_{0}\in(1, \eta)$. Below we shall prove that $T_{W}$ satisfies the desired property. 
	
	Let $(x,t)$ be an orbit segment satisfying the assumptions of Lemma~\ref{l.hyptime}. Then, by Lemma~\ref{l.robust.hyp}, in particular Equation \eqref{e.hyp2}, we have
	\begin{equation}\label{e.5.5}
		\prod_{i=0}^{\floor{t}-1}	 \left\|\psi^*_{1}|_{E_N(x_{i})}\right\| \le  \eta^{-\floor{t}}.
	\end{equation}
	
	Consider the sequence
	$$
	a_n = -\log\left\|\psi^*_{1}|_{E_N(x_{n-1})}\right\|, n=1,\ldots,\floor{t}.
	$$
	Setting $A=\log K_1^*$ with $K_1^*$ given by~Proposition~\ref{p.tubular3}, $b_2 =\log\eta$ and $b_1=\log\lambda_0$, then we have $a_j\le A, j=1,\ldots, \floor{t}$; furthermore, \eqref{e.5.5} implies that
	$$
	\sum_{i=1}^{\floor{t}}a_i\ge \floor{t}b_2.
	$$ 
	Then, Theorem~\ref{t.pliss} gives $\theta_0>0$, and times $n_1,\ldots,n_\ell$ with $\ell>\theta_0\floor{t}$ such that for each $n_j$, one has
	$$
	\sum_{i=n_j}^{n-1}a_i\ge \floor{t}b_1, \forall n =n_j+1,\ldots, \floor{t}.
	$$
	This shows that
	$$
	\prod_{i=0}^{k-1}	 \left\|\psi^*_{1,x_{n_j-1+i}}|_{E_N(x_{n_j+i})}\right\| \le  \lambda_0^{-k},\,\, \forall j=1,\ldots,\floor{t_j},
	$$
	that is, each $x_{n_j-1}$ is a $(\lambda_0,E)$-forward hyperbolic time for the orbit segment $(x_{n_j-1},t-n_j+1)$. 
\end{proof}

We conclude this subsection with the following two lemmas concerning the contraction and expansion of distance near a hyperbolic time. We assume that $a>0$ is taken small enough such that 
\begin{equation}\label{e.lambda1}
	\lambda_1:=\frac{\lambda_0}{a+1}>1,
\end{equation}
where $\lambda_0$ is given by Lemma~\ref{l.hyptime}.

Without loss of generality, we assume that for all $x\in\bM$ it holds 
\begin{equation}\label{e.inject.r}
	\frac12 \le \inf_ym(D_y\exp_x) <\sup_y\|D_y\exp_x\|\le 2.
\end{equation}
where the supremum and infimum are taken over $y\in B_{\mathfrak{d}_0}(x)$ with $\mathfrak{d}_0$ being the injectivity radius (decrease it if necessary).

The following lemma shows that if $x$ is a $E$-forward hyperbolic time for $(x,t)$, then the map $\cP_{s,x}$ exponentially contracts distance up to the flow speed on the fake leaf $\cF^E_{x,\cN}(x)$ up to time $t$. Recall the constant $K_0$ from Proposition~\ref{p.tubular}. 
\begin{lemma}\label{l.hyp.fakeEleaf}
	Let $\lambda_1$ be given by~\eqref{e.lambda1} for $a>0$ small enough. Then, there exist $\rho_1\in (0,\rho_0)$ and $\alpha_0>0$ such that every $\rho\in(0,\rho_1]$ and $\alpha\in (0,\alpha_0]$ have the following property: 
	
	Assume that $x\in\Lambda$ is a regular point and is a $(\lambda_0,E)$-forward hyperbolic time for $(\psi_t^*)$ and the orbit segment $(x,t)$. Then every point $y$ contained in the $\rho |X(x)|/(4K_0)$-disk of the fake leaf $\cF_{x,\cN}^E(x)$ centered at $x$ is $\rho$-scaled shadowed by the orbit of $x$ up to time $t$. Furthermore, for every $k\in [0,\floor{t}]\cap \NN$ it holds that 
	\begin{equation}\label{e.4.4a}
	d^E_{x_k}(y_{\tau_{x,y}(k)}) %= d_{\cF^E_{x_k, \cN}(x_k)}\left(x_k, \cP_{k,x}(y)\right)
	\le 4 \frac{|X(x_{k})|}{|X(x)|}\lambda_1^{-k} \cdot d^E_{x}(y) \le \lambda_1^{-k} K_0^{-1} \rho |X(x_{k})|,
	\end{equation}
	where $\tau_{x,y}(\cdot)$ is given by Definition \ref{d.shadow1}.

	A similar statement holds for orbit segments $(x,t)\in \cO(U)\setminus\Lambda\times\RR^+$ satisfying Lemma \ref{l.hyptime}, Assumption (1) and such that $x$ is a $(\lambda_0,E)$-forward hyperbolic time for $(\psi_t^*)$. 
\end{lemma}

We point out that this lemma is an alternate version of \cite[Lemma 7.1]{PYY23} (see Lemma~\ref{l.hyptime.dist} below) with a major difference: in \cite[Lemma 7.1]{PYY23} the $\rho-$scaled shadowing property is an assumption; here we shall prove it for points on the fake leaf of $x$ (note that this statement may not hold on fake leaves of a nearby point $x\ne z\in \cN_{\rho|X(x)|}(x)$!). This difference calls for a more careful treatment at each step of iteration; therefore we provide the proof below. We also remark that for non-singular flows and diffeomorphisms (without the term ${|X(x_{k})|/|X(x)|}$, of course), this is a well-known result. See~\cite{ABV00, PS}.

\begin{proof}
	First we consider $x\in\Lambda$. We prove inductively that \eqref{e.4.4a} holds. 
	
	For $k=0$, the assumption that $d^E_{x_k}(y_k)< \rho|X(x)|/(4K_0) $ means that the orbit of $y$ is $\rho$-scaled shadowed by the orbit of $x$ up to time one, and it is clear that \eqref{e.4.4a} holds. 
	
	Now assume that \eqref{e.4.4a} holds for $k=j$. This implies that the orbit of $y$ is $\rho$-scaled shadowed by the orbit of $x$ up to time $j+1$ due to Proposition~\ref{p.tubular1} and Lemma~\ref{l.shadowing2}. 
	Recall the definition of $P_{1,x}$ on the normal bundle by~\eqref{e.p}. To simplify notation, for $i=0,\ldots, j+1$ we write $\tilde y^i = \cP_{x_i}(y_i)$ and $\overline y^i = \exp_{x_i}^{-1}(\tilde y^i)$ for its image on the normal bundle. Then we have
	$$
	P_{1,x_i}(\overline y^{i-1}) = \overline y^{i},\,\, i=1,\ldots, j+1. 
	$$ 
	If we further define $\overline y^{i,*} = \overline y^i / |X(x_i)|$, then~\eqref{e.p*} means that 
	$$
	P^*_{1,x_i}(\overline y^{i-1,*}) = \overline y^{i,*},\,\, i=1,\ldots, j+1. 
	$$
	Below we shall first prove
	that 
	$$
	|D_{P_{\overline x}(\overline y)}P_{j+1, \overline x}(u)|\le  \frac{|X(x_{j+1})|}{|X(x)|}\lambda_1^{-(j+1)} |u|,
	$$
	for every vector $u$ such that $D_{P_{\tilde x}(\tilde y)}P_{i, \tilde x}(u)\in C_{\alpha}(E_N)$, $i=0,\ldots, j+1$. Note that this property is satisfied by tangent vectors of the fake leaf $\cF_{x,\cN}^E(x)$.

	Recall that $D_z P^*_{1,x}$ is uniformly continuous at a uniform scale (Proposition~\ref{p.tubular3} (2)). Also note that $D_{0_x} P^*_{1,x}= \psi^*_{1,x}$. Then, for every $a$ sufficiently small,  there exist $\alpha_0>0$ and $\rho_1>0$ such that for all $\rho \le \rho_1$, $\alpha\le\alpha_0$, $v\in C_\alpha(E_N(\overline y^{i,*}))$ with $u= D_{\overline y^{i,*}}P^*_{1, x_{i}}(v)\in C_\alpha(E_N(\overline y^{i+1,*}))$, we have 
	$$
	|D_{\overline y^{i,*}} P^*_{1, x_{i}}(v)|\le  (1+a) \left\|D_{\overline x_i} P^*_{1, x_{i}}|_{E_N(x_i)}\right\|  |v| =  (1+a)\left\|\psi^*_{1,x_i}|_{E_N(x_i)}\right\|  |u|.
	$$
	This implies that for every vector $u$ such that $D_{P_{\tilde x}(\tilde y)}P_{i, \tilde x}(u)\in C_{\alpha}(E_N)$, $i=0,\ldots, j+1$,
	\begin{align*}
		|D_{\overline y^{0,*}} P^*_{j+1, x_t}(u)|&= \left|\left( \prod_{i=1}^{j+1}D_{\overline y^{i-1,*}} P^*_{1, x_{i-1}}\right)(u)\right| \\
		%&\le \prod_{i=1}^{j+1} \left\|D_{\overline y^{i-1,*}} P^*_{1, x_{i-1}}|_{E_N(x_{t-i})}\right\||u|\\
		&\le (1+a)^{j+1}\prod_{i=1}^{j+1}\left\|\psi^*_{1,x_{i-1}}|_{E_N(x_{i-1})}\right\||u|\\
		&\le \left((1+a)\lambda_0^{-1}\right)^{j+1}|u|,
	\end{align*}
	where the last inequality follows from the assumption that $x$ is a $(\lambda_0,E)$-forward hyperbolic time for $(\psi_t^*)$. This shows that 
	$$
	|D_{\overline y^{0}} P_{j+1, x}(u)|\le \frac{|X(x_{j+1})|}{|X(x)|}\left((1+a)\lambda_0^{-1}\right)^{j+1}|u| =  \frac{|X(x_j)|}{|X(x_t)|}\left(\lambda_1^{-1}\right)^{j+1}|u|,
	$$
	as required. In particular, by~\eqref{e.inject.r} we have 
	$$
	|D_{y} \cP_{j+1, x}(u)|\le 4 \frac{|X(x_j)|}{|X(x_t)|}\left(\lambda_1^{-1}\right)^{j+1}|u|
	$$
	for every tangent vector of the fake leaf $\cF_{x,\cN}^E(x)$. This shows that \eqref{e.4.4a} holds for $j+1$, finishing the induction.

	For orbit segments in $\cO(U)\setminus \Lambda\times \RR^+$ satisfying Lemma \ref{l.hyptime}, Assumption (1), note that Lemma \ref{l.robust.dom.spl} and \ref{l.robust.hyp} provides the invariant bundles. The fake foliations on the normal plane have been constructed in Section \ref{sss.fakefoliation.nbhd} and form a local product structure, and therefore $d^E_{x_s}$ are well-defined. The same argument as before applies, giving the desired result. 
\end{proof}

For points that are on $\cN_{\rho_0|X(x)|}(x)$ but not on $\cF_{x,\cN}^E(x)$, we have the following lemma which was originally stated for $F$-backward hyperbolic times in~\cite{PYY23}. To obtain this version, one only need to consider $-X$.

\begin{lemma}\cite[Lemma 7.1]{PYY23}\label{l.hyptime.dist}
	Let $\lambda_1$ be given by~\eqref{e.lambda1} for $a>0$ small enough. Then, there exist $\rho_1\in (0,\rho_0)$ and $\alpha_0>0$ such that for every $\rho\in(0,\rho_1]$ and $\alpha\in (0,\alpha_0]$, the following property hold. 
	
	Let $x\in\Lambda$ be a regular point that is a $(\lambda_0,E)$-forward hyperbolic time for the orbit segment $(x,t)$ for $(\psi_t^*)$. Let $y\in \cN_{\rho|X(x)|}(x)$ be $\rho$-scaled shadowed by the orbit of $x$ up to time $t$. Then, for every $j=0,\ldots,\floor{t}$ and every vector $u\in C_\alpha(E_N(y))$  such that  $D_y\cP_{j,x}(u) \in C_\alpha(E_N(\cP_{j,x}(y))),$ one has 
	\begin{equation}\label{e.Fcont}
		|D_y\cP_{j, x}(u)|\le 4 \frac{|X(x_{j})|}{|X(x)|}\lambda_1^{-j} |u|.
	\end{equation}
	Consequently, for every $j=0,\ldots,\floor{t}$,
	\begin{equation}\label{e.F.dist.cont}
		d^E_{x_{j}}(y_{\tau_{x,y}(j)})\le 4 \frac{|X(x_{j})|}{|X(x)|}\lambda_1^{-j} \cdot d^E_{x}(y).
	\end{equation}

	A similar statement holds for orbit segments $(x,t)\in \cO(U)\setminus\Lambda\times\RR^+$ satisfying Lemma \ref{l.hyptime}, Assumption (1) and such that  $x$ is a $(\lambda_0,E)$-forward hyperbolic time for $(\psi_t^*)$. 
\end{lemma}

The proof can be found in~\cite{PYY23} and is omitted here. We remark that the $E$ and $F$-length of $\cP_{x_s}(y_s)$  are well-defined for every $s\in [0,t]$ since $y$ is assumed to be $\rho$-scaled shadowed by the orbit of $x$ up to time $t$. 

To conclude this subsection, we point out that both lemmas can be applied to $-X$ to obtain corresponding statements for $(\lambda_0,F)$-backward hyperbolic times. 

\subsection{Forward and backward recurrence Pliss times}\label{ss.recpliss}
In~\cite{PYY23} we introduced backward recurrence Pliss times that characterize the ``good recurrence property'' to a small neighborhood of $\Sing_\Lambda(X)$. In this paper, we shall consider recurrence Pliss times defined for both forward and backward orbits.

\begin{definition}\label{d.RecPliss}
	Given $\beta\in(0,1)$, a set $W\subset \bM$ and an orbit segment $(x,t)$ with $t\ge 1$, 
	\begin{itemize}
		\item 	we say that $x$ is a {\em $(\beta, W)$-forward recurrence Pliss time for the orbit segment $(x,t)$}, if for every $s=0,1,\ldots, \floor{t}$, it holds 
		\begin{equation}\label{e.RecPliss1}
			\frac{\# \{\tau\in[0,s]\cap\NN: x_{\tau}\in W\}}{s+1}\le \beta. 
		\end{equation}
		\item 	we say that $x_t$ is a {\em $(\beta, W)$-backward recurrence Pliss time for the orbit segment $(x,t)$}, if for every $s=0,1,\ldots, \floor{t}$, it holds 
		\begin{equation}\label{e.RecPliss2}
			\frac{\# \{\tau\in[0,s]\cap\NN: x_{t-\tau}\in W\}}{s+1}\le \beta. 
		\end{equation}
	\end{itemize}

\end{definition}
The backward recurrence Pliss time coincides with the recurrence Pliss time defined in~\cite[Section 4.2]{PYY23}. Furthermore, it is clear from the definition that $x_t$ is a  $(\beta, W)$-backward recurrence Pliss time for the orbit segment $(x,t)$ if and only if $y=x_t$ is a $(\beta, W)$-forward recurrence Pliss time for the orbit segment $(y,t)$ under the flow $-X$. 

\begin{remark}\label{r.RecPliss}
	It should be noted that the if $x$ is a $(\beta,W)$-forward recurrence Pliss time, then $x\notin W$; indeed, let $s\in[0,\floor{t}-1]\cap \NN$ be the smallest integer such that $x_{s}\in W$, then it follows that $\frac{1}{s+1}\le \beta$, i.e., $s\ge \beta^{-1}-1$. In other words, the orbit segment $\{ x,x_{1}, \ldots,x_{\floor{t}}\}$ (as an orbit segment for the time-$1$ map $f_{1}$) must spend at least $\floor{\beta^{-1}}$ many iterates outside $W$ before entering it for the first time. Similarly, if $x_t$ is a $(\beta,W)$-backward recurrence Pliss time, then the orbit segment $x_{t}, x_{t-1}, \ldots, x_{t-\floor{t}}$ (as an orbit segment for the time-$(-1)$ map $f_{-1}$) must spend at least $\floor{\beta^{-1}}$ many iterates outside $W$ before entering it for the first time.
\end{remark}

The next lemma gives the existence of $(\beta, W)$-recurrence times with density arbitrarily close to one, provided that the overall time that the orbit segment spends in $W$ is small enough.

\begin{lemma}\cite[Lemma 4.6]{PYY23}\label{l.rec}
	For any $\beta_0\in(0,1)$ and $\kappa\in(0,1)$, there exists $\beta_1 \in(0,\beta_0)$ such that for any set  $W\in\bM$ and any orbit segment $(x,t)\in\bM\times\RR^+$, the following statements hold:
	\begin{itemize}
		\item if 
		\begin{equation}\label{e.pliss1}
			\frac{\#\{s=0,\ldots, \floor{t}-1: x_{s}\in W\}}{\floor{t}} \le \beta_1,
		\end{equation}
		then there exist at least $\kappa (\floor{t}+1)$ many $n_i$'s, such that each $x_{n_i}$ is a $(\beta_0,W)$-forward recurrence Pliss time for the orbit segment $(x_{n_i},t-n_i)$ and the set $W$. 
		\item if 
		\begin{equation}\label{e.pliss2}
			\frac{\#\{s=0,\ldots, \floor{t}-1: x_{t-s}\in W\}}{\floor{t}} \le \beta_1,
		\end{equation}
		then there exist at least $\kappa (\floor{t}+1)$ many $m_i$'s, such that each $x_{t-m_i}$ is a $(\beta_0,W)$-backward recurrence Pliss time for the orbit segment $(x,t-m_i)$ and the set $W$. 
	\end{itemize}
\end{lemma}
The proof for backward recurrence Pliss times can be found in~\cite[Section 4.2]{PYY23} and is omitted. The proof for forward recurrence Pliss times follows from symmetry. 

%\begin{remark}
%	When $t\in \NN$, \eqref{e.pliss1} and \eqref{e.pliss2} are the same.
%\end{remark}

\subsection{Existence of simultaneous Pliss times}
The following definition will play a central role in the next section when we construct the good orbit segments collection $\cG$.
\begin{definition}\label{l.bi-Pliss}
	Given $\lambda_0>1$, $\beta_0\in(0,1)$ and a set $W\subset \bM$:
	\begin{itemize}
		\item we say that $x$ is a $(\lambda_0,E, \beta_0,W)$-simultaneous forward Pliss time for the orbit segment $(x,t)$, if $x$ is a $(\lambda_0,E)$-forward hyperbolic time for $(\psi_t^*)$, as well as a $(\beta_0, W)$-forward recurrence Pliss time for the orbit segment $(x,t)$.
		\item we say that $x_t$ is a $(\lambda_0,F, \beta_0,W)$-simultaneous backward Pliss time for the orbit segment $(x,t)$, if $x_t$ is a $(\lambda_0,F)$-backward hyperbolic time for $(\psi_t^*)$, as well as a $(\beta_0, W)$-backward recurrence Pliss time for the orbit segment $(x,t)$.
	\end{itemize}

\end{definition}

The existence of such times is given by the following lemma. For this purpose, let $\lambda_0>1$ be given by Lemma~\ref{l.hyptime}.

\begin{lemma}\cite[Lemma 4.7]{PYY23}\label{l.bihyptime}
	For every $\beta_0\in(0,1)$ there exists $\beta_1\in (0,\beta_0)$ such that for every open isolating neighborhood $W$ of $\Sing_\Lambda(X)$ that is contained in $B_{r_0}(\Sing_\Lambda(X))$, the constant $T_{W}$ given by  Lemma \ref{l.robust.hyp} satisfies that for every orbit segment $(x,t)\in\cO(U)$ satisfying the following properties:
	\begin{enumerate}
		\item  there exist $t_1,t_2\ge 0$ such that $(x_{-t_1}, t_1+t+t_2)\in \cO(U)$ and $x_{-t_1}\notin B_{r_0}(\Sing_\Lambda(X)), x_{t+t_2}\notin B_{r_0}(\Sing_\Lambda(X))$;
		\item $x,x_t\notin W$ and $t>T_{W}$;
	\end{enumerate}  the following statements hold: 
	\begin{itemize}
		\item assume that $$\frac{\#\{s=0,\ldots, \floor{t}: x_{s}\in W\}}{\floor{t}+1} \le \beta_1,$$ then there exists $s\in[0,t]\cap\NN$ such that $x_s$ is a $(\lambda_0,E, \beta_0,W)$-simultaneous forward Pliss time for the orbit segment $(x_s,t-s)$ and the set $W$.
		\item assume that $$\frac{\#\{s=0,\ldots, \floor{t}: x_{t-s}\in W\}}{\floor{t}+1} \le \beta_1,$$ then there exists $s\in[0,t]\cap\NN$ such that $x_{t-s}$ is a $(\lambda_0,F, \beta_0,W)$-simultaneous backward Pliss time for the orbit segment $(x,t-s)$ and the set $W$.
	\end{itemize}
	In both cases, given any $\epsilon>0$ one can further decrease $\beta_1$ so that the number of simultaneous Pliss times is at least $(\theta_0-\epsilon)\floor{t};$ i.e., the density of such Pliss times can be made arbitrarily close to $\theta_0$.
\end{lemma}

\begin{proof}
	These two cases are again symmetric, so we shall only consider the first case. 
	
	Let $\theta_0\in(0,1)$ be given by Lemma~\ref{l.hyptime}. Given any $\epsilon>0$ small, we take $\kappa>1-\epsilon$. Then 
	$$
	(1-\kappa) + (1-\theta_0) < 1-(\theta_0-\epsilon).
	$$
	Now we apply Lemma~\ref{l.rec} to obtain $\beta_1$ so that the density of forward recurrence Pliss times is at least $\kappa$. For every orbit segment satisfying both assumptions of Lemma~\ref{l.bihyptime}, Lemma~\ref{l.rec} provides $\kappa\floor{t}$ many $n_i\in \NN$ such that each $x_{n_i}$ is a $(\beta_0, W)$-forward recurrence Pliss time for the orbit segment $(x_{n_i},t-n_i)$. Moreover, by Lemma~\ref{l.hyptime} we have $\theta_0(\floor{t}+1)$  many natural numbers  $k_i$ such that each $x_{k_j}$ is a $(\lambda_0,E)$-forward hyperbolic time for the orbit segment $(x_{k_j},t-k_j)$. Then there are at least $(\theta_0-\epsilon)\floor{t}$ many places where $n_i$ must coincide with one of the $k_j$. Such a point is a $(\lambda_0,E, \beta_0,W)$-simultaneous forward Pliss time; the density of such points is at least $\theta_0-\epsilon,$ as required.

\end{proof}

We conclude this subsection with the following lemma which is an immediate consequence of Lemma~\ref{l.gluepliss}, and therefore the proof is omitted.

\begin{lemma}\label{l.gluePliss}
	Let $t,t'>1$ with $t\in \NN$. Assume that $x$ is a $(\lambda_0,E, \beta_0,W)$-simultaneous forward Pliss times for the orbit segment $(x,t)$, and that $x_{t}$ is a $(\lambda_0,E, \beta_0,W)$-simultaneous forward Pliss times for the orbit segment $(x_t,t')$. Then $x$ is a $(\lambda_0,F, \beta_0,W)$-simultaneous backward Pliss times for the orbit segment $(x,t+t')$. 
	
	A similar statement holds for backward Pliss times. 
\end{lemma}

\subsection{Infinite hyperbolic  times}\label{ss.infinitePliss}
Finally, let us turn our attention to points that are infinite Pliss times. They can be considered as the limit points of hyperbolic times for finite orbit segments, when the length of said orbits tend to infinity. Such infinite hyperbolic times will play a central role in the proof of the specification property; see Section~\ref{s.spec}.

\begin{definition}\label{d.infinitePliss}
	Let $\lambda>1$. For $x\in\Reg(X)$:
	\begin{itemize}
		\item we say that $x$ is a $(\lambda,E)$-forward infinite hyperbolic time for $(\psi_t^*)$, if for every $j\in\NN$ we have 
		\begin{equation}\label{e.inf.Epliss}
			\prod_{i=0}^{j-1}	 \|\psi^*_{1}\mid_{E_N(x_{i})}\| \le \lambda^{-j};
		\end{equation}
		\item we say that $x$ is a $(\lambda,F)$-backward infinite hyperbolic time for $(\psi_t^*)$, if for every $j\in\NN$ we have 
		\begin{equation}\label{e.inf.Fpliss}
			\prod_{i=0}^{j-1}	 \|\psi^*_{-1}\mid_{F_N(x_{-i})}\| \le \lambda^{-j}.
		\end{equation}
	\end{itemize} 
\end{definition}

\begin{remark}\label{r.inf.hyp}
Assume that $N = E^X_N\oplus F^X_N$ is a singular dominated splitting for $X$. Then $N = E^{-X}_N\oplus F^{-X}_N$, with $E^{-X}_N = F^X_N$ and $ F^{-X}_N = E^X_N $ is a singular dominated splitting for $-X$; furthermore, 
	\begin{itemize}
		\item Every $(\lambda,E^X)$-forward infinite hyperbolic time for $(\psi_t^*)$ is a $(\lambda,F^{-X})$-backward infinitely hyperbolic time for $(\psi_{-t}^*)$. 
		\item $x$ is a $(\lambda,E^X)$-forward infinite hyperbolic time for  $(\psi_{t}^*)$ if and only if for every $t\ge 1$, $x$ is a $(\lambda,E^X)$-forward hyperbolic time for the orbit segment $(x,t)$ and $(\psi_{t}^*)$.
	\end{itemize}
\end{remark}

Let $\lambda>1$ be fixed. We denote by
$$
\Lambda^E(\lambda) = \{x\in\Lambda:  \mbox{$x$ is a $(\lambda,E)$-forward infinite hyperbolic time for $(\psi_t^*)$}\},
$$
and 
$$
\Lambda^F(\lambda) = \{x\in\Lambda:  \mbox{$x$ is a $(\lambda,F)$-backward infinite hyperbolic time for $(\psi_t^*)$}\}.
$$
Note that $\Lambda^*(\lambda)$, $*=E,F$ may not be compact. This is because $\Reg(X)$ is not compact, and a sequence of infinite hyperbolic times may tend to a singularity (this phenomenon does not appear for diffeomorphisms and non-singular flows). To this end, we have the following result.  
\begin{proposition}\label{p.hyptime.cpt}
	Let $\Lambda$ be a multi-singular hyperbolic compact invariant set for the flow $X$. Then, for every $\lambda >1$ and $*=E,F$, the set  $\Lambda^*(\lambda)\cup\Sing_\Lambda(X)$ is compact. 
\end{proposition}

\begin{proof}
	We make the following observation: %, which itself follows from continuity of $(\psi^*_t)$ at regular points and the previous remark: 
	let $\{x^n\}\subset\Lambda$ be a sequence of points such that each $x^n$ is a $(\lambda,E)$-forward hyperbolic time for the orbit segment $(x^n,t)$ and $(\psi_{t}^*)$. Assume that $x^n\to x$ where $x$ is not a singularity; then $\psi_s^*$ is continuous at $x$ for every $s\ge 0$. As a result, $x$ is a $(\lambda,E)$-forward hyperbolic time for the orbit segment $(x,t)$. 
	
	Using this observation, we see that if $\{x^n\}\subset\Lambda$ is a sequence of points such that each $x^n$ is a $(\lambda,E)$-forward hyperbolic time for the orbit segment $(x^n,t_n)$ and $(\psi_{t}^*)$, with $t_n\to\infty$, and if $x$ is a limit point of $\{x^n\}$ which is not a singularity, then $x$ is a $(\lambda,E)$-forward hyperbolic time for the orbit segment $(x,t)$ and $(\psi_t^*)$ for every $t>0$. This shows that $x$ is a $(\lambda,E)$-forward infinite hyperbolic time.  In particular, this shows that $\Lambda^E(\lambda)\cup\Sing_\Lambda(X)$ is compact. Similar argument holds for $\Lambda^F(\lambda)\cup\Sing_\Lambda(X).$
\end{proof}

Next, we show that  $\Lambda^*(\lambda)$  is non-empty. Clearly this would require us to carefully choose $\lambda$. Indeed, in the following proposition we shall prove that there are plenty of infinite hyperbolic times as typical points of every ergodic measure that is not  a point mass of a singularity.

\begin{proposition}\label{p.hyptime.density}
	Let $\Lambda$ be a multi-singular hyperbolic compact invariant set, and $\eta>1$ be given by Definition~\ref{d.multising} (2). Then, for every $\lambda\in(1,\eta)$ there exists $\theta_\lambda>0$ such that for every ergodic invariant measure $\mu$ support on $\Lambda$ that is not a point mass of a singularity, one has 
	\begin{equation}\label{e.4.12a}
	\mu(\Lambda^*({\lambda}))\ge\theta_\lambda,\,\, *=E,F.
	\end{equation}
	Furthermore, \eqref{e.4.12a} also holds for all ergodic invariant measure of the time-one map $f_1$ (which may not be invariant for $X$) that is not a point mass of a singularity.
\end{proposition}
\begin{remark}\label{r.lambda0}
	From the proof it will be clear that one could take $\lambda = \lambda_0$ as in  Lemma~\ref{l.hyptime}, in which case $\theta_\lambda = \theta_0$.  
\end{remark}

\begin{proof}
	We shall prove the ``furthermore'' part: for every ergodic invariant measure of the time-one map $f_1$ (which may not be invariant for $X$) that is not a point mass of a singularity, we have $\mu(\Lambda^*({\lambda}))\ge\theta_\lambda,\,\, *=E,F$.
	
	First let us demonstrate how the proposition follows from its ``furthermore'' part. 	Let $\mu$ be an ergodic invariant measure of $X$ that is not a point mass of a singularity. Note that ergodicity for the flow does not imply ergodicity for the time-one map $f_1$; however, we can always take a typical ergodic component of $\mu$ for $f_1$, which we denote by $\tilde\mu$. Then for every measurable set $A$ we have 
	$$
	\mu(A) = \int_0^1 (f_t)_*\tilde\mu (A)\,dt. 
	$$
	Furthermore, each $\tilde \mu_s : = (f_s)_*\tilde\mu$ is ergodic and invariant for $f_1$. Then the proposition follows from the ``furthermore'' part by integrating over $t$.
	 
	It remains to prove the ``furthermore'' part. For this purpose we let $\mu$ be an ergodic invariant measure of $f_1$ that is not a point mass of a singularity, and take a Birkhoff typical point $x$ of $\mu$ under the map $f_1$. This means that the empirical measures 
	$$
	\mu_n:= \frac1n \sum_{i=0}^{n-1} \delta_{x_i}
	$$ 
	converge in weak-* topology to $\mu$; here $\delta_{x_i}$ is the point mass on $x_i=f_i(x)$. As before we shall only consider the case $*=E$, and the other case follows from symmetry.
	
	Let $\lambda\in (1,\eta)$ and $\theta_\lambda$ be given by Theorem~\ref{t.pliss} with $A = K_1^*, b_2 = \log \eta$ and $b_1 = \log\lambda$. 
	Let $W \subset B_{r_0}(\Sing_\Lambda(X))$ be an isolating neighborhood of $\Sing_\Lambda(X)$ and 
	let $T_{W}$ be given by Lemma \ref{l.robust.hyp}. Changing from $x$ to another point $y$ on the forward orbit of $x$, we may assume that $x\notin B_{r_0}(\Sing_\Lambda)(X)$.\footnote{Note that since $x$ is a Birkhoff typical point of a regular measure $\mu$, $x$ cannot be in the stable manifold of any singularity.} For each integer $n> T_{W}$ such that $x_n\notin W$ (there are infinitely many such $n$'s because $W$ is an isolating neighborhood of $\Sing_\Lambda(X)$, and $\mu$ is not a point mass on $\Sing_\Lambda(X)$), we have
	$$
	\prod_{i=0}^{n-1}\left\|\psi^*_1|_{E_N(x_i)}\right\|\le \eta^{-n}.
	$$
	By Lemma~\ref{l.hyptime}, there are at least $\theta_\lambda \cdot n$ many $k_i$'s such that each $x_{k_i}$ is a $(\lambda,E)$-forward hyperbolic time for the orbit segment $(x_{k_i}, n-k_i)$. 
	
	Now fix any $c\in (0,\theta_\lambda)$ and define the set 
	\begin{equation*}
		\begin{split}
			I_n(c)=\{x_{k}:\,\, & k\in[0, (1-c)(n-1)]\cap\NN, x_k \mbox{ is a $(\lambda,E)$-forward }\\
			& \mbox{ hyperbolic time for the orbit segment $(x_{k}, n-k)$}\}.
		\end{split}
	\end{equation*}
	Then we have 
	$$
	\mu_n(I_n(c)) \ge \theta_\lambda-c. 
	$$
	
	Take a sequence $n_i\to\infty$ %such that $x_{n_i}\notin W$ 
	and consider the set 
	$$
	\tilde I(c) = \bigcup_i I_{n_i}(c)
	$$
	and denote by $I(c)$ the set of all limit points of $\tilde I(c)$. In other words, $I(c)$ is obtained by take the limit of any sequence $\{x_{k_i}\}_{i=1}^\infty$ with $x_{k_i}\in I_{n_i}(c)$. Then it is clear that $I(c)$ is closed. 
	
	\medskip
	\noindent Claim 1. Every $y\in I(c)\cap\Reg(\Lambda)$ is a  $(\lambda,E)$-forward infinite hyperbolic time.
	
	\noindent Proof of Claim 1. Note that each $x_{k_i}\in I_{n_i}(c)$ is a $(\lambda,E)$-forward hyperbolic time for an orbit segment with length at least $c\cdot n_i$ which tends to infinity as $i\to\infty$. Then any limit point $y$ along such a sequence must be an infinite hyperbolic time as we proved in Proposition~\ref{p.hyptime.cpt}. 
	
	\medskip
	\noindent Claim 2. We have 
	$$
	\mu(I(c)) \ge \theta_\lambda - c.  
	$$
	
	\noindent Proof of Claim 2. Take an open neighborhood $U\supset I(c)$. Then there exists $N>0$ such that $I_{n_i}(c)\subset U$ for all $i> N$ (otherwise we have a limit point outside $U$, a contradiction). This shows that $\mu_{n_i}(\overline U)\ge\mu_{n_i}( U)\ge \theta_\lambda -c$, and consequently $\mu(\overline U)\ge \theta_\lambda-c$ since $\overline U$ is compact. Then take a decreasing sequence of neighborhoods $U_k$ such that $\bigcap_k \overline U_k = I(c)$ which is possible since $I(c)$ is compact, we obtain the desired inequality. 
	
	Combining these two claims, we obtain 
	$$
	\mu(\Lambda^E(\lambda)) \ge \mu(I(c)) \ge \theta_\lambda-c.
	$$
	Since $c$ is arbitrary, we get $\mu(\Lambda^E(\lambda)) \ge\theta_\lambda$ as required. 
\end{proof}

The following theorem of Liao shows for points in $\Lambda^*(\lambda)$, there exist invariant manifolds at a certain relative uniform size.

\begin{definition}
	Let $x$ be a regular point of the vector field $X$. An embedded submanifold $W_{loc, \cN}^s(x)\subset \cN(x)$ is called the local stable manifold of $x$ on the normal plane, if:
	\begin{itemize}
		\item there exists $\rho>0$ such that every $y\in W_{loc, \cN}^s(x)$ is $\rho$-scaled shadowed by the orbit of $x$ up to time $t$ for every $t>0$;
		\item for every $y\in W_{loc, \cN}^s(x)$, we have 
		$$
		d_{\cN(x_t)}(\cP_{t,x}(y), x_t) \to 0 \mbox{ as } t\to+\infty.
 		$$
	\end{itemize}
	The local unstable manifold can be defined similarly by considering $-X$. 
\end{definition}

\begin{proposition}\label{p.inv.mld}
	For each $\lambda>1$, there exists a constant $\rho_\textit{inv} = \rho_\textit{inv}(\lambda)\in (0,\rho_0] $ such that:
	\begin{itemize}
		\item for every $x \in \Lambda^E(\lambda)$, the local stable manifold of $x$ exists and has diameter at least $\rho_\textit{inv}|X(x)|$; and 
		\item for every $x\in \Lambda^F(\lambda)$, the  local unstable manifold of $x$ exists and has diameter at least $\rho_\textit{inv}|X(x)|$.
	\end{itemize}
\end{proposition}

The proof follows from the Hadamard-Perron theorem and the uniform continuity of $DP_{1,x}^*$ in Liao's scaled tubular neighborhood (see Proposition~\ref{p.tubular3} (2)). Also see Lemma~\ref{l.hyp.fakeEleaf}.

We conclude this section with the following proposition concerning the relation between invariant manifolds of infinite hyperbolic times and fake leaves at hyperbolic times of a finite orbit segment.

Below we shall use $\cF^*_{x,\cN}(x,r)$ to denote the $r$-ball centered at $x$ inside the fake leaf $\cF^*_{x,\cN}(x)$, $*=E,F$.

\begin{proposition}\label{p.fake.inv}
	Let $\{(x^n,t_n)\}$ be a sequence of orbit segments in $\cO(U)$ satisfying Lemma \ref{l.hyptime}, Assumption (1). Furthermore, 
	assume that $\{x^n\}$ converges to a regular point $x$, $t_n \to\infty$,
	and $x^n$ is a $(\lambda_0,E)$-forward hyperbolic time for the orbit segment $(x^n,t_n)$. Then $x$ is a $(\lambda_0,E)$-forward infinite hyperbolic time; furthermore,  we have
	$$
	\lim_n\cF_{x_n,\cN}^{E, (x,t_n)}(x_n,\rho_\textit{inv}|X(x_n)|)= W_{loc,\cN}^s(x),\footnote{	When $x_n\in\Lambda$, $\cF_{x_n,\cN}^{E, (x,t_n)}$ is replaced by $\cF_{x_n,\cN}^{E}$.}
	$$
	where the limit is taken in the $C^0$ topology in the space of $C^1$ embeddings of $\dim E_N$-dimensional disks. 

\end{proposition}

Since every infinite hyperbolic time $x$ is also a hyperbolic time for the orbit segment $(x,t)$ for every $t>0$, we see that 
$$
\cF_{x,\cN}^E(x,\rho_\textit{inv}|X(x)|)= W_{loc,\cN}^s(x).
$$

The first part of the proposition, namely that $x$ is a infinite hyperbolic time, is due to the continuity of $(\psi_t^*)$ at $x$ for every $t$ (we already used this fact in Proposition~\ref{p.hyptime.cpt}). The proof of the second part is essentially the continuity of the local invariant manifold given by the Hadamard-Perron Theorem, and is therefore omitted.

\section{Proof of Theorem~\ref{m.B}}\label{s.mainthm.proof}
This section contains the proof of  Theorem~\ref{m.B}. Unless otherwise specified (for example, see Theorem~\ref{t.Bowen}), we shall always assume that $\Lambda$ is a multi-singular hyperbolic compact invariant set that is an isolated chain recurrence class. 

\subsection{Orbit segments with the Bowen property and specification}\label{ss.BS}
In this section we construction two orbit segments collections, $\cG_B$ and $\cG_S$, on which the Bowen property and specification holds, respectively. 

\subsubsection{Definition of $\cG_B$ and the Bowen property}\label{sss.bowen}

Let $r_0>0$ and $U$ be the isolating neighborhood of $\Sing_\Lambda$(X) given by Lemma~\ref{l.robust.hyp}, and recall the definition of $W_r\subset B_{r_0}(\Sing_\Lambda(X))$ from \eqref{e.W.def}. Given $\lambda_0>1$, $\beta\in (0,1)$ and $0<r<r_0$, we let $\cG_B=\cG_B(\lambda_0,\beta, r_0, W_{r})$ be the collection of orbit segments $(x,t)\in \cO(U)$ with $t\in\NN$, such that:
\begin{enumerate}
	\item there exist $t_1,t_2\ge 0$ such that $(x_{-t_1}, t_1+t+t_2)\in \cO(U)$ and $x_{-t_1}\notin B_{r_0}(\Sing_\Lambda(X)), x_{t+t_2}\notin B_{r_0}(\Sing_\Lambda(X))$; in particular, by Lemma~\ref{l.robust.dom.spl} and \ref{l.robust.hyp}, the ``larger'' orbit segment $(x_{-t_1}, t_1+t+t_2)$ that contains $(x,t)$ has a dominated splitting $E_N\oplus F_N$ on its normal bundle, and consequently, fake foliations are well-defined on the normal planes with size proportional to the flow speed;
	\item $x$ is a $(\lambda_0,E, \beta,W_r)$-simultaneous forward Pliss times for the orbit segment $(x,t)$ and the set $W_r$;
	\item $x_t$ is a $(\lambda_0,F, \beta,W_r)$-simultaneous backward Pliss times for the orbit segment $(x,t)$ and the set $W_r$;
\end{enumerate}
Since $x$ is a forward recurrence Pliss time, by Remark~\ref{r.RecPliss} we have $x\notin W_r$; the same statement holds for $x_t$. Also note that there is a natural symmetry in the definition of $\cG_B$: $\cG_B$ defined for $-X$ is the same as $\cG_B$ for $X$. This observation will significantly simplify the proof of the next theorem.

The following theorem gives the Bowen property on $\cG_B$ for proper choice of parameters. 
\begin{theorem}\label{t.Bowen}
	Let $\Lambda$ be a multi-singular hyperbolic set with all singularities active. Then, there exists $\beta_0\in(0,1)$, such that for every $\beta\in (0,\beta_0]$, $r_0>0$ small enough and $U\supset \Lambda$ small enough, there exists $\overline r\in (0,r_0)$, such that for every $r\in (0,\overline r)$, there exists $\vep_B>0$ such that $\cG_B(\lambda_0,  \beta,r_0, W_r)$ has the Bowen property for every H\"older continuous function at scale $\vep_B$.
\end{theorem}
Here the smallness of $r_0$ and $U$ is given by Lemma \ref{l.nbhd.1} to \ref{l.robust.hyp} (keep in mind Remark \ref{r.U}).

{We remark that $\beta_0$ is a constant that only depends on the hyperbolicity near $\Lambda$ ($\eta$ in \eqref{e.hyp1} and \eqref{e.hyp2}) and $K_1^*$, the $C^0$ norm of $\psi_1^*$ (Proposition \ref{p.tubular3}. Also, in Theorem~\ref{t.Bowen} we do not need Assumption (B) of Theorem \ref{m.B}, nor do we need $\Lambda$ to be isolated or chain-transitive.} The proof of Theorem~\ref{t.Bowen} occupies Section~\ref{s.bowen}.

\subsubsection{Definition of $\cG_S$ and specification}
Next, we turn our attention to the specification property. To start with, recall that $\theta_0$ is the density of Pliss times given by Lemma~\ref{l.hyptime} (also keep in mind Proposition~\ref{p.hyptime.density} and Remark~\ref{r.lambda0}) with $\lambda_0$ fixed as before. Let $\beta_0$ be given by Theorem~\ref{t.Bowen}. Below we will take $\beta_1\ll\beta_0$
as given by Lemma~\ref{l.bihyptime} with $\epsilon = \theta_0/1000 $. In other words, if $t\in\NN$ and
\begin{equation}\label{e.choice.beta1a}
	\frac{\#\{s=0,\ldots, \floor{t}: x_{s}\in W\}}{\floor{t}+1} \le \beta_1
\end{equation}
then for $*=E,F$, the density of $(\lambda_0,*,\beta_0,W)$-simultaneous Pliss times in the time interval $[0,t]\cap\NN$ is at least $\frac{999}{1000}\theta_0$. We shall further decrease $\beta_1$ so that $\beta_1<\theta_0/1000.$

Recall that $\Lambda^*(\lambda_0)$,  $*=E,F$, are the sets of $(\lambda_0,*)$-infinite hyperbolic times in $\Lambda$. On these sets, points have invariant manifolds whose sizes are proportional to the flow speed. For $x\in \Lambda^E(\lambda_0)$ and $\delta_x< \rho_\textit{inv}(\lambda_0) |X(x)|$ (recall $\rho_\textit{inv}$ from Proposition~\ref{p.inv.mld}), we write $W^s_{\delta_x,\cN}(x)$ for the $\delta_x$-disk centered at $x$ inside $W^s_{loc,\cN}(x)$. $W^u_{\delta_x,\cN}(x)$ is defined similarly for points in $\Lambda^F(\lambda_0)$. For a hyperbolic periodic orbit $\gamma$ and $d>0$, we write $W^*_d(\gamma)$ for the disk of $W^*(\gamma)$ with diameter $d$, $*=s,u$. Their dimensions satisfy
$$
\dim W^s_{\delta_x,\cN}(x) = \dim E_N = \dim W^s_d(\gamma)-1.
$$ 

Below we introduce a proposition concerning the transversal intersection between invariant manifolds of points in $\Lambda^*(\lambda_0)$ and a hyperbolic periodic orbit $\gamma$. 
%This result provides a key mechanism for producing transversal intersections in the absence of a global hyperbolic structure and will play a central role in the proof of the specification property. 
This result constitutes a key new ingredient of the paper, providing a mechanism to produce transversal intersections and playing a central role in the proof of the specification property.
The proof shall be postponed to Section \ref{ss.transversal}.

\begin{proposition}\label{p.transversal}
	Under Assumption (A) and (B) of Theorem~\ref{m.B}, for every isolating open neighborhood $W\subset B_{r_0}(\Sing_\Lambda(X))$ of $\Sing_\Lambda(X)$ and every $\delta>0$ sufficiently small, there exist compact sets $K^*_W\subset \Lambda^*(\lambda_0)\cap W^c$, a hyperbolic periodic orbit $\gamma$ and a constant $d_0>0$, such that the following properties hold: 
	\begin{enumerate}
		\item every $x\in K^E_W$ satisfies $W^s_{\delta,\cN}(x)\pitchfork  W^u_{d_0}(\gamma)\ne\emptyset$, and every $x\in K^F_W$ satisfies $W^u_{\delta,\cN}(x)\pitchfork  W^s_{d_0}(\gamma)\ne\emptyset$;
		\item for every $f_1$-invariant measure $\mu$ (not necessarily invariant for $X$) with $\mu\left(
		W\right)<\beta_1$, it holds that 
		$$
		\mu(\Lambda^*(\lambda_0)\setminus K^*_W)< 2\beta_1; 
		$$
		moreover, $\mu(K^*_W)>\theta_0-3\beta_1 \ge \frac{997}{1000}\theta_0$.
	\end{enumerate}
\end{proposition}
%
%\begin{remark}\label{r.homoclinic.related}
%	Let $\mu_1$, $\mu_2$ be two measures with $\mu_i\left(W\right)<\beta_1$, $i=1,2.$ Using the Inclination Lemma and the previous proposition, we see that $\mu_1$ and $\mu_2$ are homoclinically related in the following sense: there exist a $\mu_1$ positive measure set $A_1$ and a $\mu_2$ positive measure set $A_2$ such that for every $x\in A_1$ and $y\in A_2$, the stable manifold of $\Orb(x)$ exists and transversely intersect with the unstable manifold of $\Orb(y).$ Furthermore, one can choose $A_1$, $A_2$ in such a way that the size of invariant manifolds for which the transversal intersection happens is uniformly bounded. 
%\end{remark}

Now we are ready to define $\cG_S$. For any open neighborhoods $U^*$ ($*=E,F$) of $K^*_W$, we define $\cG_S = \cG_S(\lambda_0,r_0,W, U^E,U^F)$ as the collection of orbit segments $(x,t)\in{\cO(U)}$ with $t\in\NN$ such that following properties hold:
\begin{enumerate}
	\item there exist $t_1,t_2\ge 0$ such that $(x_{-t_1}, t_1+t+t_2)\in \cO(U)$ and $x_{-t_1}\notin B_{r_0}(\Sing_\Lambda(X)), x_{t+t_2}\notin B_{r_0}(\Sing_\Lambda(X))$;
	\item $x\in U^E\cap {W^c}$ is a $(\lambda_0,E)$-forward hyperbolic time for the orbit segment $(x,t)$ and $(\psi_t^*)$;
	\item $x_t\in U^F\cap {W^c}$ is a $(\lambda_0,F)$-backward hyperbolic time for the orbit segment $(x,t)$ and $(\psi_t^*)$.
\end{enumerate}

The next theorem gives the specification on $\cG_S$. 

\begin{theorem}\label{t.spec}
Assume that Theorem~\ref{m.B}, Assumption (A) and (B) hold. Then, for $r_0>0$ sufficiently small, every isolating open neighborhood $W\subset B_{r_0}(\Sing_\Lambda(X))$ and every $\delta>0$ sufficiently small, {there exist $U_1\subset U$  open neighborhood of $\Lambda$},   open neighborhoods $U^*\supset K^*_W$, $*=E,F$,  such that:
\begin{enumerate}
	\item  tail specification at scale $\delta$ holds on $\cG_S(\lambda_0,r_0,W, U^E,U^F)\cap \cO(U_1)$ with shadowing orbits contained in $\cO(U)$; and 
	\item  tail specification at scale $\delta$ holds on $\cG_S(\lambda_0,r_0,W, U^E,U^F)\cap \Lambda\times\RR^+$ with shadowing orbits contained in $\cO(U_1)$.
\end{enumerate}
%{In other words, Conditions $(I_0)$ and $(I_1)$ of Theorem~\ref{t.improvedCL1} hold for $\cG_S\cap\Lambda\times\RR^+$ and $\cG_S\cap\cO(U_1)$, respectively..}

%Here $(\cG_S)^1$ is the set 
%$$
%(\cG_S)^1 = \{(x,t): \mbox{ there exists $p,s\in [0,1]$ such that $(x_p,t-s-p)\in \cG_S$}\},
%$$
%similar to Section \ref{s.CT}

\end{theorem}
We point out that in Theorem~\ref{t.spec} we do not require $x$ or $x_t$ to be recurrence Pliss times. Instead, we need Assumption (B) of Theorem~\ref{m.B} concerning the transversal intersection between invariant manifolds of periodic orbits in $\Lambda$. 

In the rest of this section, we will prove Theorem \ref{m.B} assuming that Theorem~\ref{t.Bowen}, \ref{t.spec} and Proposition~\ref{p.transversal} hold; the proof of these theorems will be postponed to Section \ref{s.bowen} and \ref{s.spec}.
	
\subsection{Choice of parameters and the $(\cP,\cG,\cS)$-decomposition}\label{ss.para}
Now we describe the choice of $\vep$ (the scale for the Bowen property) and $\delta$ (the scale for specification). We will also construct the $(\cP,\cG,\cS)$-decomposition on the space of orbit segments $\cO(U_1)$ for a neighborhood $U_1\subset U$ of $\Lambda$.

Given an invariant probability measure $\mu$ and a continuous function $\phi:\Lambda\to \RR$, we define its metric pressure as
$$
P_\mu(\phi) = h_\mu(f_1)+\int\phi\,d\mu.
$$
The next lemma is taken from~\cite{PYY23}. It states that for potential functions $\phi$ satisfying Assumption (C) of Theorem~\ref{m.B}, if a measure assigns a slightly larger weight to $B_{r_0}(\Sing_\Lambda(X))$ then it must have small pressure and therefore cannot be an equilibrium state.

As before, $\Cl$ denotes the closure of a set. 
\begin{lemma}\cite[Lemma 5.3]{PYY23}\label{l.gap1}
	Let $\phi:\bM\to\RR$ be a H\"older continuous function satisfying Assumption (C) of Theorem~\ref{m.B}. Then for any $b>0$, there exist $r_0>0$ and $a_0>0$ such that for any $r\le r_0$ and any invariant probability measure $\mu$ satisfying $\supp\mu\subset \Lambda$ and $\mu\left(\Cl\left( B_{r}(\Sing_\Lambda(X))\right)\right)\ge b$, we have $P_\mu(\phi)<P(\phi)-a_0$. 
\end{lemma}

We remark that the proof of \cite[Lemma 5.3]{PYY23} does not rely on the sectional-hyperbolic structure of $\Lambda$. Instead, it only requires the upper semi-continuity of $h_\mu(f_1)$ as a function of invariant measures and the assumption that the singletons at singularities are not equilibrium states. In our setting this is given by Theorem~\ref{m.A}; the rest of the proof remains the same. 

We also need the next lemma which improves Lemma~\ref{l.bihyptime}. 

\begin{lemma}\label{l.bipliss2}
	Let $r_0>0$ be small and $U_1$ be a sufficiently small neighborhood of $\Lambda$.  For any isolating neighborhood $W$ of $\Sing_\Lambda(X)$ in $B_{r_0}(\Sing_\Lambda(X))$,  
	$\delta>0$ and neighborhood $U^E$ of $K^E_{W}$, there exists a constant $T_U\ge T_W$ where $T_W$ is given by Lemma \ref{l.robust.hyp}, such that the following hold.

	\noindent For any orbit segment $(x,t)\in \cO(U_1)$ satisfying
	\begin{enumerate}
		\item  there exist $t_1,t_2\ge 0$ such that $(x_{-t_1}, t_1+t+t_2)\in \cO(U)$ and $x_{-t_1}\notin B_{r_0}(\Sing_\Lambda(X)), x_{t+t_2}\notin B_{r_0}(\Sing_\Lambda(X))$;
		\item $x,x_t\notin W$ and $t>T_{U}$;
		\item the following inequality holds
		\begin{equation}\label{e.5.5a}
		\frac{1}{\floor{t}+1}\#\left\{ i\in[0,\floor{t}]\cap\NN: x_i\in W\right\} < \beta_1;
		\end{equation}
	\end{enumerate}
	there exist $x_i\in U^E, i\in[0,\floor{t}]\cap\NN$ such that $x_i$ is a $(\lambda_0,E,\beta_0,W)$-simultaneous forward Pliss time for the orbit segment $(x_i, t-i)$ and $(\psi_t^*)$. 
	
	A similar statement holds for any neighborhood $U^F$ of $K^F_{W}.$
\end{lemma}

Comparing to  Lemma~\ref{l.bihyptime}, the previous lemma provides more information on where simultaneous Pliss times occur: they must appear near infinite hyperbolic times. 

\begin{proof}	
	By Lemma~\ref{l.bihyptime}, there  exist  $(\lambda_0,E,\beta_0,W)$-forward hyperbolic times as long as $t> T_W$. We only need to show that such times can be found in $U^E$. The proof below is similar to the proof of Proposition~\ref{p.hyptime.density}.
	
	Towards a contradiction, we assume that there exists a sequence of orbit segments $(x^n,t_n)$ with $t_n\to\infty$ such that $x^n,(x^n)_{t_n}\notin  W$, and \eqref{e.5.5a} hold for each $n$, but there is no $(\lambda_0,E,\beta_0,W)$-simultaneous forward Pliss time in $U^E$.

	For simplicity's sake we assume that $t_n\in\NN$. For each $n$ we consider the following sets: 
	\begin{equation*}
		\begin{split}
			I_n = \Bigg\{(x^n)_i:\,\, & i\in \left[0,\left(1-\frac{1}{1000}\theta_0\right)t_n\right]\cap\NN, (x^n)_i \mbox{ is a $(\lambda_0,E, \beta_0,W)$-}\\
			&\mbox{ simultaneous forward Pliss time for the orbit}\\
			&	\mbox {   segment  } ((x^n)_i,t_n-i)	\Bigg\}.
		\end{split}
	\end{equation*}
%	and
%	\begin{equation*}
%		\begin{split}
%			J_n = \Bigg\{(x^n)_i:\,\, & i\in \left[0,\frac{999}{1000}t_n\right]\cap\NN, (x^n)_i \mbox{ is a $(\lambda_0,E)$-forward}\\
%				&\mbox{ hyperbolic time for the orbit segment  } ((x^n)_i,t_n-i), \\ &\mbox{ but is not a simultaneous Pliss time}
%				\Bigg\}.
%		\end{split}
%	\end{equation*}
%	In other words, $I_n\cup J_n$ is the collection of all $(\lambda_0,E)$-forward hyperbolic time contained in the time interval $[0,0.999\cdot t_n]$.
	By  Lemma~\ref{l.bihyptime} %and the choice of $\beta_1$ by~\eqref{e.choice.beta1a} 
	(see also Lemma~\ref{l.rec}) we have $\#I_n \ge  \frac{998}{1000} \theta_0t_n$. % furthermore by Lemma~\ref{l.rec} we get  $\#J_n \le  \frac{1}{1000} \theta_0t_n$.

	Now consider the empirical measure 
	$$
	\mu_n:= \frac{1}{t_n} \sum_{i=0}^{t_n-1}\delta_{(x^n)_i}
	$$	
	where $\delta_y$ is the point mass of $y$. By taking subsequence if necessary, we assume that  $\{\mu_n\}$ converges to an $f_1$-invariant measure $\mu$. Furthermore, \eqref{e.5.5a} shows that 
	$$
	\mu_n(W) < \beta_1,\forall n>0. 
	$$
	As a result, 
	\begin{equation}\label{e.5.5b}
	\mu(W)\le \beta_1 <\frac{\theta_0}{1000}.  
	\end{equation}
	
	Define $\tilde I = \bigcup_n I_n$, and denote by $I$ the set of all limit points of $\tilde I$. %Define $\tilde J$ and $J$ in the same way. 
	Then $I$ is compact. 
	
	\medskip 
	\noindent Claim 1. Every $y\in I\cap\Reg(\Lambda)$ is a $(\lambda_0,E)$-forward infinite hyperbolic time; consequently $I\subset \Lambda^E(\lambda_0)$. %Same holds for $J$.
	
	\noindent This claim immediately follows from the continuity of $(\psi_t^*)$ at regular points. Here the choice of $\left(1-\frac{1}{1000}\theta_0\right)t_n$ in the definition of $I_n$ guarantees that those forward hyperbolic times corresponds to orbit segments with length at least $0.001\theta_0t_n$, which tends to infinity. 
	
	\medskip 
	\noindent Claim 2. We have 
	$$
	\mu(I)\ge \frac{998}{1000}\theta_0;\,\, %\mbox{ therefore } \mu(\lambda_0,E\setminus I)\le 
	$$
	
	\noindent Proof of Claim 2. Take any open neighborhood $V\supset I$, we have $I_n\subset V$ for all $n$ large enough. Then we obtain $\mu_n(\overline V)\ge \mu_n(I_n) \ge \frac{998}{1000}\theta_0$ for all $n$ large. Therefore the same inequality holds for $\mu(\overline V)$. Since $V$ is arbitrary, we get the same inequality for $I$ by taking a sequence of decreasing neighborhoods. 

	On the other hand, by Proposition~\ref{p.transversal} (2) and Equation \eqref{e.5.5b} we see that 
	$$
	\mu\left(I\cap K^E_W\right)\ge \mu(I) - \mu\left(\Lambda^E(\lambda_0)\setminus  K^E_W\right)\ge \frac{996}{1000}\theta_0.
	$$
	Since $U^E$ is an open neighborhood of $K^E_W$, we must have $I_n\cap U^E\ne\emptyset$ for $n$ large enough,  a contradiction.
	
\end{proof}

We continue with the choice of parameters. 
\begin{enumerate}
	\item We have already chosen $\beta_0$ given by Theorem~\ref{t.Bowen} and let $\beta_1 < \min\{\beta_0,\theta_0/1000\}$ be given by Lemma~\ref{l.bihyptime} so that the density of simultaneous Pliss times (both backward and forward) is at least $\frac{999}{1000}\theta_0$ on orbit segments satisfying \eqref{e.choice.beta1a}.
	\item We then apply Lemma~\ref{l.gap1} with $b = \beta_1/2$ to obtain $r_0$ (which must also satisfy Lemma \ref{l.robust.dom.spl} and \ref{l.robust.hyp}) and $a_0>0$ so that for any invariant probability measure $\mu$ supported on $\Lambda$, we have 
	\begin{equation}\label{e.gap1}
		\mu(\Cl(B_{r_0}(\Sing_\Lambda(X))))\ge \beta_1/2\implies P_\mu(\phi)<P(\phi)-a_0.
	\end{equation}	
	\item {Next we obtain $U$, the open neighborhood of $\Lambda$ from Lemma \ref{l.robust.dom.spl} and \ref{l.robust.hyp}. For any orbit segment $(x,t)$ in $\cO(U)$ with $x,x_t\notin B_{r_0}(\Sing_\Lambda(X))$ there exists a dominated splitting on the normal bundle of $(x,t)$ with desired contraction/expansion property, namely \eqref{e.hyp1} and \eqref{e.hyp2}, as long as the time is sufficiently long (depending on $W$ which will be chosen in the next step). In particular, the normal plane at $x_s,s\in[0,t]$ has fake foliations that form a local product structure with size proportional to $|X(x_s)|$.  If necessary, we shall decrease $r_0$ and $U$ so that Lemma~\ref{l.bipliss2} holds. Keep in mind Remark \ref{r.U}.}
	\item Next we apply Theorem~\ref{t.Bowen} with $\beta = \beta_0$, $r_0$ and $ U$ (further decrease them if necessary, and note that $\beta_0$ does not depend on them) to obtain $r\in (0, \overline r]$, $\vep_B>0$ and open set $W_r = \bigcup_{\sigma\in\Sing_\Lambda(X)} W_r(\sigma)\subset B_{r_0}(\Sing_\Lambda(X)) $ an isolating neighborhood of $\Sing_\Lambda(X)$  so that the Bowen property holds on $\cG_B(\lambda_0,\beta_0,r_0,W_r)$ at scale $\vep_B$; note that the Bowen property also holds at all smaller scales. 
	\item We shall further decrease $\vep_B$ if necessary, such that $X|_ \Lambda$ is almost expansive at scale $\vep_B$ (recall Theorem~\ref{m.A}); furthermore, since $\phi$ is continuous, we require
	\begin{equation}\label{e.vep}
		\sup_{x,y\in\bM, d(x,y)\le\vep} |\phi(x)-\phi(y)|< \frac{a_0}{2}.
	\end{equation}
	The purpose of this step is to get rid of the second scale in the pressure $P([\cC], \phi,\delta,\vep)$.
	\item We pick $\delta_0>0$ which is the scale of the specification; fix
	$$
	\delta_0\le \frac{\vep_B}{1000L_X}
	$$
	where $L_X$ is defined by \eqref{e.lip}; decreasing $\delta_0$ if necessary, we apply Proposition~\ref{p.transversal} with $W = W_r$ to obtain compact sets $K^*_W\subset \Lambda^*(\lambda)\cap (W_r)^c$, $*=E,F$.
	\item Finally we apply Theorem~\ref{t.spec} on $U_1\subset U$, an open neighborhood of $\Lambda$ in $U$, and $\delta_0$ (decrease it if necessary) to obtain  open neighborhoods $U^*\supset K^*_W$ and the constant $T_U\ge T_W$, such that $\cG_S = \cG_S(\lambda_0,r_0,W_r,E^E,U^F)\subset \cO(U)$ satisfies 
	\begin{itemize}
		\item $\cG_S\cap\cO(U_1)$ has tail specification at scale $\delta_0$ with shadowing orbits contained in $\cO(U)$; and 
		\item $\cG_S\cap\Lambda\times\RR^+$ has tail specification at scale $\delta_0$ with shadowing orbits contained in $\cO(U_1)$.
	\end{itemize}
\end{enumerate}

Now we construct the $(\cP,\cG,\cS)$ decomposition on $\cO(U_1)$ and describe $\cG$, the collection of ``good'' orbit segments required in~Theorem~\ref{t.improvedCL1}. In view of Theorem~\ref{t.Bowen} and Theorem~\ref{t.spec}, we take $\cG = \cG_B\cap \cG_S\cap\cO(U_1)$. More precisely, $\cG$ consists of orbit segments in $\cO(U_1)$ with the following properties:
\begin{itemize}
	\item $t\in\NN$.
	\item there exist $t_1,t_2\ge 0$ such that $(x_{-t_1}, t_1+t+t_2)\in \cO(U)$ and $x_{-t_1}\notin B_{r_0}(\Sing_\Lambda(X)), x_{t+t_2}\notin B_{r_0}(\Sing_\Lambda(X))$; therefore the dominated splitting $E_N\oplus F_N$ and the associated fake foliations are well-defined on $(x_{-t_1}, t_1+t+t_2)$ and hence on $(x,t)$.
	
	\item $x\in U^E\cap W_r^c$ and $x_t\in U^F\cap W_r^c$;
	\item $x$ is a $(\lambda_0,E, \beta,W_r)$-simultaneous forward Pliss times for the orbit segment $(x,t)$ and the set $W_r$.
	\item $x_t$ is a $(\lambda_0,F, \beta,W_r)$-simultaneous backward Pliss times for the orbit segment $(x,t)$ and the set $W_r$.
\end{itemize}

As an immediate corollary, we obtain:
\begin{proposition}\label{p.bs}
	The orbit segment collection $\cG$ defined above has the Bowen property at scale $\vep_B$ and tail specification at scale $\delta_0$ with shadowing orbits contained in $\cO(U)$. Furthermore, $\cG\cap \Lambda\times\RR^+$ has tail specification at scale $\delta_0$ with shadowing orbits contained in $\cO(U_1)$.
\end{proposition}

Next we describe the collection of orbit segments $\cD$ on which there exists a $(\cP,\cG,\cS)$-decomposition. To do this, we will introduce some ``bad'' orbit segments along the way, namely $\cB_1$ to $\cB_3$.

Let $(x,t)\subset \cO(U_1)$. We define the following two functions:
\begin{equation}\label{e.ps1}
	\begin{split}
		&\tilde p(x,t) = \min\{s\in[0,t]: x_s\notin {B_{r_0}(\Sing_\Lambda(X))}\}, \mbox{ and }\\
		&\tilde s(x,t) = \max\{s\in[0,t]: x_s\notin {B_{r_0}(\Sing_\Lambda(X))}\}.
	\end{split}
\end{equation}
 By replacing $\tilde s$ by some $\tilde s'$ with $|\tilde s'-\tilde s|<1$, we may assume that $\tilde s - \tilde p\in\NN$ while $x_{\tilde s}\notin W_r$.
Then define (recall $T_U$ from Lemma \ref{l.bipliss2})
\begin{equation}\label{e.B1}
	\cB_1 = \{(x,t)\in\Lambda\times \RR^+: \mbox{ either $\tilde p$ or $\tilde s$ does not exist, or $\tilde s - \tilde p\le T_{U}$}\}.
\end{equation}
%where $T_{W}>0$ is given by Lemma~\ref{l.bipliss2}. 

For an orbit segment that is not in $\cB_1$, we have $x_{\tilde p}\notin {B_{r_0}(\Sing_\Lambda(X))}$, $x_{\tilde s} \notin {B_{r_0}(\Sing_\Lambda(X))}$ and $\tilde s-\tilde p > T_{U}\ge T_W.$   In particular, Lemma \ref{l.robust.dom.spl} and \ref{l.robust.hyp} applies. Also recall that $W_r \subset B_{r_0}(\Sing_\Lambda(X))$.
Now consider Lemma~\ref{l.bipliss2} for the orbit segment $(x_{\tilde p}, \tilde s - \tilde p)$, and define 
\begin{equation}\label{e.B2}
	\cB_2 = \left\{(x,t) \notin \cB_1: \frac{1}{\tilde s-\tilde p+1}\#\{i\in [0,\tilde s-\tilde p]\cap\NN: x_{\tilde p+i}\in W_r\}\ge {0.9}\beta_1\right\}.
\end{equation}
If $(x,t)\notin \cB_2$ then by Lemma~\ref{l.bipliss2} one can find an integer $k\ge 0$ such that $x_{\tilde p+k}\in U^E$ is a $(\lambda_0,E,\beta_0,W_r)$-simultaneous forward Pliss time for the orbit segment $(x_{\tilde p+k}, \tilde s -\tilde p - k)$ and $(\psi_t^*)$. Let $\tilde k$ be the smallest such $k$ and define 
$$
p(x,t) = \tilde p(x,t) + \tilde k.
$$
Note that $x_p$ is a $(\beta_0,W_r)$-forward recurrence Pliss time, and therefore $x_p\notin W_r$ (see Remark \ref{r.RecPliss}). 

\begin{lemma}\label{l.firsthyptime}
	One of the following statement hold for the orbit segment $(x,t)\notin \cB_1\cup\cB_2$:
	\begin{itemize}
		\item either $p-\tilde p \le T_{U};$ 
		\item or $$
		 \frac{1}{p-\tilde p + 1}\,\#\{i\in [0,p-\tilde p]\cap\NN: x_{\tilde p+i}\in W_r\}\ge \beta_1.
		$$
	\end{itemize}
\end{lemma}

\begin{proof}
	Note that $x_{\tilde p}, x_p\notin W_r$. Assuming the contrary, we can apply Lemma \ref{l.bipliss2} to the orbit segment $(x_{\tilde p}, p-\tilde p)$ to find an integer $\ell\in [0, p-\tilde p -1]$ so that $x_{\tilde p+\ell}\in U^E$ is a $(\lambda_0,E,\beta_0,W_r)$-simultaneous forward Pliss time for the orbit segment $(x_{\tilde p+\ell}, p -\tilde p - \ell)$ and $(\psi_t^*)$. Since $x_p$ is also  a $(\lambda_0,E,\beta_0,W_r)$-simultaneous forward Pliss time for the orbit segment $(x_{p}, \tilde s - p)$, we see that $x_{\tilde p+\ell}\in U^E$ is a $(\lambda_0,E,\beta_0,W_r)$-simultaneous forward Pliss time for the orbit segment $(x_{\tilde p+\ell}, \tilde s -\tilde p - \ell)$ and $(\psi_t^*)$ thanks to Lemma~\ref{l.gluePliss}. Since $\ell < k$, this contradicts the minimality of $\tilde k$ in the definition of $p$. 
\end{proof}

Now that we have defined $p(x,t)$, we consider the orbit segment $(x_p,\tilde s-p)$ and note that $x_p,x_{\tilde s}\notin W_r$. Note that 
$$
\frac{1}{\tilde s-\tilde p+1}\#\{i\in [0,\tilde s-\tilde p]\cap\NN: x_{\tilde p+i}\in W_r\}< {0.9}\beta_1
$$
since we are not in $\cB_2$. By  Lemma \ref{l.firsthyptime}, the orbit segment $(x_{\tilde p}, p-\tilde p)$ has two cases: either the time is shorter than $T_U$, or the density of visits to $W_r$ is at least  $\beta_1.$ In the second case of we have 
$$
\frac{1}{\tilde s- p + 1}\,\#\{i\in [0,\tilde s- p]\cap\NN: x_{\tilde p+i}\in W_r\}< \beta_1.
$$ 	
In the first case we have the same inequality as long as $\tilde s - p > 10 T_U$.\footnote{Here the factor 10 is obtained by considering the worst case: $p-\tilde p= T_U$, and the corresponding orbit segment does not visit $W_r$.}
Either way, we can apply Lemma~\ref{l.bipliss2} to the orbit segment $(x_p,\tilde s - p)$ if $\tilde s-p> 10T_{U} $. Put  
\begin{equation}\label{e.B3}
	\cB_3 = \left\{(x,t) \notin \cB_1\cup\cB_2: \tilde s-p \le 10T_{U}\right\}.
\end{equation}
For $(x,t)\notin \cB_1\cup\cB_2\cup\cB_3$ we use Lemma~\ref{l.bipliss2} (and note that $\tilde s-p\in\NN$) for $-X$ to find $j\in [0,\tilde s-p]$ so that $x_{p+j}\in U^F$ is a $(\lambda_0,F,\beta_0,W_r)$-simultaneous backward Pliss time for the orbit segment $(x_{p}, j)$. Let $\tilde j$ be the largest such $j$, and define 
$$
s(x,t) = p(x,t)+\tilde j.
$$

\begin{lemma}\label{l.lasthyptime}
	For $(x,t)\notin \cB_1\cup\cB_2\cup \cB_3$, one of the following statement holds for the orbit segment $(x_{\tilde p}, p-\tilde p)$:
	\begin{itemize}
		\item either $\tilde s-s \le T_{U};$ 
		\item or $$
		\frac{1}{\tilde s-s + 1}\,\#\{i\in [0,\tilde s-s]\cap\NN: x_{s+i}\in W_r\}\ge \beta_1.
		$$
	\end{itemize}
\end{lemma}
The proof follows the same lines as the proof of Lemma~\ref{l.firsthyptime} applied to $-X$, and it therefore omitted.

The process above gives us four functions on $D: = (\cB_1\cup\cB_2\cup\cB_3)^c$:
$$
0\le \tilde p(x,t) \le p(x,t) < s(x,t) \le \tilde s(x,t)\le t.
$$
By Remark~\ref{r.RecPliss}, $x_p, x_s\notin W_r$ since they are recurrence Pliss times (forward and backward, respectively). 

The next lemma follow immediately from the construction (and note that $s-p = \tilde j\in\NN$).
\begin{lemma}
	For $(x,t)\in\cD$, the orbit segment $(x_p,s-p)$ is  in $\cG.$
\end{lemma}

We are left to show that the pressure gap property holds. This is the content of Section~\ref{ss.gap}.

\subsection{The pressure gap property and the proof of Theorem~\ref{m.B}}\label{ss.gap}

Finally we are ready to prove Theorem~\ref{m.B}. We do this by verifying the assumption of the (further) improved CT criterion (Theorem~\ref{t.improvedCL1}). 

First, Theorem~\ref{m.A} shows that $X|_{\Lambda}$ is almost expansive at scale $\vep$.
Furthermore, we already constructed the $(\cP,\cG,\cS)$-decomposition on $\cD  = (\cB_1\cup\cB_2\cup\cB_3)^c$ and showed that $\cG$ has the Bowen property at scale $\vep$ (Theorem \ref{t.Bowen}) and specification at $\delta_0 < \vep/(1000L_X)$ (Theorem \ref{t.spec}). This means that $(\cG)^1$ (recall the definition by \eqref{e.g11}) has specification at scale $\delta:=\delta_0 L_X<\vep/1000$. So it remains to prove the pressure gap property, namely (recall the definition of $[\cC]$ from \eqref{e.[C]})
\begin{equation}\label{e.gap11}
	P(\cB_1\cup\cB_2\cup\cB_3\cup[\cP]\cup[\cS],\phi,\delta,\vep) < P(\phi, X|_\Lambda). 
\end{equation} 

We make two observations on the pressure.

\medskip
\noindent {\em Observation 1}. 
\begin{equation*}
	\begin{split}
		&P(\cB_1\cup\cB_2 \cup\cB_3\cup [\cP]\cup[\cS],\phi,\delta,\vep)\\ &\le \max\left\{P(\cB_1,\phi,\delta,\vep), P(\cB_2,\phi,\delta,\vep),P(\cB_3,\phi,\delta,\vep), P([\cP],\phi,\delta,\vep), P([\cS],\phi,\delta,\vep)
		\right\}.
	\end{split}
\end{equation*}

\noindent {\em Observation 2}:
For any $\cC\subset \bM\times\RR^+$, by~\eqref{e.vep} we have $P(\cC,\phi,\delta,\vep)< P(\cC,\phi,\delta) + \frac{a_0}{2}$.

Therefore, in order to prove~\eqref{e.gap11} we must prove that $P(\cC, \phi,\delta) < P(\phi, X|_\Lambda)-\frac{a_0}{2}$ for $\cC = \cB_1,\cB_2,\cB_3,[\cP]$ and $[\cS]$. The proof relies on the following general results concerning the pressure of an orbit segments collection.

Let $\cC\subset \bM\times \RR^+$ be a collection of orbit segments. Recall that the pressure of the potential $\phi$ on $\cC$ is defined as (here the second scale is zero, in view of Observation 2 above)
$$
\Lambda(\cC,\phi, \delta,t) = \sup\left\{\sum_{x\in E_t}\exp(\Phi_0(x,t)): E_t \mbox{ is a $(t,\delta)$-separated set of } \cC_t \right\}, 
$$
and
$$
P(\cC,\phi, \delta)  =\limsup_{t\to\infty} \frac1t \log \Lambda(\cC,\phi, \delta,t),
$$
where $(\cC)_t = \{x:(x,t)\in\cC\}$.

For each $t>0$ we choose $E_t$ a $(t,\delta)$-separated set of $(\cC)_t$ with 
$$
\log\sum_{x\in E_t}\exp(\Phi_0(x,t)) \ge \log \Lambda(\cC,\phi, \delta,t)-1.
$$
Then we consider 
\begin{equation}\label{e.mu}
	\begin{split}
		\nu_t :&= \frac{\sum_{x\in E_t}\exp(\Phi_0(x,t))\cdot \delta_x}{\sum_{x\in E_t}\exp(\Phi_0(x,t))}, \mbox{ and }\\
		\mu_t :&= \frac1t \int_0^t(f_s)_*\nu_t\, ds\\
		& = \frac{1}{\sum_{x\in E_t}\exp(\Phi_0(x,t))} \sum_{x\in E_t}\exp(\Phi_0(x,t))\cdot \left(\frac1t\int_0^t\delta_{x_s} \,ds\right).
	\end{split}
\end{equation}
Let $\mu$ be any limit point of the integer-indexed sequence $(\mu_t)_{t\in \NN}$, under the weak-* topology.

\begin{lemma}\cite[Proposition 5.1]{BCTT},\cite[Lemma 5.8]{PYY23}\label{l.var.princ}
	It holds that 
	$$
	P(\cC, \phi,\delta)\le P(\mu): = h_\mu(X) + \int \phi\,d\mu.
	$$
	%Furthermore, for any compact set $A$ and $b>0$, if every empirical measure $m_{(x,t)} := \frac1t\int_0^t\delta_{x_s}\,ds$ for $(x,t)\in\cC$ satisfies  that  $m_{(x,t)}(A)\ge b$, then we have $\mu(A)\ge b$.
\end{lemma}
The proof follows the proof of the variational principle~\cite[Theorem 8.6]{Wal} closely. Details can be found in~\cite[Appendix A]{PYY23}.

For simplicity, below we will denote by 
$$
\delta_{(x,t)} = \frac1t \int_0^t \delta_{x_s} \,ds
$$
the empirical measure on the orbit segment $(x,t)$. It should not be confused with $\delta_{x_t}$ which is the point mass at the point $x_t = f_t(x)$. The next five lemmas deal with the pressure of $\cB_1, \cB_2, \cB_3, [\cP]$ and $[\cS]$, respectively. Their proof all follow the same argument: show that every $(x,t)\in\cC$ with $t$ sufficiently large (note that short orbit segments do not contribute towards the pressure due to the pressure being defined using $(\cC)_t$ for $t$ sufficiently large) must spent at least $0.5\beta_1$ portion of its time inside $W_r$ or $B_{r_0}(\Sing_\Lambda(X))$. This forces the limiting measure $\mu$ to satisfy $\mu(\Cl(B_{r_0}(\Sing_\Lambda(X))))\ge 0.5\beta_1$. Then we use Lemma~\ref{l.var.princ} and \eqref{e.gap1} to show that the pressure on $\cC$ is at most $P(\phi, X|_\Lambda)-a_0$.

\begin{lemma}\label{l.B1}
	We have $P(\cB_1, \phi,\delta) < P(\phi, X|_\Lambda)- a_0$.
\end{lemma}
\begin{proof}
		Recall that $\cB_1$ (see~\ref{e.B1}) consists of orbit segments $(x,t)$ such that either $\tilde p$ or $\tilde s$ does not exist, or $\tilde s-\tilde p\le T_{U}$. In the first two case, for every $(x,t)\in\cB_1$ we have $\delta_{(x,t)}(B_{r_0}(\Sing_\Lambda(X))) = 1$. In the last case,  note that the orbit segment from $0$ to $\tilde p$ and from $\tilde s$ to $t$ are spent in $B_{r_0}(\Sing_\Lambda(X))$. Consequently, for all $t$ sufficiently large one has 
		$$
		\delta_{(x,t)}(B_{r_0}(\Sing_\Lambda(X))) > 0.6\beta_1. 
		$$
		Either way,  we see that for any $(t,\delta)$-separated set $E_t$ of $(\cB_1)_t$, the measure $\mu_t$ defined as in~\eqref{e.mu} is a convex combination of empirical measures $\dxt$ for $x\in E_t$, and consequently, satisfies
		$$
		\mu_t(B_{r_0}(\Sing_\Lambda(X))) > 0.6\beta_1. 
		$$
		This shows that for any limit point $\mu$ of $\mu_t$ (as $t\to\infty$), it holds that
		$$
		\mu\left(\Cl(B_{r_0}(\Sing_\Lambda(X)))\right)\ge  0.6\beta_1,
		$$ 
		which, according to~\eqref{e.gap1}, implies that $P(\mu) < P(\phi, X|_\Lambda) - a_0$. Then we obtain from Lemma~\ref{l.var.princ} that
		$$
		P(\cB_1, \phi,\delta) \le P(\mu) < P(\phi, X|_\Lambda) - a_0,
		$$
		as required.
\end{proof}

\begin{lemma}\label{l.B2}
	$P(\cB_2, \phi,\delta) < P(\phi, X|_\Lambda) - a_0$.
\end{lemma}
\begin{proof}
	Recall the definition of $\cB_2$ from \eqref{e.B2}: the portion of time spent in $W_r$ (w.r.t.\,the time-one map $f_1$) in the time interval $[\tilde p,\tilde s]$ is at least $0.9\beta_1$. 	The definition of $W_r$ (see~\eqref{e.W.def}) guarantees that if $x_{i}$ and $x_{i+1} = f_1(x_i)$ are both in $W_r^c$, then the orbit segment $(x_i,1)$ does not intersect with $W_r$.	Also note that the orbit segment for the time interval  $[0,\tilde p]$ and $[\tilde s, t]$ are contained in $B_{r_0}(\Sing_\Lambda(X))$, a fact we already used in the previous lemma. 
	
	Then, for any $(t,\delta)$-separated set $E_t$ of $(\cB_2)_t$ and every $x\in E_t$, it holds 
	$$
	\dxt(B_{r_0}(\Sing_\Lambda(X))) \ge  0.9\beta_1.
	$$
	Similar to the previous lemma, we then have 
	$$
	\mu\left(\Cl\left(  B_{r_0}(\Sing_\Lambda(X))\right)\right)  \ge \frac{\beta_1}{2},
	$$
	where $\mu$ is any limit point of the sequence of measures $(\mu_t)_t$ defined by~\eqref{e.mu}. Now Lemma~\ref{l.var.princ} and Equation~\eqref{e.gap1} implies that 
	$$
	P(\cB_2, \phi,\delta) < P(\phi, X|_\Lambda) - a_0,
	$$
	as required.
	
\end{proof}

\begin{lemma}\label{l.B3}
	$P(\cB_3, \phi,\delta) < P(\phi, X|_\Lambda) - a_0$.
\end{lemma}

\begin{proof}
	Recall the definition of $\cB_3$ from \eqref{e.B3}; it consists of orbit segments such that $\tilde s-p\le 10T_{U}$, where $p$ is given by the first simultaneous forward Pliss time in the time interval $[\tilde p,\tilde s]$. 
	
	We make the following observation:
	\begin{itemize}
		\item the time interval $[0,\tilde p]$ is spent in $B_{r_0}(\Sing_\Lambda(X))$;
		\item for the time interval $[\tilde p,p]$, Lemma~\ref{l.firsthyptime} shows that either $p-\tilde p\le T_{U}$, or the portion of time spent in $W_r$ is larger than $\beta_1$;
		\item the time interval $[p,\tilde s]$ has length at most $10T_{U}$;
		\item the time interval $[\tilde s,t]$ is spent in $B_{r_0}(\Sing_\Lambda(X))$.
	\end{itemize}
	This shows that the overall time of $(x,t)$ spent in $B_{r_0}(\Sing_\Lambda(X))$ is at least $0.9\beta_1$, as long as $t$ is sufficiently large. Then the same argument as the previous two lemmas shows that $\mu(B_{r_0}(\Sing_\Lambda(X)))\ge 0.9\beta_1 $ 	where $\mu$ is any limit point of the sequence of measures $(\mu_t)_t$ defined by~\eqref{e.mu}, meaning that the pressure is less than $P(\phi, X|_\Lambda)-a_0.$
\end{proof}

\begin{lemma}\label{l.P}
	$P([\cP], \phi,\delta) < P(\phi, X|_\Lambda) - a_0$.
\end{lemma}
\begin{proof}
	Recall the definition of  $[\cP]$ from~\eqref{e.[C]}.  Also recall that $\cP$ consists of orbit segments of the form $(x,p(x,t))$ where $p$ is given by the first simultaneous forward Pliss time in the time interval $[\tilde p,\tilde s]$.
	
	Similar to the previous lemma, we have 
	\begin{itemize}
		\item the time interval $[0,\tilde p]$ is spent in $B_{r_0}(\Sing_\Lambda(X))$;
		\item for the time interval $[\tilde p,p]$, Lemma~\ref{l.firsthyptime} shows that either $p-\tilde p\le T_{U}$, or the portion of time spent in $W_r$ is at least  $\beta_1$;
	\end{itemize}
	In particular, the overall time of $(x,p(x,t))$ spent in $B_{r_0}(\Sing_\Lambda(X))$ is at least $\beta_1$, as long as $t$ is sufficiently large. The same argument as before shows that the pressure is less than $P(\phi, X|_\Lambda)-a_0$.
\end{proof}

\begin{lemma}\label{l.S}
	$P([\cS], \phi,\delta) < P(\phi, X|_\Lambda)- a_0$.
\end{lemma}

\begin{proof}
	$\cS$ consists of orbit segments of the form $(y,\tau)$ where for some $(x,t)\in\cD$,  $y = x_{s(x,t)}$ and $\tau = t-s(x,t)$. Here $s$ is given by the last simultaneous backward Pliss time in the time interval $[p,\tilde s]$.

	The following observations hold for the orbit segment $(x,t)$:
	\begin{itemize}
		\item for the time interval $[s,\tilde s]$, Lemma~\ref{l.lasthyptime} shows that either $\tilde s-\tilde s\le T_{U}$, or the portion of time spent in $W_r$ is at least  $\beta_1$;
		\item the time interval $[\tilde s, t]$ is spent in $B_{r_0}(\Sing_\Lambda(X))$.
	\end{itemize}
	In particular, we are in a symmetric case comparing to Lemma~\ref{l.P}, and the overall time of $(y,\tau ) = (x_{s},t-s)$ spent in $B_{r_0}(\Sing_\Lambda(X))$ is at least $\beta_1$, as long as $\tau$ is sufficiently large. The same argument as before shows that the pressure is less than $P(\phi, X|_\Lambda)-a_0$.
\end{proof}

Collecting the previous five lemmas and keeping in mind Observation 1 and 2, we have 
\begin{align*}
	&P(\cD^c\cup[\cP]\cup[\cS],\phi,\delta,\vep)\\ 
	&\le \max\left\{P(\cB_1,\phi,\delta,\vep),  P(\cB_2,\phi,\delta,\vep), P(\cB_3,\phi,\delta,\vep), P([\cP],\phi,\delta,\vep),  P([\cS],\phi,\delta,\vep)\right\}\\
	&\le \max\left\{P(\cB_1,\phi,\delta),  P(\cB_2,\phi,\delta),  P(\cB_3,\phi,\delta),  P([\cP],\phi,\delta),  P([\cS],\phi,\delta)\right\} + \frac{a_0}{2}\\
	&\le P(\phi, X|_\Lambda)-a_0+ \frac{a_0}{2}\\
	&= P(\phi, X|_\Lambda)-\frac{a_0}{2}.
\end{align*}
This verifies Assumption (III) of the improved CT criterion (Theorem~\ref{t.improvedCL1}) and concludes the proof of Theorem~\ref{m.B}.

\section{The Bowen property}\label{s.bowen}
In this section we prove Theorem~\ref{t.Bowen}, namely the Bowen property on $\cG_B$. Recall that an orbit segment $(x,t)\in\cO(U)$ is in $\cG_B$ if and only 
\begin{itemize}
	\item $(x,t)$ is contained in a larger orbit segment $(x_{-t_1}, t_1+t+t_2) \in \cO(U)$ with endpoints $x_{-t_1}, x_{t+t_2}$ outside $B_{r_0}(\Sing_\Lambda(X))$. This guarantees the existence of a dominated splitting $E_N\oplus F_N$ on the normal bundle, and therefore the existence of fake foliations and the local product structure on the normal planes.
	\item  $x$ is a simultaneous forward Pliss time for the set $W_r$.
	\item  $x_t$ is a simultaneous backward Pliss time for the set $W_r$. 
\end{itemize}

\subsection{Fitting Bowen balls into Liao's tubular neighborhoods}\label{ss.key}

In this subsection, we implement the classical Bowen strategy in the multi-singular setting, which requires a mechanism to control the geometry of orbits near singularities, where the flow slows down and the usual local stable and unstable manifolds may be truncated. We state and prove the key result of Section~\ref{s.bowen}, Proposition \ref{p.key}, which provides precisely this control: for orbit segments in the good set $\cG_B$, every point in the Bowen ball $B_{t,\vep_B}(x)$ is $\rho$-scaled shadowed by the orbit of $x$ up to time $t$. This property allows us to use the fake foliations constructed in Section~\ref{ss.fake} as a local product structure on the normal bundle of the orbit segment $(x,t)$, and to apply $E$- and $F$-hyperbolic times to track contraction and expansion along orbits. 
These ingredients yield the bounded distortion estimates required for the Bowen property and provide the bridge between the ambient metric and the scaled neighborhoods near singularities that allows the classical Bowen argument to be carried out in this setting.

First we fix some parameters. Recall the choice of $a>0$ and $\lambda_1$ as in~\eqref{e.lambda1}. Let $\rho_1$ be given by Proposition~\ref{p.tubular3} (1) and (2) applied to $\rho_0$. Fix $\rho\in (0,\rho_1]$, and let $\rho'$ be given by Proposition~\ref{p.tubular3} (5) applied to $\rho K_0^{-2}$ where $K_0$ is given by Proposition~\ref{p.tubular}. In particular, if $y\in \cN_{\rho'|X(x)|}(x)$ then $\tilde y_{-1}:=\cP_{-1,x_{-1}}(y)$ exists and is contained in $\cN_{\rho K_0^{-2}|X(x_{-1})|}(x_{-1})$; consequently, $\tilde y_{-1}$ is $\rho$-scaled shadowed by the orbit of  $x_{-1}$ up to time $T=2$. We also assume that $\alpha_0>0$ is taken small enough so that the construction in Section~\ref{ss.fake} and, in particular, Lemma \ref{l.Elarge}, \ref{l.Flarge} and \ref{l.hyptime.dist} hold.

\begin{proposition}[The Main Proposition]\label{p.key}
	There exists $\beta_0\in(0,1)$, such that for every $\beta\in(0,\beta_0]$, $r_0>0$ small enough and $U$ small enough, there exists $\overline r\in (0,{r_0})$, such that for every $r \in (0,\overline r]$, there exists $\vep_B>0$ such that for every  $(x,t)\in \cG_B = \cG_B(\lambda_0,\beta, r_0, W_{r })$, we have that 
	every $y\in \cP_{x}\left(B_{t,\vep_B}(x)\right)$ is $\rho$-scaled shadowed by the orbit of $x$ up to  time $t$.
\end{proposition}

The proof of this proposition occupies the rest of Section~\ref{ss.key}.

To begin with, we fix $r_0>0$ and $U$ small enough so that Lemma~\ref{l.Elarge} and~\ref{l.Flarge} hold, and obtain $\overline r\in (0,r_0)$ and $\vep_1>0$. In particular, the choice of $\overline r$ does not depend on $\beta$ and $\beta_0$. 

Now we let 
\begin{equation}\label{e.vep2}
	\vep_B = \vep_B(r ,r_0,\rho) = \frac12 \min\left\{ \rho K_0^{-1} \inf_{z\in W_{r }^c}\{|X(z)|\},\,\,\vep_1\right\}.
\end{equation}
The reason behind this choice is similar to that of $\vep_{\Exp}$ in Section~\ref{ss.choice.vexp} (see \eqref{e.choice.vep}): if the orbit segment $(x,t)$ does not enter $W_r$, then for every $y\in B_{t,\vep_B}(x)$, $\cP_x(y)$ is $\rho$-scaled shadowed by the orbit of $x$ up to time $t$. Also note that Lemma \ref{l.reg} and~\ref{l.expansion.reg} apply to such orbit segments.

Given any $(x,t)\in\cG_B$, we parse the time interval $[0,t]$ in the same way as we did in Section~\ref{s.fakefoliation}, but keep in mind that $x,x_t\notin W_r$ (Remark \ref{r.RecPliss}):
\begin{equation}\label{e.parsing}
0< T_1^i < T_1^o < T_2^i < T_2^o <\cdots < T_N^o < t  ,
\end{equation}
such that for $k=1,\ldots, N$:
\begin{itemize}
	\item the orbit segment $(x_{T_k^i+1}, T_k^o-T_k^i-2)$ is contained in $\overline W_r$; 
	\item  the orbit segment $(x_{T_k^o}, T_{k+1}^i-T_k^o)$ is not in $W_r$;
	\item $T_1^i\ge \floor{\beta^{-1}}$ and $t - T_N^o \ge \floor{\beta^{-1}}$.
\end{itemize}
We shall further assume that $T_k^*\in\NN$, $k=1,\ldots, N$, $* = i,o$. Due to our choice of $\vep_B$, the product structure and the $E$, $F$-length of $y$ are well defined for the sub-orbit segment corresponding to the time interval $[T_k^o, T_{k+1}^i]$  for each $k$.

The next lemma is summarized from the proof of Theorem~\ref{t.exp}.

\begin{lemma}\label{l.Flarge1}
	Assume that for some $k\in[1,N]\cap\NN$, we have 
	$$
	d^F_{x_{T_k^o}}\left(\cP_{x_{T_k^o}}(y_{T_k^o})\right)\ge 	d^E_{x_{T_k^o}}\left(\cP_{x_{T_k^o}}(y_{T_k^o})\right).
	$$
	Then for every $j>k$ and $*=i,o$, one has 
	$$
	d^F_{x_{T_j^*}}\left(\cP_{x_{T_j^*}}(y_{T_j^*})\right)\ge d^E_{x_{T_j^*}}\left(\cP_{x_{T_j^*}}(y_{T_j^*})\right).
	$$
	
	Similarly, if 	
	$$
	d^E_{x_{T_k^o}}\left(\cP_{x_{T_k^o}}(y_{T_k^o})\right)\ge 	d^F_{x_{T_k^o}}\left(\cP_{x_{T_k^o}}(y_{T_k^o})\right),
	$$
	then for every $j\le k$ and $*=i,o$, one has 
	$$
	d^E_{x_{T_j^*}}\left(\cP_{x_{T_j^*}}(y_{T_j^*})\right)\ge d^F_{x_{T_j^*}}\left(\cP_{x_{T_j^*}}(y_{T_j^*})\right).
	$$
\end{lemma}

The proof is omitted. One only need to repeatedly use Lemma~\ref{l.Elarge}, \ref{l.Flarge} for orbit segments inside $W_r$, and use Lemma~\ref{l.reg} outside $W_r$. Note that these lemmas do not require $x\in\Lambda$; instead we only need the dominated splitting for orbit segments in $\cO(U)$, namely Lemma \ref{l.robust.dom.spl}.

Now we fix $\beta_0$ sufficiently close to zero, such that 
\begin{equation}\label{e.beta0'}
	\left((K_1^*)^{\frac{\beta_0}{1-\beta_0}}\lambda_1^{-1}\right)^{\beta_0^{-1}-1}\rho_1<\frac{1}{16}\rho', 
\end{equation}
where $K_1^*$ is given by Proposition~\ref{p.tubular3}. Here~\eqref{e.beta0'} is possible because the base satisfies 
$$
(K_1')^{\frac{\beta_0}{1-\beta_0}}\lambda_1^{-1}<1
$$
as long as $\beta_0$ is close to zero, meanwhile $(\beta_0^{-1}-1)$ can be made arbitrarily large. Then we fix any $\beta\in(0,\beta_0]$. Below we shall prove that Proposition~\ref{p.key} holds with the choice of parameters describe above.

The following lemma is the crucial step in the proof of Proposition~\ref{p.key}.
\begin{lemma}[The Key Lemma]\label{l.key}
	Under the assumptions of Proposition~\ref{p.key}, assume in addition that for some $k\in [1,N]\cap\NN$, 
	$$
	d^F_{x_{T_k^o}}\left(\cP_{x_{T_k^o}}(y_{T_k^o})\right)\ge 	d^E_{x_{T_k^o}}\left(\cP_{x_{T_k^o}}(y_{T_k^o})\right).
	$$
	Then $\cP_{x_{T_k^i}}(y_{T_k^i})$ is $\rho$-scaled shadowed by the orbit of $x_{T_k^i}$ up to time $t-T_k^i$.
\end{lemma}
Note that the conclusion is stated for $T_k^i<T_k^o$ instead of $T_k^o$. 

Before proving the lemma, we first demonstrate how to obtain Proposition~\ref{p.key} from Lemma~\ref{l.key}. 

\begin{proof}[Proof of Proposition~\ref{p.key}, assuming Lemma~\ref{l.key}]

Let 
$$
s = \sup\big\{\tau\in[0,t]: \mbox{ $y$ is $\rho$-scaled shadowed by $x$ up to time $\tau$}\big\}.
$$
Then $s\ge T_1^i>0$. Furthermore, our choice of $\vep_B$ means that $s\in (T_k^i,T_k^o)$ for some $k\in [1,N]\cap\NN$. 

Now consider the following two cases: 

\medskip 
\noindent Case 1.	$$
d^F_{x_{T_k^o}}\left(\cP_{x_{T_k^o}}(y_{T_k^o})\right)\ge 	d^E_{x_{T_k^o}}\left(\cP_{x_{T_k^o}}(y_{T_k^o})\right).
$$

In this case, Lemma~\ref{l.key} shows that $\cP_{x_{T_k^i}}(y_{T_k^i})$ is $\rho$-scaled shadowed by the orbit of $x_{T_k^i}$ up to time $t-T_k^i$. On the other hand, our choice of $s$ means that $y$ is $\rho$-scale shadowed by the orbit of $x$ up to time $s \ge  T_k^i$. Then by Lemma~\ref{l.shadowing} we see that $y$ is $\rho$-scale shadowed by the orbit of $x$ up to time $t$, which is a contradiction.

\medskip 
\noindent Case 2. $$
d^E_{x_{T_k^o}}\left(\cP_{x_{T_k^o}}(y_{T_k^o})\right)\ge 	d^F_{x_{T_k^o}}\left(\cP_{x_{T_k^o}}(y_{T_k^o})\right).
$$

By Lemma~\ref{l.Flarge1}, we have
\begin{equation}\label{l.6.1.11}
d^E_{x_{T_k^i}}\left(\cP_{x_{T_k^i}}(y_{T_k^i})\right)\ge 	d^F_{x_{T_k^i}}\left(\cP_{x_{T_k^i}}(y_{T_k^i})\right).
\end{equation}
Now we consider $-X$.  Note that the assumptions of Proposition~\ref{p.key} holds for $-X$ with the same parameters and the orbit segment collection $\cG_B$ (although the time is reversed on each orbit segment). Furthermore, \eqref{l.6.1.11} means that for $-X$, the $F$-length of $y_{_{T_k^i}}$ is larger than its $E$-length (note that $T_k^i$, when considering $-X$, is the $(n-k+1)$th time that the orbit {\em leaves} $W_r$). This allows us to apply Lemma~\ref{l.key} for $-X$ to conclude that the orbit of $y$ is $\rho$-scaled shadowed by $x$ for the time interval $[T_k^i,T_k^o]$ (which, for $-X$, is the $(n-k+1)$th visit to $W_r$). In particular, Lemma~\ref{l.shadowing} shows that $y$ is $\rho$-scale shadowed by the orbit of $x$ up to time $T_k^o$. This contradicts with the maximality of  $s\in (T_k^i, T_k^o)$.

We conclude the proof of Proposition~\ref{p.key} assuming Lemma~\ref{l.key}.

\end{proof}

Now it remains to prove Lemma~\ref{l.key}.

\begin{proof}[Proof of Lemma~\ref{l.key}]	
	For the sake of contradiction, assume that there exists an orbit segment $(x,t)\in\cG_B$  such that for some $y\in \cP_x(B_{t,\vep_B}(x))$,  one has 	$$
	d^F_{x_{T_k^o}}\left(\cP_{x_{T_k^o}}(y_{T_k^o})\right)\ge 	d^E_{x_{T_k^o}}\left(\cP_{x_{T_k^o}}(y_{T_k^o})\right).
	$$
	However, $\cP_{x_{T_k^i}}(y_{T_k^i})$ is not $\rho$-scaled shadowed by the orbit of $x_{T_k^i}$ up to time $t-T_k^i$.
	
	We use Lemma~\ref{l.Flarge1} to conclude that when $i\ge k,$ 
	\begin{equation}\label{e.Flarge3}
	d^F_{x_{T_i^o}}\left(\cP_{x_{T_i^o}}(y_{T_i^o})\right)\ge d^E_{x_{T_i^o}}\left(\cP_{x_{T_i^o}}(y_{T_i^o})\right). \hspace{0.5cm} \mbox{($F$ large on the forward orbit)}
	\end{equation}

	For contradiction's sake, we define
	\begin{align*}
		s=\min\{\tau\in[0,t]:  \exists \tilde s\in(0,t) \mbox{ such that } y_{\tilde s} \in \cN_{\rho_0|X(x_\tau)|}(x_\tau) \\
		\mbox{ is $\rho$-scaled shadowed by $x_\tau$ up to time $t-\tau$}\}. 
	\end{align*}
	We must have $s\in ( T_{j}^i, T_{j}^o)$ for some $j\in [k,N]\cap\NN$. Furthermore, if we choose $\tau>0$ such that  
	$$
	y_\tau = f^{-X}_{\tau^{-X}_{\tilde x,\tilde y}(\floor{\tilde s})}(\tilde y)\in\cN_{\rho_0|X(x_{\roof{s}})|}(x_{\roof{s}}),\footnote{The purpose of this discussion is to avoid using $\cP_{x_s}(y_s)$ which may not be well-defined since $x_s$ may be very close to a singularity.}
	$$
	then the previous discussion shows that $y_\tau$ is $\rho$-scaled shadowed by the orbit of $x_{\roof{s}}$ up to time $t-\roof{s} = \floor {\tilde s}$.

Let
$$
s_0: = \roof{s} -1 \in\NN,
$$

Below we will prove that 
$$
y_{\tau}\in \cN_{\rho'|X(x_{\roof{s}})|}(x_{\roof{s}})
$$
where $\rho'$ is given by Proposition~\ref{p.tubular3} (5) applied to $\rho K_0^{-2}$; this shows that  $\cP_{1,x_{s_0}}^{-1}\left(y_{\tau }\right)$ exists and is contained in $\cN_{\rho K_0^{-2}|X(x_{s_0})|}(x_{s_0}); 
$
this, together with Lemma~\ref{l.shadowing2}, will give the desired contradiction.

To simplify notation, we write 
$$
t^+ = T_j^o \in (s,t).
$$
Since $t^+>s$, the orbit of $\cP_{x_{t^+}}(y_{t^+})$ is $\rho$-scaled shadowed by the orbit of $x_{t^+}$ up to the  time $t-t^+$. Also recall that $x_t$ is a $(\lambda_0, F)$-backward hyperbolic time for the orbit segment $(x,t)$ and the scaled linear Poincar\'e flow $(\psi_t^*)$. This allows us to apply Lemma~\ref{l.hyptime.dist} with $j=t^+$
and obtain
\begin{align*}
	d_{x_{t^+}}^F\left(\cP_{x_{t^+}}(y_{t^+})\right)&\le4 \frac{|X(x_{t^+})|}{|X(x_t)|}\lambda_1^{-(t-t^+)}d_{x_{t}}^F\left(\cP_{x_{t}}(y_{t})\right)\\
	&\le4  \frac{|X(x_{t^+})|}{|X(x_t)|}\lambda_1^{-(t-t^+)}\vep_B\\
	&\le 4\frac{|X(x_{t^+})|}{|X(x_t)|}\lambda_1^{-(t-t^+)}\rho_1|X(x_t)|\\
	&=4\lambda_1^{-(t-t^+)}\rho_1|X(x_{t^+})|.
\end{align*}
Here the second line is due to  $\cP_{x_{t}}(y_{t})\in \cN_{\vep_B}(x_t), $ and the next line is due to the choice of $\vep_B$; see~\eqref{e.vep2} and recall that $x_t\notin W_r$.
In particular, by~\eqref{e.Flarge3} we have 
\begin{align*}
	d_{\cN(x_{t^+})}\left(x_{t^+}, \cP_{x_{t^+}}(y_{t^+}) \right) &\le d_{x_{t^+}}^F\left(\cP_{x_{t^+}}(y_{t^+})\right) + d_{x_{t^+}}^E\left(\cP_{x_{t^+}}(y_{t^+})\right)\\
	&\le 2 d_{x_{t^+}}^F\left(\cP_{x_{t^+}}(y_{t^+})\right)\\
	\numberthis\label{e.dist}	&\le 8\lambda_1^{-(t-t^+)}\rho_1|X(x_{t^+})|. 
\end{align*}

To estimate the change of distance on the normal plane along the orbit segment between the times $\roof{s}$ and $t^+$, recall that $K^*_1>1$ is the upper-bound of $|D(P^*_{1,z})^{-1}|$ for all $z\notin \Sing(X)$ given by Proposition~\ref{p.tubular3} (4). Then we have  
\begin{equation}\label{e.Dsmall}
	\|D(\cP_{(t^+-\roof{s}), x_{\roof{s}}})^{-1}\|\le \frac{|X(x_{\roof{s}})|} {|X(x_{t^+})|}(K_1^*)^{t^+-\roof{s}}. 
\end{equation}
Combining~\eqref{e.dist} and~\eqref{e.Dsmall}, we obtain
\begin{align*}
	d_{\cN(x_{\roof{s}})}\left(x_{\roof{s}}, y_{\tau }\right) &\le
	\frac{|X(x_{\roof{s}})|} {|X(x_{t^+})|}(K_1^*)^{t^+-\roof{s}} \cdot d_{\cN(x_{t^+})}\left(x_{t^+}, \cP_{x_{t^+}}(y_{t^+}) \right) \\
	&\le 8\frac{|X(x_{\roof{s}})|} {|X(x_{t^+})|}(K_1^*)^{t^+-\roof{s}}\lambda_1^{-(t-t^+)}\rho_1|X(x_{t^+})|\\
	\numberthis\label{e.dist1}	&=8(K_1^*)^{t^+-\roof{s}}\lambda_1^{-(t-t^+)}\rho_1|X(x_{\roof{s}})|.
\end{align*}

On the other hand, since $x_t$ is a $\beta$-recurrence Pliss time for the orbit segment $(x,t)$, we have 
$$
\frac{t^+-\roof{s}}{t-\roof{s}}<\beta,
$$
which leads to 
\begin{equation}\label{e.smalltime}
	t^+-\roof{s}\le \frac{\beta}{1-\beta} (t-t^+). 
\end{equation}
Combining~\eqref{e.smalltime} with~\eqref{e.dist1}, we see that 
\begin{align*}
	d_{\cN(x_{\roof{s}})}\left(x_{\roof{s}}, y_{\tau }\right) &\le 8\left((K_1^*)^{\frac{\beta}{1-\beta}}\right)^ {(t-t^+)}\lambda_1^{-(t-t^+)}\rho_1|X(x_{\roof{s}})|\\
	&=8\left((K_1^*)^{\frac{\beta}{1-\beta}}\lambda_1^{-1}\right)^ {(t-t^+)}\rho_1|X(x_{\roof{s}})|\\
	&	\le\frac12 \rho'|X(x_{\roof{s}})|,
\end{align*}
where the last inequality follows from Equation~\eqref{e.beta0'} and the observation that 
$$
t-t^+>\beta^{-1}-1
$$
which is due to Remark~\ref{r.RecPliss}. As a result of Lemma~\ref{p.tubular3} (5),  the point 
$$
\overline y:=\cP_{1,x_{s_0}}^{-1}\left(y_{\tau }\right) \in \cP_{1,x_{s_0}}^{-1}\left(\cN_{\rho'|X(x_{\roof{s}})|}(x_{\roof{s}})\right)
$$ 
exists and is contained in $\cN_{\rho K_0^{-2}|X(x_{s_0})|}(x_{s_0})$. By Proposition~\ref{p.tubular} and Remark~\ref{r.shadowing}, $\overline y$ is $\rho$-scaled shadowed by the orbit of $x_{s_0}$ up to  time $\tau=2$. Recall that $y_\tau$ is $\rho$-scaled shadowed by the orbit of $x_{\roof{s}}$ up to  time $t-\roof{s}$. Applying Lemma~\ref{l.shadowing}, we see that $\overline y$ is indeed $\rho$-scaled shadowed by the orbit of $x_{s_0}$ up to  time $t-s_0$. Because $s_0<s$, this contradicts the minimality of $s$ (i.e., the maximality of $\tilde s$), and concludes the proof of Lemma~\ref{l.key}. 
\end{proof}

Now the proof of Proposition~\ref{p.key} is complete.

\subsection{The Bowen property}
In this section, we use Proposition~\ref{p.key} to prove Theorem~\ref{t.Bowen}.

Note that our definition of hyperbolic times only control iterates at integer times. To deal with this, we use the following  lemma.

\begin{lemma}\label{l.holder}\cite[Lemma 7.7]{PYY23}
	Let $\phi$ be a H\"older continuous function with H\"older index $\gamma$. Then $\Phi_0(x,1) = \int_{0}^1 \phi(x_s)\,ds$, as a function of $x$, is also H\"older continuous with H\"older index $\gamma$.
\end{lemma}

\begin{proof}[Proof of Theorem~\ref{t.Bowen}]
	We take the parameters as in Proposition~\ref{p.key} and obtain $\vep_B>0$. 
	Let $(x,t)\in\cG_B$ and $y\in B_{t, \vep_B}(x)$, and recall that $t\in \NN$.
	To simplify notations,  we shall assume that $y\in \cN(x)$; this results in an error of $\Phi_0(y,t)$ of no more than $\|g\|_{C^0}$, which is independent of $(x,t)$ and $y$.

	By Proposition~\ref{p.key}, $y$ is $\rho$-scaled shadowed by the orbit of $x$ up to  time $t$.  In particular, there exists a strictly increasing continuous function $\tau_{x,y}(\cdot)$ that is differentiable (see Lemma~\ref{l.tau}) w.r.t.\,$y$ with $\tau_{x,y}(0)=0$ such that 
	$$
	y_{\tau_{x,y}(s)} \in \cN_{\rho|X(x_s)|}(x_s), \forall s\in[0,t]. 
	$$
	To simplify notation, for the moment we shall drop the sub-indices and write $\tau(s)=\tau_{x,y}(s)$. 
	Then for $i\in [0,t]\cap \NN$, the points $y_{\tau(i)}^j, j=E,F$ are well-defined and satisfy $y_{ \tau(i)} = [y_{ \tau(i)}^E,y_{ \tau(i)}^F]$. Furthermore, by Equation~\eqref{e.Px2} we have 
	$$
	\cP_{x_t}(y_t) = y_{ \tau(t)}. 
	$$
	%and consequently, $|t-\tau(t)|<1$.
		
	Now let $\phi$ be any H\"older continuous function with H\"older index $\gamma\in (0,1)$. First we estimate $|\Phi_0(x,t) - \Phi_0(y,\tau(t))|$:\footnote{Here we slightly abuse notation and let $\Phi_0(y,t) = -\int^0_{t} g(y_s)\,ds$ if $t<0$. }
	\begin{align*}
			&	|\Phi_0(x,t) - \Phi_0(y,\tau(t))|\\
		\numberthis\label{e.step1}	\le \,&\sum_{i=0}^{t-1} \left| \Phi_0(x_i,1) - \Phi_0\left(y_{ \tau(i)}, \tau(i+1) -  \tau(i)\right)\right|.
	\end{align*}
	To control each summand on the right-hand side, we write
	\begin{align*}
		&\left|\Phi_0(x_i,1) - \Phi_0\left(y_{ \tau(i)}, \tau(i+1) -  \tau(i)\right)\right|\\ &\le  \left|\Phi_0(x_i,1) - \Phi_0(y_{\tau(i)}, 1)\right| + \left| \int_{1}^{\tau(i+1)-\tau(i)} \phi(y_s)\,ds\right|\\
		&\le  \left|\Phi_0(x_i,1) - \Phi_0(y_{\tau(i)}^F, 1)\right| +   \left|\Phi_0(y_{\tau(i)}^F,1) - \Phi_0(y_{\tau(i)}, 1)\right| \\
		&\quad + \|\phi\|_{C^0}\cdot  \left|\tau(i+1)-\tau(i)-1\right|\\
		&= I + II + III.
	\end{align*}

	%From now on, to simplify notation we will write $d$ for the distance within an appropriate normal plane. 
	To estimate I, we use the fact that $x_t$ is a $(\lambda_0,F)$-backward hyperbolic time for the orbit segment $(x,t)$ under the scaled linear Poincar\'e flow. By Lemma~\ref{l.hyptime.dist} (applied to $-X$) , for all $i\in [0,t]\cap \NN$:
	\begin{align*}
		d(x_i, y_{\tau(i)}^F) \le d_{\cN(x_i)}(x_i, y_{\tau(i)}^F) &\le \frac{|X(x_i)|}{|X(x_t)|} c_2\lambda_1^{i-t} \cdot d_{\cN(x_t)}(x_t, y_{\tau(t)}^F)\\
		&\le \frac{|X(x_i)|}{|X(x_t)|} c_2\lambda_1^{i-t}\cdot\rho |X(x_t)|\\
		\numberthis\label{e.FF}	&=c_2\rho\lambda_1^{i-t}|X(x_i)|. 
	\end{align*} 
	Consequently, by Lemma~\ref{l.holder} ($c_1$ is the H\"older constant of $\Phi_0(x,1)$):
	\begin{align*}
		\left|\Phi_0(x_i,1) - \Phi_0(y_{\tau(i)}^F, 1)\right| &\le c_1  d(x_i,y_{\tau(i)}^F)^\gamma \le c_1c_2^\gamma \rho^\gamma \left(\lambda_1^\gamma\right)^{i-t} \sup_{z\in\bM}\{|X(z)|\}^\gamma\\
		\numberthis\label{e.xyF}	&= c_3 \left(\lambda_1^\gamma\right)^{i-t},
	\end{align*}
	where the constant $c_3$ is independent of $x$ and $t$.
	
	To estimate II, we use Lemma~\ref{l.hyptime.dist} again and the assumption that $x$ is a $(\lambda_0,E)$-forward hyperbolic time for the orbit segment $(x,t)$ and $\psi^*$ to obtain, for  all $i\in[0,t]\cap\NN$:
	\begin{align*}
		d(y_{ \tau(i)}^F,y_{ \tau(i)}) \le d_{\cN(x_i)}(y_{ \tau(i)}^F,y_{ \tau(i)}) &\le \frac{|X(x_i)|}{|X(x)|} c_4\lambda_1^{-i} \cdot d_{\cN(x)}( y_{\tau(0)}^E, y_{\tau(0)})\\
		&\le  \frac{|X(x_i)|}{|X(x)|} c_4\lambda_1^{-i} \cdot\rho |X(x)|\\
		\numberthis\label{e.EE}	&=c_4\rho\lambda_1^{-i} |X(x_i)|.
	\end{align*} 
	As a result, we have 
	\begin{align*}
		\left|\Phi_0(y_{\tau(i)}^F,1) - \Phi_0(y_{\tau(i)}, 1)\right| &\le c_1 d(y_{\tau(i)}^F,y_{\tau(i)})^\gamma  \le c_1 c_4^\gamma \rho^\gamma \left(\lambda_1^\gamma \right)^{-i} \sup_{z\in\bM}\{|X(z)|\}^\gamma \\
		\numberthis\label{e.yFy}	&= c_5  \left(\lambda_1^\gamma \right)^{-i},
	\end{align*}
	where the constant $c_5$ does not depend on $x$ or $t$.
	
	We are left with III. First, note that 
	$$
	\tau_{x,y}(i+1) - \tau_{x,y}(i)  = \tau_{x_i,y_{\tau(i)}}(1).  
	$$
	Also, we have 
	$$
	 \tau_{x_i,x_{i}}(1) = 1.
	$$
	Then we apply Lemma~\ref{l.tau} to obtain
	\begin{align*}
		&\quad \left|\tau_{x,y}(i+1)-\tau_{x,y}(i)-1\right|\\ 
		&= \left| \tau_{x_i,y_{\tau(i)}}(1) - \tau_{x_i,x_i}(1)   \right|\\
		&\le  \sup\left\{\left\|D_z \tau_{x_i,z}(1)\right\| : z\in\cN_{\rho_0K_0^{-1}|X(x_i)|}(x_i) \right\}\cdot d_{\cN(x_i)}(y_{\tau(i)}, x_i)\\
		&\le \frac{1}{|X(x_i)|} K_\tau\cdot d_{\cN(x_i)}(y_{\tau(i)}, x_i)\\
		&\le \frac{1}{|X(x_i)|} K_\tau\left(d_{\cN(x_i)}(x_i,y_{\tau(i)}^F) + d_{\cN(x_i)}(y_{\tau(i)}^F,y_{\tau(i)}) \right)\\
		&\le  \frac{1}{|X(x_i)|} K_\tau\left(c_2\rho \lambda_1^{i-t} + c_4\rho \lambda_1^{-i}\right)|X(x_i)|\\
		\numberthis\label{e.III}	&= c_6 \rho\lambda_1^{i-t} + c_7\rho \lambda_1^{-i},
	\end{align*}
	where the sixth line follows from~\eqref{e.FF} and~\eqref{e.EE}. 
	
	Collecting~\eqref{e.xyF},~\eqref{e.yFy} and~\eqref{e.III}, we have 
	\begin{align*}
		& \left|\Phi_0(x_i,1) - \Phi_0\left(y_{ \tau(i)}, \tau(i+1) -  \tau(i)\right)\right|\\
		&\le  c_3 \left(\lambda_1^\gamma\right)^{i-t}+c_5  \left(\lambda_1^\gamma \right)^{-i}+\rho\|\phi\|_{C^0}\left( c_6 \lambda_1^{i-t} + c_7 \lambda_1^{-i}\right),
	\end{align*} 
	where all the involved constants are independent of $x$ and $t$. 
	
	Summing over $i\in [0,t]\cap \NN$, we finally obtain
	\begin{align*}
		&\quad|\Phi_0(x,t) - \Phi_0(y,\tau(t))|\\
		&\le  \|\phi\|_{C^0} +\sum_{i=0}^{t-1} \left(c_3 \left(\lambda_1^\gamma\right)^{i-t}+c_5  \left(\lambda_1^\gamma \right)^{-i}+\rho\|\phi\|_{C^0}\left( c_6 \lambda_1^{i-t} + c_7 \lambda_1^{-i}\right)\right)\\
		&\le c_8,
	\end{align*}
	for some constant $c_8>0$ that depends on $\phi$ but not on $(x,t)\in\cG_B$ or $y\in B_{t,\vep_B}(x)$.

	Next we consider $\left|\Phi_0(y,\tau(t))-\Phi_0(y,t)\right|$. We write
	\begin{align*}
		\left|\Phi_0(y,\tau(t))-\Phi_0(y,t)\right|&\le \|\phi\|_{C^0}\cdot|\tau(t)-t|\\
		&\le \|\phi\|_{C^0}\cdot\sum_{i=0}^{t-1} |\tau(i+1) - \tau(i)-1|.
	\end{align*}
	The right-hand side is precisely III estimated earlier. We conclude from Equation \eqref{e.III} that 
	$$
	\left|\Phi_0(y,\tau(t))-\Phi_0(y,t)\right|\le  \rho\|\phi\|_{C^0}\cdot\sum_{i=0}^{t-1} \left( c_6 \lambda_1^{i-t} + c_7 \lambda_1^{-i}\right) \le c_9
	$$
	with a constant $c_9$ that does not depend on $(x,t)$ or $y$. 
	
	In summary, we have $\left|\Phi_0(y,t)-\Phi_0(y,t)\right|\le c_8+c_9$ and conclude that $\cG_B$ has the Bowen property at scale $\vep_B$ for every H\"older continuous function $\phi$, finishing the proof of Theorem~\ref{t.Bowen}.

\end{proof}

Indeed the estimate on III is interesting in its own right. We summarize this as the following proposition.
\begin{proposition}\label{p.timedrift}
	For a $C^1$ vector field $X$, let $\overline \rho_0$ be given by Proposition \ref{p.Fx}, and $\lambda>1$. Then there exists $C>0$ with the following property:
	
	\noindent Assume that $(x,t)$ is an orbit segment satisfying: 
	\begin{enumerate}
		\item The normal bundle of $(x,t)$ has a dominated splitting $E_N\oplus F_N$ in the sense of \eqref{e.dom.spl1}.
		\item $x$ is a $(\lambda,E)-$forward hyperbolic time for $(\psi^*_t)$ and the orbit segment $(x,t).$
		\item $x_t$ is a $(\lambda,F)-$backward hyperbolic time for $(\psi^*_t)$ and the orbit segment $(x,t).$
	\end{enumerate}
	Then, for every $\rho\in(0,\overline \rho_0/3)$ and every $y\in \cN_{\rho|X(x)|}(x)$ that is $\rho-$scaled shadowed by the orbit of $x$ up to time $t$, we have 
	$$
	|\tau_{x,y}(t)-t|\le C\rho. 
	$$
	In particular, $\tau_{x,y}(t)-t$ converges to zero as $\rho\to 0$, uniformly in $x,y$ and $t$.
\end{proposition}

The proof is skipped. Just note that when controlling $(III)$ earlier, the recurrence Pliss times are only used to obtain that $y$ is $\rho-$scaled shadowed by $x$ (Proposition \ref{p.key}). Also note that $(\psi_t^*)$ cannot be replaced by $(\psi_t)$, the (unscaled) linear Poincar\'e flow due to the estimate of $\|D\tau\| \le C/|X(x)|$ leading to \eqref{e.III} .

We conclude this section with the following remark concerning the proof of the Bowen property.

\begin{remark}\label{r.Bowen}
	The proof of Theorem~\ref{t.Bowen} does not involve Assumption (B) of Theorem~\ref{m.B}, namely all periodic orbits in $\Lambda$ are homoclinically related. 
\end{remark}

\section{Specification}\label{s.spec}

In this section, we prove the specification property (Theorem~\ref{t.spec}) for orbit segments in $\cG_S$. We assume that Assumption (B) of Theorem~\ref{m.B} holds, so that all periodic orbits in $\Lambda$ are pairwise homoclinically related. This assumption provides a structural backbone that allows the construction of shadowing orbits, even in the presence of singularities.

The proof relies on two main ingredients:
\begin{itemize}
\item {Infinite hyperbolic times (introduced in Section~\ref{ss.infinitePliss}), whose associated invariant manifolds have non-empty transversal intersections with the invariant manifolds of a given hyperbolic periodic orbit $\gamma$;}
\item “Good” ergodic measures that assign small mass to a chosen neighborhood $W$ of $\Sing_\Lambda(X)$, ensuring that typical points along these measures experience sufficiently many infinite hyperbolic times (Lemma~\ref{l.measure.hyp}).
\end{itemize}

Using these ingredients, Proposition~\ref{p.transversal} provides a uniform construction of transversal intersections between fake leaves of orbits in $\cG_S$ and the invariant manifolds of the hyperbolic periodic orbit $\gamma$. Although orbit segments in $\cG_S$ have only finite hyperbolic times, they are shown to lie near points with infinite hyperbolic times, which allows the transversal intersections to be established at a uniform scale.

{The periodic orbit $\gamma$ then acts as a ``bridge'' to create the shadowing orbit, similar in spirit to the approach in~\cite{PYY23}, but now in a more general and technically challenging multi-singular hyperbolic context. This mechanism ultimately enables us to establish the specification property for finite orbit segments in a neighborhood of $\Lambda$, completing the proof of Theorem~\ref{t.spec}.}

\subsection{Transversal intersection at infinite hyperbolic times away from singularities: proof of Proposition~\ref{p.transversal}}\label{ss.transversal}

The goal of this subsection is to prove Proposition~\ref{p.transversal}: there exist compact sets $K_W^*$ with $*=E,F$, away from $\Sing_\Lambda(X)$, such that the stable (resp. unstable) manifold of  points in $K_W^E$ (resp. $K_W^F$) have non-empty transversal intersection with the unstable (resp. stable) manifold of a given periodic orbit. %furthermore, the scale at which the transversal intersection happens have uniform control. 

Let $W\subset B_{r_0}(\Sing_\Lambda(X))$ be any isolating open neighborhood of $\Sing_\Lambda(X)$ (for the moment, think of $W$ as the union of $W_r$ over all singularities in $\Lambda$, with proper choices of $r$ and $r_0$). We define, for $*=E,F:$
$$
\Lambda^*_W(\lambda_0) : \Lambda^*(\lambda_0)\cap W^c,
$$
where $\Lambda^*(\lambda_0)$ are infinite hyperbolic times defined in Section~\ref{ss.infinitePliss}. 
Then it follows from Proposition \ref{p.hyptime.cpt} that $\Lambda^*_W(\lambda_0)$ are compact, and points in $\Lambda^E_W(\lambda_0)$ are away from singularities. As a result, their stable manifolds have diameter at least 
$$
\rho_W := \rho_\textit{inv}(\lambda_0)\cdot \inf_{y\in \Lambda \cap W^c} |X(y)|>0, 
$$
where $\rho_\textit{inv}$ is given by Proposition~\ref{p.inv.mld}.

Consider the following family of measures:
\begin{equation}\label{e.def.cM}\begin{split}
		\cM_{\beta_1, W}& = \{\mu: \mu \mbox{ is invariant {for $f_1$}, and } \mu({W})< \beta_1\} \\
		&= \{\mu:\mu\left({W^c}\right)\ge 1- \beta_1\}.
	\end{split}
\end{equation}

Since $W$ is open, we see that $\cM_{\beta_1, W}$ is compact under the weak-$^*$ topology.

\begin{lemma}\label{l.measure.hyp}
	For any $\mu\in \cM_{\beta_1, W}$, we have 
	$$
	\mu(\Lambda^*_W(\lambda_0)) > \theta_0 - 2\beta_1.
	$$
\end{lemma}

\begin{proof}
Recall that $W\subset B_{r_0}(\Sing_\Lambda(X))$ is an isolating open neighborhood of $\Sing_\Lambda(X)$.	Given $\mu\in \cM_{\beta_1, W}$, consider the ergodic decomposition of $\mu$ with respect to the time-one map $f_1$ (here we view $\mu(\nu)$ as a measure on $\cM^e$, the set of ergodic invariant measures for $f_1$):
	\begin{align*}
	\mu & = \int_{\cM^e} \nu \,d\mu(\nu) \\
\numberthis \label{e.7.2a}		& = \int_{\cM^e_{\beta_1,W}} \nu \,d\mu(\nu) + \int_{\cM^e_{\Sing}} \nu \,d\mu(\nu) + \int_{\cM^e_{1}} \nu \,d\mu(\nu)  ,
	\end{align*}
	where \begin{itemize}
		\item $\cM^e_{\beta_1,W} =\cM_{\beta_1, W} \cap \cM^e$; measures here satisfy, by Proposition \ref{p.hyptime.density}, 
		$$
		\nu(W)< \beta_1, \mbox{ and } \nu(\Lambda^*(\lambda_0))\ge \theta_0;
		$$		
		\item $\cM_{\Sing}^e$ consists of the point masses of singularities; if $\nu \in \cM_{\Sing}^e$, we have $\nu(\Lambda^*_W(\lambda_0)) = 0$ and 
		$$
		\int_{\cM_{\Sing}^e} 1\, d\mu(\nu) \le \mu(\Sing_\Lambda(X)) <\beta_1;
		$$
		\item $\cM^e_{1} = \cM^e\setminus\left(\cM^e_{\beta_1,W} \cup\cM_{\Sing}^e\right)$; here we have  $\nu(\Lambda^*(\lambda_0))\ge \theta_0$ and so 
		$$
		\nu(\Lambda_W^*(\lambda_0))\ge \theta_0 - \nu(W).
		$$
	\end{itemize}
	
	Combining \eqref{e.7.2a}, Proposition~\ref{p.hyptime.density} and the discussion above, we obtain
	\begin{align*}
	&	\mu(\Lambda^*_W(\lambda_0))\\
		 & = \int_{\cM^e_{\beta_1,W}} \nu((\Lambda^*_W(\lambda_0))) \,d\mu(\nu) + \int_{\cM^e_{1}} \nu((\Lambda^*_W(\lambda_0))) \,d\mu(\nu)\\
		& \ge  \int_{\cM^e_{\beta_1,W}} (\theta_0 - \beta_1) \,d\mu(\nu) + \int_{\cM^e_{1}} (\theta_ 0 - \nu(W)) \,d\mu(\nu)\\
		& \ge \int_{\cM^e_{\beta_1,W}} (\theta_0-\beta_1) \,d\mu(\nu)  +  \int_{\cM^e_{1}} (\theta_ 0 - \beta_1) \,d\mu(\nu)  -\int_{\cM^e_{1}} (\nu(W)) \,d\mu(\nu) \\  
		& = \int_{\cM^e} (\theta_0-\beta_1) \,d\mu(\nu) - \int_{\cM^e_\Sing} (\theta_0-\beta_1) \,d\mu(\nu)-\int_{\cM^e_{1}} (\nu(W)) \,d\mu(\nu) \\ 
		&\ge \theta_0 - \beta_1 - \left(\int_{\cM_{\Sing}^e} 1\,d\mu(\nu) +  \int_{\cM^e_{1}}\nu(W) \,d\mu(\nu) \right)\\
		&\ge \theta_0 - \beta_1 - \mu(W) \ge \theta_0-2\beta_1. 
	\end{align*}
\end{proof}

Recall that for points in $\Lambda^E_W(\lambda_0)$, the existence of the local stable manifold (on the normal plane) at scale $\rho_W$ is given by Proposition~\ref{p.inv.mld}. 
For $0<\iota\le  \rho_W$ and $x\in \Lambda^E_W(\lambda_0)$, we denote by $W_{\iota, \cN}^s(x)\subset \cN_{\rho_0|X(x)|}(x)$ the stable manifold at $x$ with diameter $\iota$.

\begin{lemma}\label{l.measure.intersect}
	For every isolating open neighborhood $W$ of $\Sing_\Lambda(X)$, every $\delta>0$ sufficiently small and every ergodic invariant regular measure $\mu$, there exists a subset $\tilde\Lambda^E_W(\mu,\lambda_0)\subset \Lambda^E_W(\lambda_0)$ satisfying  $\mu(\tilde\Lambda^E_W(\mu,\lambda_0)) = \mu(\Lambda^E_W(\lambda_0))$, such that for every $x\in \tilde\Lambda^E_W(\mu,\lambda_0)$, %and every $\vep_x\in (0,\rho_W^i)$, 
	there exists a hyperbolic periodic orbit $\gamma_x$ such that  
	$$
	W_{{\delta/100}, \cN}^s(x)\pitchfork W^u(\gamma_x)\ne\emptyset. 
	$$
	A similar conclusion holds for $x\in \tilde\Lambda^F_W(\lambda_0)$ by considering $-X$.
\end{lemma}

\begin{remark}
	In Lemma \ref{l.measure.intersect} we do not need $\mu(W)$  to be small.
\end{remark}

For $C^{1+\alpha}$ diffeomorphisms and non-singular flows, it is well-known that invariant manifolds (in the sense of Pesin) of typical points of any hyperbolic measure have transversal intersections with some hyperbolic periodic orbit (Katok's Shadowing Lemma). For singular star flows that are only $C^1$, the proof uses Liao's shadowing lemma and can be found in Appendix \ref{s.A2}.

The following lemma summarizes the construction so far and immediately leads to  Proposition~\ref{p.transversal}.

\begin{lemma}\label{l.cpt.intersect}
	For every isolating open neighborhood $W$, every $\delta>0$ sufficiently small and every hyperbolic  periodic orbit $\gamma\in\Lambda$, there exists a compact subset $K^E_W\subset \Lambda^E_W(\lambda_0)$ and a constant $d_0>0$ with the following properties:
	\begin{enumerate}
		\item for every $x\in K^E_W$, $W_{\delta/8,\cN}^s(x)\pitchfork W^{u}_{d_0/2}(\gamma)\ne\emptyset$;
		\item for every $\mu\in\cM_{\beta_1, W}$, $\mu\left(\Lambda^E_W(\lambda_0) \setminus K^E_W\right) < \beta_1$; consequently, we have  $\mu(K^E_W)\ge \theta_0-3\beta_1$.
	\end{enumerate}
\end{lemma}
Again, the same statement holds for a compact subset $K^F_W\subset \Lambda^E_W(\lambda_0)$ by  considering $-X$. 

\begin{proof}
	Let $\delta>0$ be fixed. For every $\mu\in \cM_{\beta_1, W}$, there is a compact set $K^E_W(\mu)\subset \tilde \Lambda^E_W(\mu,\lambda_0)$ where the latter is given by Lemma~\ref{l.measure.intersect}, such that $\mu\left(\Lambda^E_W(\lambda_0) \setminus K^E_W(\mu)\right) < \beta_1/2$, and every $y\in K^E_W(\mu)$ satisfies 
	$$
	W^s_{\delta/32,\cN}(y)\pitchfork W^u(\gamma_y)\ne \emptyset  
	$$
	for some $\gamma_y$ depending on $y$. Let $\gamma$ be a periodic orbit in $\Lambda$ that will be fixed throughout this proof. Then, using Assumption (B) of Theorem~\ref{m.B} and the Inclination Lemma (also known as the $\lambda$-Lemma; see for instance~\cite[Proposition 6.2.23]{Katok}), one has 
	$$
	W^s_{\delta/16,\cN}(y)\pitchfork W^u(\gamma)\ne \emptyset.
	$$
	By compactness of $K^E_W(\mu)$ and the continuity of transversal intersection, there exists $D(\mu)>0$ such that  
	$$
	W^s_{\delta/16,\cN}(y)\pitchfork W^u_{D(\mu)}(\gamma)\ne \emptyset, \,\,\forall y\in K^E_W(\mu).
	$$
	Furthermore, by compactness of $K^E_W(\mu)$ and the continuity of invariant manifolds, for any open neighborhood $U_\mu\supset K^E_W(\mu)$ that is sufficiently small and for every $z\in \Lambda^E_W(\lambda_0)\cap  \overline U_\mu$, 
	$$
	W^s_{\delta/8,\cN}(z)\pitchfork W^u_{2D(\mu)}(\gamma)\ne \emptyset.
	$$
	%Without loss of generality we may assume that $\mu(U_\mu)=0$. 
	
	Below we will construct $K^E_W$ and show that the previous conclusion holds with a uniform constant $d_0$ independent of $\mu$. Note that $\Lambda^E_W(\lambda_0)\setminus U_\mu$ is compact with $\mu$ measure less than $\beta_1/2$. Therefore, one can find a small open neighborhood $\cV_\mu$ in the space of ${f_1}-$invariant probability measures, such that every $\nu\in \cV_\mu$ satisfies $\nu(\Lambda^E_W(\lambda_0)\setminus U_\mu) < \beta_1$. 
	
	Recall that the set $\cM_{\beta_1, W}$ is compact, and the collection of such $\cV_\mu$ form an open covering of $\cM_{\beta_1, W}$. Let $\{\cV_{\mu_i}\}_{i=1}^N$ be a finite sub-covering, and define 
	$$
	K^E_W = \bigcup_{i=1}^N \left( \Lambda^E_W(\lambda_0)\cap  \overline U_\mu\right), 
	$$
	and $d_0 = \max_i\{2D(\mu_i)\}$. Then we have $\mu\left( \Lambda^E_W(\lambda_0)\setminus K^E_W\right)< \beta_1$ for every $\mu\in\cM_{\beta_1, W}$ by construction; moreover, for every $z\in K^E_W$, 
	$$
	W^s_{\delta/8,\cN}(z)\pitchfork W^u_{d_0/2}(\gamma)\ne \emptyset,
	$$
	as desired.

\end{proof}
Now that the lemma is proven, we conclude the proof of Proposition~\ref{p.transversal}.

\subsection{Specification on $\cG_S$}\label{ss.spec.G}
Recall that $\Lambda$ is an isolated chain recurrence class. In this section we prove Theorem~\ref{t.spec}.  This is summarized as the following propositions:

\begin{proposition}\label{p.spec}
	For every isolating open neighborhood $W$ of $\Sing_\Lambda(X)$ and every $\delta>0$ sufficiently small, there exist an open sets $U^E,U^F$ with $U^E\supset K^E_W$ and $ U^F\supset K^F_W$, and a constant $L>0$, such that specification at scale $\delta$ holds for the collection of orbit segments $(x,t)\in\cO(U)$ with $t\in\NN$ satisfying the following properties:
	\begin{enumerate}
		\item there exist $t_1,t_2\ge 0$ such that $(x_{-t_1},t_1+t+t_2)\in \cO(U)$ and $x_{-t_1}\notin B_{r_0}(\Sing_\Lambda(X))$, $x_{t+t_2}\notin B_{r_0}(\Sing_\Lambda(X))$; 
		\item $t>L$;
		\item $x\in U^E{\cap W^c}$ is a $(\lambda_0,E)$-forward hyperbolic time for the orbit segment $(x,t)$ and $(\psi_t^*)$;
		\item $x_t\in U^F{\cap W^c}$ is a $(\lambda_0,F)$-forward hyperbolic time for the orbit segment $(x,t)$ and $(\psi_t^*)$.
 	\end{enumerate}
\end{proposition}
Note that the same shadowing property automatically holds if we replace $\cO(U)$ by  by $\Lambda\times\RR^+$ or $\cO(U_1)$ for any $U_1\subset U$. Then, to finish the proof of Theorem \ref{t.spec} one also need the following proposition.

\begin{proposition}\label{p.nbhd}
	Under the assumptions of the previous proposition, for $\delta>0$ sufficiently small, one can choose $U_1$, an open neighborhood of $\Lambda$ that is contained in $U$, such that:
	\begin{itemize}
		\item For orbit segments $(x^i,t_i)$ in $\Lambda\times\RR^+$, the shadowing orbit $(x,t)$ can be taken from $\cO(U_1)$; 
		\item For orbit segments $(x^i,t_i)$ in $\cO(U_1)$, the shadowing orbit $(x,t)$ can be taken from $\cO(U)$.
	\end{itemize}
\end{proposition}

The proof of these two propositions occupy the rest of this section.

Recall that under Assumption (1) of Proposition \ref{p.spec}, the fake foliations are well-defined on the normal planes of points in $(x,t)$. Also recall $\rho_0,\overline\rho_0$ from Proposition \ref{p.Fx}.

%Recall that the singular dominated splitting $E_N\oplus F_N$ is robust (since they come from a dominated splitting of the extended linear Poincar\'e flow on $\mathfrak B(\Lambda)\subset G^1$ which is compact; see~\cite[Section 3]{CDYZ}). Furthermore, the cone fields $C_\alpha(F_N)$ and $C_\alpha(E_N)$, defined on the normal bundle $N_\Lambda$,  are (forward and backward, respectively) invariant under $(\psi_1^*)$. We will extended these cone fields to a small neighborhood $U_\Lambda$ of $\Lambda$ such that they are still invariant under $(\psi_1^*)$.

\begin{definition}\label{d.cone.disk}
	Given $\alpha>0$ and an embedded disk $D\subset \cN_{\rho_0|X(x)|}(x)$ at a regular point $x$, we say that $D$ is {\em tangent to the $(\alpha,E_N)$-cone} (alternatively, a {\em $(\alpha,E_N)$-disk}), if $\dim D  = \dim E_N$, and $\exp_{x}^{-1}(D)$ is the graph of a $C^1$ function $g: \{v\in E_N: |v|<R\}\to F_N$ for some $R\in (0,\overline\rho_0|X(x)|)$, with $\|Dg\|_{C^0}\le \alpha$. 
	$(\alpha,F_N)$-disks can be defined similarly. 
\end{definition}

Given $(x,t)\in\cO(U)$ satisfying  Assumption (1) of Proposition \ref{p.spec}, $y\in \cN_{\rho_0|X(x)|}(x)$ and $\iota>0$, we will use $\cF_{x,\cN}^{*,(x,t)}(y,\iota)$, $*=E,F$, to denote the disk in the fake leaf $\cF_{x,\cN}^{*,(x,t)}(y)$ with radius $\iota$ centered at $y$. Below the point $x$ will be taken outside $W$, therefore $\rho_0|X|$ is bounded away from zero. In particular, for fixed $\rho\in (0,\rho_0)$ one can find $\iota$ such that  $\cF_{x,\cN}^{*,(x,t)}(y,\iota)\subset \cN_{\rho_0|X(x)|}(x)$ whenever $x\notin W$. 
Due to the construction of the fake foliations, every such disk is tangent to the $\alpha$-cone of the corresponding bundle.

To simplify notation, for fixed $W$ we define 
$$
\epsilon = \inf_{z\notin W}\rho_0 |X(z)|.
$$
Throughout this section, all $(\alpha,*)$-disks ($*=E,F$) are disks contained in an appropriate normal plane $\cN_{\epsilon}(x)$ at some $x\notin W$, and are tangent to the $\alpha$-cone of the corresponding bundle (with the cone fields extended to every point on $\cN_\epsilon(x)$).

\begin{lemma}\label{l.udisk}\cite[Lemma 8.1]{PYY23}
	For any open set $W$ that satisfies  $\Sing_\Lambda(X)\subset  W\subset  B_{r_0}(\Sing_\Lambda(X))$ and any $\alpha>0$, $\delta>0$, $\xi>0$ small enough, 
	%there exists $\zeta_1>0$, such that 
	for every $\zeta>0$,  there  exists  $L>0$ such that the following property holds:
	
	Assume that $t>L$, $x\in\Lambda\cap W^c$ is a $(\lambda_0,E)$-forward hyperbolic time and $x_t\in \Lambda\cap W^c$ is a $(\lambda_0,F)$-backward hyperbolic time for $(x,t)$ and $(\psi_t^*)$.  
	Let $D$ be a $(3\alpha,F_N)$-disk centered at $z\in\cF_{x,\cN}^{E,(x,t)}(x, \delta/2)$ with size at least $\zeta$. Then $D$ contains a smaller disk $D'$ centered at $z$ such that every point in $D'$ is  $\rho_0$-scaled shadowed by the orbit of $x$ up to time $t$. Furthermore, $\cP_{t,x}(D')$ contains a disk with size at least $\delta$ centered at $\cP_{t,x}(z)$ and has transversal intersection with $\cF_{x_t,\cN}^{E,(x,t)}(x_t,\xi)$ at $\cP_{t,x}(z)$. 
	
	%In addition, if $D$ contains a disk with radius at least $\zeta/2$ centered at $y\in D\pitchfork W^{s,N}_\xi(x)$, then it is possible to choose $D'$ such that $P_{t,x}(D')$ contains a disk with radius at least $\zeta$ centered at $z\in P_{t,x}(D') \pitchfork W^{s,N}_\xi(x_t)$.
\end{lemma}

%\todo{ADD A PICTURE}

The proof is an adapted version of the Hadamard-Perron theorem and can be found in \cite{ABV00}, and is therefore omitted. We remark that in \cite[Lemma 8.1]{PYY23} the stable bundle $E_N$ is uniformly contracted by $(\psi_t^*)$ (a trait of sectional hyperbolicity); therefore the stable manifold $W^s_{N}(x)$ exists on the normal space of every regular point $x$. Here we assume that $x$ is a forward hyperbolic time, and the role of the stable manifold is played by the fake leaf $\cF^{E,(x,t)}_{x,\cN}(x)$. However, this causes little difference in the proof, thanks to the exponential contraction of $\cP_{s,x}$ on this fake leaf up to time $t$ as proven in Lemma~\ref{l.hyp.fakeEleaf}.

We also need the following lemma regarding the shadowing at the uniform scale $\delta$.  This does not immediately follow from the previous lemma due to the time change $\tau_{x,y}(t)$ involved in the scaled shadowing property. In particular, we need to obtain an upper bound on $|\tau_{x,y}(t)-t|$.

\begin{lemma}\label{l.bowen1}\cite[Lemma 8.2]{PYY23}
	Let $\delta>0$. Under the assumptions of the previous lemma, one can decrease $%\xi,
	\zeta$ such that for every $y\in D'$, it holds that  $y\in B_{t,\delta}(x)$.

\end{lemma}

The proof follows verbatim of \cite[Lemma 8.2]{PYY23} with the stable manifold of $x$ replaced by the fake leaf, and use Lemma~\ref{l.hyp.fakeEleaf} to obtain the exponential contraction of distance. See also Proposition \ref{p.timedrift}.

Next, we establish the transversal intersection between fake leaves of a hyperbolic time (not infinite) with invariant manifolds of $\gamma.$

\begin{lemma}\label{l.fake.intersect}
	For every isolating open neighborhood $W$ of $\Sing_\Lambda(X)$, every $\delta>0$ sufficiently small and every hyperbolic  periodic orbit $\gamma\in\Lambda$, there exist open neighborhoods $U^E, U^F$ with $K^E_W\subset U^E$ and $K^F_W\subset U^F$ and constants  $L>0,d_0>0$, such that the following holds for every $\xi$ sufficiently small:
	\begin{enumerate}
		\item For every $x\in U^E$ that is a $(\lambda_0,E)$-forward hyperbolic time for the orbit segment $(x,t)$ and $(\psi_t^*)$ with $t>L$, one has  
		$$
		\cF^{E,(x,t)}_{x, \cN}(x,\delta/4)\pitchfork W^{u}_{d_0}(\gamma)\ne\emptyset
		$$
		\item  Furthermore, assume that $y=x_t\in U^F$ is a $(\lambda_0,F)$-backward hyperbolic time for the orbit segment $(x,t)$ and $(\psi_t^*)$ with $t>L$, and $D_0$ is any $(3\alpha,F_N)$-disk centered at some $x'\in\cF_{x,\cN}^{E,(x,t)}(x, \delta/2)$ with size at least $\zeta$. Then the smaller disk $D'_0\subset D_0$ given by Lemma \ref{l.udisk} satisfies that $D:=\cP_{t,x}(D'_0)$ contains a disk $D'$ with size at least $ \delta/4$ centered at $y':=\cP_{t,x}(x')\subset \cF_{y,\cN}^{E,(x,t)}(x,\delta/2)$,
		% and has transversal intersection with $\cF_{x_t,\cN}^{E,(x,t)}(x_t,\xi)$ at $\cP_{t,x}(z)$
		%every $z\in \cF_{y,\cN}^{E,(x,t)}(y,\xi)$, %and every disk $D$ that is contained in $\cF_{x_t,\cN}^{F,(x,t)}(z)$ with radius $\delta/4$, 
		satisfying 
		$$
				\overline D' \pitchfork W^{s}_{d_0}(\gamma)\ne\emptyset.
		$$
	\end{enumerate}
	Furthermore, the angles of the transversal intersection in both cases are bounded away from zero.
\end{lemma}

\begin{proof}
	Given $W$ and $\delta,$ we apply Lemma~\ref{l.cpt.intersect} to $X$ and $-X$ to find compact sets $K^*_W$, $*=E,F$ such that
	\begin{equation*}
		\begin{split}
			&W_{\delta/8,\cN}^s(y)\pitchfork W^{u}_{d_0/2}(\gamma)\ne\emptyset, \forall y\in K^E_W;\,\, \mbox{ and }\\
			&W_{\delta/8,\cN}^u(y)\pitchfork W^{s}_{d_0/2}(\gamma)\ne\emptyset, \forall y\in K^F_W.
		\end{split}
	\end{equation*}
	
	Case (1) of the Lemma follows from the continuity of transversal intersections, the compactness of $K^E_W$ and Proposition~\ref{p.fake.inv} on fake $E$ leaves converging to the stable manifold of points in $K^E_W$.

	Case (2) follows from the continuity of transversal intersections, the compactness of $K^E_W$ and Proposition~\ref{p.fake.inv} (applied to $-X$) on fake $F$ leaves converging to the unstable manifold of points in $K^F_W$.  Finally, note that as $t\to\infty$, $\cP_{t,x}(D_0')$ approximates the fake unstable leaf at $y=x_t$ of size $\delta$.
	
%	following statement concerning continuity of transversal intersections (note that the space $E_N$ varies continuously at regular points):
%	
%	\noindent Assume that $W_{\delta/8,\cN}^u(y)\pitchfork W^{s}_{d_0/2}(\gamma)\ne\emptyset$ for $y\notin W$. Then there exist open neighborhoods $U_y, V_y$ of $y$ with $\overline U_y\subset V_y$  such that for every $x\in U_y$ and $z \in V_y\cap \cN_{\rho_0|X(x)|}(x)$, for every $(\alpha,F_N)$-disk $D$ centered at $z$ with radius $\delta/4$, one has 
%	$$
%	D\pitchfork W^{s}_{d_0}(\gamma)\ne\emptyset.
%	$$ 
%	Then we can use the compactness of $K^F_W$ to obtain a finite covering of $K^F_W$ by $U_{y^i},i=1,\ldots, N$. Then take $U^F = \cup_{i=1}^N U_{y^i}$, and let $\xi\le \min\{ d(\overline U_{y^i}, (V_{y^i})^c): i=1,\ldots, N\}$. 
\end{proof}

The next lemma concerns the maximum gap size, namely the transition time $\tau$ from one orbit segment to the next. 
For this lemma, we use the notation $\pitchfork^N$ to denote the transversal intersection inside $\cN_\epsilon(z)$ between two submanifolds of $\cN_\epsilon(z)$.

%Let $\zeta_1>0$ be given by Lemma~\ref{l.udisk}.

%\todo{Triple check the dependence of parameters}
\begin{lemma}\label{l.D_intersection}
	For isolating open neighborhood  $ W\subset  B_{r_0}(\Sing_\Lambda(X))$ and $\xi>0$ %$\zeta_1>0$ 
	small enough,  there exist $\zeta>0$ and
	$\tau>0$ with the following property:
	
	 Let $z\in U^E$ be any $(\lambda_0,E)-$forward hyperbolic time for the orbit segment $(z,t')$. Let $(x,t)$ with $t>L$ be such that $x$ is a $(\lambda_0,E)-$forward hyperbolic times, $y=x_t\in U^F$ is a $(\lambda_0,F)$-backward hyperbolic time for the orbit segment $(x,t)$ and $(\psi_t^*)$.
		%every $z\in \cF_{y,\cN}^{E,(x,t)}(y,\xi)$,
	Finally, Let $D_0$ be any $(3\alpha,F_N)$-disk centered at some $x'\in\cF_{x,\cN}^{E,(x,t)}(x, \delta/2)$ with size at least $\zeta$. Denote by $D'_0\subset D_0$ the smaller disk given by Lemma \ref{l.udisk} and $D:=\cP_{t,x}(D'_0)$. %contains a disk $D'$ with size at least $ \delta/4$ centered at $z:=\cP_{t,x}(x')\subset \cF_{y,\cN}^{E,(x,t)}(x,\delta/2)$,
	Then there exists $s\in[0,\tau]$ such that $ \cP_z(f_s(D)\cap B_\epsilon(z))$ contains a $(3\alpha,F_N)-$disk $D''$ with size at least $\zeta$, centered at some $z'\in \cP_z\big(f_s(D)\cap B_\epsilon(z)\big)\pitchfork^N \cF^{E,(x,t)}_{z,\cN}(z,\delta/2)$.
\end{lemma}

%\todo{Draw a new picture!}
The proof resembles that of \cite[Lemma 8.6]{PYY23} with stable manifolds replaced by fake $E$ leaves.

\begin{proof}
	We will use the hyperbolic periodic orbit $\gamma$ as a ``bridge'' to build the transversal intersection.

	First, we apply Lemma~\ref{l.fake.intersect} (2) to obtain 
	$$
	\overline D'\pitchfork W^s_{d_0}(\gamma)\ne\emptyset,
	$$ where $\overline D'\subset D$ is the $\delta/4$-disk centered at $y' :=\cP_{t,x}(x')\subset \cF_{y,\cN}^{E,(x,t)}(x,\delta/2)$.
	The intersection happens at a point
	$\tilde y'\in  D$ that is $\delta/4$-close to $y'$, and at an angle $\theta_T$ that is bounded away from zero. Furthermore, there is a sub-disk $D'$ contained in $D$ centered at $\tilde y'$ with size at least  $\delta/2$. See Figure~\ref{f.spec1}.

	\begin{figure}[h!]
		\centering
		\def\svgwidth{\columnwidth}
		\includegraphics[scale=0.93]{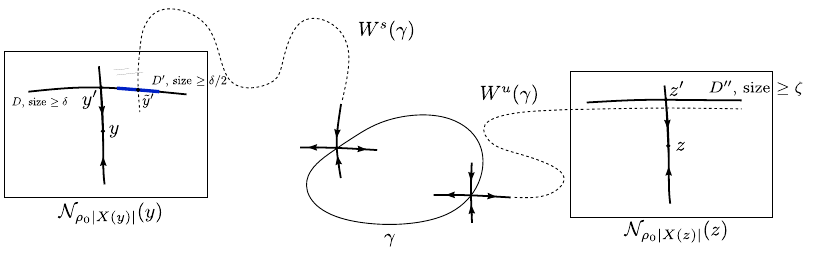}
		\caption{$\gamma$ as a ``bridge''.}
		\label{f.spec1}
	\end{figure}
	
	On the other hand, at $z\in U^E$ we have  
	$$
	W^u_{d_0}(\gamma)\pitchfork \cF^{E,(x,t)}_{z,\cN}(z,\delta/4)\ne\emptyset
	$$
	by Lemma~\ref{l.fake.intersect} (1). Note that both transversal intersections near $y$ and $z$ happen at uniform sizes. 
	Applying the Inclination Lemma (see, for instance,~\cite[Proposition 6.2.23]{Katok}),  there exists $\tau>0$ depending on $\delta$, $\alpha$ (the size of the cone field) and $\theta_T$ but not on $D$ and $y$, such that for some $s\in[0,\tau]$, $f_s(f_{[-\rho_0,\rho_0]}(D'))$ is  $\min\{\alpha, \delta/4\}$-approximated by $W^u_{2 d_0}(\gamma)\cap B_{\epsilon}(z)$ in $C^1$ topology, whose projection to the normal plane, namely $\cP_z\left(W^u_{2 d_0}(\gamma)\cap B_{\epsilon}(z)\right)$, is a disk with uniform size $\epsilon$ and has a transversal intersection with $\cF^{E,(x,t)}_{z,\cN}(z,\delta/4)$ inside $\cN_\epsilon (z)$. This shows that $\cP_z\left(f_s(D')\cap B_{\epsilon}(z)\right)$ is a disk with uniform size and has a transversal intersection with $\cF^{E,(x,t)}_{z,\cN}(z,\delta/2)$ inside $\cN_\epsilon(z)$. Changing $s$  by no more than $\rho_0$ (recall the definition of $\cP_{z}$ from~\eqref{e.Px1}), we may assume that the point of intersection is $z' \in f_{s}(D').$
	In addition,  $\cP_{z}(f_s(D')\cap B_\epsilon(z)) \subset \cN_\epsilon (z)$ contains a disk $D''$ centered at $z'$ with size at least $\zeta$, as desired.

\end{proof}

\begin{proof}[Proof of Proposition~\ref{p.spec}]
	At this point we have proven that \cite[Lemma 8.6]{PYY23} has a counterpart in out setting, namely Lemma~\ref{l.D_intersection}, which has (almost) identical statement and the same logical order in the choice of parameters. Then one only need to follow verbatim the proof of~\cite[Section 8, Proof of Theorem 5.2]{PYY23}. For this reason we shall not include the detailed proof below, but rather outline its structure.
	
	The proof is motivated by the  work of Bowen~\cite{B75_2}. Letting $(x^i,t_i)$, $i=1\ldots, N$	be a sequence of orbit segments satisfying the assumptions of Proposition~\ref{p.spec}, we start with a $(3\alpha,F)$-disk $D_1$ transversely intersecting $\cF^{E,(x,t)}_{x^1,\cN}(x^1,\delta/2)$.\footnote{Indeed one could take $D_1$  to be a sub-disk of $\cF^{F,(x,t)}_{x^1,\cN}(x^1,\delta/2)$.} We first use Lemma~\ref{l.udisk} to obtain a $(\alpha,F)$-disk $\tilde D_1$ in the normal plane of $(x^1)_{t_1}$ with size at least $\zeta_1.$ Then, applying Lemma~\ref{l.D_intersection} we obtain a disk $D_2$ transversely intersects $\cF^{E,(x,t)}_{x^2,\cN}(x^2,\delta/2)$ with uniform size $\zeta$. Furthermore, Lemma~\ref{l.D_intersection} provides a uniform bound on the iterates between $\tilde D_1$ and $D_2$. We are in a position to apply Lemma~\ref{l.udisk} to $D_2$. Recursively, we obtain a sequence of $(3\alpha, F)$-disks $D_k$ in the normal plane of $x^k$, and a sequence of $(\alpha,F)$-disks $\tilde D_k$ in the normal plane of $(x^k)_{t_k}$. They all have uniform sizes ($\zeta$ and $\zeta_1$, respectively) and have transversal intersection with the fake $E$-leaf of their respective reference point $x^k$ or $(x^k)_{t_k}$ at scale $\delta/2$ and angles uniformly bounded away from zero.  Moreover, each $D_{k+1}$ belongs to the forward image of $\tilde D_k$ (therefore, belongs to the forward image of $D_1$) with a uniformly bounded time of iteration. By Lemma~\ref{l.bowen1}, each $D_k$ contains a sub-disk $D_k'$ whose points are in the $(\delta,t_k)$-Bowen ball of $x^k$. This allows us to find a small disk in $D_1$, whose points shadows each orbit segment $(x^k,t_k)$ at scale $\delta$ and has uniformly bounded transition time between consecutive orbit segments. 
\end{proof}

Next we prove Proposition~\ref{p.nbhd} and consequently finish the proof of Theorem \ref{t.spec}. This is the only place where we need $\Lambda$ to be a chain recurrence class. In particular, we need the following result:

\begin{lemma}\label{l.crc}\cite{Con} (see also \cite{BGW})
	Every chain recurrence class $\Lambda$ admits arbitrarily small filtrating neighborhoods $U$, with the property that once an orbit leaves $U$, it will never come back to $U$.
\end{lemma}
\begin{remark}\label{r.crc1}
	The filtrating neighborhoods of $\Lambda$ have the form $U = U^+\cap U^-$ where $U^\pm$ are open sets such that 
	$$
	f_1(\Cl(U^+))\subset U^+, \,\, \mbox{ and } f_{-1}(\Cl(U^-))\subset U^-.
	$$ 
	See for instance \cite[Section 2, Section 3.2]{BGW}. Note that the inclusions above persist under $C^1$ perturbation of $X$, and hence each filtrating neighborhood $U$ remains a filtrating neighborhood for all nearby vector fields $Y$. 
\end{remark}

\begin{proof}[Proof of Proposition~\ref{p.nbhd}] 
	Let $U$ be the open neighborhood of $\Lambda$ as before. We may assume w.l.o.g.\ that $U$ is a filtrating neighborhood of $U$, keeping in mind Remark \ref{r.U}. Next, we take another filtrating neighborhood $U_1\subset U$ of $\Lambda$ such that (shrink $\delta$ in Proposition \ref{p.spec} when necessary):
	\begin{itemize}
		\item $B_{\delta}(\Lambda)\subset U_1$;
		\item $B_{\delta}(U_1)\subset U$.
	\end{itemize} 
	Below we prove the first item of Proposition \ref{p.nbhd}.
	
	Let $(x^i,t_i)$, $i=1,\ldots, n$ be a finite collection of orbit segments taken from $\Lambda\times\RR^+$ that satisfy the assumptions of Proposition \ref{p.spec}, and $(x,t)$ an orbit segment given by Proposition \ref{p.spec} that shadows each $(x^i,t_i)$ at scale $\delta$ with gap times $\tau_i,i=1,\ldots, n-1$ satisfying $\tau_i\le \tau.$ Define 
	$$
	s_i = \sum_{j=1}^{i} t_i + \sum_{j=1}^{i-1} \tau_i.
	$$
	From the construction of $U_1$, we only need to show that $(x_{s_i},\tau_i)\subset U_1$ for every $1\le i\le n-1$. However, this is an immediate consequence of $U_1$ being a filtrating neighborhood of $\Lambda$: if $(x_{s_i},\tau_i)$ escapes $U_1$ at any point, the forward orbit of $x_{s_i}$ can never come back to $U_1$, contradicting with the fact that $(x_{s_i+\tau_i},t_{i+1})\subset U_1$.
	
	The proof of the second item is analogous.
\end{proof}

\section{Robust uniqueness: proof of Theorem~\ref{m.C}}\label{s.opendense}
In this section we prove Theorem~\ref{m.C}.
 First, recall from Lemma \ref{l.generic} and Theorem \ref{t.dichotomy} %\footnote{Theorem \ref{t.dichotomy}, taken from \cite{PYY23a}, was stated for $C^1$ generic star flows. However, the argument in \cite{PYY23a} can be directly applied to the dynamics in a small neighborhood of a multi-singular hyperbolic chain recurrence class with minimal changes. } 
	that $C^1$ generically, every non-trivial, isolated chain recurrence class $\Lambda$ that is multi-singular hyperbolic must be the homoclinic class of a periodic orbit $\gamma$ and therefore has positive topological entropy. Furthermore, by the main result of \cite{Abd}, there exists a small isolating neighborhood $U$ of $\Lambda$ such that for every $Y$ that is in a small  $C^1$ neighborhood $\mathcal U$ of $X$, the maximal invariant set of $Y$ in $U$, denoted by $\Lambda_Y$, is also a chain recurrence class.

Fix $\mathcal R\subset \mathscr X^1(\bM)$ the residual set such that the discussion above holds, and let $X\in\mathcal R$, $\Lambda\subset \bM$ a non-trivial isolated chain recurrence class. We take $\cU$ a $C^1$ small neighborhood of $X$, and $U$ a small isolating neighborhood of $\Lambda$.
The smallness of $\cU$ and $U$ will be changed a finite number of times in the proof below. For now, we assume that $\cU$ and $U$ are taken so that the discussion above holds; we also assume that Theorem~\ref{t.robust.hyp} and, consequently, Lemma \ref{l.nbhd.1}, \ref{l.robust.dom.spl}, \ref{l.robust.hyp} hold for every $Y\in\mathcal U$. Furthermore, every singularity $\sigma_Y$ in $U$ is hyperbolic and is the continuation of some singularity of $X$. %We shall also assume that Assumption (A) holds for $Y\in\cU$. 

Note that Assumption (A) of Theorem \ref{m.B} holds for $X$ and also for every $Y\in\cU$ (shrink $\cU$ if necessary), since $\Lambda_Y$ is a chain recurrence class and singularities being non-degenerate is a $C^1$ open condition. 

Let $\phi:\bM\to \RR$ be a H\"older continuous function that satisfies $\phi(\sigma)<P(\phi, X|_\Lambda)$, for all $\sigma\in\Sing_\Lambda(X)$. Then by Theorem \ref{m.A1} we see that (shrink $\cU$ if necessary)
$$
\phi(\sigma_Y)<P(\phi, Y|_{\Lambda_Y}), \forall \sigma_Y\in \Sing_{\Lambda_Y}(Y),
$$
since both sides vary continuously w.r.t.\ the vector field.
In other words, Assumption (C) also holds for every $Y\in\cU$.

Next, we claim that Assumption (B) of Theorem \ref{m.B} holds for $X\in\mathcal R$ (if necessary, replace $\cR$ by $\cR'\subset\cR$ which is also a residual set). This is because, being multi-singular hyperbolic implies that all periodic orbits are hyperbolic and have the same stable index. Then this statement is a direct consequence of the connecting lemma for chain recurrence classes \cite{Cr06} and \cite[Lemma 3.2]{GW03}.

%First, note that $ \Lambda_Y$ may not be transitive or even chain transitive (meaning that $x\mapsto y$ for every $x,y\in\Lambda_Y$), and singularities in $ \Lambda_Y$ are not necessarily active. However, due to \eqref{e.phiY}, point masses on singularities cannot be equilibrium states. Therefore, one can replace $\Lambda_Y$ by the maximal invariant set in the open set $U\setminus\{\Cl (B_r(\sigma_Y)): \sigma_Y\mbox{ is not active in } \Lambda_Y\}$ where $r>0$ is taken small enough so that $\Cl (B_r(\sigma_Y))$ is an isolating neighborhood of $\sigma_Y$. The resulting set, still denoted by $\Lambda_Y$, only contains active singularities. Note that equilibrium states must exist on this smaller set. As a result, we shall assume from now on that all singularities in $\Lambda_Y$ are active, i.e., Assumption (A) holds on $\Lambda_Y$.

However, note that Assumption (B) may not hold for $Y\in\cU$. This requires us to reproduce the argument in Section~\ref{s.mainthm.proof} and \ref{s.spec} for $Y$ without directly invoking Assumption~(B) of Theorem~\ref{m.B}. Note that Theorem~\ref{t.Bowen}, namely the Bowen property on $\cG_B$ (defined in the same way for $Y\in \cU$) remains valid in this case, since it does not require Assumption (B). The key step below is to prove a robust version of Theorem~\ref{t.spec}. See Theorem~\ref{t.spec.robust} below.

A remark on the notation in this section: since we will consider $C^1$ small perturbation of $X$, we use the following notations to highlight the dependence on the vector field:
\begin{itemize}
	\item $f^Y_t$ is the flow generated by the vector field $Y$; $x_{t,Y} = f_t^Y(x)$.
	\item $(x,t;Y)$ is the orbit segment of $Y$ starting at $x$ with time $t$.
	\item $\psi_{t,Y}, \psi^*_{t,Y}$ are the associated (scaled) linear Poincar\'e flow for $Y$.
	\item %$E^Y,F^Y$ are the continuation of $E$ and $F$; similarly, 
	$E^Y_N\oplus F^Y_N$ is the singular dominated splitting on the normal bundle of $Y$; they can be seen as the continuation of $E_N$ and $F_N$ for $X$; when no confusion is caused, we shall still write $E$ and $F$ for the invariant bundles of $Y$.
	\item For $*=E^Y,F^Y$, $(\lambda_0,*;Y)$-hyperbolic times are defined using $\psi^*_{t,Y}$.
	\item $\cF^{*,(x,t;Y)}_{x,\cN,Y}$, $*=E,F$ are the fake leaves for the orbit segment $(x,t;Y)$ defined for $Y$ and are contained in the normal plane $\cN^{Y}(x) = \exp_{x}\left(\langle Y(x)\rangle^\perp\right)$.
	\item  As before, $r_0>0$ and $U$ are given by Lemma \ref{l.nbhd.1} to \ref{l.robust.hyp}. 
	\item { Note that $W_r$ defined by \eqref{e.W.def} may not be flow-saturated by $Y$. To fix this issue we define, for each $\sigma\in \Sing_\Lambda(Y)$ and $0<r<r_0$,
	$$
	W_r^Y(\sigma) = \bigcup_{x\in B_{r }(\sigma)} \{x_{s, Y}:s\in(-t^-,t^+)\}.
	$$
	The proof of Theorem \ref{t.Bowen} remains unchanged. }
	\item {\ Lemma \ref{l.Elarge} and \ref{l.Flarge} only uses the fact that singularities in $\Lambda$ are Lorenz-like and hence also hold for $Y\in \cU$, with all constants uniform in $\cU$.}
	\item { We let $\beta_0$ be defined by \eqref{e.beta0'} as in Section \ref{s.bowen}. Note that $\rho_1$ and $K_1^*$ given by Proposition \ref{p.tubular3} can be chosen uniformly on $\cU$, and hence $\beta$ is also uniform on $\cU.$ }
\end{itemize}

\subsection{Robust Bowen property} 
The definition of $\cG_B(Y) = \cG_B(\lambda_0,\beta, r_0, W^Y_{r};Y)$ is the same as in Section \ref{sss.bowen}. The Bowen property on $\cG_B(Y)$ follows directly from Theorem \ref{t.Bowen}, whose assumptions are satisfied by every $Y\in\cU$ (with all $W_r$ replaced by $W_r^Y$). We recall that the main step in proving Theorem \ref{t.Bowen} is to show that every $y\in B_{t,\vep_B(x)}$ is scaled shadowed by the orbit of $x$ up to time $t$ (Proposition \ref{p.key}). With this in mind, it is easy to check that the all parameters in Theorem \ref{t.Bowen} are uniform in $\cU.$

\subsection{Robust specification}\label{ss.spec.robust}
Recall that all the results in Section \ref{ss.BS}, in particular Proposition~\ref{p.transversal}, apply to $X\in\cR$ and $\Lambda$. For any given isolating neighborhood $W\subset B_{r_0}(\Sing_\Lambda(X))$ and $\delta>0$ small, we let $K^*_W(X)$, $*=E,F$ be the compact subsets of $\Lambda^*(\lambda_0,X)\cap W^c$, $\gamma = \gamma_X$ a hyperbolic periodic orbit, and $d_0>0$ a constant given by Proposition~\ref{p.transversal} applied to $X$. Here $W$ should not be confused with $W_r^Y$ in the previous subsection: $W$ is chosen for $X$ once and for all, while $W_r^Y$ depends on $Y\in\cU$.

 %Here we slighted changed the notation to highlight the dependence of various objects on $X$.

Now we are ready to define $\cG_S$ for $C^1$ vector fields $Y$ close to $X$. Let $\cU$ be a small $C^1$ neighborhood of $X$ which shall be specified later, and $U$ a small open isolating neighborhood of $\Lambda$ such that all previous discussions  (in particular Lemmas \ref{l.nbhd.1} to \ref{l.robust.hyp}) hold. For any isolating neighborhood $W$ of $\Sing_\Lambda(X)$,\footnote{Note that $W$ is also an isolating neighborhood of $\Sing_{\Lambda_Y}(Y)$ for $Y$ sufficiently close to $X$.} any open neighborhoods $U^*$ ($*=E,F$) of $K^*_W(X)$ and $Y\in\cU$, we define $\cG_S(Y) = \cG_S(\lambda_0,r_0,W, U^E,U^F;Y)$ as the collection of orbit segments $(x,t;Y)$,  $t\in\NN$, with the following properties:
\begin{enumerate}
	\item There exists $t_1,t_2\ge0$ such that $(x_{-t_1,Y}, t_1+t+t_2;Y)\subset \cO(U)$;
	\item $x\in U^E\cap W^c$ is a $(\lambda_0,E^Y;Y)$-forward hyperbolic time for the orbit segment $(x,t;Y)$ and $(\psi_{t,Y}^*)$;
	\item $x_{t,Y}\in U^F\cap W^c$ is a $(\lambda_0,F^Y;Y)$-backward hyperbolic time for the orbit segment $(x,t;Y)$ and $(\psi_{t,Y}^*)$;
\end{enumerate}

\begin{remark}
	It is worth pointing out that $U^E,U^F$ may not contain any infinite hyperbolic time of $Y$, since the set of infinite hyperbolic times varies upper semi-continuously w.r.t.\ the system. However, all the proofs below will only involve hyperbolic times of {\em finite orbit segments} of $Y$, which are close to infinite hyperbolic times of $X$ as long as $Y$ is close to $X$ and the time is sufficiently long. 
\end{remark}

The next theorem generalizes Theorem~\ref{t.spec} to nearby vector fields that do not necessarily satisfy Assumption (B) of Theorem~\ref{m.B}.

\begin{theorem}\label{t.spec.robust}
	Assume that the assumptions of Theorem~\ref{m.C}. Then, for $r_0>0$ sufficiently small, for every isolating open neighborhood $W\subset B_{r_0}(\Sing_\Lambda(X))$ and every $\delta>0$ sufficiently small, there exist open sets $U^*\supset K^*_W$, $*=E,F$ and a $C^1$ small neighborhood $\cU$ of $X$, such that for every $Y\in\cU$, there exists an open set $U_1\subset U$ such that%\todo{It's better to have $U^*$ independent of $Y$.} 
	\begin{enumerate}
		\item tail specification at scale $\delta$ holds on $\cG_S(\lambda_0,r_0,W, U^E,U^F;Y)\cap \cO(U_1)$ with shadowing orbits contained in $\cO(U)$; and  
		\item tail specification at scale $\delta$ holds on $\cG_S(\lambda_0,r_0,W, U^E,U^F;Y)\cap \Lambda_Y\times\RR^+$ with shadowing orbits contained in $\cO(U_1)$.
	\end{enumerate}

\end{theorem}
The key point here is that the sets $W, K_W^*, U^*$ are chosen for $X$ and thus do not depend on $Y$. Meanwhile, the open neighborhood $\cU$ depends on $W$.

By the choice of $\cU$, $U$ and Theorem~\ref{t.robust.hyp}, $\Lambda_Y$ is multi-singular hyperbolic. For the sake of simplicity, we shall assume that $\lambda_0>1$, $\theta_0\in (0,1)$ are taken so that Lemma \ref{l.hyptime} holds for every $Y\in\cU$. This is possible because $\eta>1$ and $T_W>0$ given by Definition~\ref{d.multising} and Lemma \ref{l.robust.hyp} can be taken continuously w.r.t.\,$X$ in $C^1$ topology, a fact hidden in the proof of \cite[Theorem B]{CDYZ}. Meanwhile, $\theta_0$, the density given by the Pliss lemma (Theorem \ref{t.pliss}), only depends on $\lambda_0$, $\eta$ and the norm of $\psi^*_1$ which is also continuous in $C^1$ topology. We also take $\lambda_1$ by \eqref{e.lambda1}. With these choices, it is straightforward to check that Lemma \ref{l.udisk} and \ref{l.bowen1} apply to every $Y\in \cU$ (shrink $\cU$ if necessary), with all parameters taken independent of $Y$.\footnote{ Here one can slightly increase $r_0$ and shrink $\cU$ such that the prescribed set $W$ satisfies $W\subset B_{r_0}(\Sing_{\Lambda_Y}(Y))$.  }
	
Next we consider the counterpart of Lemma~\ref{l.fake.intersect} in this setting. In particular, we shall prove the following lemma.

{
\begin{lemma}\label{l.fake.intersect.robust}
	For every isolating open neighborhood $W$ of $\Sing_\Lambda(X)$, every $\delta>0$ sufficiently small and every hyperbolic  periodic orbit $\gamma\in\Lambda$ of $X$, there exist open neighborhoods $U^E, U^F$ with $K^E_W\subset U^E$ and $K^F_W\subset U^F$, constants  $L>0,d_0>0$, and a $C^1$ neighborhood $\cU$ of $X$ such that the following holds for every $\xi$ sufficiently small and every $Y\in\cU$ (here $\gamma_Y$ is the continuation of $\gamma$):
	\begin{enumerate}
		\item for every $x\in \Lambda_Y \cap U^E$ that is a $(\lambda_0,E^Y;Y)$-forward hyperbolic time for the orbit segment $(x,t;Y)$ and $(\psi_{t,Y}^*)$ with $t>L$, one has  
		$$
		\cF^{E, (x,t;Y)}_{x, \cN,Y}(x,\delta/4)\pitchfork W^{u}_{d_0,Y}(\gamma_Y)\ne\emptyset;
		$$
		\item for every $y\in U^F$, every $z\in \cF_{y,\cN,Y }^{E, (x,t;Y)}(y,\xi)$ and every $(\alpha,F^Y_N;Y)$-disk $D$ with radius $\delta/4$ centered at $z$, one has 
		$$
		D\pitchfork W^{s}_{d_0,Y}(\gamma_Y)\ne\emptyset.
		$$
	\end{enumerate}
	Furthermore, the angles of the transversal intersection in both cases are bounded away from zero.
\end{lemma}

\begin{proof}
	The proof essentially follows the same lines as the proof of Lemma~\ref{l.fake.intersect}. Note that we still have 
	\begin{equation*}
		\begin{split}
			&W_{\delta/8,\cN,X}^s(y)\pitchfork W^{u}_{d_0/2,X}(\gamma)\ne\emptyset, \forall y\in K^E_W,\,\, \mbox{ and }\\
			&W_{\delta/8,\cN,X}^u(y)\pitchfork W^{s}_{d_0/2,X}(\gamma)\ne\emptyset, \forall y\in K^F_W,
		\end{split}
	\end{equation*}
	due to $X$ satisfying Lemma~\ref{l.cpt.intersect}. Then, observe that (shrink $\cU$ when necessary):
	\begin{itemize}
		\item one can take $\cU$ small enough so that $W$ is an isolating neighborhood of $\Sing_{\Lambda_Y}(Y)$ for every $Y\in\cU$; furthermore, $\frac{|X(x)|}{|Y(x)|}$ is close to one uniformly for $Y\in\cU$ and $x\notin W$;
		\item the invariant manifolds of $\gamma$ varies continuously in $C^1$ topology;
		\item the bundles $E^X$ varies continuously w.r.t.\,$X$ in $C^1$ topology; 
		\item hyperbolic times vary upper semi-continuously in $C^1$ topology: assume that $Y_n\to X$ in $C^1$ topology, $x^n$ is a $(\lambda_0,E^{Y_n};Y_n)$-forward hyperbolic time for the orbit segment $(x^n,t_n;Y_n)$  with $x^n\to x\notin \Sing(X)$ and $t_n\to\infty$, then $x$ is a $(\lambda_0,E^X;X)$-forward infinite hyperbolic time; this is due to the continuity of $\psi_{s,X}^*$ with respect to both the base point and the vector field;
		\item under the setting of the previous item, the fake leaves at $x^n$ for the vector field $Y_n$ converges to the stable manifold $W^s_{\loc, \cN,X}(x)$ in $C^0$ topology in the space of $C^1$ embeddings of $\dim E_N-$dimensional disks; this is because invariant manifolds given by the Hadamard-Perron Theorem
		vary continuously with respect to the system in $C^1$ topology.
		%\item the $\alpha$-cone of $E=E^X$ is invariant under $Df_t^Y$ for every $Y\in\cU$; same holds for the $\alpha$-cone of $F=F^X;$ in particular,  the open neighborhoods $U_y$ in Case (2) of Lemma~\ref{l.fake.intersect} can be chosen uniformly for $Y\in\cU.$
	\end{itemize}
	Then the lemma follows from the robustness  of transversal intersections. Furthermore, one can shrink $U^*$ so that the desired property holds for every $Y\in\cU$.
\end{proof}

\begin{proof}[Proof of Theorem~\ref{t.spec.robust}]
	With Lemma~\ref{l.fake.intersect.robust} replacing Lemma~\ref{l.fake.intersect}, it is straightforward to verify that Lemma~\ref{l.D_intersection} holds when the vector field $X$ is replaced by $Y$. 
	
	At this point we have re-established all relevant lemmas in Section \ref{ss.spec.G} for the perturbed vector field $Y$, and one only need to repeat verbatim the proof of Proposition~\ref{p.spec} to obtain the specification property on $\cG_S(Y) = \cG_S(\lambda_0,r_0,W, U^E,U^F;Y)$. On the other hand, by Lemma \ref{l.crc} and Remark \ref{r.crc1}, $U_1$ and $U$, which are chosen for $X$, are still filtrating neighborhoods of $\Lambda_Y$ for the vector field $Y$. One can then repeat the proof of Proposition \ref{p.nbhd} to get that the shadowing orbits for orbit segments in $\Lambda_Y\times\RR^+$ (resp. $\cO(U_1)$) are in $\cO(U_1)$ (resp. $\cO(U)$), finishing the proof of Theorem \ref{t.spec.robust}.
\end{proof}

}

\subsection{Proof of Theorem~\ref{m.C}}\label{ss.thmC.proof}

Now that we have re-established the Bowen property on $\cG_B(Y)$ and tail specification on $\cG_S(Y)$, one can attempt to reproduce the argument in Section~\ref{ss.para} and \ref{ss.gap}. We only remark that the constant $a_0$ from Equation \eqref{e.gap1} can be chosen for $X$ but applies to all $Y$ sufficiently close to $X$ due to the continuity of the topological pressure as a function of the vector field. Similarly, $U$, the open neighborhood of $\Lambda$, is also chosen for $X$ but applies to all $Y$ close to $X$. This is because \eqref{e.hyp1} and \eqref{e.hyp2} hold for all nearby vector fields, as well as Remark \ref{r.crc1}.

The only issue is that Proposition~\ref{p.transversal} is stated for $X$ but not for nearby vector field $Y$,\footnote{Recall that $K^*_W$ are only defined for $X$ and may not even be contained in $\Lambda_Y$.} and we need this proposition in order to obtain the crucial Lemma~\ref{l.bipliss2}, i.e., to find simultaneous Pliss times of $Y$ inside $U^*$, $*=E,F$. Below we will prove a counterpart of Lemma~\ref{l.bipliss2} directly using a semi-continuity argument. 

All parameters below, e.g., $\theta_0,\lambda_0, \beta_1,\beta_0$ are taken in the same order as they appeared in Section~\ref{s.mainthm.proof}, and recall that they can be chosen uniformly for vector fields $Y$ in a small $C^1$ neighborhood $\cU$ of $X$.

{ 
\begin{lemma}\label{l.bipliss2.robust}
	For any isolating neighborhood $W\subset B_{r_0}(\Sing_\Lambda(X))$, $\delta>0$ and neighborhood $U^E$ of $K^E_{W}$, there exist a constant $T_U\ge T_W$ and a $C^1$ neighborhood $\cU$ of $X$, such that for every $Y\in\cU$, if $(x,t;Y)$ is an orbit segment in $\Lambda_Y$ satisfying
	\begin{enumerate}
		\item  there exist $t_1,t_2\ge 0$ such that $(x_{-t_1,Y}, t_1+t+t_2;Y)\in \cO(U)$ and $x_{-t_1,Y}\notin B_{r_0}(\Sing_\Lambda(X)), x_{t+t_2,Y}\notin B_{r_0}(\Sing_\Lambda(X))$;
		\item $x,x_{t,Y}\notin W$ and $t>T_U$;
		\item  the following holds:
		\begin{equation}\label{e.5.5a.Y}
			\frac{1}{\floor{t}+1}\#\left\{ i\in[0,\floor{t}]\cap\NN: x_{i,Y}\in W\right\} < \beta_1;
		\end{equation}
	\end{enumerate}
	then there exist $x_{i,Y}\in U^E, i\in[0,\floor{t}]\cap\NN$ such that $x_{i,Y}$ is a $(\lambda_0,E^Y,\beta_0,W;Y)$-simultaneous forward Pliss time for the orbit segment $(x_i, t-i;Y)$ and $(\psi_{t,Y}^*)$. 
	
	A similar statement holds for $(\lambda_0,F^Y,\beta_0,W;Y)$-simultaneous backward Pliss times.
\end{lemma}

\begin{proof}
	Towards a contradiction, assume that there exist some neighborhood $U^E$ of $K^E_W$, a sequence of times $t_n\nearrow+\infty$, a sequence of $C^1$ vector fields $Y_n\to X$ and a sequence of points $x^n$, for which we have
	\begin{itemize}
		\item $x^n,(x_n)_{t_n,Y_n}\notin W$;
		\item the following holds:
		\begin{equation}\label{e.5.5a.cont}
			\frac{1}{\floor{t_n}+1}\#\left\{ i\in[0,\floor{t_n}]\cap\NN: (x^n)_{i,Y_n}\in W\right\} < \beta_1.
		\end{equation}
	\end{itemize} 
	However, there does not exist any $(\lambda_0,E^{Y_n},\beta_0,W;Y_n)$-simultaneous forward Pliss time for the orbit segment $((x^n)_{i_n}, t_n-i_n;Y_n)$ and $(\psi_{t,Y_n}^*)$ that is contained in $U^E$.
	
	The proof below is similar to Lemma~\ref{l.bipliss2}, except that one need to get a contradiction using Proposition~\ref{p.transversal} applied to $X$. We define
	\begin{equation*}
		\begin{split}
			I_n^{Y_n} = \Bigg\{(x^n)_{i,Y_n}:\,\, & i\in \left[0,\left(1-\frac{1}{1000}\theta_0\right)t_n\right]\cap\NN, (x^n)_{i,Y_n} \mbox{ is a $(\lambda_0,E^Y, \beta_0,W;Y_n)$-}\\
			&\mbox{ simultaneous forward Pliss time for the orbit}\\
			&	\mbox {   segment  } \left((x^n)_{i,Y_n},t_n-i;Y_n\right)	\Bigg\}.
		\end{split}
	\end{equation*}
	Note that Lemma~\ref{l.bihyptime} does not rely on Assumption (B) of Theorem~\ref{m.B} and can be applied to $Y_n$. By the choice of $\beta_1$ by~\eqref{e.choice.beta1a} we have $\#I_n^{Y_n} \ge  \frac{998}{1000} \theta_0t_n$. % furthermore by Lemma~\ref{l.rec} we get  $\#J_n \le  \frac{1}{1000} \theta_0t_n$.

	Now consider the empirical measures
	$$
	\mu_n^{Y_n}:= \frac{1}{t_n} \sum_{i=0}^{t_n-1}\delta_{(x^n)_{i,Y_n}}
	$$	
	where $\delta_y$ is the point mass of $y$. So we have $\mu_n^{Y_n}(I_n^{Y_n})\ge \frac{998}{1000} \theta_0$ by construction.
	By taking a subsequence if necessary, we may assume that  $\{\mu_n^{Y_n}\}$ converges in weak-* topology to an $f^X_1$-invariant measure $\mu^X$, where $f_1^X$ is the time-one map of $X$. Furthermore, \eqref{e.5.5a.cont} means that 
	$$
	\mu_n^{Y_n}(W) < \beta_1,\forall n>0. 
	$$
	As a result, 
	\begin{equation}
		\mu^X\left(B_{r_0}(\Sing_\Lambda(X))\right)\le \beta_1. 
	\end{equation}
	
	Define $\tilde I = \bigcup_n I_n^{Y_n}$, and denote by $I$ the set of all limit points of $\tilde I$. %Define $\tilde J$ and $J$ in the same way. 
	Then $I$ is compact. 
	
	\medskip 
	\noindent Claim 1. Every $y\in I\cap\Reg(\Lambda)$ is a $(\lambda_0,E^X;X)$-forward infinite hyperbolic time; consequently $I\subset \Lambda^E(\lambda_0)$ where $\Lambda^E(\lambda_0)$ is defined for $X$. %Same holds for $J$.
	
	\noindent Proof of Claim 1. Note that $\psi^*_{1,Y_n}$ converges to $\psi^*_{1,X}$ as $n\to\infty$ in operator norm; furthermore, $E^{Y_n}$ converges to $E^X$. This shows that for fixed $T>0$, if $\{y^n\}$ is a sequence of $(\lambda_0,E^{Y_n};Y_n)$-forward hyperbolic times for the orbit segment $(y^n,T;Y_n)$, then the limit point $y$ is a $(\lambda_0,E^{X};X)$-forward hyperbolic time for the orbit segment $(y,T;X)$. The choice of $0.999 t_n$ in the definition of $I^{Y_n}_n$ guarantees that the previous argument applies to orbit segments with arbitrarily large length. Therefore the limit point is an infinite hyperbolic time for $X$.
	
	\medskip 
	\noindent Claim 2. We have 
	$$
	\mu^X(I)\ge \frac{998}{1000}\theta_0.\,\, %\mbox{ therefore } \mu(\lambda_0,E\setminus I)\le 
	$$
	
	\noindent Proof of Claim 2. Take any open neighborhood $U\supset I$, we have $I^{Y_n}_n\subset U$ for all $n$ large enough. Then we obtain $\mu_n^{Y_n}(\overline U)\ge \mu_n^{Y_n}(I^{Y_n}_n) \ge \frac{998}{1000}\theta_0$ for all $n$ large. Therefore the same inequality holds for $\mu^X(\overline U)$. Since $U$ is arbitrary, we get the same inequality for $I$ by taking a sequence of decreasing neighborhoods. 
	
	On the other hand, by Proposition~\ref{p.transversal} (2) applied to $X$ and Equation \eqref{e.5.5b} we see that 
	$$
	\mu^X\left(I\cap K^E_W\right)\ge \mu^X(I) - \mu^X\left(\Lambda^E(\lambda_0)\setminus  K^E_W\right)\ge \frac{996}{1000}\theta_0.
	$$
	Since $U^E$ is an open neighborhood of $K^E_W$, we must have $I^{Y_n}_n\cap U^E\ne\emptyset$ for $n$ large enough,  a contradiction.
	
\end{proof}

At this point we have re-established, for the vector field $Y\in\cU$, all major lemmas leading to Theorem~\ref{m.B}. One only need to repeat verbatim the argument in Section~\ref{ss.para} and \ref{ss.gap}, and Theorem~\ref{m.C} follows.

}

\section{Application: equilibrium states for singular star flows}\label{s.star}
In this section we apply Theorems \ref{m.A} to \ref{m.C} to singular star flows, and prove Theorem~\ref{mc.C} to~\ref{mc.A}. We will need the following lemma which verifies Assumption (B) of Theorem~\ref{m.B}.

\begin{lemma}\label{l.generic.HR}
	There exists a residual set $\cR$ in the space of $C^1$ star vector fields, such that for every $X\in\cR$ and every chain recurrence class $C$ of $X$, all periodic orbits in $C$ are pairwise homoclinically related. 
\end{lemma}
\begin{proof}
	Recall from Lemma~\ref{l.generic} that $C^1$ generically, any chain recurrence class that contains a  periodic orbit coincide with the homoclinic classes of this periodic orbit. Also note that the star properties guarantees that all periodic orbits in $C$ have the same stable index (see for instance \cite{SGW}). Then this lemma is a direct consequence of the connecting lemma for chain recurrence classes \cite{Cr06} and \cite[Lemma 3.2]{GW03}.
\end{proof}

\subsection{Proof of Theorem~\ref{mc.A}}\label{ss.maincorA}
\begin{proof}[Proof of Theorem \ref{mc.A}]
	As we have seen in Theorem~\ref{t.BD} (see also \cite[Theorem 1.8]{CDYZ}), $C^1$ open and densely, every star vector field $X$ is multi-singular hyperbolic. Also note that Theorem~\ref{m.A} does not require (chain) transitivity nor singularities being active. Therefore one can apply Theorem~\ref{m.A} to the chain recurrent set $\mbox{CR}(X)$ which is a compact invariant set to obtain that $X$ is robustly almost expansive at some scale $\vep_{\Exp}>0$, finishing the proof of Theorem~\ref{mc.A}.
\end{proof}

We remark that the only reason $C^1$ open and dense is involved in Theorem~\ref{mc.A} is because of Theorem~\ref{t.BD}. In particular, if one can prove that all star flows are multi-singular hyperbolic (such a proof is recently announced in dimension three; see \cite{BDJ}. Also see \cite[Question 1]{BD21} and \cite[Question 1.9]{CDYZ}), then one immediately obtains that all star flows are robustly almost expansive.

\subsection{Finiteness of equilibrium states for open and dense star flows: proof of Theorem~\ref{mc.C} and \ref{mc.C1}}\label{ss.maincorB}

We state the precise result on the finiteness of equilibrium states for star vector fields. Recall that $\mathscr{X}^{1,*}(\bM)$ is the collection of $C^1$ star vector fields, %$\mathscr{X}^{1,*}_+(\bM)$ denotes the set of $C^1$ star vector fields with positive topological entropy, 
and $\mathscr{X}^{1,*}_\phi(\bM)$ is the set of  $C^1$ star vector fields such that 
\begin{equation}\label{e.pgap.global}
	\phi(\sigma)< P(\phi,X),\forall \sigma\in\Sing(X).
\end{equation}

%
%\begin{theorem}\label{t.continuity}
%	Let $\phi:\bM\to \RR$ be a continuous function. Then there exists a $C^1$ open and dense set  $\cU\subset \mathscr{X}^{1,*}_+(\bM)\cap \mathscr{X}^{1,*}_\phi(\bM)$ such that the topological pressure of $\phi$ varies continuously on w.r.t. $X\in\cU$ in $C^1$ topology.
%\end{theorem}
%
%\begin{proof}
%	By Theorem \ref{mc.A}, the topological entropy of $\phi$, as a function of the vector field, varies upper semi-continuously for a $C^1$ open and dense subset $\cU$ of star vector fields. We only need to show that on $\cU\cap  \mathscr{X}^{1,*}_+(\bM)\cap \mathscr{X}^{1,*}_\phi(\bM)$, $P(\phi,X)$ is also lower semi-continuous. 
%	
%	Note that there exists a chain recurrence class $C$ such that $P(\phi,X) = P(\phi,X|_C)$. If the topological entropy of $X|_C$ is zero then by \eqref{e.pgap.global}, $C$ 
%\end{proof}
%
%

\begin{theorem}\label{t.mc.B}
There exists a $C^1$ residual set $\cR\subset \mathscr{X}^{1,*}(\bM)$ such that  for every $X\in\cR$, every H\"older continuous function $\phi:\bM\to\RR$ satisfying \eqref{e.pgap.global} has only finitely many ergodic equilibrium states.
\end{theorem}
\begin{proof}
	Let $\cR$ be the residual set given by Theorem~\ref{t.dichotomy}. Passing to a generic subset, we may assume that for every $X\in\cR$ and every chain recurrence class $C$ of $X$, all periodic orbits in $C$ (if exist) are pairwise homoclinically related due to Lemma~\ref{l.generic.HR}. 
	
	%Define $\cR^* = \cR\cap \mathscr{X}^{1,*}(\bM)$.	
	For $X\in\cR,$ let $\phi:\bM\to\RR$ be a H\"older continuous function satisfying~\eqref{e.pgap.global}, and $\mu$ be an ergodic equilibrium state of $\phi$.  Then from Section~\ref{ss.crc} we see that the support of $\mu$ is contained in a chain recurrence class $C$. We claim that if $C$ is non-trivial (i.e., it is not a singularity or a periodic orbit), then $h_{top}(X|_C)>0$; by Theorem~\ref{t.dichotomy} (1), this implies that $C$ is an isolated homoclinic class.
	
	Assume by contradiction that $C$ is non-trivial, yet $h_{top}(X|_C)=0$. Then Theorem~\ref{t.dichotomy} (2) shows that the only invariant measures support on $C$ are the point masses of singularities. By~\eqref{e.pgap.global}, such measures cannot be equilibrium states of $\phi$. This is a contradiction. 
	
	It remains to show the finiteness. Assume by contradiction that there exist infinitely many distinct, ergodic equilibrium states of $\phi$. We may take a sequence of such measures, which we denote by $\{\mu^i: i\in\NN\}$. Without loss of generality, we may assume that $\mu^i\to \mu$ in weak-$^*$ topology. By Theorem~\ref{m.A} and~\cite{B72}, the metric entropy $h_\mu(X)$ varies upper  semi-continuously as a function of $\mu$ in the weak-* topology. As a result, the metric pressure 
	$$
	P_\mu(X):= h_\mu(X) + \int \phi\,d\mu
	$$
	is also upper semi-continuous. This implies that 
	$$
	P_\mu(X) \ge \limsup_n P_{\mu^n}(\phi) = P(\phi),
	$$ 
	i.e., $\mu$ is also an equilibrium state although not necessarily ergodic. Take a typical ergodic component $\tilde\mu$ of $\mu$, then $\tilde\mu$ is an equilibrium state because the metric pressure is an affine map of the invariant measure. Note that $\supp\tilde\mu \subset \limsup_n C_n$ where $C_n$ is the chain recurrence class containing $\supp\mu^n$, and the limit is taken in the Hausdorff topology.  
	On the other hand, $\tilde\mu$ is supported on a chain recurrence class $\tilde C$. If $\tilde C$ is non-trivial, then the previous claim shows that it is an isolated homoclinic class, which is a contradiction; if it is trivial, then it cannot be a singularity due to~\eqref{e.pgap.global}; this means that it is a hyperbolic periodic orbit $\gamma$ whose homoclinic class is trivial. Either way, $\tilde C$ is isolated,\footnote{In the second case where $\tilde C = \gamma$, it is isolated due to the following reason. Assume by contradiction that $\{D_n\}$ is a sequence of distinct chain recurrence classes of $X$ converging to $\gamma$ in Hausdorff topology. By Lemma~\ref{l.generic} (3), each $D_n$, if non-trivial, must be the Hausdorff limit of periodic orbits. This shows that $\gamma$ is approximated in Hausdorff topology by periodic orbits $\gamma_n$ of $X$ that are different from $\gamma$. By the star property, for $n$ sufficiently large $\gamma$ and $\gamma_n$ must have the same index and are homoclinically related; this is impossible since the homoclinic class of $\gamma$ is assumed to be trivial.} meaning that $\tilde C=C_n$ for all $n$ large enough. However, this is impossible due to Theorem~\ref{m.B}, which states that $\tilde C$ supports a unique equilibrium state. This finishes the proof of Claim 1.
\end{proof}

\begin{theorem}\label{t.mc.B1}
	 For every H\"older continuous function $\phi:\bM\to\RR$,  there exists a $C^1$ residual set $\cR_\phi \subset  \mathscr{X}^{1,*}_\phi(\bM)$ such that for every $X\in \cR_\phi$, there exists a $C^1$ neighborhood $\cU_X$ such that every $Y\in\cU_X$ has only finitely many ergodic equilibrium states. Furthermore, the number of ergodic equilibrium states for $Y$ is no more than the number of ergodic equilibrium states for $X$.
\end{theorem}

\begin{proof}
	By Theorem \ref{mc.A}, there exists a $C^1$ open and dense subset $\cU\subset \mathscr{X}^{1,*}(\bM)$, such that for every continuous function $\phi:\bM\to \RR$, the topological pressure $P(\phi,X)$ varies upper semi-continuously w.r.t. $X\in\cU$ in $C^1$ topology. Therefore there exists a $C^1$ residual subset $\tilde \cR_\phi\subset \cU$ consists of the point of continuity of $P(\phi,X)$ (as a function of $X$). In particular, \eqref{e.pgap.global} holds in a $C^1$ neighborhood of $X$.

	Let $\cR_\phi = \tilde \cR_\phi \cap \cR\cap \mathscr{X}^{1,*}_\phi(\bM)$ where $\cR$ is given by Theorem \ref{t.mc.B}. Every $X\in \cR_\phi$ has only finitely many ergodic equilibrium states $\mu_{1,X},\ldots, \mu_{N,X}$. Suppose that each $\mu_{i,X}$ is contained in the chain recurrence class $C_i$. Then $C_i$, $i=1,\ldots, N$ are all distinct due to Theorem~\ref{m.B}. 
	
	For $i=1,\ldots, N$, let $\cU_i$ be the $C^1$ neighborhood of $X$ and $U_i$ the open neighborhood of $C_i$ given by Theorem~\ref{m.C}. W.l.o.g.\ we may assume that $U_i$ are pairwise disjoint. Define $\tilde \cU = \bigcap_i \cU_i$ an open neighborhood of $X$. We claim that some open subset $\cU\subset \tilde \cU$ that contains $X$ has the desired property. 
	
	Towards a contradiction, assume that $\{Y_n\}\subset \tilde\cU$ is a sequence of vector fields converging to $X$ in $C^1$ topology, such that  $\phi(\sigma)< P(\phi,Y_n),\forall \sigma\in\Sing(Y_n)$ and for every $n$; however, each $Y_n$ has at least $N+1$ ergodic equilibrium states. By Theorem~\ref{m.C}, each $Y_n$ can have at most one equilibrium state in each $U_i$. Consequently, each $Y_n$ has at least one ergodic equilibrium state outside $\bigcup_i U_i$. Denote this equilibrium state by $\mu^n_{Y_n}$, and assume by taking a subsequent that $\mu^n_{Y_n}\to\mu$ in weak-*. By Theorem~\ref{m.A}, the metric entropy is an upper semi-continuous function w.r.t.\,both the measure and the vector field; also recall that $P(\phi,X)$ is continuous at $X$. Therefore $\mu$ is an equilibrium state for $X$. However, the support of $\mu$ does not intersect with $\supp \mu_{i,X}$, $i=1,\ldots, N$. So $X$ has at least $N+1$ ergodic equilibrium states. This is a contradiction.
\end{proof}

Note that Theorem \ref{mc.C} is a special case of Theorem \ref{mc.C1} with $\phi\equiv 0$. Therefore we will only prove  Theorem \ref{mc.C1}.

\begin{proof}[Proof of Theorem~\ref{mc.C1}]
Given a H\"older continuous function $\phi:\bM\to \RR$, let $\cR_\phi\subset \mathscr{X}^{1,*}_\phi(\bM)$ be given by Theorem \ref{t.mc.B1}. Below we will modify each $X\in \cR_\phi$ slightly so that it has only a unique ergodic equilibrium state for the potential function $\phi.$ 

To this end, 
fix any $X\in \cR$ and denote by $\mu_{1,X},\cdots, \mu_{N,X}$ the ergodic equilibrium states of $X$ for the potential function $\phi.$, and take $C_i$ the chain recurrence class of $X$ containing $\supp\mu_{i,X}$, $U_i$ the open neighborhood of $C_i$ as before. For every $\iota>0$ sufficiently small, we can take a smooth function $\tau^\iota:\bM\to \RR^+$ with $\tau^\iota|_{C_1} = 1+\iota$, $\tau^\iota(x)\in [1,1+\iota] $ for all $x\in U_1$, and $\tau^\iota|_{U_1^c} = 1.$ Then we consider the vector field $X^\iota = \tau^\iota\cdot X$. In other words, $\tau^\iota$ slightly accelerates the flow near $C_1$ and leave the dynamics unchanged outside $U_1$. For this reason, $X^\iota$ and $X$ have the exactly same space of invariant measures. Furthermore, $Y\mapsto Y^\iota$ defines a smooth surjective map from a $C^1$ small neighborhood of $X$ to a $C^1$ small neighborhood of $X^\iota.$ This shows that the topological pressure is also continuous at $X^\iota$ provided that $\iota$ is sufficiently small, and all periodic orbits of $X^\iota$ are still pairwise homoclinically related. This means $X^\iota\in \cR_\phi$ for $\iota$ sufficiently small. 

For each $\iota>0$, $X^\iota$ has exactly one ergodic equilibrium state for the potential $\phi$, which is $\mu_{1,X}$. Furthermore, as $\iota\to 0$, $X^\iota$ converges to $X$ in $C^1$. This shows that $\cR_\phi$ has a dense subset $\cD$ on which the equilibrium state for $\phi$ is unique. Then, the same argument as in the proof of Theorem \ref{t.mc.B1} shows that there exists an open neighborhood for each $Y\in \cD$ in which the equilibrium state for $\phi$ is unique. Then, one defines $\cU$ as the union of all such open neighborhoods, which is an open and dense subset of $\mathscr{X}_\phi^{1,*}(\bM)$. This finishes the proof of Theorem \ref{mc.C1}.

\end{proof}

\subsection{On the number of chain recurrence classes with positive  entropy: proof of Theorem \ref{mc.B}} {
We state precisely the result to be proven in this section, which immediate leads to Theorem \ref{mc.B}.
\begin{theorem}\label{t.mc.B0}
	There exists a residual set $\cR_1\subset \mathfrak{X}^{1,*}_+(\bM)$ such that
	\begin{enumerate}
		\item For every $X\in\cR_1$ and every $h>0$, there are only finitely many chain recurrence classes with topological entropy at least  $h$. Denote this number by $N^{\ge h}(X)$.
		\item Given $X\in\cR_1$ and $h>0$, there exists a $C^1$ neighborhood $\cU_X\subset \mathfrak{X}^{1,*}_+(\bM)$  such that  every $Y\in \cU_X$  has only finitely many chain recurrence classes with topological entropy greater than $h$. Denoting this number by $N^{\ge h}(Y)$, we have 
		$$
		N^{\ge h}(Y) \le N^{\ge h}(X).
		$$
	\end{enumerate} 
	In both cases, each such chain recurrence class $C$ supports a unique measure of maximal entropy for $X|_C$.
\end{theorem}

The proof of this theorem requires the following lemmas.

\begin{lemma}\label{l.n.1}
	Let $X\in \mathfrak{X}^{1,*}(\bM) $. Then for every chain recurrence class $C$ of $X$, there exists an open neighborhood $U$ of $C$ and a $C^1$ open neighborhood $\cU$ of $X$, such that every $Y\in\cU$ has at least one chain recurrence class in $U$. 
\end{lemma}

\begin{proof}
	First recall that chain recurrence classes vary upper semi-continuously in the following sense: if $X_n\to X$ in $C^1$ topology and $C_n$ is a chain recurrence class of $X_n$, then ${\limsup\limits_n}\, C_n$ is a chain recurrence class of $X$ (here the limit is taken in Hausdorff topology).
	
	Now assume that $X\in \mathfrak{X}^{1,*}(\bM)$ and $C$ is a chain recurrence class of $X$. Then there are two cases: 
	
	\medskip 
	\noindent {\em Case 1.} $C$ contains no singularity. In this case, $C$ is a uniformly hyperbolic set and therefore contains at least one periodic orbit $\gamma$. Take any small neighborhood $U$ of $C$, then for every vector field $Y$ $C^1$ close to $X$, the continuation of $\gamma$, denoted by $\gamma_Y$, belongs to a chain recurrence class $C(\gamma_Y)$ which must be contained in $U$ if $Y$ is sufficiently close to $X$. 
	
	\medskip 
	\noindent {\em Case 2.} $C$ contains a singularity $\sigma$ which must be hyperbolic. In this case, the continuation of $\sigma$, denoted by $\sigma_Y$, belongs to a chain recurrence class $C(\sigma_Y)$ which must be contained in $U$ if $Y$ is sufficiently close to $X$. 
	
\end{proof}

\begin{remark}
	Note that the conclusion of the lemma holds if $U$ is replaced by any $V\subset U$ that is also an isolating neighborhood of $C$ (shrink $\cU$).
\end{remark}

The next lemma deals with the continuity of the number of chain recurrence classes for $C^1$ generic star vector fields.

\begin{lemma}\label{l.n.2}
	There exists a $C^1$ residual set $\cR_0\in \mathfrak{X}^{1,*}(\bM) $	such that for every $X\in \cR_0$ and for any isolated chain recurrence classes $C_1,\ldots, C_n$, there exist isolating neighborhoods $U_i$ of $C_i$ and a $C^1$ neighborhood $\cU$ of $X$, such that every $Y\in \cU$ has exactly one chain recurrence class $C^Y_i$ in $U_i$.
\end{lemma}

\begin{proof}
	Let $\{U_1, U_2,\ldots \}$ be a countable topological basis of $\bM$. For any finite index set $I\subset \cN$, define 
	$$
	\cU^I = \left\{f\in \mathfrak{X}^{1,*}(\bM): f \mbox{ has at least two chain recurrence classes in } \bigcup_{i\in I} U_i\right\}.
	$$

We claim that $\cU^I$ is open for every finite $I\subset \NN$.

\medskip 
\noindent 
{\em Proof of the claim.}
Let $X\in \cU^I$ and let $C_1, C_2$ be two chain recurrence classes of $X$ in $U^I: = \bigcup_{i\in I}U_i$. By Lemma \ref{l.n.1}, there exists isolating neighborhoods $U\supset C_1$, $V\supset C_2$ and open neighborhoods $\cU_1, \cU_2$ of $X$, such that every $Y\in \cU_1\cap \cU_2$ has at least one chain recurrence class in $U$, and at least one chain recurrence class in $V$. as a result, we have $\cU_1\cap \cU_2\subset \cU^I$. 

Now let $\displaystyle \cV^I = \left(\Cl\left(\cU^I \right)\right)^c$ where $\Cl$ denotes the closure in $C^1$ topology, and the closure is taken in $\mathfrak{X}^{1,*}(\bM)$. Then $\cV^I\cup \cU^I$ is open and dense in $\mathfrak{X}^{1,*}(\bM)$. Furthermore, if $X\in \cV^I$ then for every $Y$ that is $C^1$ close to $X$, $Y$ has at most one chain recurrence class in $U^I$. 

Define 
$$
\cR_0 = \bigcap_{I\subset \NN,\ |I|<\infty} \left(\cV^I\cup \cU^I\right).
$$
Then $\cR_0$ is a residual subset of $\mathfrak{X}^{1,*}(\bM)$. We claim that for every isolated chain recurrence class $C$ of $X\in\cR_0$, there exists an isolating neighborhood $U\supset C$ and a $C^1$ neighborhood $\cU$ of $X$ such that every $Y\in \cU$ has {\em exactly one} chain recurrence class in $U$.

\medskip 
\noindent {\em Proof of the claim.} Let $X\in \cR_0$ and $C$ be a chain recurrence class of $X$. Then $\{U_1,U_2,\ldots\}$ is an open covering of $C$, and we can select a finite covering $\{U_i:i\in I\subset \NN, |I|<\infty\}$. We may also assume that $U^I = \bigcup_{i\in I} U_i$ is an isolating neighborhood of $C$ and satisfies the conclusion of Lemma \ref{l.n.1}. Since $X\in \cR_0$ and has only one chain recurrence class in $U^I$, we must have $X\in \cV^I.$ Consequently, there exists a $C^1$ neighborhood $\cU^C_X$ of $X$ such that every $Y\in \cU^C_X$ has at most one chain recurrence class in $U^I$. By Lemma \ref{l.n.1}, $Y$ has exactly one chain recurrence class in $U^I$.

Now if we are given finitely many isolated chain recurrence classes of $X$, we take $\cU_X$ as the intersection of $\cU^C_X$'s, and the conclusion of the lemma follows. 

\end{proof}

\begin{remark}
	In general, systems satisfying the conclusion of Lemma \ref{l.n.2} is called {\em tame}. Such systems have been well-studied under the $C^1$ generic context; see for instance \cite{Abd}. 
\end{remark}

Finally we are ready to prove Theorem \ref{t.mc.B0}.
\begin{proof}[Proof of Theorem \ref{t.mc.B0}]
	Let $\cR$ be the residual set given by Theorem \ref{t.dichotomy} and $\cR_0$ be given by Lemma \ref{l.n.2}. Let $\cR_1 = \cR_0\cap\cR$. We will show that $\cR_1$ satisfies the conclusion of the theorem.
	
	Let $h>0$ be fixed and take $X\in\cR_1$. Assume that $X$ has distinct chain recurrence classes $C_1,C_2,\ldots$ all with topological entropy at least $h$. Taking a subsequence, we may assume that $C_n\to C$ in Hausdorff topology. Then $C$ is also a chain recurrence class with topological entropy at least $h$, due to Theorem \ref{mc.A}. By Theorem \ref{t.dichotomy}, $C$ must be isolated, which is a contradiction. This finishes the proof of (1).
	
	For (2), we let $C_1,\ldots, C_N$ be all the chain recurrence classes of $X$ with topological entropy at least $h$. Each $C_i$ is isolated, so we can take isolating neighborhoods $U_i\supset C_i$ and a $C^1$ neighborhood $\cU_X$ of $X$ given by Lemma \ref{l.n.2}. Every $Y$ in $\cU_X$ has exactly one chain recurrence class in each $U_i$. Below we will show that there exists a open set $\cU_X'\subset \cU_X$ that contains $X$, such that the number of chain recurrence classes of $Y\in \cU'_X$ with topological entropy at least $h$ is at most $N$. 
	
	Towards a contradiction, assume that there exists a sequence of $C^1$ vector fields $Y_k\to X$ in $C^1$ topology, and each $Y_k$ has at least $(N+1)$-many chain recurrence classes with topological entropy at least $h$. Denote them by $C^{Y_k}_i,i=1,\ldots,N+1$. Then, at least one of them must be outside $\bigcup_{i=1}^nU^i$. Say it is $C^{Y_k}_{N+1}$. Then by taking a subsequence if necessary, we assume that $C^{Y_k}_{N+1}\to C_0$ as $k\to\infty$ in Hausdorff topology. Then by Theorem \ref{mc.A}, $C_0$ is a chain recurrence class of $X$ with topological entropy at least $h$. However, $C_0$ is outside $\bigcup_{i=1}^n U_i$ and therefore is distinct from all $C_i,i=1,\ldots, N$, which is a contradiction.
	
	Now the last statement on the uniqueness of MME is a direct corollary of \ref{m.C}. We conclude the proof of Theorem \ref{t.mc.B0}.
	
\end{proof}

\begin{proof}[Proof of Theorem \ref{mc.B}]
	Theorem \ref{mc.B} is a corollary of Theorem \ref{t.mc.B0} by taking $\cU = \bigcup_{X\in \cR_1} \cU_X$. 
\end{proof}
}

\appendix
\section{On multi-singular hyperbolicity}\label{A.multi-sing}

\subsection{Proof of Proposition~\ref{p.equivalent}}

In this section we will prove Proposition~\ref{p.equivalent}, namely Definition~\ref{d.multising} and~\cite[Definition 1.6]{CDYZ} are equivalent. 

Let $\Lambda$ be a compact invariant set. Under the assumption that all singularities are active and hyperbolic, it is proven in~\cite[Theorem D and E]{CDYZ} that \cite[Definition 1.6]{CDYZ} is equivalent to the original definition of Bonatti and da Luz (Definition~\ref{d.multising.BD} below; also see~\cite[Definition 2]{BD21} and~\cite[Definition 5.5]{CDYZ}). Below we will introduce the original definition of Bonatti and da Luz on multi-singular hyperbolicity, and prove that it implies item (2) in Definition~\ref{d.multising}. 

Recall the extended linear Poincar\'e flow $\psi_t$ defined in Section~\ref{ss.active}. At regular points, the extended linear Poincar\'e flow can be naturally identified with the linear Poincar\'e flow, and thus is denoted by the same notation.  Also recall $\mathfrak{B}(\Lambda)$ defined in~\eqref{e.BLambda}. $\mathfrak{B}(\Lambda)\subset G^1$ is a compact set that is invariant under $(\psi_t)$.

Let $\Lambda$ be a compact invariant set of $X$. Define
$$
\mathfrak M(\Lambda) = \mathfrak{B}(\Lambda)\cup \bigcup_{\sigma\in\Sing_\Lambda(X)}G^1(\sigma).
$$
$\mathfrak M(\Lambda) $ is also compact and invariant under $(\psi_t)$. A {\em (real-valued) multiplicative cocycle over the extended linear Poincar\'e flow $(\psi_t)_{t\in\RR}$} is a continuous function $H: \mathfrak M(\Lambda) \times\RR \to  (0,+\infty)$ such that
$$
H(L, t+s) = H(L,t)\cdot H(\psi_t(L),s),\,\,\forall t,s\in\RR\mbox{ and }L\in\mathfrak M(\Lambda).
$$
Below we write  $h^t(L) = H(L,t)$. 

\begin{definition}\cite[Definition 5.3]{CDYZ}; \cite[Definition 5]{BD21}\label{d.cocycle}
	Let $\sigma$ be a hyperbolic singularity of $X$. A multiplicative cocycle $(h^t)$ over $(\psi_t)$ is a {\em local rescaling cocycle at $\sigma$}, if
	\begin{itemize}
		\item there exist a neighborhood $U_\sigma$ of $\sigma$ and a constant $C>1$, such that for any orbit segment $(x,t)$ with $x,x_t\in U_\sigma$, $L\in G^1(x)\cap\mathfrak M(\Lambda)$, it holds that 
		$$
		\frac1C \le \frac{h^t(L)}{\|Df_t\mid_{\langle X(x) \rangle}\|}\le C;
		$$
		\item for any small neighborhood $V$ of $\sigma$, there exists $C_V>0$ such that for any $(x,t)$ with $x,x_t\notin V$, $L\in G^1(x)\cap\mathfrak M(\Lambda)$, one has
		$$
		\frac1C_V\le h^t(L)\le C_V.
 		$$ 
	\end{itemize}	
\end{definition}
The existence of local rescaling cocycle is proven in~\cite[Theorem 1]{BD21}. Furthermore, it is proven that for every $\sigma$, the local rescaling cocycle is unique up to multiplication
by a cocycle bounded away from $0$ and $+\infty$.

\begin{definition}\cite[Definition 5.5]{CDYZ}; \cite[Definition 8]{BD21}\label{d.multising.BD}
A compact invariant set $\Lambda$ is multi-singular hyperbolic in the sense of Bonatti and da Luz, if:
\begin{itemize}
	\item all singularities in $\Lambda$ are hyperbolic, and we fix a local rescaling cocycle $(h^t_\sigma)$ at each $\sigma\in\Sing_\Lambda(X)$;
	\item the extended linear Poincar\'e flow admits a dominated splitting $\tilde E_N\oplus \tilde F_N$ over $\mathfrak{B}(\Lambda)$;
	\item there exists a subset $S_+\subset \Sing_\Lambda(X)$, such that for $h^t_+ : = \prod_{\sigma\in S_+} h^t_\sigma$, the cocycle $\big(h^t_+\cdot \psi_t|_{\tilde E_N}\big)_{t\in\RR}$ is uniformly contracting;
	\item there exists a subset $S_-\subset \Sing_\Lambda(X)$, such that for $h^t_- : = \prod_{\sigma\in S_-} h^t_\sigma$, the cocycle $\big(h^t_-\cdot \psi_t|_{\tilde F_N}\big)_{t\in\RR}$ is uniformly expanding.
\end{itemize}
\end{definition}

\begin{remark}
	At each regular point $x\in\Lambda$, the splitting $\tilde E_N\oplus \tilde F_N$ can be viewed as a dominated splitting on $N(x)$. This gives the singular dominated splitting $E_N\oplus F_N$ on the normal bundle $N_\Lambda$.
\end{remark}

\begin{remark}
	Under the assumption that all singularities in $\Lambda$ are active, the set $S_{\pm}$ coincides with $\Sing^\pm_\Lambda(X)$. See~\cite[Remark 5.6, Proposition 5.7]{CDYZ}.
\end{remark}

Now we are ready to prove Proposition~\ref{p.equivalent}.

\begin{proof}[Proof of Proposition~\ref{p.equivalent}]
	\noindent \item  Definition~\ref{d.multising}  (2) $\implies$ (2)': this direction is trivial.
	
	%Let $\eta$ be given by Definition~\ref{d.multising} (2), and take any compact isolating neighborhood $V$ of $\Sing_\Lambda(X)$. Then there exists $T_V>0$ such that 
	%$$
	%\prod_{i=0}^{\floor{t}-1}\left\|\psi^*_1|_{E_N(x_i)}\right\|\le \eta^{-\floor{t}}, 
	%$$
	%whenever $x,x_t\in \Lambda\cap V^c$ and $t>T_V$. Let $C_1$ be given by Lemma~\ref{l.psi.bdd} with $\tau =1$, then
	%\begin{align*}
	%\|\psi_t|_{E_N(x)}\|& = \|\psi^*_t|_{E_N(x)}\|\cdot \|Df_t|_{\langle X(x)\rangle}\|\\
	%&\le C_1\prod_{i=0}^{\floor{t}-1}\left\|\psi^*_1|_{E_N(x_i)}\right\|\cdot\frac{|X(x_t)|}{|X(x)|}\\
	%&\le C_1\sup_{y,z\in \Lambda\cap V^c} \frac{|X(y)|}{|X(z)|}\cdot \eta^{-\floor{t}}.
	%\end{align*}
	%A similar estimate hold on $F_N$ by considering $-X$. Note that points in  $\Lambda\cap V^c$ are away from all singularities; therefore the constant $C_1\sup_{y,z\in \Lambda\cap V^c} \frac{|X(y)|}{|X(z)|}$ is bounded from above. Then (2)' follows by taking any $\eta'\in (1,\eta)$ and increasing $T_V$. 

	\medskip 
	\noindent \item  \cite[Definition 1.6]{CDYZ} $\implies$ Definition~\ref{d.multising} (2): Under the assumption that all singularities in $\Lambda$ are active, \cite[Definition 1.6]{CDYZ} is equivalent to the original definition of Bonatti and da Luz. Therefore we only need to show that (2) follows from Definition~\ref{d.multising.BD}.
 	
 	Recall that $\mathfrak M(\Lambda)$ and $\mathfrak B(\Lambda)$ are both compact. By choosing an appropriate Riemannian metric (alternatively, by speeding up the flow $X$) we may assume that there exists $\eta>1$ such that 
 	$$
	\big\|h^1_+\cdot \psi_1|_{\tilde E_N}\big\|\le \eta^{-1},\,\mbox{ and }\big\|h^{-1}_-\cdot \psi_{-1}|_{\tilde F_N}\big\| \le \eta^{-1}.
 	$$
 	Then we have, for all $y\in \Reg(X)\cap\Lambda$ (where we naturally identify the extended linear Poincar\'e flow with the linear Poincar\'e flow), 
 	\begin{align*}
 		\big\|\psi^*_{-1}|_{F_N(y)}\big\| &=  \frac{\big\|\psi_{-1}|_{F_N(y)}\big\|}{\|Df_{-1}|_{\langle X(y)\rangle}\|}\\
 		&=\frac{\big\|h^{-1}_-\cdot \psi_{-1}|_{\tilde F_N(y)}\big\|}{h^{-1}_-(\langle X(y)\rangle)}\cdot \frac{|X(y)|}{|X(y_{-1})|}.		
 	\end{align*}
	Multiplying over the forward orbit of $x\in \Reg(X)\cap \Lambda$, we obtain (here w.l.o.g. we assume that $t\in\NN$; otherwise one introduces a constant $L_X = \sup\{\|DX\|\}$ and increase $T_V$):
	\begin{align*}
		\prod_{i=0}^{t-1}\left\|\psi^*_{-1}|_{F_N(x_{t-i})}\right\| &\le \prod_{i=0}^{t-1}\frac{\big\|h^{-1}_-\cdot \psi_{-1}|_{\tilde F_N(x_{t-i})}\big\|}{h^{-1}_-(\langle X(x_{t-i})\rangle)}\cdot \frac{|X(x_{t-i})|}{|X(x_{t-i-1})|}\\
		& = \frac{1}{h_-^{-t}(\langle X(x_{t})\rangle)}\frac{|X(x_t)|}{|X(x)|} \cdot \prod_{i=0}^{t-1}\big\|h^{-1}_-\cdot \psi_{-1}|_{\tilde F_N(x_{t-i})}\big\|\\
		&\le  \frac{1}{h_-^{-t}(\langle X(x_{t})\rangle)}\frac{|X(x_t)|}{|X(x)|}\cdot \eta^{-t}.
	\end{align*}
	Now let $V$ be any open neighborhood of $\Sing_\Lambda(X)$ and assume that  $x,x_t\in \Lambda\cap V^c$. Then we have 
	$$
	\sup_{y,z\in \Lambda\cap V^c}\frac{|X(y)|}{|X(z)|}<\infty,\,\, \mbox{ and } \frac{1}{h_-^{-t}(\langle X(x_{t})\rangle)} \le C_V,
	$$
	where $C_V>0$ is the constant given by Definition~\ref{d.cocycle}. Then, fix any $\eta'\in(1,\eta)$, there exists $T_V>0$ such that whenever $t>T_V$, we have 
	$$
	\prod_{i=0}^{t-1}\left\|\psi^*_{-1}|_{F_N(x_{t-i})}\right\|\le (\eta')^{-t}, 
	$$
	as desired. The same argument applies to $E_F$ by considering  $-X$. This proves \eqref{e.hyp}. 
	
	\eqref{e.hyp0} can be proven along the same lines without the telescoping factor $\frac{|X(y)|}{|X(y_{-1})|}$. With this we conclude the proof of Proposition~\ref{p.equivalent}.
\end{proof}

\subsection{Proof of Lemma \ref{l.nbhd.1}, \ref{l.robust.dom.spl} and \ref{l.robust.hyp}}
\begin{proof}[Proof of Lemma \ref{l.nbhd.1}]
	In this section we will stop using $x_t$ to denote $f_t(x)$ since we need $x_k$ to denote the $k$th coordinate of the Euclidean space $\RR^n$. 
	
	Let $\sigma^+\in\Sing_\Lambda^+(X)$. For simplicity we treat $\sigma^+$ as the origin of $\RR^n$, and assume that $E^{ss}_{\sigma^+} = \{(x_1,\ldots, x_n): x_i = 0, i\ge k\} $, $E^c_{\sigma^+} = \{(x_1,\ldots, x_n): x_i = 0, i\ne  k\} $ and $E^{uu}_{\sigma^+} = \{(x_1,\ldots, x_n): x_i = 0, i\le  k\} $.

	By Lemma \ref{l.lgw} we have that every $L\in \mathfrak B_{\sigma^+}(\Lambda)$ is contained in $E^c_{\sigma^+}\oplus E^{uu}_{\sigma^+}$. Fix any small $\alpha>0$. Then by continuity, one can find $r_0(\sigma)>0$ small enough so that for every $x = (x_1,\ldots, x_n)\in \partial B_{r_0(\sigma^+)}(\sigma^+)\cap \Lambda$, we have $\frac{|x_i|}{r_0(\sigma^+)}<\frac{\alpha}{2}$ for $i=1,\ldots, k-1.$ Note that this property also holds if $r_0(\sigma^+)$ is replaced by any smaller radius.	Similarly, one can define $r_0(\sigma^-)$ for $\sigma^-\in \Sing_\Lambda^-(X)$ by considering $-X.$
	
	Now let $r_0 = \min_{\sigma\in\Sing_\Lambda(X)} r_0(\sigma)>0$. 
%	We also define, for $\sigma^+\in\Sing_\Lambda^+(X)$:
%	$$
%	U^+(\sigma) = \{x\in B_{r_0}(\sigma^+): \}
%	$$	
	Then, one can take a small open neighborhood $\tilde U$ of $ \Lambda \cap \left(\bigcup_{\sigma\in\Sing_\Lambda(X)} B_{r_0}(\sigma)\right)^c$ so that:
	$$
	\mbox{if } y\in \tilde U\cap \partial B_{r_0}(\sigma^+) \mbox{ for some } \sigma^+\in \Sing_\Lambda^+(X), \mbox{ then }\frac{|y_i|}{r_0}< {\alpha}, i = 0,\ldots, k-1.
	$$
	Also, 
	$$
	\mbox{if } y\in \tilde U\cap \partial B_{r_0}(\sigma^-) \mbox{ for some } \sigma^-\in \Sing_\Lambda^-(X), \mbox{ then }\frac{|y_i|}{r_0}< {\alpha}, i = k+1,\ldots,n.
	$$
	In other words, orbit segments in $\cO(\tilde U)$ that start outside of $B_{r_0}(\Sing_\Lambda(X))$ can only enter and leave $B_{r_0}(\sigma^+)$ for any $\sigma^+\in\Sing_\Lambda^+(X)$ in the $\alpha-$cone of the $E^c_{\sigma^+}\oplus E^{uu}_{\sigma^+}-$plane. Similarly, it can only enter and leave 
	$B_{r_0}(\sigma^-)$ for any $\sigma^-\in\Sing_\Lambda^-(X)$ in the $\alpha-$cone of the $E^{ss}_{\sigma^-}\oplus E^{c}_{\sigma^-}-$plane.
	
	Now let $U = \tilde U\cup \bigcup_{\sigma\in\Sing_\Lambda(X)} B_{r_0}(\sigma)$ which is an open neighborhood of $\Lambda$. Since the balls $ B_{r_0}(\sigma)$ are open, orbit segments in $U$ that start outside all $B_{r_0}(\sigma)$ still have the property described above. 	
	At this point, one can repeat verbatim the proof of \cite[Lemma 2.16]{PYY23} to get that if $r<r_0$ is small enough, then any orbit in $\cO(U)$ that starts outside $B_{r_0}(\sigma^+)$ and enters $B_{r}(\sigma^+)$ must enter $B_{r_0}(\sigma^+)$ in the $\alpha-$ cone of $E^c_{\sigma^+}$, as required.
	
\end{proof}

\begin{proof}[Proof of Lemma \ref{l.robust.dom.spl}]
	Recall from \eqref{e.zeta} and \eqref{e.BLambda} that $\mathfrak B(\Lambda)\subset G^1$ is the closure of the lift of $\Lambda$ to $G^1$. Furthermore, it is proven in \cite[Section 3]{CDYZ} the singular dominated splitting $N_\Lambda = E_N \oplus F_N$ extends continuously to a dominated splitting on $\mathfrak B(\Lambda)$ for the extended linear Poincar\'e flow. Then by \cite[Appendix B]{BDV}, there exists an open neighborhood $\mathfrak U$ of $\mathfrak B(\Lambda)$ in $G^1$, such that for any orbit segments $(x,t)$ satisfying  $\zeta(x,t)\subset \mathfrak U$, there is a dominated splitting that is the continuation of $E_N \oplus F_N$. Here $\zeta:\Reg(X)\to G^1$ is the map defined by \eqref{e.zeta}.
	
	Now the main issue is that $\mathfrak U$ cannot be treated as a neighborhood of $\Lambda$ since $T_\sigma\bM\nsubseteq
	 \mathfrak U$ for $\sigma\in\Sing_\Lambda(X)$. However, by taking $\tilde U\supset \Lambda \cap \left(\bigcup_{\sigma\in\Sing_\Lambda(X)} B_{r_0}(\sigma)\right)^c$ in the proof of  Lemma \ref{l.nbhd.1} small enough, one can require that  $\tilde U \cap \Sing(X) = \emptyset$, and every $x\in \tilde U$ satisfies that $\zeta(x)\in \mathfrak U$. We claim that %, by taking $\alpha$ in Lemma  \ref{l.nbhd.1} sufficiently small, 
	every $(x,t)\in\cO(U)$ with $x,x_t\notin B_{r_0}(\Sing_\Lambda(X))$ satisfies $\zeta((x,t))\subset \mathfrak U$, and the Lemma \ref{l.robust.dom.spl} follows from the choice of $\mathfrak U$.

	Let $(x,t)\in\cO(U)$ with $x,x_t\notin B_{r_0}(\Sing_\Lambda(X))$. Write 
	$$
	0 < t_1^- < t_1^+ < \cdots < t_n^- < t_n^+ < t,
	$$
	where $ (t_i^-, t_i^+)$ corresponds to the $i$th time that the orbit segment spends in $B_{r_0}(\Sing_\Lambda(X))$. Due to our choice of $\tilde U$, we only need to show that, writing $y^i = f_{t_i^-} (x)$ and $s_i = t_i^+ - t_i^-$, we have $\zeta((y^i, s_i))\subset \mathfrak U$ for every $i$.
	
	We will only show the case $i=1$ as all subsequent cases are the same. Assume that $(y^1, s_1)\subset B_{r_0}(\sigma)$. We will also only consider $\sigma\in\Sing_\Lambda^+(X)$. The case of $\sigma\in\Sing_\Lambda^-(X)$ follows by considering $-X$. 
	
	Like before, we treat $\sigma$ as the origin of $\RR^n$. Recall from the  proof of  Lemma \ref{l.nbhd.1} that $y^1 = (y_1,\ldots, y_n)\in \partial B_{r_0}(\sigma)$ is contained in the $\alpha-$cone of $E^c_\sigma\oplus E^{uu}_\sigma;$ in other words we have $\frac{|y_i|}{r_0} <\alpha, i= 1,\ldots, k-1$ where $k-1=\dim E^{ss}_\sigma$. Define 
	$$
	C(\alpha) = \{z= (z_1,\ldots,z_n): |z|\in  (0, r_0), \frac{|z_i|}{|z|}<\alpha: i=1\ldots, k-1\}. 
	$$
	If $\alpha$ is small enough, then we have $\zeta (C(\alpha))\subset \mathfrak U$. Furthermore, since $E^{ss}_\sigma \oplus \left(E^c_\sigma\oplus E^{uu}_\sigma\right)$ is a dominated splitting, we have that $C(\alpha)$ is invariant in the sense that if $z\in C(\alpha)$ and $(z,\tau)\subset B_{r_0}(\sigma) $, then $f_{\tau} (z)\in C(\alpha)$. Since we have $y^1\in C(\alpha)$ and $(y^1,s_1)\subset  B_{r_0}(\sigma)$, it follows that $\zeta((y^1,s_1))\subset \mathfrak U$, as needed. 
\end{proof}

\begin{proof}[Proof of Lemma \ref{l.robust.hyp}]
	Recall that $\psi_t,\psi^*_t$ are continuous cocycles on $ \mathfrak B(\Lambda)$ which is a compact subset of $G^1$. Consequently, $\psi_t,\psi^*_t$ are uniformly continuous. 		
	Also note that for orbit segments in a small neighborhood of $\mathfrak B(\Lambda)$, the dominated splitting $N_{(x,t)} = E_N\oplus F_N$ (lifted to $G^1$) varies continuously with respect to the base point. Finally, we remark the local rescaling cocycles $\{h_\sigma^t:\sigma\in \Sing(X)\}$ constructed in \cite[Section 6]{BD21} are in fact defined on $\mathfrak M(\bM)$ which is
	$$
	\mathfrak M(\bM) = \Cl\Big\{ \langle X(x)\rangle: x\in\Reg(X)\Big\}\cup\bigcup_{\sigma\in\Sing(X)} \mathbb G^1(\sigma),
	$$
	and is continuous. In particular, they are well-defined in any neighborhood of $\mathfrak B(\Lambda)$, uniformly continuous, and have the cocycle property. Using Definition \ref{d.multising.BD},  we get
	$$
	\big\|h^1_+\cdot \psi_1|_{E_N}\big\|\le (\eta')^{-1},\,\mbox{ and }\big\|h^{-1}_-\cdot \psi_{-1}|_{F_N}\big\| \le (\eta')^{-1}
	$$
	in a small neighborhood $\mathfrak U$ of $\mathfrak B(\Lambda)$ for some $(\eta')\in (1,\eta)$. 
	Repeating the proof of Proposition \ref{p.equivalent}, we see that for any isolating neighborhood $W$ of $\Sing_\Lambda(X)$ and for all orbit segments $(x,t)$ satisfying $\zeta(x,t)\subset \mathfrak U$, $x,x_t\notin W$ and $t> T_W$ for some constant $T_W>0$, we have 
 	$$
 		\prod_{i=0}^{\floor{t}-1}\left\|\psi^*_1|_{E_N(x_i)}\right\|\le (\eta')^{-\floor{t}}, \mbox{ and } \prod_{i=0}^{\floor{t}-1}\big\|\psi_{-1}^*|_{F_N((x_{\floor{t}-i}))}\big\|\le (\eta')^{-t}.
 	$$	
	Same holds if $\psi^*$ is replaced by $\psi$.
	
	Now, let $(x,t)$ be an orbit segment satisfying the assumptions of Lemma \ref{l.robust.hyp}; in particular,  there exists $t_1,t_2\ge 0$ such that $(x_{-t_1}, t_1+t+t_2)\in \cO(U)$ and $x_{-t_1}\notin B_{r_0}(\Sing_\Lambda(X)), x_{t+t_2}\notin B_{r_0}(\Sing_\Lambda(X))$. Then, as in the proof of Lemma \ref{l.robust.dom.spl} we get (shrink $U$ if necessary) that $\zeta(x_{-t_1}, t_1+t+t_2)\subset \mathfrak U$, meaning that $\zeta(x,t)\subset \mathfrak U$, and the conclusion of the lemma follows. 
\end{proof}

\section{Proof of Lemma \ref{l.measure.intersect}}\label{s.A2}

The proof of Lemma \ref{l.measure.intersect} calls for the following shadowing lemma, which is a modification of Liao and Gan's Shadowing Lemma\cite{Liao, Gan02}. See also Katok's Shadowing Lemma~\cite{Katok80} and its adaptation to $C^{1+\alpha}$ singular flows in a recent work \cite{LLL}.

\begin{lemma}\label{l.Liao.shadowing}
	Let $X$ be a $C^1$ vector field and $\Lambda$ a compact invariant set on which there is  a dominated splitting $N_\Lambda = E_N\oplus F_N$ for the linear Poincar\'e flow. For any compact yet not necessarily invariant set $\Lambda_0$ with $\Lambda_0 \cap \Sing(X) = \emptyset$ and $\lambda>1$, there exists $\tilde\delta_0>0$ and $L>0$, such that for any $\varepsilon>0$ sufficiently small, there exist $\tilde\delta>0$ such that for any orbit segment $(x,T)$ with the following properties:
	\begin{enumerate}
		\item $x,x_T\in \Lambda_0$, and $d(x,x_T)<\tilde \delta$;
		\item $x$ is a $(\lambda,E)$-forward infinite hyperbolic time for $(\psi_t^*)$;
		\item $x_T$ is a $(\lambda,F)$-backward infinite hyperbolic time for $(\psi_t^*)$;
	\end{enumerate}
	then there exist a point $p$ and a $C^1$ strictly increasing function $\theta:[0,T] \to \RR$, such that
	\begin{enumerate}[label=(\alph*)]
		\item $\theta(0)=0$ and $|\theta'(t)-1|<\varepsilon$; furthermore, there exists a constant $C>0$ independent of $T$, such that $|\theta(T) - T|\le C\vep$.
		\item $p$ is a hyperbolic periodic point with period $\theta(T)$;
		\item $p$ has stable and unstable manifold with size at least $\tilde\delta_0$.
		\item $p$ is $\vep$-scaled shadowed by the orbit of $x$ up to time $T$;
		\item $d(x_t, p_{\theta(t)})\le Ld(x,x_T)$;
		\item for $\iota_x: = 2Ld(x,x_T)$, the unstable manifold of $p$ of size $\iota_x$ has non-empty transversal intersection with the stable manifold of $x$ of size $\iota_x$, and the stable manifold of $p$ of size $\iota_x$ has non-empty transversal intersection with the unstable manifold of $x_T$ of size $\iota_x$.
	\end{enumerate}
	%Furthermore, the result remains true with the same constants $\delta$, $L$ and $\delta_0>0$ if $\Lambda_0$ is replaced by a subset of  $\Lambda_0$.
\end{lemma}  

The version cited here can be found in~\cite[Section 4.2, Lemma 4.5]{PYY}; see \cite{Liao81} and \cite{Gan02, LLL} for the proof.

\begin{proof}[Proof of Lemma~\ref{l.measure.intersect}]
	Let $\mu$ be an ergodic invariant regular measure. We only consider the case $\mu(\Lambda^*_W(\lambda_0))\ne 0$ for $* = E,F$, otherwise one take $\tilde \Lambda^*_W(\lambda_0) = \emptyset$ and there is nothing to prove. Since $\mu$ is ergodic, there exists a time $\tau_0\ge 0$ such that $f_{\tau_0}(\Lambda^F_W)\cap \Lambda^E_W$ has positive measure.  Denote the closure of this set by $\tilde \Lambda_W$ and note that $\tilde \Lambda_W$ contains no singularity. 
	
	Fix any $\lambda>1$ and let $(y,t)\in\Lambda\times\RR^+$  be such that $y, z := f_t(y)\in f_{\tau_0}(\Lambda^F_W)\cap \Lambda^E_W$. Then, treating it as an orbit segment of the accelerated vector field $Y = \tau \cdot X$ for some $\tau \gg \tau_0$ depending $\lambda$ but not on $(y,t)$, we see that $y$ is a $(\lambda,E)$-forward infinite hyperbolic time and $z$ is a $(\lambda,E)$-backward infinite hyperbolic time for the vector field $Y$. Applying Lemma \ref{l.Liao.shadowing} to the vector field $Y$ and the compact set $\tilde \Lambda_W$, one obtains $L>1$ and for any $\vep>0$ one obtains $\tilde \delta >0$. Note that $L$ does not depend on $\vep$ or $\tilde\delta$.
	
	Now let $y\in f_{\tau_0}(\Lambda^F_W)\cap \Lambda^E_W$ be a Birkhoff typical point of $\mu$. Then by Poincar\'e Recurrence Theorem, one can find $t_y>0$ such that $f^Y_{t_y}(y) \in B_{\tilde \delta}(y)\cap f_{\tau_0}(\Lambda^F_W)\cap \Lambda^E_W$; here  $(f^Y_{t})$ denote the flow for the vector field $Y$. Lemma \ref{l.Liao.shadowing} provides a hyperbolic periodic orbit $\gamma$ of $Y$ (and therefore of $X$) such that for $\iota_y : = 2Ld(y,f^Y_{t_y}(y)) < 2L\tilde\delta$,  the unstable manifold of $\gamma$ with size $\iota_y$ has non-empty transversal intersection with the stable manifold of $y$ with size $\iota_y$. In other words, the conclusion of Lemma \ref{l.measure.intersect} holds for points in a full measure subset of $f_{\tau_0}(\Lambda^F_W)\cap \Lambda^E_W \subset \Lambda^E_W.$ Also note that $\tilde \delta \to 0$ as $\vep \to 0$, so consequently $\iota_y$ can be made arbitrarily small (although $\gamma$ depends on $\vep$). 
	
	Let $\delta>0$ be fixed. Now we show that the same conclusion holds for a full measure subset of $\Lambda^E_W$ with the point of intersection taken inside $W_{\delta, \cN}^s(x)$. Let $x\in\Lambda^E_W$ be a Birkhoff typical point of $\mu$. Then there exists $t_x > 0$ such that $y: =f_{t_x}(x) \in f_{\tau_0}(\Lambda^F_W)\cap \Lambda^E_W $. By taking $\vep$ small enough, one can require $\iota_y$ to be so small such that $W_{\iota_y,\cN}^s(y)\subset \cP_{t_x,x}\left( W_{\delta, \cN}^s(x)\right)$. Here $\cP_{t_x,x}\left( W_{\delta, \cN}^s(x)\right)$ is well-defined because $x$ is a forward infinite hyperbolic time, and therefore every point in the stable manifold of $x$ is $\rho_0$-scaled shadowed by the orbit of $x$ for all $t>0$ (Lemma \ref{l.hyp.fakeEleaf}). The previous argument shows that there exists a hyperbolic periodic orbit $\gamma$ such that $W_{\iota_y,\cN}^s(y)\pitchfork W^u(\gamma)\ne\emptyset.$ By invariance, we get $W_{\delta,\cN}^s(x)\pitchfork W^u(\gamma)\ne\emptyset.$

\end{proof}

\end{document}